\documentclass{amsart}

\usepackage{amssymb}
\usepackage[all]{xy}
\usepackage{mathrsfs}
\usepackage{upgreek}
\usepackage{stmaryrd}
\usepackage[colorlinks = true]{hyperref}
\usepackage{rotating}
\usepackage{tikz}

\newcommand{\bk}{\Bbbk}
\newcommand{\Z}{\mathbb{Z}}
\newcommand{\C}{\mathbb{C}}

\newcommand{\Q}{\mathbb{Q}}
\newcommand{\F}{\mathbb{F}}
\newcommand{\scO}{\mathscr{O}}
\newcommand{\scK}{\mathscr{K}}
\newcommand{\Gm}{\mathbb{G}_{\mathrm{m}}}

\newcommand{\sA}{\mathsf{A}}
\newcommand{\sB}{\mathsf{B}}
\newcommand{\sC}{\mathsf{C}}
\newcommand{\sD}{\mathsf{D}}
\newcommand{\sE}{\mathsf{E}}
\newcommand{\sa}{\mathsf{a}}

\newcommand{\spp}{\mathsf{p}}
\newcommand{\snn}{\mathsf{n}}
\newcommand{\smm}{\mathsf{m}}
\newcommand{\bR}{\mathsf{R}}
\newcommand{\Rn}{\mathsf{Rn}}

\newcommand{\fg}{\mathfrak{g}}
\newcommand{\fb}{\mathfrak{b}}
\newcommand{\ft}{\mathfrak{t}}
\newcommand{\fn}{\mathfrak{n}}
\newcommand{\fm}{\mathfrak{m}}
\newcommand{\fp}{\mathfrak{p}}
\newcommand{\fnt}{\dot{\mathfrak{n}}}
\newcommand{\fN}{\mathfrak{N}}
\newcommand{\fZ}{\mathfrak{Z}}

\newcommand{\weyl}{\mathsf{M}}
\newcommand{\coweyl}{\mathsf{N}}
\newcommand{\tilt}{\mathsf{T}}
\newcommand{\irr}{\mathsf{L}}
\newcommand{\St}{\mathrm{St}}

\newcommand{\Fr}{\mathrm{Fr}}
\newcommand{\Dist}{\mathrm{Dist}}
\newcommand{\tcN}{{\widetilde{\mathcal{N}}}}
\newcommand{\se}{\mathsf{e}}

\newcommand{\bX}{\mathbf{X}}

\newcommand{\aff}{{\mathrm{aff}}}
\newcommand{\bXpp}{\bX^{+,\mathrm{reg}}}
\newcommand{\dom}{\mathsf{dom}}
\newcommand{\convo}{\mathsf{conv}^0}
\newcommand{\conv}{\mathsf{conv}}
\newcommand{\Waff}{W_{\mathrm{aff}}}
\newcommand{\WaffCox}{W_{\mathrm{aff}}^{\mathrm{Cox}}}
\newcommand{\Waffmin}{{}^0 \hspace{-1pt} \Waff}
\newcommand{\WaffminI}{{}^0 \hspace{-1pt} \Waff^I}

\newcommand{\cF}{\mathcal{F}}
\newcommand{\cE}{\mathcal{E}}
\newcommand{\cG}{\mathcal{G}}
\newcommand{\cL}{\mathcal{L}}
\newcommand{\cO}{\mathcal{O}}

\newcommand{\IC}{\mathcal{IC}}
\newcommand{\bQ}{\mathbf{Q}}
\newcommand{\cJ}{\mathcal{J}}
\newcommand{\cT}{\mathcal{T}}

\newcommand{\Gr}{\mathrm{Gr}}
\newcommand{\Fl}{\mathrm{Fl}}
\newcommand{\sph}{{\mathrm{sph}}}
\newcommand{\Sat}{\mathcal{S}}
\newcommand{\Iw}{\mathrm{Iw}}

\newcommand{\bL}{{\mathsf{\Lambda}}}
\newcommand{\bS}{\mathbf{S}}
\newcommand{\Koszul}{\mathsf{Koszul}}

\newcommand{\cC}{\mathcal{C}}
\newcommand{\lmod}{\text{-}\mathsf{mod}}
\newcommand{\ldgmod}{\text{-}\mathsf{dgmod}}
\newcommand{\mix}{{\mathrm{mix}}}
\newcommand{\Perv}{\mathsf{Perv}}
\newcommand{\Par}{\mathsf{Parity}}
\newcommand{\PerPar}{\mathsf{PervParity}}
\newcommand{\Rep}{\mathsf{Rep}}
\newcommand{\Repd}{\mathsf{Rep}^{\mathrm{disc}}}
\newcommand{\Repf}{\mathsf{Rep}^{\mathrm{f}}}
\newcommand{\Vect}{\mathsf{Vect}}
\newcommand{\Coh}{\mathsf{Coh}}
\newcommand{\QCoh}{\mathsf{QCoh}}

\newcommand{\Db}{D^{\mathrm{b}}}
\newcommand{\Kb}{K^{\mathrm{b}}}
\newcommand{\Dfg}{D^\mathrm{fg}}
\newcommand{\Dmix}{D^\mix}
\newcommand{\Stein}{\mathrm{Stein}}

\DeclareMathOperator{\Hom}{Hom}
\DeclareMathOperator{\Ext}{Ext}
\DeclareMathOperator{\End}{End}
\DeclareMathOperator{\Sym}{Sym}
\DeclareMathOperator{\Ind}{Ind}
\DeclareMathOperator{\For}{For}
\DeclareMathOperator{\pr}{pr}
\DeclareMathOperator{\incl}{in}
\DeclareMathOperator{\inc}{inc}

\newcommand{\id}{\mathrm{id}}
\newcommand{\simto}{\xrightarrow{\sim}}
\newcommand{\la}{\langle}
\newcommand{\ra}{\rangle}
\newcommand{\aq}{\sslash}
\newcommand{\rcap}{{\stackrel{_R}{\cap}}}
\newcommand{\coH}{\mathsf{H}}
\newcommand{\qis}{\mathrm{qis}}

\newcommand{\rmZ}{\mathrm{Z}}

\newcommand{\bK}{\mathbb{K}}
\newcommand{\bO}{\mathbb{O}}
\newcommand{\bE}{\mathbb{E}}

\makeatletter
\def\lotimes{\@ifnextchar_{\@lotimessub}{\@lotimesnosub}}
\def\@lotimessub_#1{\mathchoice{\mathbin{\mathop{\otimes}^L}_{#1}}%
  {\otimes^L_{#1}}{\otimes^L_{#1}}{\otimes^L_{#1}}}
\def\@lotimesnosub{\mathbin{\mathop{\otimes}^L}}
\makeatother

\numberwithin{equation}{section}
\newtheorem{thm}{Theorem}[section]
\newtheorem{lem}[thm]{Lemma}
\newtheorem{prop}[thm]{Proposition}
\newtheorem{cor}[thm]{Corollary}
\newtheorem{conj}[thm]{Conjecture}
\theoremstyle{definition}
\newtheorem{defn}[thm]{Definition}

\theoremstyle{remark}
\newtheorem{rmk}[thm]{Remark}

\title[Reductive groups, loop Grassmannian, Springer resolution]{Reductive groups, the loop Grassmannian, and the Springer resolution}

\author{Pramod N. Achar}
\address{Department of Mathematics\\
  Louisiana State University\\
  Baton Rouge, LA 70803\\
  U.S.A.}
\email{pramod@math.lsu.edu}

\author{Simon Riche}
\address{Universit\'e Clermont Auvergne, CNRS, LMBP, F-63000 Clermont-Ferrand, France.
}
\email{simon.riche@uca.fr}

\thanks{P.A. was supported by NSA Grant No.~H98230-15-1-0175 and NSF Grant No.~DMS-1500890.  S.R. was partially supported by ANR Grant 
No.~ANR-13-BS01-0001-01. This project has received
funding from the European Research Council (ERC) under the European Union's Horizon 2020
research and innovation programme (grant agreement No 677147).
}

\begin{document}

\begin{abstract}
In this paper we prove equivalences of categories relating the derived category of a block of the category of representations of a connected reductive algebraic group over an algebraically closed field of characteristic $p$ bigger than the Coxeter number and a derived category of equivariant coherent sheaves on the Springer resolution (or a parabolic counterpart). In the case of the principal block, combined with previous results, this provides a modular version of celebrated constructions due to Arkhipov--Bezrukavnikov--Ginzburg for Lusztig's quantum groups at a root of unity. As an application, we prove a ``graded version'' of a conjecture of Finkelberg--Mirkovi{\'c} describing the principal block in terms of mixed perverse sheaves on the dual affine Grassmannian, and deduce a new proof of Lusztig's conjecture in large characteristic.
\end{abstract}

\maketitle

\tableofcontents


\section{Introduction}

\subsection{Main players}
\label{ss:intro-main-players}

Let $G$ be a connected reductive algebraic group over an algebraically closed field $\bk$ of characteristic $\ell$, and let $T \subset B \subset G$ be a maximal torus and a Borel subgroup.  Assume that $\ell > h$, where $h$ is the Coxeter number of $G$, and that the derived subgroup of $G$ is simply connected. Under these assumptions, most of the combinatorial data for the category $\Repf(G)$ of finite-dimensional algebraic $G$-modules (in particular, characters of simple and indecomposable tilting modules) can be deduced from the corresponding data in the ``principal block'' $\Rep_\varnothing(G)$,  i.e.~the Serre subcategory generated by the simple modules whose highest weight has the form $w (\rho)-\rho + \ell \lambda$ for $\lambda \in X^*(T)$ and $w \in W=N_G(T)/T$. (Here, as usual $\rho$ is the half sum of positive roots.)

In the hope of computing these data, it has long been desired to have a ``geometric model'' for this category, in the spirit of what is known for representations of complex semisimple Lie algebras~\cite{bb,bk}, affine Kac--Moody Lie algebras~\cite{kt}, quantum groups at a root of unity~\cite{abg:qglg}, and reductive Lie algebras in positive characteristic~\cite{bmr1, bmr2, bm}. The main goal of the present paper is to provide such a model. 

More precisely, let $\dot G$ denote the Frobenius twist of $G$, and let $\dot G^\vee$ be the complex connected reductive group whose root datum is dual to that of $\dot G$.  (Thus, the coweight lattice for $\dot G^\vee$ is identified with $\ell X^*(T)$.)  This paper is concerned with the categories and functors in the following diagram:
\begin{equation}
\label{eqn:main-players}
\xymatrix@C=27pt{
\Dmix_{(\Iw)}(\Gr,\bk) \ar[r]_-P^-{\sim} \ar@/^6ex/[rrr]^-{\text{graded Finkelberg--Mirkovi\'c conjecture}}_-{\bQ}&
\Db\Coh^{\dot G \times \Gm}(\tcN) \ar[r]_-{F} &
\Db_{\Stein}(B) \ar[r]_-{R\Ind_B^G}^-{\sim} &
\Db\Rep_\varnothing(G).
}
\end{equation}
Here, $\Gr$ is the affine Grassmannian for $\dot G^\vee$, $\Iw \subset \dot G^\vee(\C[ \hspace{-1pt} [z] \hspace{-1pt} ])$ is an Iwahori subgroup, and $\Dmix_{(\Iw)}(\Gr,\bk)$ is the mixed derived category of $\bk$-sheaves on $\Gr$ which are constructible with respect to the stratification by $\Iw$-orbits (in the sense of~\cite{modrap2}). Next, $\tcN$ is the Springer resolution for $\dot G$, with its natural action of $\dot G \times \Gm$, and $\Db_{\Stein}(B)$ is the derived category of complexes of $B$-modules whose cohomology is trivial on the first Frobenius kernel $B_1 \subset B$.  

The functor $P$ in~\eqref{eqn:main-players} is an equivalence of triangulated categories that was established by the first author and L.~Rider (see~\cite{ar:agsr}) and by C.~Mautner and the second author (see~\cite{mr:etsps}) independently.  The other two functors in this diagram are the topics of two of the main results in this paper.  The \emph{formality theorem} asserts that $\Db\Coh^{\dot G \times \Gm}(\tcN)$ is a graded version of $\Db_{\Stein}(B)$, and the \emph{induction theorem} asserts that $R\Ind_B^G : \Db_{\Stein}(B) \to \Db\Rep_\varnothing(G)$ is an equivalence of categories.  In the last section of the paper, we will study the composition $\bQ := R\Ind_B^G{} \circ F \circ P$, and we will prove a graded analogue of the Finkelberg--Mirkovi\'c conjecture~\cite{fm:sif1}, describing $\Rep_\varnothing(G)$ in terms of $\Perv^\mix_{(\Iw)}(\Gr,\bk)$.

Statements analogous to those above were established by Arkhipov--Bezrukav\-nikov--Ginzburg~\cite{abg:qglg} for quantum groups at a root of unity.  Their work has significant consequences for the representation theory of quantum groups: they lead to alternative proofs of Lusztig's character formula for simple modules (see~\cite[\S1.2]{abg:qglg}) and of Soergel's character formula for tilting modules (using~\cite{yun}).  We believe that the results of the present paper will likewise have consequences for the representation theory of $G$. In particular, we expect to use them to establish the character formulas for simple and tilting $G$-modules conjectured by the second author and G.~Williamson in~\cite{rw}. See~\S\ref{ss:characters} below for details.

\subsection{Statements and strategy}
\label{ss:statements}

Let us now state our results more precisely.

The diagram~\eqref{eqn:main-players} is inspired by the ideas in~\cite{abg:qglg}, but the proofs in this paper are quite different from those in~\cite{abg:qglg}.  In particular, a central theme of this paper is the importance of ``wall-crossing functors.'' Most of the categories and functors in~\eqref{eqn:main-players} have analogues associated to parabolic subgroups. When we construct the various functors in~\eqref{eqn:main-players}, we will simultaneously construct their parabolic analogues, and we will construct commutative diagrams that relate the Borel version to a parabolic version (or two parabolic versions to each other).  Wall-crossing functors play an essential role in the argument, even if one is interested only in the Borel versions of the theorems, because they let us reduce difficult calculations (in, say, $\Db_\Stein(B)$) to easier cases. 

At several points, we will need the notion of a degrading functor. Let $\cC$ and $\cC'$ be triangulated categories, and suppose $\cC$ is equipped with an autoequivalence $\{ 1\}: \cC \to \cC$.  A triangulated functor $\varphi: \cC \to \cC'$ is called a \emph{degrading functor} (with respect to $\{1\}$) if (i)~its image generates $\cC'$ as a triangulated category, and (ii)~there is a natural isomorphism $\varphi \cong \varphi \circ \{1\}$ that induces, for any $X, Y \in \cC$, an isomorphism
\[
\bigoplus_{n \in \Z} \Hom_{\cC} (X,Y\{n\}) \simto \Hom_{\cC'}(\varphi X, \varphi Y).
\]

Let $S$ be the set of simple reflections in the Weyl group $W$ of $(G,T)$.  For any subset $I \subset S$, we let $P_I \subset G$ be the corresponding standard parabolic subgroup, $\dot P_I$ be its Frobenius twist, and $\fnt_I$ be the Lie algebra of the unipotent radical of $\dot P_I$. The Frobenius morphism of $P_I$ will be denoted $\Fr$, and for $V \in \Rep(\dot P_I)$ we will denote by $\Fr^*(V)$ the $P_I$-module obtained from $V$ by composition with $\Fr$. We will use similar notation for other groups below.

Let $\Db_{\Stein}(P_I)$ be the full triangulated subcategory of the derived category $\Db \Repf(P_I)$ of finite-dimensional algebraic $P_I$-modules generated by the objects of the form $\St_I \otimes \Fr^*(V)$ for $V$ in $\Repf(\dot P_I)$.  Here, $\St_I$ is a fixed Steinberg module for $P_I$, i.e.,~the module $\Ind_B^{P_I}((\ell-1)\varsigma_I)$, where $\varsigma_I$ is a fixed character of $T$ such that for any simple coroot $\alpha^\vee$, we have
\[
\la \alpha^\vee,\varsigma_I\ra =
\begin{cases}
1 & \text{if $\alpha^\vee$ corresponds to a reflection $s_\alpha \in I$,} \\
0 & \text{if $\alpha^\vee$ corresponds to a reflection $s_\alpha \notin I$.}
\end{cases}
\]
Finally, let $\tcN_I := \dot G \times^{\dot P_I} \fnt_I$.  For a $\dot G \times \Gm$-equivariant coherent sheaf $\cF$ on $\tcN_I$, let $\cF\la 1\ra$ be the sheaf obtained by twisting the $\Gm$-action.  The following statement combines parts of Theorems~\ref{thm:formality}, \ref{thm:parabolicJ}, and~\ref{thm:parabolicI}.

\begin{thm}[Formality theorem]\label{thm:intro-formality}
For any subset $I \subset S$, there is a functor 
\[
F_I: \Db\Coh^{\dot G \times \Gm}(\tcN_I) \to \Db_\Stein(P_I)
\]
that is a degrading functor with respect to $\la 1\ra[1]$ and such that for any $V \in \Repf(\dot G)$, there is a natural isomorphism
\[
F_I(\cF \otimes V) \cong F_I(\cF) \otimes \Fr^*(V).
\]
If $J \subset I \subset S$ (so that $P_J \subset P_I$), there is a commutative diagram
\begin{equation}\label{eqn:intro-formality-comm}
\vcenter{\xymatrix@C=2cm{
 \Db\Coh^{\dot G \times \Gm}(\tcN_J) \ar[r]^-{F_J} \ar[d]_-{\Pi_{J,I}} & \Db_{\Stein}(P_J)\ar[d]|{R\Ind_{P_J}^{P_I} \bigl( (-) \otimes \bk(\varsigma_J - \varsigma_I) \bigr)} \\
 \Db\Coh^{\dot G \times \Gm}(\tcN_I) \ar[r]^-{F_I}Ê& \Db_{\Stein}(P_I).
}}
\end{equation}
\end{thm}

In this commutative diagram, $\Pi_{J,I}$ is a functor that is defined using the intermediate space $\tcN_{J,I} := \dot G \times^{\dot P_J} \fnt_I$ and the correspondence
\[
\tcN_J \leftarrow \tcN_{J,I} \to \tcN_I;
\]
see~\S\ref{ss:Ind-Res-Springer} for details.

Next, let $\Rep_I(G)$ be the Serre subcategory of $\Repf(G)$ generated by the simple modules whose highest weight has the form $w(\rho-\varsigma_I) -\rho + \ell \lambda$ with $\lambda \in X^*(T)$. This subcategory is a direct summand of $\Repf(G)$, and it ``has singularity $I$'' in the sense that the stabilizer of $-\varsigma_I$ for the dot-action of the affine Weyl group is the parabolic subgroup $W_I$ of $W$ generated by $I$. In particular, when $I=\varnothing$, $\Rep_\varnothing(G)$ is a sum of regular blocks of $\Repf(G)$. If $J \subset I \subset S$, then we have a natural translation functor $T_J^I : \Rep_J(G) \to \Rep_I(G)$.

The following statement combines parts of Lemma~\ref{lem:theta-phi-iso} (see also Proposition~\ref{prop:Theta-psi}) and Theorem~\ref{thm:induction-thm}.

\begin{thm}[Induction theorem]\label{thm:intro-induction}
For any subset $I \subset S$, the functor
\begin{equation}\label{eqn:intro-induction-equiv}
R\Ind_{P_I}^G : \Db_\Stein(P_I) \to \Db\Rep_I(G)
\end{equation}
is an equivalence of categories.  Moreover, for any $V \in \Repf(\dot G)$, there is a natural isomorphism
\[
R\Ind_{P_I}^G(M \otimes \Fr^*(V)) \cong R\Ind_{P_I}^G(M) \otimes \Fr^*(V).
\]
If $J \subset I \subset S$, there is a commutative diagram
\begin{equation}\label{eqn:intro-induction-comm}
\vcenter{\xymatrix@C=2cm{
 \Db_\Stein(P_J) \ar[r]^-{R\Ind_{P_J}^G} \ar[d]|{R\Ind_{P_J}^{P_I} \bigl( (-) \otimes \bk(\varsigma_J - \varsigma_I) \bigr)} & \Db\Rep_J(G) \ar[d]^{T_J^I} \\
 \Db_\Stein(P_I) \ar[r]^-{R\Ind_{P_I}^G}Ê& \Db\Rep_I(G).
}}
\end{equation}
\end{thm}

\begin{rmk}
T.~Hodge, P.~Karuppuchamy and L.~Scott have obtained a different proof that~\eqref{eqn:intro-induction-equiv} is an equivalence in the case $I=\varnothing$, see~\cite{hks}. Their proof
is closer to the proof of the quantum case in~\cite{abg:qglg}. (It does not directly apply to other parabolic subgroups, as far as we understand.) 
\end{rmk}

Combining Theorems~\ref{thm:intro-formality} and~\ref{thm:intro-induction} with the main results of~\cite{ar:agsr, mr:etsps}, one sees immediately that the functor $\bQ := R\Ind_B^G {}\circ F \circ P$ is a degrading functor.  We will discuss further properties of $\bQ$ in \S\ref{ss:graded-fm-intro}.

\subsection{Koszul duality and the formality theorem}
\label{ss:koszul-formality}

We now discuss in more detail the ingredients in the proof of Theorem~\ref{thm:intro-formality}.  Given a subset $I \subset S$, consider the exterior algebra
\[
\bL_I:=\bigwedge \hspace{-3pt} {}^\bullet \, \fnt_I,
\]
regarded as a dg-algebra with trivial differential and with $\fnt_I$ placed in degree $-1$. For any subset $J \subset I$, the group $\dot P_J$ acts on $\bL_I$.  Let $\Dfg_{\dot P_J}(\bL_I)$ be the derived category of $\dot P_J$-equivariant $\bL_I$-dg-modules with finitely generated cohomology.  

The proof of Theorem~\ref{thm:intro-formality} involves breaking up the commutative diagram into subdiagrams as shown in Figure~\ref{fig:intro-formality-proof}. In the middle row, $P_JM_{I,1}$ is the (scheme-theoretic) preimage of $\dot P_J$ under the Frobenius morphism $\Fr : P_I \to \dot P_I$. (The notation will be explained in~\S\ref{ss:statement-equiv-formality}.)

\begin{figure}
\[
\vcenter{\xymatrix@C=1.5cm{
\Db\Coh^{\dot G \times \Gm}(\tcN_J) \ar[d] \ar[r]^-{\varkappa_J} \ar@/_10ex/[dd]_(.3){\Pi_{J,I}} \ar@/^5ex/[rr]^-{F_J} \ar@{}[dr]|{\S\ref{sec:koszul};\S\ref{sec:exotic}} &
  \Dfg_{\dot P_J}(\bL_J) \ar[d] \ar[r]^-{\psi_J}_-{\S\ref{sec:formality-non-equiv};\S\ref{sec:equiv-formality}} \ar@/_5ex/[dd]_(.2){\Theta_{J,I}}|\hole \ar@{}[dr]|{\S\ref{sec:compatibility}} &
  \Db_\Stein(P_J) \ar[d]|{\hphantom{\scriptstyle \bigl( (-) \otimes \bk(\varsigma_J - \varsigma_I) \bigr)}R\Ind_{P_J}^{P_JM_{I,1}} \bigl( (-) \otimes \bk(\varsigma_J - \varsigma_I) \bigr)} \\
\Db\Coh^{\dot G \times \Gm}(\tcN_{J,I}) \ar[d] \ar[r]^-{\varkappa_{J,I}}  \ar@{}[dr]|{\S\ref{sec:koszul};\S\ref{sec:exotic}} &
  \Dfg_{\dot P_J}(\bL_I) \ar[d] \ar[r]^-{\psi_{J,I}}  \ar@{}[dr]|{\S\ref{sec:compatibility}} &
  \Db_\Stein(P_JM_{I,1}) \ar[d]|{R\Ind_{P_JM_{I,1}}^{P_I}} \\
\Db\Coh^{\dot G \times \Gm}(\tcN_I) \ar[r]^-{\varkappa_I} \ar@/_5ex/[rr]_-{F_I} &
  \Dfg_{\dot P_I}(\bL_I) \ar[r]^-{\psi_I}_-{\S\ref{sec:formality-non-equiv};\S\ref{sec:equiv-formality}} &
  \Db_\Stein(P_I)}}
\]
\caption{Setup for the proof of Theorem~\ref{thm:intro-formality}}\label{fig:intro-formality-proof}
\end{figure}

The left half of Figure~\ref{fig:intro-formality-proof} is essentially a study of Koszul duality.  Recall that $\Coh^{\dot G \times \Gm}(\tcN_I)$ is equivalent to $\Coh^{\dot P_I \times \Gm}(\fnt_I)$.  The latter is, in turn, identified with the category of finitely-generated graded $\dot P_I$-equivariant modules over the symmetric algebra $\bS_I:=\Sym(\fnt_I^*)$.  The functor $\varkappa_I$ and its variants are degrading functors that are close to the well-known Koszul duality relating $\bS_I$ to $\bL_I$, see~\cite{bgs, gkm}. The appropriate theory, including the commutativity of the squares in the left half of the figure, is developed in Sections~\ref{sec:koszul} and~\ref{sec:exotic}, building on~\cite{gkm,mr2}.

The right half of Figure~\ref{fig:intro-formality-proof} involves the study of a certain dg-algebra $\Rn_I$.  This algebra is equipped with a homomorphism $\sigma_I: \Rn_I \to \bL_I$, as well as a quasi-isomorphism $\pi_I$ to the distribution algebra of the first Frobenius kernel $N_{I,1}$ of $N_I$. We can therefore consider the composition
\begin{equation}\label{eqn:intro-formality-non-equiv}
\Dfg(\bL_I) \xrightarrow{\sigma_I^*} \Dfg(\Rn_I) \xrightarrow[\sim]{\pi_{I*}} \Db \Repf(N_{I,1}).
\end{equation}
We will build the right half of the figure in three steps.  First, in Section~\ref{sec:formality-non-equiv}, we use the functors in~\eqref{eqn:intro-formality-non-equiv} to construct an equivalence of categories
\begin{equation}
\label{eqn:equiv-phi-intro}
\varphi_I : \Dfg(\bL_I) \simto \Db_{\Stein}(P_{I,1}),
\end{equation}
where the right-hand is the subcategory of the bounded derived category of finite-dimensional representations of the Frobenius kernel $P_{I,1}$ of $P_I$ generated by $\St_I$.
Next, in Section~\ref{sec:equiv-formality}, we study the action of $\dot P_I$ or $\dot P_J$ on the various algebras in~\eqref{eqn:intro-formality-non-equiv} in order to construct $\psi_I$ and show that it is an equivalence.  Finally, the commutativity of the two squares in the right half of Figure~\ref{fig:intro-formality-proof} is shown in Section~\ref{sec:compatibility}.

\begin{rmk}
Let us briefly explain the origin of the name ``formality theorem'' for Theorem~\ref{thm:intro-formality} (which we took from~\cite{abg:qglg}). For simplicity we restrict to the case $I=\varnothing$. In this case, a well-known result due to Friedlander--Parshall~\cite{fp:claag} asserts that there exists a graded algebra isomorphism
\begin{equation}
\label{eqn:cohomology-B1}
\Ext_{B_1}^\bullet(\bk,\bk) \cong \Sym(\fnt_\varnothing^*),
\end{equation}
where in the right-hand side $\fnt_\varnothing^*$ is placed in degree $2$. On the other hand, it follows from abstract nonsense that the category $\Db_{\Stein}(B_1)$ can be described in terms of dg-modules over the dg-algebra $R\Hom_{B_1}(\bk, \bk)$. In view of~\eqref{eqn:cohomology-B1}, if we could prove that this dg-algebra is formal (i.e.~quasi-isomorphic to its cohomology), then this would prove that $\Db_{\Stein}(B_1)$ can be described in terms of the dg-algebra $\Sym(\fnt_\varnothing^*)$ (with trivial differential). Combining this with some form of Koszul duality would provide an approach to proving equivalence~\eqref{eqn:equiv-phi-intro}. In practice, however this is \emph{not} the way we construct this equivalence, and in fact we will not prove the formality of any dg-algebra.
\end{rmk}

\subsection{Exotic sheaves and the induction theorem}
\label{ss:exotic-induction}

We saw in~\S\ref{ss:koszul-formality} that in the proof of Theorem~\ref{thm:intro-formality}, the proof that $F_I$ is a degrading functor is quite separate from the proof that~\eqref{eqn:intro-formality-comm} commutes.  In contrast, for Theorem~\ref{thm:intro-induction}, the commutativity of~\eqref{eqn:intro-induction-comm} must be established first. This plays an essential role in the proof that $R\Ind_{P_I}^G: \Db_\Stein(P_I) \to \Db\Rep_I(G)$ is an equivalence.

The commutativity of~\eqref{eqn:intro-induction-comm} is established in Section~\ref{sec:translation}, as part of a larger effort concerned with the diagram in Figure~\ref{fig:intro-induction-proof}.  This figure also depicts the left adjoints of $\Pi_{\varnothing, I}$, $\Theta_{\varnothing,I}$, and $T_\varnothing^I$. 

\begin{figure}
\[
\vcenter{
\xymatrix{
\Db\Coh^{\dot G \times \Gm}(\tcN_\varnothing) \ar[r]^-{\varkappa_\varnothing} \ar@<1ex>[d]^{\Pi_{\varnothing,I}} \ar@/^5ex/[rr]^-{F_\varnothing} \ar@{}[dr]|{\S\ref{sec:exotic}} & 
  \Dfg_{\dot B}(\bL_\varnothing) \ar[r]^-{\psi_\varnothing} \ar@<1ex>[d]^{\Theta_{\varnothing,I}} \ar@{}[drr]|{\S\ref{sec:translation}} &
  \Db_{\Stein}(B) \ar[r]^-{R\Ind_B^G} & 
  \Db\Rep_\varnothing(G) \ar@<1ex>[d]^{T_\varnothing^I} \\
\Db\Coh^{\dot G \times \Gm}(\tcN_I) \ar[r]^-{\varkappa_I} \ar@<1ex>[u]^{\Pi^{\varnothing,I}\la -n_I\ra[-n_I]} \ar@/_5ex/[rr]_-{F_I} &
  \Dfg_{\dot P_I}(\bL_I) \ar@<1ex>[u]^{\Theta^{\varnothing,I}}\ar[r]^-{\psi_I} &
  \Db_\Stein(P_I) \ar[r]^-{R\Ind_{P_I}^G} &
  \Db\Rep_{I}(G). \ar@<1ex>[u]^{T_I^\varnothing}
}
}
\]
\caption{Setup for the proof of Theorem~\ref{thm:intro-induction}}\label{fig:intro-induction-proof}
\end{figure}

The main result of Section~\ref{sec:translation} asserts, in addition to the commutativity of~\eqref{eqn:intro-induction-comm}, that when $\#I=1$,\footnote{A posteriori, this assumption can be removed; see Remark~\ref{rmk:thm-translation}.} the middle and rightmost parts of Figure~\ref{fig:intro-induction-proof} form a \emph{commutative diagram of adjoint pairs}.  This means that there is a pair of natural isomorphisms that intertwine the units (or the counits) for the adjoint pairs $(\Theta^{\varnothing,I}, \Theta_{\varnothing,I})$ and $(T_I^\varnothing, T_\varnothing^I)$.  Similarly, we will show in Section~\ref{sec:exotic} that the leftmost square in that figure is a commutative diagram of adjoint pairs.

Let us now return to the problem of showing that $R\Ind_{P_I}^G: \Db_{\Stein}(P_I) \to \Db\Rep_I(G)$ is an equivalence. It is easy to see that the essential image of this functor generates $\Db\Rep_I(G)$ as a triangulated category, so it is enough to show that it is fully faithful. If we had a rich enough supply of objects in $\Db_{\Stein}(P_I)$ whose $\Ext$-groups and images under $R\Ind_{P_I}^G$ were understood, we could try to prove full faithfulness by direct calculation. Unfortunately, it is unclear (at least to us) how to produce such objects in $\Db_{\Stein}(P_I)$.\footnote{In the case $I=\varnothing$, the proof in~\cite{abg:qglg} essentially proceeds in this way, but it turns out that the ``direct calculation'' is not so easy.}

Figure~\ref{fig:intro-induction-proof} suggests looking instead at $\Db\Coh^{\dot G \times \Gm}(\tcN_I)$.  Note that, since we already know that $F_I$ is a degrading functor, $R\Ind_{P_I}^G$ is fully-faithful if and only if $R\Ind_{P_I}^G {} \circ F_I$ is a degrading functor.  Moreover, in the special case $I = \varnothing$, there is a rich supply of objects with favorable $\Ext$-properties in $\Db\Coh^{\dot G \times \Gm}(\tcN_\varnothing)$: namely, the standard and costandard objects in the heart of the \emph{exotic t-structure}, which has been introduced by Bezrukavnikov~\cite{bez:ctm} and studied further in~\cite{ar:agsr, mr:etspc}.  Using the special case $I = \{s\}$ in Figure~\ref{fig:intro-induction-proof}, we prove in Section~\ref{sec:induction-thm} that $R\Ind_B^G {}\circ F_\varnothing$ takes standard (resp.~costandard) exotic sheaves to Weyl (resp.~dual Weyl) modules.  That gets us most of the way to finishing the proof of Theorem~\ref{thm:intro-induction} (in the case $I=\varnothing$).

For general $I$, we introduce some ``parabolic analogues'' of the standard and costandard exotic sheaves, and study how they behave under the functors $\Pi_{\varnothing,I}$ and $\Pi^{\varnothing,I}$. Using the case $I=\varnothing$, in this case also we prove that the functor $R\Ind_{P_I}^G{} \circ F_I$ takes standard (resp.~costandard) exotic sheaves to Weyl (resp.~dual Weyl) modules, and we finish the proof as before.
(These parabolic exotic sheaves might be of independent interest. In particular they allow one to define an ``exotic t-structure'' on $\Db\Coh^{\dot G \times \Gm}(\tcN_I)$, which might have other applications.)

\subsection{The graded Finkelberg--Mirkovi\'c conjecture}
\label{ss:graded-fm-intro}

Recall that $\dot G^\vee$ is the complex connected reductive group which is Langlands-dual to $\dot G$, and that $\Gr=\dot G^\vee(\C( \hspace{-1pt} (z) \hspace{-1pt} ))/\dot G^\vee(\C[ \hspace{-1pt} [z] \hspace{-1pt} ])$ is its affine Grassmannian. Let $\Perv_\sph(\Gr,\bk)$ be the abelian category of $\dot G^\vee(\C[ \hspace{-1pt} [z] \hspace{-1pt} ])$-equivariant $\bk$-perverse sheaves on $\Gr$. This category admits a natural convolution product $\star$, and the celebrated \emph{geometric Satake equivalence}, due in this setting to Mirkovi{\'c}--Vilonen~\cite{mv:gld}, asserts that there exists an equivalence of monoidal categories
\[
\Sat : (\Perv_{\sph}(\Gr, \bk), \star) \simto (\Repf(\dot G), \otimes).
\]
The category $\Repf(\dot G)$ embeds naturally in the category $\Rep_\varnothing(G)$ via the functor $V \mapsto \Fr^*(V)$. On the other hand, $\Perv_{\sph}(\Gr, \bk)$ embeds in the category $\Perv_{(\Iw)}(\Gr,\bk)$ of $\bk$-perverse sheaves on $\Gr$ which are constructible with respect to the $\Iw$-orbits (where $\Iw$ is an Iwahori subgroup, as in~\S\ref{ss:intro-main-players}).
The Finkelberg--Mirkovi{\'c} conjecture~\cite{fm:sif1}
predicts that the equivalence $\Sat$ can be ``extended'' to an equivalence of highest-weight categories
\[
Q : \Perv_{(\Iw)}(\Gr,\bk) \simto \Rep_\varnothing(G)
\]
which satisfies $Q(\cF \star \cG) \cong Q(\cF) \otimes \Fr^*(\Sat(\cG))$ for any $\cF$ in $\Perv_{(\Iw)}(\Gr,\bk)$ and $\cG$ in $\Perv_{\sph}(\Gr, \bk)$. (Here, $\star$ also denotes the natural convolution action of $\Perv_{\sph}(\Gr, \bk)$ on $\Perv_{(\Iw)}(\Gr,\bk)$.)

As an application of our constructions, we prove a ``graded version'' of this conjecture. Namely, consider the abelian category $\Perv^\mix_{(\Iw)}(\Gr, \bk)$ of mixed $\bk$-perverse sheaves on $\Gr$ which are constructible with respect to the $\Iw$-orbits, in the sense of~\cite{modrap2}, and let $\langle 1 \rangle$ be its ``Tate twist'' autoequivalence.  This category is a graded highest weight category in a natural way. Moreover there exists a natural action of $\Perv_{\sph}(\Gr, \bk)$ on $\Dmix_{(\Iw)}(\Gr, \bk)$ (induced by convolution), see~\S\ref{ss:geometric-Satake}, and we prove in Proposition~\ref{prop:convolution-exact} that this action restricts to an action on $\Perv^\mix_{(\Iw)}(\Gr, \bk)$.

\begin{thm}[Graded Finkelberg--Mirkovi\'c conjecture]
\label{thm:intro-fm-conjecture}
There is an exact functor
\[
\bQ : \Perv^\mix_{(\Iw)}(\Gr, \bk) \to \Rep_\varnothing(G)
\]
with the following properties:
\begin{enumerate}
\item
the functor $\bQ$ sends standard, costandard, simple, and indecomposable tilting objects in $\Perv^\mix_{(\Iw)}(\Gr,\bk)$ to standard, costandard, simple, and indecomposable tilting objects in $\Rep_\varnothing(G)$ respectively;
\item
there is an isomorphism $\varepsilon : \bQ \circ \langle 1 \rangle \simto \bQ$ that induces, for any $\cF,\cG$ in $\Perv^\mix_{(\Iw)}(\Gr, \bk)$ and any $k \in \Z$, an isomorphism
\[
\bigoplus_{n \in \Z} \Ext^k_{\Perv^\mix_{(\Iw)}(\Gr, \bk)}(\cF,\cG \langle n \rangle) \simto \Ext^k_{\Rep_\varnothing(G)}(\bQ(\cF), \bQ(\cG));
\]
\item
there exists a functorial isomorphism
\[
\bQ(\cF \star \cG) \cong \bQ(\cF) \otimes \Fr^*(\Sat(\cG))
\]
for any $\cF$ in $\Perv^\mix_{(\Iw)}(\Gr, \bk)$ and $\cG$ in $\Perv_{\sph}(\Gr, \bk)$.
\end{enumerate}
\end{thm}

As in~\eqref{eqn:main-players}, we define $\bQ$ to be the composition $R\Ind_B^G {}\circ F_\varnothing \circ P$. Then parts~(2)  and~(3) follow quite easily from Theorems~\ref{thm:intro-formality} and~\ref{thm:intro-induction}, combined with the main result of~\cite{ar:agsr,mr:etspc}. (Part~(2) is essentially a restatement of the fact that $\bQ$ is a degrading functor with respect to the Tate twist.)  The papers~\cite{ar:agsr,mr:etspc} also tell us how $P$ interacts with exotic sheaves on $\tcN$.  Combining this with the study of exotic sheaves in the proof of Theorem~\ref{thm:intro-induction} leads to a proof of t-exactness for $\bQ$, and of part~(1) of the theorem above.

\begin{rmk}
The natural analogue of the Finkelberg--Mirkovi\'c conjecture in the setting of quantum groups at a root of unity is proved by Arkhipov--Bezrukavnikov--Ginzburg in~\cite{abg:qglg}, using their versions of Theorems~\ref{thm:intro-formality},~\ref{thm:intro-induction}, and of the results of~\cite{ar:agsr,mr:etspc}. However, since they do not consider the role of the exotic t-structure in this picture, they have to work harder to prove the exactness of their version of our functor $\bQ$; see~\cite[\S 9.10]{abg:qglg}.
\end{rmk}

\subsection{Relationship with the Bezrukavnikov--Mirkovi{\'c}--Rumynin theory of localization in positive characteristic}
\label{ss:relation-BMR}

The papers~\cite{bmr1, bmr2, bm} build a ``localization theory'' for modules over the enveloping algebra $\mathcal{U}(\fg)$ of the Lie algebra $\fg$ of $G$; in other words they provide a ``geometric model'' for the representation theory of this algebra. Building on these results, in~\cite{riche} the second author has obtained a geometric model for the representation theory of the restricted enveloping algebra $\mathsf{g}$ of $\fg$, i.e.~the quotient of $\mathcal{U}(\fg)$ by the trivial character of the Frobenius center (or equivalently the distribution algebra of the Frobenius kernel $G_1$). In this subsection we briefly explain the (philosophical) relation between our results and those of~\cite{bmr1, bmr2, bm, riche}.

Let as above $I \subset S$ be a subset, and consider the category $\mathsf{g}\lmod^{\mathrm{fg}}_I$ of finite-dimensional $\mathsf{g}$-modules with generalized Harish-Chandra character $-\varsigma_I$. (This category would be denoted $\mathrm{Mod}^{\mathrm{fg}}_{-\varsigma_I}((\mathcal{U} \fg)_0)$ in the conventions of~\cite[\S 3.2]{riche}.) Consider also the Grothendieck resolution
\[
\widetilde{\fg}_I := {\dot G} \times^{\dot P_I} {\dot \fp}_I,
\]
where $\dot \fp_I$ is the Lie algebra of $\dot P_I$.
Then by~\cite[Theorem~3.4.14]{riche} there exists an equivalence of triangulated categories
\begin{equation}
\label{eqn:equiv-Ug}
\mathsf{DGCoh} \bigl( \widetilde{\fg}_I \, \rcap_{\dot \fg \times \dot G / \dot P_I} \, \dot G / \dot P_I \bigr) \simto \Db \bigl( \mathsf{g}\lmod^{\mathrm{fg}}_I \bigr),
\end{equation}
where the left-hand side is the (derived) category of coherent dg-sheaves on the dg-scheme obtained as the derived intersection of $\widetilde{\fg}_I$ and the zero-section $\dot G / \dot P_I$ in $\dot \fg \times \dot G / \dot P_I$; see~\cite[\S 1.8]{riche} for details on this construction.

A construction similar to that of the functor $\varkappa_I$ in~\S\ref{ss:koszul-formality} (involving Koszul duality) provides a functor
\[
\overline{\varkappa}_I : \Db \Coh^{\Gm}(\tcN_I) \to \mathsf{DGCoh} \bigl( \widetilde{\fg}_I \, \rcap_{\dot \fg \times \dot G / \dot P_I} \, \dot G / \dot P_I \bigr)
\]
with properties similar to those of $\varkappa_I$, see~\cite{riche}. Composing this functor with~\eqref{eqn:equiv-Ug} we obtain a functor
\begin{equation}
\label{eqn:equiv-Ug-2}
\Db \Coh^{\Gm}(\tcN_I) \to \Db \bigl( \mathsf{g}\lmod^{\mathrm{fg}}_I \bigr)
\end{equation}
which is a degrading functor.

Now we have a natural forgetful functor
\[
\Db \Coh^{\dot G \times \Gm}(\tcN_I) \to \Db \Coh^{\Gm}(\tcN_I),
\]
and differentiation of the $G$-action provides a natural functor
\[
\Db \Rep_I(G) \to \Db \bigl( \mathsf{g}\lmod^{\mathrm{fg}}_I \bigr).
\]
It is reasonable to expect that the following diagram is commutative:
\begin{equation}
\label{eqn:diagram-localization}
\vcenter{
\xymatrix@C=2.5cm{
\Db \Coh^{\dot G \times \Gm}(\tcN_I) \ar[d] \ar[r]^-{R\Ind_{P_I}^G \circ F_I} & \Db \Rep_I(G) \ar[d] \\
\Db \Coh^{\Gm}(\tcN_I) \ar[r]^-{\eqref{eqn:equiv-Ug-2}} & \Db \bigl( \mathsf{g}\lmod^{\mathrm{fg}}_I \bigr).
}
}
\end{equation}
This would explain the relationship between the results of the present paper and localization theory.

We will not attempt to prove the commutativity of~\eqref{eqn:diagram-localization}. One difficulty in trying to prove such a relationship is that the construction of the equivalence~\eqref{eqn:equiv-Ug} depends on the choice of a ``splitting bundle'' for some Azumaya algebra; in order to prove some compatibility result we would most likely have in particular to understand this choice better, and see how one can choose the bundle in a more canonical way.

\subsection{Application: a character formula for tilting modules}
\label{ss:characters}


The results of this paper open the way to geometric approaches to various deep problems in the representation theory of $G$, either via constructible sheaves or via coherent sheaves.

First, in~\cite{rw}, the second author and G.~Williamson conjecture that the multiplicities of standard/costandard modules in indecomposable tilting modules in $\Rep_\varnothing(G)$ can be expressed in terms of the values at $1$ of some $\ell$-Kazhdan--Lusztig polynomials (in the sense of~\cite{jw}), which compute the dimensions of the stalks of some indecomposable parity complexes on the affine flag variety $\Fl$ of $\dot G^\vee$. This conjecture is proved in the case $G=\mathrm{GL}_n(\bk)$ in~\cite{rw}, but the methods used in this proof do not make sense for a general reductive group. 

Note that, as was noticed by Andersen, from the characters of indecomposable tilting $G$-modules one can deduce (at least if $\ell \geq 2h-2$) character formulas for simple $G$-modules, see~\cite[\S 1.8]{rw}; hence the conjectural tilting character formula provides a replacement for Lusztig's conjecture~\cite{lusztig}, which was recently shown to be false for some values of $\ell$, see~\cite{williamson}.

Theorem~\ref{thm:intro-fm-conjecture} is a first step towards a proof of this character formula valid for any reductive group. Namely, this result reduces the computation of multiplicities of tilting objects in $\Rep_\varnothing(G)$ to the similar problem in $\Perv^\mix_{(\Iw)}(\Gr, \bk)$. In a work in progress with S.~Makisumi and G.~Williamson we develop a modular analogue of the (geometric) Koszul duality for Kac--Moody groups of Bezrukavnikov--Yun~\cite{by}, and deduce in particular an equivalence of graded additive categories between the category of tilting objects in $\Perv^\mix_{(\Iw)}(\Gr, \bk)$ and the category of Iwahori--Whittaker parity complexes on $\Fl$ as considered in~\cite[\S 11.7]{rw}, see~\cite{amrw}. Since the combinatorics of the latter category is known to be governed by the appropriate $\ell$-Kazhdan--Lusztig polynomials (see~\cite[Theorem~11.13]{rw}), this implies the conjectural character formula for tilting $G$-modules of~\cite{rw}.

In a different direction, in a joint work with W.~Hardesty~\cite{ahr} we use the relation between the category $\Rep_\varnothing(G)$ and exotic coherent sheaves on $\tcN_\varnothing$ to obtain first results towards a proof of a conjecture of Humphreys~\cite{humphreys-conj} on support varieties of tilting $G$-modules.


We conclude this paper with a direct application of our results to characters of simple $G$-modules, independent of~\cite{amrw}. In particular, we give a new proof of Lusztig's conjecture~\cite{lusztig} for $\ell$ large (with no explicit bound), as already proved by Andersen--Jantzen--Soergel~\cite{ajs} (building on work of Kazhdan--Lusztig~\cite{kl}, Lusztig~\cite{LUSMon} and Kashiwara--Tanisaki~\cite{kt}), Fiebig~\cite{fiebig} and Bezrukavnikov--Mirkovi\'c~\cite{bm} (as part of a broader picture). But we obtain slightly more than what was known until now: 
\begin{enumerate}
\item
a geometric character formula valid for all simple modules in the principal block and in all characteristics $\ell>h$ (in terms of mixed intersection cohomology complexes on $\Gr$), see Proposition~\ref{prop:ch-simples};
\item
and an equivalence between the validity of Lusztig's conjecture and parity-vanishing properties of some (ordinary) intersection cohomology complexes on $\Gr$, see Theorem~\ref{thm:criterion-lusztig}.
\end{enumerate}


\subsection{Acknowledgments}

This paper began as a joint project with Ivan Mirkovi\'c. We thank him for his encouragement, and inspiring discussions at early stages of our work. As should be clear already, this paper owes much to the ideas of Bezrukavnikov and his collaborators, in particular those of~\cite{abg:qglg}. We also thank Geordie Williamson for stimulating discussions. Finally, we thank Terrell~Hodge, Paramasamy~Karuppuchamy and Leonard~Scott for keeping us informed of their progress on~\cite{hks}.

\subsection{Contents}

This paper is divided into 3 parts, which each begin with an overview of their content. Part~\ref{pt:prelim} is devoted to preliminaries. Part~\ref{pt:formality} is concerned with the proof of the formality theorem. Finally, Part~\ref{pt:induction} is devoted to the proof of the induction theorem and of the graded analogue of the Finkelberg--Mirkovi\'c conjecture.

\newpage

\part{Preliminary results}
\label{pt:prelim}

\textbf{Overview.}
Section~\ref{sec:dgalg} contains background material on (module categories for) dg-algebras equipped with actions of algebraic groups. In Section~\ref{sec:reductive}, we fix notation and conventions for reductive groups and related objects. We also prove a number of lemmas on the behavior of Steinberg modules for Levi subgroups under various functors. These modules play an important role in Part~\ref{pt:formality}.  Finally, in Section~\ref{sec:koszul}, we study some version of the familiar Koszul duality for symmetric and exterior algebras on a vector space equipped with a group action. In particular, we show that Koszul duality is compatible (in a suitable sense) with a change of vector space.

\section{Dg-algebras and dg-modules}
\label{sec:dgalg}

Throughout this section, we let $\bk$ be a field. 

\subsection{Dg-modules}
\label{ss:notation}

If $\sA$ is a ring, we denote by $\sA\lmod$ the abelian category of $\sA$-modules. If $\sA$ is a dg-algebra, we denote by $\sA\ldgmod$ the category of (left) $\sA$-dg-modules, and by $D(\sA)$ the corresponding derived category. If the cohomology algebra $\coH^\bullet(\sA)$ is left Noetherian, we denote by $\Dfg(\sA) \subset D(\sA)$ the full subcategory of differential graded modules whose cohomology is finitely generated over $\coH^\bullet(\sA)$.

Let $f: \sA \to \sB$ be a homomorphism of dg-algebras.  We denote by
\[
f^*: \sB\ldgmod \to \sA\ldgmod
\]
the functor that regards a $\sB$-module as an $\sA$-module via $f$. This functor is exact, and we denote similarly the induced functor from $D(\sB)$ to $D(\sA)$.

The functor $f^*$
has a right adjoint
\[
f_*: \sA\ldgmod \to \sB\ldgmod \qquad\text{given by}\qquad f_*(M) = \Hom^\bullet_{\sA}(\sB,M),
\]
where the $\sB$-module structure is induced by right multiplication of $\sB$ on itself.
(The functor $f^*$ also has a left adjoint $M \mapsto \sB \otimes_{\sA} M$, but we will not use any special notation for this functor.)
It is well known that, if $\sA$ is concentrated in nonpositive degrees (i.e.~if $\sA^i=0$ for $i>0$), then the category $\sA\ldgmod$ has enough K-injective objects (see~\cite[Proposition~3.11]{spaltenstein} for the simpler case of modules over a ring, or~\cite[Theorem~1.3.6]{riche} for the more complicated case of sheaves of dg-modules); therefore the functor $f_*$ admits a right derived functor
\[
Rf_* : D(\sA) \to D(\sB).
\]
Arguments similar to those in~\cite{spaltenstein} or~\cite{riche} show that $Rf_*$ is right adjoint to $f^*$. Also, if $f : \sA \to \sB$ and $g : \sB \to \sC$ are morphisms of dg-algebras concentrated in nonpositive degrees, then we have a canonical isomorphism
\[
f^* \circ g^* \cong (g \circ f)^*.
\]
By adjunction we deduce an isomorphism
\begin{equation}
\label{eqn:transitivity-direct-image}
R(g \circ f)_* \cong Rg_* \circ Rf_*.
\end{equation}

\subsection{Normal subalgebras and quotients}
\label{ss:normal-subalg}

Let $\sA$ be a $\bk$-dg-algebra concentrated in nonpositive degrees and endowed with a counit $\varepsilon: \sA \to \bk$ (assumed to be a morphism of complexes), and let $\sA_+=\ker(\varepsilon)$ be the augmentation ideal. Let $\sa \subset \sA$ be a normal dg-subalgebra, i.e., a dg-subalgebra with the property that $\sA \cdot (\sa \cap \sA_+) = (\sa \cap \sA_+) \cdot \sA$.  Let $\sA \aq \sa := \sA/\sA\cdot (\sa \cap \sA_+)$. For any $\sA$-dg-module $M$, we consider the complex
\[
\Hom^\bullet_{\sa}(\bk, M),
\]
where $\bk$ is considered as an $\sa$-dg-module via the restriction of $\varepsilon$. This complex identifies with the sub-$\sA$-dg-module of $M$ consisting of elements $m \in M$ satisfying $a \cdot m = \varepsilon(a) m$ for all $a \in \sa$. In particular, it has a natural structure of $\sA \aq \sa$-dg-module. The assignment $M \mapsto \Hom^\bullet_{\sa}(\bk,M)$ defines a functor from the category of $\sA$-dg-modules to the category of $\sA \aq \sa$-dg-modules; we denote its right derived functor by
\[
R\Hom^\bullet_{\sa}(\bk, -) : D(\sA) \to D(\sA \aq \sa).
\]
(This functor can be computed by means of K-injective resolutions.) 

If $p : \sA \to \sA \aq \sa$ is the natural surjection, then we have a natural isomorphism of functors $p_* \cong \Hom^\bullet_{\sa}(\bk, -)$; we deduce a canonical isomorphism
\begin{equation}
\label{eqn:direct-image-normal-subalg}
Rp_* \cong R\Hom^\bullet_{\sa}(\bk, -).
\end{equation}
The following lemma justifies our choice of a special notation for this functor. (In practice we will always work under the assumption of this lemma; otherwise the notation might be misleading.)

\begin{lem}
Assume that $\sA$ is K-flat as a right $\sa$-dg-module, and consider the embedding $i : \sa \to \sA$.
For any $M$ in $D(\sA)$, the image in $D(\bk)$ of the  $\sA \aq \sa$-dg-module $R\Hom^\bullet_{\sa}(\bk, M)$ coincides with the complex $R\Hom^\bullet_{\sa}(\bk, i^* M)$.
\end{lem}

\begin{proof}
The claim follows from the fact that, under our assumption, if $M$ is a K-injective $\sA$-dg-module then $i^* M$ is also K-injective (as an $\sa$-dg-module), since the functor $\sA \otimes_{\sa} (-)$ sends acyclic dg-modules to acyclic dg-modules.
\end{proof}

One can restate the fact that the functor $p^*$ is left adjoint to $Rp_*$ by saying that there exists a functorial isomorphism
\begin{equation}
\label{eqn:RHom-normal-subalg}
\Hom_{D(\sA \aq \sa)} \bigl( M, R\Hom^\bullet_{\sa}(\bk, N) \bigr) \cong \Hom_{D(\sA)}( M,N)
\end{equation}
for any $M$ in $D(\sA \aq \sa)$ and any $N$ in $D(\sA)$ (where we omit the functor $p^*$ in the right-hand side).

\subsection{Semidirect products}
\label{ss:semi-direct-products}

Let $\sD$ be a Hopf algebra over $\bk$, and let $\sA$ be a $\bk$-dg-algebra that is also a $\sD$-module in such a way that 
\begin{itemize}
\item
the differential of $\sA$ commutes with the $\sD$-action;
\item
$d \cdot 1=\varepsilon(d) \cdot 1$ for any $d \in \sD$;
\item
the multiplication map $\sA \otimes \sA \to \sA$ is a homomorphism of $\sD$-modules.
\end{itemize}
One can then form the \emph{semidirect product} or \emph{crossed product} $\sA \rtimes \sD$, namely the dg-algebra which coincides with $\sA \otimes \sD$ as a complex of $\bk$-vector spaces (where $\sD$ is considered as a complex concentrated in degree $0$, with trivial differential), and with multiplication given by
\[
(a \rtimes d) \cdot (b \rtimes e) = \sum a(d_{(1)}\cdot b) \rtimes d_{(2)} e.
\]
Here we are using Sweedler's notation, with $\Delta(d) = \sum d_{(1)} \otimes d_{(2)}$.

Consider now two Hopf algebras $\sD$ and $\sE$ over $\bk$ and a $\bk$-linear morphism of Hopf algebras $\varphi : \sD \to \sE$. Let $\sA$ and $\sB$ be $\bk$-dg-algebras endowed with actions of $\sE$ as above, and $f : \sA \to \sB$ be a $\bk$-linear morphism of dg-algebras which commutes with the $\sE$-actions. Then one can consider the commutative square
\[
\xymatrix@C=1.5cm{
\sA \rtimes \sD \ar[r]^-{f \rtimes \id_{\sD}} \ar[d]_-{\id_{\sA} \rtimes \varphi} & \sB \rtimes \sD \ar[d]^-{\id_{\sB} \rtimes \varphi} \\
\sA \rtimes \sE \ar[r]^-{f \rtimes \id_{\sE}} & \sB \rtimes \sE
}
\]
of dg-algebras and morphisms of dg-algebras.

\begin{lem}
\label{lem:base-change}
Consider the setting described above, and assume that $\sA$ and $\sB$ are concentrated in nonpositive degrees. Then there exists an isomorphism of functors
\[
(f \rtimes \id_{\sE})^* \circ R(\id_{\sB} \rtimes \varphi)_* \cong R(\id_{\sA} \rtimes \varphi)_* \circ (f \rtimes \id_{\sD})^*.
\]
\end{lem}

\begin{proof}
Adjunction and isomorphism~\eqref{eqn:transitivity-direct-image} provide a morphism of functors
\begin{multline*}
R(\id_{\sB} \rtimes \varphi)_* \to R(\id_{\sB} \rtimes \varphi)_* R(f \rtimes \id_{\sD})_* (f \rtimes \id_{\sD})^* \cong \\
R\bigl( (f \rtimes \id_{\sE}) \circ (\id_{\sA} \rtimes \varphi) \bigr)_* (f \rtimes \id_{\sD})^* \cong R(f \rtimes \id_{\sE})_* R(\id_{\sA} \rtimes \varphi)_* (f \rtimes \id_{\sD})^*,
\end{multline*}
and using adjunction again we deduce a natural morphism of functors 
\begin{equation}
\label{eqn:base-change-morph}
(f \rtimes \id_{\sE})^* \circ R(\id_{\sB} \rtimes \varphi)_* \to R(\id_{\sA} \rtimes \varphi)_* \circ (f \rtimes \id_{\sD})^*.
\end{equation}
To prove that the latter morphism is invertible, we observe that the algebras $\sA \rtimes \sD$ and $\sB \rtimes \sD$ are K-flat as complexes of right $\sD$-modules (for the action induced by right multiplication of $\sD$ on itself). Moreover, there exist canonical isomorphisms of $\sA \rtimes \sD$-modules and $\sB \rtimes \sD$-modules respectively
\[
(\sA \rtimes \sD) \otimes_{\sD} \sE \cong \sA \rtimes \sE, \qquad (\sB \rtimes \sD) \otimes_{\sD} \sE \cong \sB \rtimes \sE.
\]
We deduce, for $M$ in $D(\sB \rtimes \sD)$, functorial isomorphisms in $D(\Bbbk)$:
\begin{multline*}
R\Hom^\bullet_{\sB \rtimes \sD}(\sB \rtimes \sE,M) \cong R\Hom^\bullet_{\sB \rtimes \sD}((\sB \rtimes \sD) \otimes_{\sD} \sE,M) \cong R\Hom^\bullet_{\sD}(\sE,M) \\
\cong R\Hom^\bullet_{\sA \rtimes \sD}((\sA \rtimes \sD) \otimes_{\sD} \sE,M) \cong R\Hom^\bullet_{\sA \rtimes \sD}(\sA \rtimes \sE,M).
\end{multline*}
It is easily checked that this isomorphism is induced by~\eqref{eqn:base-change-morph}, and the lemma is proved.
\end{proof}

\subsection{Induction}
\label{ss:induction}

For any affine $\bk$-group scheme $H$, we denote by $\Rep(H)$ the abelian category of (not necessarily finite-dimensional) algebraic $H$-modules, and by $\Repf(H) \subset \Rep(H)$ the subcategory consisting of finite-dimensional modules. If $\lambda : H \to \bk^\times$ is a character of $H$, we denote by $\bk_H(\lambda)$ the corresponding $1$-dimensional $H$-module. (When $\lambda$ is the trivial character, we abbreviate the notation to $\bk$.)

If $H$ and $K$ are affine $\bk$-group schemes and $\varphi : H \to K$ is a morphism of group schemes, we can consider the \emph{induction functor}
\[
\Ind_H^K : \Rep(H) \to \Rep(K)
\]
defined by $\Ind_H^K(V) = (V \otimes \mathcal{O}(K))^H$, where $\mathcal{O}(K)$ is considered as a $K \times H$-module via the action induced by
\[
(k,h) \cdot g = kg \varphi(h)^{-1} \quad \text{for $g \in K$ and $(k,h) \in K \times H$.}
\]
Note that we allow $\varphi$ to be any morphism, not necessarily an embedding of a closed subgroup (as e.g.~in~\cite{jantzen}). The functor $\Ind_H^K$ is right adjoint to the forgetful functor
\[
\For_H^K : \Rep(K) \to \Rep(H);
\]
in particular it takes injective objects to injective objects.

The (left exact) functor $\Ind_H^K$ admits a right derived functor
\[
R\Ind_H^K : D^+ \Rep(H) \to D^+ \Rep(K),
\]
which can be computed using injective resolutions, and which is right adjoint to the functor $\For_H^K : D^+ \Rep(K) \to D^+ \Rep(H)$.

This construction is transitive in the sense that if $\varphi : H \to K$ and $\psi : K \to I$ are morphisms of affine $\bk$-group schemes, then we have canonical isomorphisms of functors
\begin{equation}
\label{eqn:transitivity-RInd}
\For^K_H \circ \For^I_K \cong \For^I_H, \qquad R\Ind_K^I \circ R\Ind_H^K \cong R\Ind_H^I,
\end{equation}
where the functors $\For^I_H$ and $R\Ind_H^I$ are defined with respect to the morphism $\psi \circ \varphi : H \to I$. (In fact, the first isomorphism is obvious, and the second one follows by adjunction.)

Later on we will need the following technical lemma. Consider as above a morphism of (affine) $\bk$-group schemes $\varphi : H \to K$, and let $H' \subset H$, $K' \subset K$ be closed subgroups such that $\varphi(H') \subset K'$. Then we can consider the diagram
\[
\xymatrix@C=2cm{
D^+ \Rep(H) \ar[r]^-{R\Ind_H^K} \ar[d]_-{\For_{H'}^H} & D^+ \Rep(K) \ar[d]^-{\For_{K'}^K} \\
D^+ \Rep(H') \ar[r]^-{R\Ind_{H'}^{K'}} & D^+ \Rep(K').
}
\]

\begin{lem}
\label{lem:Ind-For}
Assume that:
\begin{enumerate}
\item
\label{it:Ind-For-ass-1}
the morphism
\[
H \times^{H'} K' \to K : [h:k] \mapsto k \varphi(h)^{-1}
\]
is an isomorphism;
\item
$H'$ is a finite group scheme.
\end{enumerate}
Then there exists a canonical isomorphism of functors
\[
\For_{K'}^K \circ R\Ind_H^K \cong R\Ind_{H'}^{K'} \circ \For_{H'}^H
\]
from $D^+ \Rep(H)$ to $D^+ \Rep(K')$.
\end{lem}

\begin{proof}
For any $M$ in $\Rep(H)$,
restriction induces a functorial morphism
\begin{equation}
\label{eqn:Ind-For-1}
\For^K_{K'} \circ \Ind_H^K(M) = \bigl( M \otimes \cO(K) \bigr)^H \to \bigl( M \otimes \cO(K') \bigr)^{H'} = \Ind_{H'}^{K'} \circ \For_{H'}^H(M).
\end{equation}
One can also define a functorial morphism
\begin{equation}
\label{eqn:Ind-For-2}
\Ind_{H'}^{K'} \circ \For_{H'}^H(M) \to \For^K_{K'} \circ \Ind_H^K(M)
\end{equation}
as follows: an element in $\Ind_{H'}^{K'} \circ \For_{H'}^H(M)$ is an $H'$-equivariant morphism $f : K' \to M$. Inducing this morphism we obtain an $H$-equivariant morphism $H \times^{H'} K' \to H \times^{H'} M$. By~\eqref{it:Ind-For-ass-1} the domain of this map identifies with $K$. Composing with the action morphism $H \times^{H'} M \to M$ we deduce an $H$-equivariant morphism $K \to M$, i.e.~an element of $\For^K_{K'} \circ \Ind_H^K(M)$.

It is straightforward to check that the morphisms~\eqref{eqn:Ind-For-1} and~\eqref{eqn:Ind-For-2} are inverse to each other, so that we obtain an isomorphism of functors
\[
\For^K_{K'} \circ \Ind_H^K \cong \Ind_{H'}^{K'} \circ \For_{H'}^H.
\]
From this isomorphism we deduce a canonical morphism of functors
\[
\For^K_{K'} \circ R\Ind_H^K \to R\Ind_{H'}^{K'} \circ \For_{H'}^H,
\]
and to prove that this morphism is an isomorphism it suffices to prove that if $M$ is an injective $H$-module then the $H'$-module $\For^H_{H'}(M)$ is acyclic for the functor $\Ind_{H'}^{K'}$.

So, let $M$ be an injective $H$-module. By~\cite[Proposition~I.3.10(a)]{jantzen}, there exists a $\bk$-vector space $V$ such that $M$ is a direct summand of $V \otimes \cO(H)$. We have a natural isomorphism
\[
R\Ind_{H'}^{K'}(V \otimes \cO(H)) \cong V \otimes R\Ind_{H'}^{K'}(\cO(H)),
\]
so that to conclude it suffices to prove that
\[
R^{>0}\Ind_{H'}^{K'}(\cO(H))=0.
\]
Now using~\cite[Proposition~I.3.10(c)]{jantzen} we see that, as complexes of vector spaces, we have
\[
R\Ind_{H'}^{K'}(\cO(H)) \cong R\mathbb{I}^{H'}(\cO(H) \otimes \cO(K')),
\]
where $\mathbb{I}^{H'} : \Rep(H') \to \Vect(\bk)$ is the functor of $H'$-invariants and where $H'$ acts diagonally on $\cO(H) \otimes \cO(K')$. From this we deduce a canonical isomorphism
\[
R\Ind_{H'}^{K'}(\cO(H)) \cong R\Ind_{H'}^H(\cO(K')).
\]
Then the desired vanishing follows from~\cite[Corollary~I.5.13(b)]{jantzen}.
\end{proof}

\begin{rmk}
\label{rmk:infinitesimal-groups-1}
Assume that $H$ and $K$ are infinitesimal affine $\bk$-group schemes in the sense of~\cite[\S I.8.1]{jantzen}. Then there exist canonical equivalences of categories
\begin{equation}
\label{eqn:Rep-Distmod}
\Rep(H) \cong \Dist(H)\lmod, \qquad \Rep(K) \cong \Dist(K)\lmod
\end{equation}
where $\Dist(-)$ denotes the distribution algebra;
see~\cite[\S\S I.8.4--6]{jantzen}. On the other hand, the morphism $\varphi : H \to K$ defines an algebra morphism $\phi : \Dist(H) \to \Dist(K)$, see~\cite[\S I.7.9]{jantzen}. It is straightforward to check that in this setting the functor $\Ind_H^K : \Rep(H) \to \Rep(K)$ corresponds to the functor $\phi_*$ defined in~\S\ref{ss:notation} under the identifications~\eqref{eqn:Rep-Distmod}.
\end{rmk}

\subsection{A spectral sequence for $H$-modules}
\label{ss:spectral-sequence-quotient}

Let $H$ be an affine $\bk$-group scheme, and let $K \subset H$ be a closed normal subgroup. Let $V$ be a finite-dimensional $H$-module. Then, for any $H$-module $V'$, the natural (diagonal) $H$-action on the vector space $\Hom_{K}(V,V')$ descends to an (algebraic) $H/K$-action. In other words, the functor $\Hom_K(V,-)$ factors through a functor $\Rep(H) \to \Rep(H/K)$, which we will denote similarly. Then the derived functors $\Ext^n_K(V,-)$ also factor through functors $\Ext^n_K(V,-) : \Rep(H) \to \Rep(H/K)$.

\begin{lem}
\label{lem:spectral-sequence}
For any $V'$ in $\Rep(H)$,
there exists a (bifunctorial) convergent spectral sequence
\[
E_2^{p,q} = \coH^p(H/K, \Ext^q_K(V,V')) \Rightarrow \Ext^{p+q}_H(V,V').
\]
\end{lem}

\begin{proof}
Using adjunction we can assume that $V$ is the trivial $H$-module. In this case the spectral sequence we wish to construct looks as follows:
\[
E_2^{p,q} = \coH^p(H/K, \coH^q(K,V')) \Rightarrow \coH^{p+q}(H,V').
\]
This spectral sequence is obtained from Grothendieck's spectral sequence for the derived functor of a composition of functors, see e.g.~\cite[Proposition~I.4.1]{jantzen}. For this we observe that we have $\mathbb{I}^H = \mathbb{I}^{H/K} \circ \mathbb{I}^K$, where as above $\mathbb{I}$ is the functor of invariants. Then we have to check that if $V'$ is an injective $H$-module then the $H/K$-module $\mathbb{I}^K(V')$ is injective. However, as in the proof of Lemma~\ref{lem:Ind-For}, we can assume that $V'=E \otimes \cO(H)$ where $E$ is a $\bk$-vector space (with trivial $H$-action). Then we have 
\[
\mathbb{I}^K(V') = \mathbb{I}^K(E \otimes \cO(H)) \cong E \otimes \cO(H/K),
\]
so that this $H/K$-module is indeed injective.
\end{proof}

From this lemma we deduce the following property.

\begin{cor}
\label{cor:dim-Ext}
For any $n \geq 0$ we have
\[
\dim(\Ext^n_H(V,V')) \leq \sum_{p+q=n} \dim \bigl( \coH^p(H/K, \Ext^q_K(V,V')) \bigr)
\]
(if the right-hand side is $<\infty$).
\end{cor}

\begin{proof}
The convergence of the spectral sequence of Lemma~\ref{lem:spectral-sequence} means that for any $n$, there is a filtration on $\Ext^n_H(V,V')$ whose associated graded is a subquotient of
\[
\bigoplus_{p+q=n} \coH^p(H/K, \Ext^q_K(V,V')).
\]
The claim follows.
\end{proof}

\subsection{Equivariant dg-modules}
\label{ss:equiv-dgmod}

Let $H$ be an affine $\bk$-group scheme, and let $\sA$ be a $\bk$-dg-algebra endowed with the structure of an $H$-module which is compatible with the grading, the differential and the multiplication. (Such a structure will be called an $H$-equivariant dg-algebra.)
Let $\sA\ldgmod_H$ be the category of $H$-equivariant $\sA$-dg-modules, i.e.~$\sA$-dg-modules $M$ endowed with the structure of an $H$-module which is compatible with the grading and the differential, and such that the action morphism $\sA \otimes M \to M$ is $H$-equivariant. (Morphisms are required to commute with the $\sA$- and $H$-actions.) We denote by $D_H(\sA)$ the corresponding derived category. If $\coH^\bullet(\sA)$ is left Noetherian, we denote by $\Dfg_H(\sA) \subset D_H(\sA)$ the full triangulated subcategory whose objects have finitely generated cohomology.

If $\sA$ is concentrated in nonpositive degrees,
we will also consider the full subcategory  $\sA\ldgmod_H^+$ of $\sA\ldgmod_H$ consisting of dg-modules which are bounded below, and the corresponding derived category $D^+_H(\sA)$. Our assumption implies that the usual truncation functors for complexes define functors on the category $\sA\ldgmod$; using these functors it is easy to check that the natural functor $D_H^+(\sA) \to D_H(\sA)$ is fully faithful, and that its essential image is the full subcategory of $D_H(\sA)$ consisting of dg-modules whose cohomology is bounded below.

We will not attempt to study the general theory of equivariant dg-modules. For instance, it is not clear to us whether, given a general $H$-equivariant dg-algebra $\sA$ as above (even if it is concentrated in nonpositive degrees), any object of $\sA\ldgmod_H$ (or even of $\sA\ldgmod^+_H$) admits a K-injective resolution. (A very special case of this question will be treated in~\S\ref{ss:Kinj-equivariant} below.) In this setting, we will restrict ourselves to easy constructions.

First we remark that if $H$ and $K$ are affine $\bk$-group schemes, $\varphi : H \to K$ is a morphism of group schemes, and $\sA$ is a $K$-equivariant dg-algebra, then $\sA$ can also be considered as an $H$-equivariant dg-algebra via $\varphi$. Moreover, the functor $\For^K_H : \Rep(K) \to \Rep(H)$ associated with $\varphi$ induces an exact functor $\sA\ldgmod_K \to \sA\ldgmod_H$. We will denote by
\[
\For^K_H : D_K(\sA) \to D_H(\sA)
\]
the induced functor on derived categories. If $\sA$ is concentrated in nonpositive degrees, then this functor restricts to a functor $D^+_K(\sA) \to D^+_H(\sA)$.

Now let $\sA$ and $\sB$ be $H$-equivariant dg-algebras, and
let $f : \sA \to \sB$ be an $H$-equivariant morphism of dg-algebras. As in the nonequivariant setting (see~\S\ref{ss:notation}) we have an exact ``restriction of scalars'' functor $f^* : \sB\ldgmod_H \to \sA\ldgmod_H$, and the corresponding derived functor
\[
f^* : D_H(\sB) \to D_H(\sA).
\]
If $\sA$ and $\sB$ are concentrated in nonpositive degrees,
this functor clearly restricts to a functor from $D_H^+(\sB)$ to $D_H^+(\sA)$. If $\sA$, $\sB$, $\sC$ are $H$-equivariant dg-algebras and $f : \sA \to \sB$, $g : \sB \to \sC$ are $H$-equivariant morphisms of dg-algebras, then we have
\begin{equation}
\label{eqn:composition-restriction-scalars}
(g \circ f)^* = f^* \circ g^*.
\end{equation}

Combining the previous two constructions, it is clear that if $\varphi : H \to K$ is a morphism of affine $\bk$-group schemes and if $f : \sA \to \sB$ is a $K$-equivariant morphism of $K$-equivariant dg-algebras, the following diagram commutes:
\begin{equation}
\label{eqn:restriction-scalars-For}
\vcenter{
\xymatrix@C=1.5cm{
D_K(\sB) \ar[r]^-{f^*} \ar[d]_-{\For^K_H} & D_K(\sA) \ar[d]^-{\For^K_H} \\
D_H(\sB) \ar[r]^-{f^*} & D_H(\sA).
}
}
\end{equation}

The following lemma is well known.

\begin{lem}
\label{lem:qis-equiv}
Let $H$ be an affine $\bk$-group scheme, let $\sA$ and $\sB$ be $H$-equivariant dg-algebras, and
let $f : \sA \to \sB$ be an $H$-equivariant morphism of dg-algebras which is a quasi-isomorphism. Then the functor $f^* : D_H(\sB) \to D_H(\sA)$ is an equivalence of categories.

If $\sA$ and $\sB$ are concentrated in nonpositive degrees then $f^*$ restricts to an equivalence $D_H^+(\sB) \simto D_H^+(\sA)$, and if $\coH^\bullet(\sA) \cong \coH^\bullet(\sB)$ is left Noetherian then $f^*$ restricts to an equivalence $\Dfg_H(\sB) \simto \Dfg_H(\sA)$.
\end{lem}

\begin{proof}[Sketch of proof]
The same procedure as for ordinary dg-modules (see~\cite{bl:esf}) shows that for any $M$ in $\sA\ldgmod_H$, there exists $M'$ in $\sA\ldgmod_H$ which is K-flat as an $\sA$-dg-module and a quasi-isomorphism $M' \xrightarrow{\qis} M$. Hence the derived functor
\[
\sB \lotimes_{\sA} (-) : D_H(\sA) \to D_H(\sB)
\]
is well defined. Then the same arguments as for~\cite[Theorem~10.12.5.1]{bl:esf} show that $f^*$ is an equivalence, with quasi-inverse given by $\sB \lotimes_{\sA} (-)$.

The final claim is clear from the fact that for $M$ in $D_H(\sB)$, $\coH^\bullet(M)$ is bounded below, resp.~finitely generated, iff $\coH^\bullet(f^*(M))$ is bounded below, resp.~finitely generated.
\end{proof}

\begin{rmk}
\label{rmk:infinitesimal-groups-2}
Consider as above an affine $\bk$-group scheme $H$ and a morphism $f : \sA \to \sB$ of $H$-equivariant dg-algebras concentrated in nonpositive degrees.
Assume also that $H$ is infinitesimal. 
We can consider the semidirect product $\sA \rtimes \Dist(H)$ as defined in~\S\ref{ss:semi-direct-products}. We also have a similar semidirect product $\sB \rtimes \Dist(H)$, and a dg-algebra morphism $f \rtimes \id : \sA \rtimes \Dist(H) \to \sB \rtimes \Dist(H)$.
Then the equivalence 
$\Rep(H) \cong \Dist(H)\lmod$
considered in~\eqref{eqn:Rep-Distmod} induces equivalences
\begin{equation}
\label{eqn:Rep-Distmod-dg}
D_H(\sA) \cong D(\sA \rtimes \Dist(H)), \qquad D_H(\sB) \cong D(\sB \rtimes \Dist(H)).
\end{equation}
In fact these equivalences also hold at the level of nonderived categories, so that K-injective resolutions exist in this setting.

Clearly, the following diagram commutes up to an isomorphism of functors:
\[
\xymatrix@C=2cm{
D_H(\sB) \ar[r]^-{f^*} \ar[d]_-{\eqref{eqn:Rep-Distmod-dg}}^-{\wr} & D_H(\sA) \ar[d]_-{\wr}^-{\eqref{eqn:Rep-Distmod-dg}} \\
D(\sB \rtimes \Dist(H)) \ar[r]^-{(f \rtimes \id)^*} & D(\sA \rtimes \Dist(H)).
}
\]
For simplicity, the functor corresponding to the functor $R(f \rtimes \id)_*$ under the identifications~\eqref{eqn:Rep-Distmod-dg} will be denoted
\[
Rf_* : D_H(\sA) \to D_H(\sB).
\]
\end{rmk}

\begin{rmk}
\label{rmk:infinitesimal-groups-3}
Let $H$ and $K$ be infinitesimal affine $\bk$-group schemes and let $\varphi : H \to K$ be a morphism of $\bk$-group schemes. Let $\sA$ be a $K$-equivariant $\bk$-dg-algebra concentrated in nonpositive degrees. Then via $\varphi$ we can also consider $\sA$ as an $H$-equivariant dg-algebra, and as in Remark~\ref{rmk:infinitesimal-groups-2} we have natural equivalences
\begin{equation}
\label{eqn:Rep-Distmod-dg-2}
D_K^+(\sA) \cong D^+(\sA \rtimes \Dist(K)), \qquad D_H^+(\sA) \cong D^+(\sA \rtimes \Dist(H)).
\end{equation}
Moreover $\varphi$ induces an algebra morphism $\phi : \Dist(H) \to \Dist(K)$, and hence a dg-algebra morphism $\id_{\sA} \rtimes \phi : \sA \rtimes \Dist(H) \to \sA \rtimes \Dist(K)$, so that we can consider the associated direct and inverse image functors relating $D^+(\sA \rtimes \Dist(K))$ and $D^+(\sA \rtimes \Dist(H))$.
It is clear that the following diagram commutes:
\[
\xymatrix@C=1.5cm{
D_K^+(\sA) \ar[d]^-{\wr}_-{\eqref{eqn:Rep-Distmod-dg-2}} \ar[r]^{\For^K_H} & D_H^+(\sA) \ar[d]_-{\wr}^-{\eqref{eqn:Rep-Distmod-dg-2}} \\
D^+(\sA \rtimes \Dist(K)) \ar[r]^-{(\id_{\sA} \rtimes \phi)^*} & D^+(\sA \rtimes \Dist(H)).
}
\]
We will denote by
\[
R\Ind_H^K : D_H^+(\sA) \to D_K^+(\sA)
\]
the functor corresponding to $R(\id_{\sA} \rtimes \phi)_*$ under the identifications~\eqref{eqn:Rep-Distmod-dg-2}.
This notation is justified by the fact that this functor is compatible with the functors $R\Ind_H^K$ of~\S\ref{ss:induction} in the obvious sense; in fact this follows from the observation that any K-injective $\sA \rtimes \Dist(H)$-dg-module is also K-injective as a complex of $\Dist(H)$-modules, since $\sA \rtimes \Dist(H)$ is K-flat as a complex of right $\Dist(H)$-modules.
\end{rmk}

\subsection{$H$-action on $\Hom$-spaces}
\label{ss:Hom-action}

Let $H$ be an affine $\bk$-group scheme, and let $\sA$ be an $H$-equivariant dg-algebra.

\begin{lem}
\label{lem:Kproj-resolution}
For any $M$ in $\sA\ldgmod_H$, there exists an object $M'$ in $\sA\ldgmod_H$ which is K-projective as an $\sA$-dg-module and a quasi-isomorphism $M' \xrightarrow{\qis} M$.
\end{lem}

\begin{proof}
The ``bar resolution'' of~\cite[\S 10.12.2.4]{bl:esf} (see also~\cite[Remark~10.12.2.7]{bl:esf}) provides a resolution with the desired properties.
\end{proof}

From now on in this subsection we assume that $\bk$ is algebraically closed and that $H$ is reduced and of finite type (in other words an algebraic group in the ``traditional'' sense). Then we can consider the abelian category $\Repd(H)$ of ``discrete'' $H$-representations, i.e.~vector spaces $V$ endowed with a group homomorphism from (the $\bk$-points of) $H$ to $\mathrm{GL}(V)$ which is not necessarily a morphism of algebraic varieties.  (A typical example is an infinite-dimensional representation that is not the union of its finite-dimensional subrepresentations, which might arise e.g.~when taking the dual of an infinite-dimensional algebraic $H$-module.)

For any $M$ in $\sA\ldgmod_H$, consider the functor
\[
\Hom^\bullet_{\sA}(-,M) : \bigl( \sA\ldgmod_H \bigr)^{\mathrm{op}} \to C(\Repd(H))
\]
(where the right-hand side is the category of complexes of objects in $\Repd(H)$), where the $H$-action is diagonal. The resolutions considered in Lemma~\ref{lem:Kproj-resolution} are split on the right for this functor, so that we can consider the associated derived functor
\[
R\Hom_{\sA}(-,M) : D_H(\sA)^{\mathrm{op}} \to D (\Repd(H)).
\]
By construction, for any $N$ in $\sA\ldgmod_H$ and any $n \in \Z$ we have a canonical isomorphism
\[
\coH^n(R\Hom_{\sA}(N,M)) \cong \Hom^n_{D(\sA)}(\For^H_{\{1\}}(N),\For^H_{\{1\}}(M)).
\]
In particular, this implies that the vector space $\Hom^n_{D(\sA)}(\For^H_{\{1\}}(N),\For^H_{\{1\}}(M))$ has a natural action of $H$ (which might be nonalgebraic).

\begin{lem}
\label{lem:action-Hom}
Let $f : \sA \to \sB$ be an $H$-equivariant morphism of $H$-equivariant dg-algebras. Then for any $M,N$ in $D_H(\sB)$, the morphism
\[
\Hom_{D(\sB)}(M,N) \to \Hom_{D(\sA)}(f^*M, f^*N)
\]
induced by the functor $f^*$
(where for simplicity we omit the functors $\For^H_{\{1\}}$) is $H$-equivariant.
\end{lem}

\begin{proof}
Let $M' \xrightarrow{\qis} M$ and $M'' \xrightarrow{\qis} f^*(M)$ be resolutions as in Lemma~\ref{lem:Kproj-resolution}. Then we have $H$-equivariant isomorphisms
\begin{align*}
\Hom_{D(\sB)}(M,N) &\cong \coH^0(\Hom^\bullet_{\sB}(M',N)), \\
\Hom_{D(\sA)}(f^* M,f^* N) &\cong \coH^0(\Hom^\bullet_{\sA}(M'',f^* N)).
\end{align*}
Moreover, the morphism under consideration is induced by the morphism of complexes
\[
\Hom^\bullet_{\sB}(M',N) \to \Hom^\bullet_{\sA}(M'',f^* N)
\]
sending a morphism $\varphi : M' \to N[k]$ to the composition
\[
M'' \xrightarrow{\qis} f^*(M) \xrightarrow{f^*(\varphi)} f^*(N)[k].
\]
This morphism is obviously $H$-equivariant, which proves the lemma.
\end{proof}

\subsection{The case of finite-dimensional dg-algebras}
\label{ss:Kinj-equivariant}

As in~\S\ref{ss:equiv-dgmod}, we let $H$ be a $\bk$-group scheme, and $\sA$ be an $H$-equivariant dg-algebra concentrated in nonpositive degrees. We assume in addition that $\dim_\bk(\sA)<\infty$.

\begin{lem}
\label{lem:K-inj-equiv}
For any bounded below $H$-equivariant $\sA$-dg-module $X$, there exists a bounded below $H$-equivariant $\sA$-dg-module $Y$ which is
\begin{itemize}
\item
K-injective as an $H$-equivariant $\sA$-dg-module;
\item
K-injective as an $\sA$-dg-module;
\item
a complex of injective $H$-modules 
\end{itemize}
and a quasi-isomorphism of $H$-equivariant $\sA$-dg-modules $\varphi : X \xrightarrow{\mathrm{qis}} Y$.
\end{lem}

\begin{proof}
We proceed in a way similar to the procedure in~\cite[Lemma~1.3.5]{riche}. Namely, we first consider a bounded below complex $V_0$ of injective $H$-modules (with the same lower bound as $X$) and an injective morphism of complexes of $H$-modules $X \hookrightarrow V_0$. This morphism defines in a natural way an injective morphism $X \hookrightarrow \Hom_{\bk}^\bullet(\sA,V_0)$. (Here, $\sA$ acts on $\Hom_{\bk}^\bullet(\sA,V_0)$ through right multiplication in $\sA$, as in the definition of the coinduction functor in~\cite[\S 1.2]{riche}, and $H$ acts diagonally.) One can easily check that $Z_0:=\Hom_{\bk}^\bullet(\sA,V_0)$ is bounded below with the same bound as $X$ and K-injective, both as an $\sA$-dg-module and as an $H$-equivariant $\sA$-dg-module. Using~\cite[Proposition~I.3.10(b)]{jantzen}, one can also check that $Z_0$ is a complex of injective $H$-modules. 

Proceeding similarly with the cokernel of the injection $X \hookrightarrow Z_0$ and repeating, we obtain $H$-equivariant $\sA$-dg-modules $Z_k$ which are bounded below with the same bound as $X$, K-injective both as $\sA$-dg-modules and as $H$-equivariant $\sA$-dg-modules, 
and whose terms are injective $H$-modules, 
and an exact sequence of $H$-equivariant $\sA$-dg-modules
\[
X \hookrightarrow Z_0 \to Z_1 \to Z_2 \to \cdots
\]
Let $Y$ be the total complex of the double complex $0 \to Z_0 \to Z_1 \to \cdots$ (where $Z_k$ is in horizontal degree $k$). Then there exists a natural morphism $X \to Y$, which is easily seen to be a quasi-isomorphism. Hence to conclude it suffices to check that $Y$ has the desired properties. Clearly each graded component of $Y$ is an injective $H$-module, so we need only consider the first two conditions.

For any $p$, we denote by $Y_p$ the total complex of the double complex $0 \to Z_0 \to Z_1 \to \cdots \to Z_{p-1} \to Z_p \to 0 \to \cdots$. Then for any $p$ we have an exact sequence
\begin{equation}
\label{eqn:ses-Kinj}
Z_{p+1}[-p-1] \hookrightarrow Y_{p+1} \twoheadrightarrow Y_p
\end{equation}
which is split as an exact sequence of $H$-equivariant graded $\sA$-modules (i.e.~when we forget differentials).

Now we can prove that $Y$ is K-injective as an $H$-equivariant $\sA$-dg-module. Let $M$ be an acyclic $H$-equivariant $\sA$-dg-module. We have, as complexes of $\bk$-vector spaces,
\[
\Hom^\bullet_{\sA\ldgmod_H}(M,Y) \cong \varprojlim_p \Hom^\bullet_{\sA\ldgmod_H}(M,Y_p).
\]
(Here, $\Hom^\bullet_{\sA\ldgmod_H}(X,X')$ is the complex whose $i$-th term consists of homogeneous morphisms of $H$-equivariant $\sA$-modules of degree $i$ from $X$ to $X'$, with the differential induced by $d_X$ and $d_{X'}$.)
For any $p$, since the exact sequence~\eqref{eqn:ses-Kinj} is split as an exact sequence of $H$-equivariant graded $\sA$-modules, it induces an exact sequence of complexes
\[
\Hom^\bullet_{\sA\ldgmod_H}(M,Z_{p+1}[-p-1]) \hookrightarrow \Hom^\bullet_{\sA\ldgmod_H}(M,Y_{p+1}) \twoheadrightarrow \Hom^\bullet_{\sA\ldgmod_H}(M,Y_p).
\]
Since $Z_{p+1}$ is K-injective, the complex $\Hom^\bullet_{\sA\ldgmod_H}(M,Z_{p+1}[-p-1])$ is acyclic. Hence the inverse system $(\Hom^\bullet_{\sA\ldgmod_H}(M,Y_p))_{p \geq 0}$ is $\mathfrak{I}$-special in the sense of~\cite[Definition~2.1]{spaltenstein}, where $\mathfrak{I}$ is the class of acyclic complexes of $\bk$-vector spaces. Using~\cite[Lemma~2.3]{spaltenstein} we deduce that its inverse limit $\Hom^\bullet_{\sA\ldgmod_H}(M,Y)$ is acyclic, which proves the desired K-injectivity.

The same arguments show that $Y$ is also K-injective as an $\sA$-dg-module, and the proof is complete.
\end{proof}

Now we consider affine $\bk$-group schemes $H$ and $K$, a morphism of $\bk$-group schemes $\varphi : H \to K$, and a finite-dimensional $K$-equivariant dg-algebra $\sA$ concentrated in nonpositive degrees. Via $\varphi$ we can also consider $\sA$ as an $H$-equivariant dg-algebra. The functor $\Ind_H^K : \Rep(H) \to \Rep(K)$ induces a functor from $\sA\ldgmod_H$ to $\sA\ldgmod_K$ (which we will also denote $\Ind_H^K$) as follows: if $M$ is in $\sA\ldgmod_H$, we consider the $\sA$-action on the complex of $K$-modules $\Ind_H^K(M)$ defined by $(a \cdot f)(k) = (k^{-1} \cdot a) \cdot f(k)$ (where elements in $\Ind_H^K(M)=\mathbb{I}^H(\cO(K) \otimes M)$ are considered as algebraic morphisms $K \to M$ as in~\cite[\S I.3.3]{jantzen}). Lemma~\ref{lem:K-inj-equiv} ensures that the right derived functor $R\Ind_H^K$ is well defined on the subcategory $D^+_H(\sA) \subset D_H(\sA)$, and that moreover the following diagram commutes up to isomorphism, where the vertical arrows are induced by the functor of forgetting the $\sA$-action:
\begin{equation}
\label{eqn:diag-for-ind}
\vcenter{
\xymatrix@C=1.5cm{
D^+_H(\sA) \ar[r]^-{R\Ind_H^K} \ar[d] & D^+_K(\sA) \ar[d] \\
D^+ \Rep(H) \ar[r]^-{R\Ind_H^K} & D^+ \Rep(K).
}
}
\end{equation}
It is also easily checked that the functor $R\Ind_H^K$ is right-adjoint to the forgetful functor $\For^K_H : D^+_K(\sA) \to D^+_H(\sA)$.

\section{Reductive algebraic groups and Steinberg modules}
\label{sec:reductive}

\subsection{Notation for algebraic groups}
\label{ss:notation-alg-groups}

From now on we assume that
$\Bbbk$ is an algebraically closed field of positive characteristic $\ell$, and let $G$ be a 
connected reductive
algebraic group over $\Bbbk$ 
with simply connected derived subgroup. 
Let $T \subset B \subset G$ be a maximal torus and a Borel subgroup, and let $B^+$ be the opposite Borel subgroup (with respect to $T$). We also denote by $N$ the unipotent radical of $B$, and by $\fg$, $\fb$, $\ft$, $\fb^+$, $\fn$ the Lie algebras of $G$, $B$, $T$, $B^+$, $N$.

We will denote by $\Phi$ the root system of $(G,T)$, by $\Phi^+ \subset \Phi$ the system of positive roots consisting of the $T$-weights in nilradical of $\fb^+$, by $\Sigma \subset \Phi$ the corresponding simple roots, by $W$ the Weyl group of $(G,T)$, and by $S \subset W$ the set of simple reflections corresponding to $\Sigma$. We will denote by
\[
s \mapsto \alpha_s, \qquad \alpha \mapsto s_\alpha
\]
the natural bijections $S \simto \Sigma$ and $\Sigma \simto S$. For any $\alpha \in \Phi$ we denote by $\fg_\alpha$ the corresponding root subspace in $\fg$, and by $\alpha^\vee$ the corresponding coroot. 

For any subset $I \subset S$, we denote by $\Sigma_I=\{\alpha_s : s \in I\} \subset \Sigma$ the corresponding subset of $\Sigma$. Then we have the corresponding root system $\Phi_I = \Phi \cap \Z \Sigma_I$ and positive roots $\Phi_I^+ =\Phi^+ \cap \Phi_I$. We also let $W_I \subset W$ be the (parabolic) subgroup generated by $I$, and $w_I$ be the longest element in $W_I$. We denote by $P_I \subset G$ the 
parabolic subgroup containing $B$ associated with $I$, and by $\fp_I$ its Lie algebra, so that
\[
\fp_I = \fb \oplus \bigoplus_{\alpha \in \Phi_I^+} \fg_\alpha.
\]
We denote by $M_I$ the Levi factor of $P_I$ containing $T$, by $\fm_I$ its Lie algebra, by $N_I$ the unipotent radical of $P_I$, and by $\fn_I$ its Lie algebra. Of course, when $I=\varnothing$ we have $P_\varnothing = B$, $M_\varnothing=T$ and $N_\varnothing=N$. When $I=\{s\}$ for some $s \in S$, we simplify the notation $P_{\{s\}}$, $M_{\{s\}}$, etc.~to $P_s$, $M_s$, etc. (This simplification will also be used for other notation depending on $I \subset S$ that will be defined later in the paper.)

We denote by $\dot G=G^{(1)}$ the Frobenius twist of $G$. Recall that by definition, as rings we have $\cO(\dot G) = \cO(G)$, but the $\bk$-actions are different: if $x \in \bk$, then $x$ acts on $\cO(\dot G)$ in the way $x^{1/\ell}$ acts on $\cO(G)$. (Here, $(-)^{1/\ell}$ is the inverse of the field automorphism of $\bk$ given by $x \mapsto x^\ell$.) The Frobenius morphism $\Fr : G \to \dot G$ is the $\bk$-scheme morphism induced by the $\bk$-algebra morphism $\cO(\dot G) \to \cO(G)$ defined by $f \mapsto f^\ell$. The $\bk$-scheme $\dot G$ has a natural structure of $\bk$-algebraic group, and $\Fr$ is an algebraic group morphism. Its kernel is (by definition) the Frobenius kernel of $G$, and will be denoted $G_1$.
It is an infinitesimal affine $\bk$-group scheme.
We use similar notation for the subgroups of $G$ introduced above. In particular, $\dot T$ is a maximal torus in $\dot G$, and $\dot B$ is a Borel subgroup in $\dot G$.

We let $\bX$ denote the lattice of characters of $T$ (or equivalently of $B$), and
$\bX^+ \subset \bX$ be the set of dominant weights.
Given a subset $I \subset S$, we set
\[
\rho_I := \frac{1}{2} \sum_{\alpha \in \Phi^+_I} \alpha \quad \in \bX \otimes_{\Z} \Q.
\]
We also choose a weight $\varsigma_I \in \bX$ such that $\langle \varsigma_I, \alpha^\vee \rangle = 1$ for all $\alpha \in \Sigma_I$.
When $I=S$, we simplify the notation to $\rho$ and $\varsigma$. 
(Starting from Section~\ref{sec:translation} we will make a more specific choice for these weights, but in Sections~\ref{sec:reductive}--\ref{sec:compatibility} they can be arbitrary.)
Throughout the paper we assume that $\ell > h$, where
$h$ is the Coxeter number of $\Phi$.

Since $\cO(\dot T) = \cO(T)$, the lattice of characters of $\dot T$ identifies canonically with $\bX$. With this identification, the morphism $\bX \to \bX$ induced by composition with the Frobenius morphism $T \to \dot T$ is given by $\lambda \mapsto \ell \lambda$. In other words, we have $\For^{\dot T}_T(\bk_{\dot T}(\lambda)) = \bk_T(\ell \lambda)$.

If $I \subset S$, we set $\fN_I := \mathcal{U}(\fn_I)$, the universal enveloping algebra of $\fn_I$. 
We denote by $\fZ_I \subset \fN_I$ the (central) subalgebra generated by elements of the form $x^\ell - x^{[\ell]}$ for $x \in \fn_I$. Then $\fZ_I$ is canonically isomorphic to $\Sym(\dot\fn_I)$ (where $\dot\fn_I$ is the Lie algebra of $\dot N_I$), and if $\bk$ is the trivial $\fZ_I$-module we have
\[
\fN_I \otimes_{\fZ_I} \bk = \snn_I,
\]
where $\snn_I$ is the restricted enveloping algebra of $\fn_I$, which identifies with the distribution algebra of $N_{I,1}$.

Note that our notation (and the rest of the notation introduced later) follows the following pattern: if $H$ is an algebraic group over $\bk$, then $\dot H$ is its Frobenius twist, $H_1$ its Frobenius kernel, $\mathfrak{h}$ its Lie algebra, $\mathfrak{H}$ the enveloping algebra of $\mathfrak{h}$, and $\mathsf{h}$ the distribution algebra of $H_1$ (or equivalently the restricted enveloping algebra of $\mathfrak{h}$).

\subsection{Steinberg modules}
\label{ss:Steinberg}

Given $I \subset S$, we can define the $P_I$-module
\[
\St_I := \Ind_{B}^{P_I} \bigl( \bk_B((\ell -1)\varsigma_I) \bigr).
\]
It is clear that $N_I \subset P_I$ acts trivially on $\St_I$, so that this module factors through an $M_I$-module (which we denote similarly.)
When $I = \varnothing$, $\St_I$ is just the one-dimensional $B$-module $\bk_B((\ell-1)\varsigma_{\varnothing})$ (i.e.~the trivial module if we have chosen $\varsigma_\varnothing=0$). When $I=S$ we omit the subscript $S$. 
For any $I$, $\St_I$ is irreducible as a $P_I$- or $M_I$-module.  When regarded as an $M_{I,1}$-module,
or as an $M_{I,1}T$-module, 
it is simple, injective, and projective (see~\cite[Proposition~II.10.2]{jantzen}).

\begin{rmk}
The results of~\cite{jantzen} cited above (as well as those cited below) are stated for the module $\Ind_{B \cap M_I}^{M_I}((\ell-1)\rho_I)$ instead of $\St_I$, assuming that $(\ell-1)\rho_I$ belongs to $\bX$. However, under our assumptions, if $I \neq \varnothing$ then $\ell$ is odd, so that $(\ell-1)\rho_I$ indeed belongs to $\bX$. And we have isomorphisms of $P_I$-modules
\[
\Ind_{B \cap M_I}^{M_I}((\ell-1)\rho_I) \cong \Ind_{B}^{P_I}((\ell-1)\rho_I) \cong \St_I \otimes \bk_{P_I}((\ell-1)(\rho_I-\varsigma_I)),
\]
since $(\ell-1)(\rho_I-\varsigma_I)$ is a character of $P_I$. These isomorphisms allow us to transfer the required results from the case of $\Ind_{B \cap M_I}^{M_I}((\ell-1)\rho_I)$ to the case of $\St_I$.
\end{rmk}

\begin{lem}
\label{lem:st-sll}
Let $B_I = B \cap M_I$, and let $B_I^+ = B^+ \cap M_I$.  Then we have isomorphisms of $M_I$-modules
\[
\St_I \cong \Ind_{B_{I,1}}^{M_{I,1}} \bk_{B_{I,1}}((\ell-1) \varsigma_I) \cong \Ind_{B_{I,1}^+}^{M_{I,1}} \bk_{B_{I,1}^+} ((\ell-1)(\varsigma_I - 2\rho_I)).
\]
\end{lem}

\begin{proof}
This follows from~\cite[II.3.18(4)--(5) \& II.3.7(4)]{jantzen}.
\end{proof}

Next, we define a $P_I$-module
\[
\rmZ_I := \For^{G}_{P_I}(\St) \otimes \bk_{P_I}((\ell-1)(\varsigma_I - 2\rho_I + 2\rho - \varsigma)).
\]
(Note that $\langle \varsigma_I - 2\rho_I + 2\rho - \varsigma, \alpha^\vee \rangle = 0$ for any $\alpha \in \Phi_I$, so that $\varsigma_I - 2\rho_I + 2\rho - \varsigma$ defines a character of $M_I$, and hence of $P_I$ via the surjection $P_I \to M_I$.)

\begin{lem}
\label{lem:pi-ind}
We have an isomorphism of $P_{I,1}$-modules $\rmZ_I \cong \Ind_{M_{I,1}}^{P_{I,1}} (\St_I)$.  Moreover, as a $P_{I,1}$-module, $\rmZ_I$ is the injective envelope of $\St_I$.
\end{lem}

\begin{proof}
By the tensor identity (see~\cite[Proposition~I.3.6]{jantzen}),
the first assertion is equivalent to the claim that
\[
\St \cong \Ind_{M_{I,1}}^{P_{I,1}} \bigl( \St_I \otimes \bk_{M_{I,1}} ((\ell-1)(2\rho_I - \varsigma_I - 2\rho + \varsigma)) \bigr)
\]
as $P_{I,1}$-modules.
By Lemma~\ref{lem:st-sll} (applied to $I$ and then to $S$) and transitivity of induction, we have isomorphisms of $P_{I,1}$-modules
\begin{multline*}
\Ind_{M_{I,1}}^{P_{I,1}} \bigl( \St_I \otimes \bk_{M_{I,1}} ((\ell-1)(2\rho_I - \varsigma_I - 2\rho + \varsigma)) \bigr)
\\
\cong \Ind_{B_{I,1}^+}^{P_{I,1}} \bk_{B_{I,1}^+}((\ell-1)(-2\rho+\varsigma))
\cong \Ind_{B_1^+}^{G_1} \bk_{B_1^+}((\ell-1)(-2\rho+\varsigma)) \cong \St
\end{multline*}
(where the second isomorphism can be deduced from~\cite[Lemma~II.3.2]{jantzen}).
Since induction takes injective modules to injective modules, $\rmZ_I$ is an injective $P_{I,1}$-module.  It is indecomposable because $\St$ is an indecomposable $N_{1}$-module (see e.g.~\cite[II.3.18(1)]{jantzen}), so the adjunction morphism $\St_I \to \Ind_{M_{I,1}}^{P_{I,1}} \St_I = \rmZ_I$ shows that it must be the injective envelope of $\St_I$.
\end{proof}

\begin{rmk}
\label{rmk:pi-ind-proj}
The $P_{I,1}$-module $\rmZ_I$ is also projective; see~\cite[\S I.8.10]{jantzen}. Using~\cite[Lemma~II.9.3]{jantzen}, we deduce that it is even projective as a $P_{I,1} T$-module.
\end{rmk}

\begin{cor}
\label{cor:Ind-Pi}
Consider the projection $P_{I,1} \to M_{I,1}$, and the associated functor $R\Ind_{P_{I,1}}^{M_{I,1}}$. Then we have $R\Ind_{P_{I,1}}^{M_{I,1}}(\rmZ_I) \cong \St_I$.
\end{cor}

\begin{proof}
Lemma~\ref{lem:pi-ind} and~\cite[Corollary~I.5.13(b)]{jantzen} imply that $\rmZ_I \cong R\Ind_{M_{I,1}}^{P_{I,1}}(\St_I)$. Using~\eqref{eqn:transitivity-RInd}, it follows that we have
\[
R\Ind_{P_{I,1}}^{M_{I,1}}(\rmZ_I) \cong R\Ind_{P_{I,1}}^{M_{I,1}} \circ R\Ind_{M_{I,1}}^{P_{I,1}} (\St_I) \cong \St_I
\]
since the composition $M_{I,1} \to P_{I,1} \to M_{I,1}$ is the identity morphism.
\end{proof}

\begin{cor}
\label{cor:StI-PiI}
There exists a nonzero morphism of $P_I$-modules $\St_I \to \rmZ_I$.
\end{cor}

\begin{proof}
Consider the vector space $\Hom_{P_{I,1}}(\St_I, \rmZ_I)$. Since $P_{I,1}$ is normal in $P_I$, and since both $\St_I$ and $\rmZ_I$ admit $P_I$-module structures, this space admits a natural $P_I$-action (by conjugation). By Lemma~\ref{lem:pi-ind}, this module has dimension $1$, so that $P_I$ necessarily acts via a character $\chi : P_I \to \Gm$. Now the same arguments as for Lemma~\ref{lem:pi-ind} (see in particular~\cite[II.10.1(4) \& \S 9.1--2]{jantzen}) show that $\rmZ_I \cong \Ind_{M_{I,1} T}^{P_{I,1} T}(\St_I)$, so that adjunction provides a nonzero morphism of $P_{I,1} T$-modules $\St_I \to \rmZ_I$. This shows that $\chi$ is trivial on $T$, and hence that it is the trivial character.
\end{proof}

\begin{lem}
\label{lem:Ind-Steinberg}
Let $I \subset I' \subset S$. Then $\St_{I'}$ is a direct summand in $\Ind_{M_{I,1}T}^{M_{I',1} T}(\St_I \otimes \bk_{M_{I,1} T}((\ell-1)(\varsigma_{I'}-\varsigma_I)))$ with multiplicity $1$. Moreover we have
\begin{multline*}
\dim_{\bk}(\Hom_{M_{I',1}T}(\St_{I'}, \Ind_{M_{I,1}T}^{M_{I',1} T}(\St_I \otimes \bk((\ell-1)(\varsigma_{I'}-\varsigma_I))))) = \\
\dim_{\bk}(\Hom_{M_{I',1}T}(\Ind_{M_{I,1}T}^{M_{I',1} T}(\St_I \otimes \bk((\ell-1)(\varsigma_{I'}-\varsigma_I))), \St_{I'}))=1.
\end{multline*}
In particular, any composition
\[
\St_{I'} \to \Ind_{M_{I,1}T}^{M_{I',1} T}(\St_I \otimes \bk((\ell-1)(\varsigma_{I'}-\varsigma_I))) \to \St_{I'}
\]
where both arrows are $M_{I',1}T$-equivariant and nonzero is itself nonzero.
\end{lem}

\begin{proof}
Set $\nu=(\ell-1)(\varsigma_{I'} - \varsigma_I)$. By adjunction, we have
\[
\Hom_{M_{I',1} T}(\St_{I'}, \Ind_{M_{I,1}T}^{M_{I',1} T}(\St_I \otimes \bk(\nu))) \cong \Hom_{M_{I,1} T}(\St_{I'} \otimes \bk(-\nu), \St_I).
\]
Since $\St_I$ is both injective and simple as an $M_{I,1} T$-module, it is its own injective envelope, and the dimension of the vector space considered above is the multiplicity of $\St_I$ as a composition factor of $\St_{I'} \otimes \bk(-\nu)$. Now the highest weights of $\St_I$ and $\St_{I'} \otimes \bk(-\nu)$ are both equal to $(\ell-1)\varsigma_I$, and the corresponding weight spaces have dimension $1$. So, the multiplicity under consideration is at most $1$. On the other hand we have $\St_{I'} = \Ind_{P_I}^{P_{I'}}(\St_I \otimes \bk(\nu))$, so adjunction provides a nonzero morphism of $P_I$-modules (hence of $M_{I,1}T$-modules) $\St_{I'} \to \St_I \otimes \bk(\nu)$, and hence the multiplicity is at least $1$.

We have thus proved that
\begin{equation}
\label{eqn:dim-Hom-1}
\dim_{\bk}( \Hom_{M_{I',1}T}(\St_{I'}, \Ind_{M_{I,1}T}^{M_{I',1} T}(\St_I \otimes \bk(\nu))) ) = 1.
\end{equation}
Any nonzero $M_{I',1} T$-equivariant morphism $\St_{I'} \to \Ind_{M_{I,1}T}^{M_{I',1} T}(\St_I \otimes \bk(\nu))$ must be injective since $\St_{I'}$ is simple. And since both $M_{I',1}T$-modules are injective, such a morphism must be the embedding of a direct summand. This proves that $\St_{I'}$ is a direct summand in $\Ind_{M_{I,1}T}^{M_{I',1} T}(\St_I \otimes \bk(\nu))$ with multplicity $1$.

It remains to compute
\begin{equation}
\label{eqn:dim-Hom-2}
\dim_{\bk}(\Hom_{M_{I',1}T}(\Ind_{M_{I,1}T}^{M_{I',1} T}(\St_I \otimes \bk(\nu)), \St_{I'})).
\end{equation}
By the same arguments as above, this dimension is the multiplicity of $\St_{I'}$ as a composition factor of $\Ind_{M_{I,1}T}^{M_{I',1} T}(\St_I \otimes \bk(\nu))$. Now since $\St_{I'}$ is also its own projective cover,~\eqref{eqn:dim-Hom-1} shows that this multiplicity is $1$, and~\eqref{eqn:dim-Hom-2} is proved.

The final assertion is an easy consequence of the previous statements.
\end{proof}

\subsection{The case of semisimple rank $1$}
\label{ss:ssrk1}

We conclude this subsection with some results in the special case where $I$ consists of a single simple reflection $s$. Recall that $\St_s$ has weights
\[
(\ell-1)\varsigma_s, (\ell-1)\varsigma_s - \alpha_s, (\ell-1)\varsigma_s - 2\alpha_s, \cdots, (\ell-1)\varsigma_s - (\ell-1)\alpha_s.
\]

For any $\lambda \in \bX$ with $\langle \lambda, \alpha_s^\vee \rangle \geq 0$, we set
\[
\coweyl_s(\lambda):=\Ind_B^{P_s}(\lambda), \qquad \weyl_s(\lambda):=(\Ind_B^{P_s}(-s \lambda))^*.
\]
Both of these modules factor through $M_s$-modules; as such they are isomorphic to the costandard and standard $M_s$-module of highest weight $\lambda$ respectively.

There exists, up to scalar, a unique nonzero morphism $\weyl_s(\lambda) \to \coweyl_s(\lambda)$; its image is the simple $M_s$-module with highest weight $\lambda$, which we denote by $\irr_s(\lambda)$. Finally, we denote by $\tilt_s(\lambda)$ the indecomposable tilting $M_s$-module of highest weight $\lambda$. The $M_s$-modules $\irr_s(\lambda)$ and $\tilt_s(\lambda)$ will sometimes be considered as $P_s$-modules via the projection $P_s \to M_s$.

\begin{lem}
\label{lem:steinberg-diff}
There exists an exact sequence of $B$-modules
\begin{equation}\label{eqn:steinberg-diff}
0 \to \bk_B(\ell\varsigma_s-\ell\alpha_s) \to \St_s \otimes \bk_B(\varsigma_s-\alpha_s) \xrightarrow{f} \St_s \otimes \bk_B(\varsigma_s) \to \bk_B(\ell\varsigma_s) \to 0
\end{equation}
which corresponds to a nonzero element of $\Ext^2_B(\bk_B(\ell\varsigma_s), \bk_B(\ell\varsigma_s-\ell\alpha_s))$.
\end{lem}
\begin{proof}
In~\cite[Proposition~II.5.2]{jantzen}, a certain basis $\{ v_0, v_1, \cdots, v_{\ell-1}\}$ of $\St_s$ is considered, where each $v_i$ is a weight vector of weight $(\ell - 1)\varsigma_s - i\alpha_s$.  Consider the linear map $f: \St_s \otimes \bk(\varsigma_s-\alpha_s) \to \St_s \otimes \bk(\varsigma_s)$ given by
\[
f(v_i \otimes 1) = 
\begin{cases}
\frac{\ell -1 - i}{\ell-1} v_{i+1} \otimes 1& \text{if $0 \le i < \ell -1$,} \\
0 & \text{if $i = \ell - 1$.}
\end{cases}
\]
According to the formulas in~\cite[Proposition~II.5.2]{jantzen}, $f$ is $B$-equivariant.  Its kernel is the span of $v_{\ell-1} \otimes 1$, which is isomorphic to $\bk_B(\ell\varsigma_s - \ell\alpha_s)$, and its cokernel is spanned by the image of $v_0 \otimes 1$; it is isomorphic to $\bk_B(\ell\varsigma_s)$.  Hence we have constructed the four-term exact sequence~\eqref{eqn:steinberg-diff}.

Before addressing the claim about $\Ext^2$, let us construct some short exact sequences.
Consider the module $\coweyl_s(\ell\varsigma_s)$, and let $\{u_0, u_1, \cdots, u_\ell\}$ be the basis for this module described in~\cite[Proposition~II.5.2]{jantzen}.  Let $g: \St_s \otimes \bk_B(\varsigma_s - \alpha_s) \to \coweyl_s(\ell\varsigma_s)$ be the map given by
\[
g(v_i \otimes 1) = u_{i+1} \quad \text{for } i \in \{0, \cdots, \ell-1\}.
\]
As in the preceding paragraph, one can check using~\cite[Proposition~II.5.2]{jantzen} that $g$ is $B$-equivariant. This map is clearly injective, so that we obtain a short exact sequence of $B$-modules
\begin{equation}\label{eqn:steinberg-diff-ext1}
0 \to \St_s \otimes \bk_B(\varsigma_s - \alpha_s) \xrightarrow{g} \coweyl_s(\ell\varsigma_s) \to \bk_B(\ell\varsigma_s) \to 0.
\end{equation}

We claim that~\eqref{eqn:steinberg-diff-ext1} is not split. Indeed, to prove this it suffices to prove that $\coweyl_s(\ell\varsigma_s)$ is indecomposable as a $B$-module. However it is clearly indecomposable as a $P_s$-module, and the functor $\For_B^{P_s}$ is fully faithful (since its right adjoint $\Ind_B^{P_s}$ satisfies $\Ind_B^{P_s} \circ \For_B^{P_s} \cong \id$). Hence $\coweyl_s(\ell\varsigma_s)$ is indeed indecomposable over $B$, which proves our claim.

Next, taking the dual of~\eqref{eqn:steinberg-diff-ext1}, and then tensoring with $\bk_B(2\ell\varsigma_s - \ell\alpha_s)$ and observing that
\[
\St_s^* \cong \St_s \otimes \bk_{P_s}((\ell-1)(\alpha_s - 2\varsigma_s))
\]
and
\[
\weyl_s(\ell\varsigma_s)=\coweyl_s(\ell\alpha_s - \ell\varsigma_s)^* \cong \coweyl_s(\ell \varsigma_s)^* \otimes \bk_{P_s}(2\ell\varsigma_s - \ell\alpha_s),
\]
we obtain a short exact sequence
\begin{equation}\label{eqn:steinberg-diff-ext2}
0 \to \bk_B(\ell\varsigma_s - \ell\alpha_s) \to \weyl_s(\ell\varsigma_s) \to \St_s \otimes \bk_B(\varsigma_s) \to 0.
\end{equation}
Since~\eqref{eqn:steinberg-diff-ext1} is not split, this short exact sequence is not split either.

By~\cite[Proposition~II.5.2 \& Corollary~II.5.3]{jantzen}, we have
\[
R\Ind_B^{P_s} \bk_B((\ell-1)\varsigma_s - \ell\alpha_s) \cong \Ind_B^{P_s} \bk_B ((\ell-1)\varsigma_s)[-1] = \St_s[-1].
\]
We therefore have
\begin{multline*}
\Ext^1_B(\St_s \otimes \bk_B(\varsigma_s), \bk_B(\ell\varsigma_s - \ell\alpha_s))
\cong \Ext^1_B(\St_s, \bk_B((\ell-1)\varsigma_s - \ell\alpha_s)) \\
\cong \Hom^1_{P_s}(\St_s, R\Ind_B^{P_s} \bk_B((\ell-1)\varsigma_s - \ell\alpha_s))
\cong \Hom_{P_s}(\St_s, \St_s) \cong \bk.
\end{multline*}
From these considerations we deduce that~\eqref{eqn:steinberg-diff-ext2} is the unique nonsplit extension of $\St_s \otimes \bk_B(\varsigma_s)$ by $\bk_B(\ell\varsigma_s - \ell\alpha_s)$, and then that~\eqref{eqn:steinberg-diff-ext1} is the unique nonsplit extension of $\bk_B(\ell\varsigma_s)$ by $\St_s \otimes \bk_B(\varsigma_s - \alpha_s)$.

We can
finally finish the proof of the lemma. Suppose for a contradiction that the element of $\Ext^2_B(\bk_B(\ell\varsigma_s), \bk_B(\ell\varsigma_s-\ell\alpha_s))$ corresponding to~\eqref{eqn:steinberg-diff} vanishes.  This means that there exists a $B$-module $V$ equipped with a filtration $0 \subset V_1 \subset V_2 \subset V$ such that~\eqref{eqn:steinberg-diff} is isomorphic to
\[
0 \to V_1 \to V_2 \to V/V_1 \to V/V_2 \to 0.
\]
Consider the short exact sequence $0 \to V_1 \to V \to V/V_1 \to 0$.  This extension cannot split, because $V_1 \cong \bk_B(\ell\varsigma_s-\ell\alpha_s)$ is not a direct summand of $V_2 \cong \St_s \otimes \bk_B(\varsigma_s-\alpha_s)$ (since $\St_s$ is indecomposable over $B$).  So from~\eqref{eqn:steinberg-diff-ext2}, we conclude that $V \cong \weyl_s(\ell\varsigma_s)$.  A similar argument using the short exact sequence $0 \to V_2 \to V \to V/V_2 \to 0$ and~\eqref{eqn:steinberg-diff-ext1} shows that $V \cong \coweyl_s(\ell\varsigma_s)$.  But now we have our contradiction, since $\weyl_s(\ell\varsigma_s)$ and $\coweyl_s(\ell\varsigma_s)$ are not isomorphic as $P_s$-modules, and hence not as $B$-modules either, since $\For_B^{P_s}$ is fully faithful. (In fact, both $\weyl_s(\ell\varsigma_s)$ and $\coweyl_s(\ell\varsigma_s)$ are nonsimple and have the simple $M_s$-module $\irr_s(\ell\varsigma_s)$ with highest weight $\ell \varsigma_s$---viewed as a $P_s$-module---as a composition factor with multiplicity $1$, but this module is the top of $\weyl_s(\ell\varsigma_s)$ and the socle of $\coweyl_s(\ell\varsigma_s)$.) This finishes the proof.
\end{proof}

The following lemma gathers well-known properties of the tilting module $\tilt_s(\ell \varsigma_s)$, see e.g.~\cite[Lemma~1.1 \& Lemma~1.3]{dh:dtpmi}.

\begin{lem}
\label{lem:tilting-SL2-l}
The $M_s$-module $\tilt_s(\ell \varsigma_s)$ is isomorphic to $\irr_s((\ell-1)\varsigma_s) \otimes \irr_s(\varsigma_s)$. This module fits into exact sequences
\[
\coweyl_s(\ell \varsigma_s - \alpha_s) \hookrightarrow \tilt_s(\ell \varsigma_s) \twoheadrightarrow \coweyl_s(\ell \varsigma_s) \quad \text{and} \quad \weyl_s(\ell \varsigma_s) \hookrightarrow \tilt_s(\ell \varsigma_s) \twoheadrightarrow \weyl_s(\ell \varsigma_s - \alpha_s).
\]
Moreover we have
\[
\coweyl_s(\ell \varsigma_s - \alpha_s) \cong \weyl_s(\ell \varsigma_s - \alpha_s) \cong \irr_s(\ell \varsigma_s - \alpha_s),
\]
and the modules $\coweyl_s(\ell \varsigma_s)$ and $\weyl_s(\ell \varsigma_s)$ have length $2$, with socle $\irr_s(\ell \varsigma_s)$ and $\irr_s(\ell \varsigma_s - \alpha_s)$ respectively, and top $\irr_s(\ell \varsigma_s - \alpha_s)$ and $\irr_s(\ell \varsigma_s)$ respectively.
\end{lem}

From Lemma~\ref{lem:tilting-SL2-l} we deduce the following fact, which is in fact a special case of a claim in~\cite[Theorem~2.1]{dh:dtpmi}.

\begin{cor}
\label{cor:tilting-SL2}
For any $\lambda \in \bX$ such that $\langle \lambda, \alpha_s^\vee \rangle \geq 0$, the $M_s$-module $\tilt_s(\ell \varsigma_s) \otimes \irr_s(\ell \lambda)$ has top $\irr_s(\ell \lambda + \ell \varsigma_s - \alpha_s)$; in particular, it is indecomposable.
\end{cor}

\begin{proof}
By Steinberg's tensor product theorem, we have
\[
\irr_s(\ell \varsigma_s - \alpha_s) \otimes \irr_s(\ell \lambda) \cong \irr_s(\ell \lambda + \ell \varsigma_s - \alpha_s).
\]
Hence from Lemma~\ref{lem:tilting-SL2-l} we deduce that $\tilt_s(\ell \varsigma_s) \otimes \irr_s(\ell \lambda)$ admits a filtration with sucessive subquotients $\irr_s(\ell \lambda + \ell \varsigma_s - \alpha_s)$, $\irr_s(\ell \varsigma_s) \otimes \irr_s(\ell \lambda)$ and $\irr_s(\ell \lambda + \ell \varsigma_s - \alpha_s)$. In particular, the simple composition factors of this module which are not isomorphic to $\irr_s(\ell \lambda + \ell \varsigma_s - \alpha_s)$ are of the form $\irr_s(\ell \mu)$ with $\mu \in \bX$ such that $\langle \mu, \alpha_s^\vee \rangle \geq 0$.

We claim that no simple $M_s$-module of the form $\irr_s(\ell \mu)$ appears in the top or the socle of $\tilt_s(\ell \varsigma_s) \otimes \irr_s(\ell \lambda)$. We will prove this claim for the top; the case of the socle is similar. We have
\[
\Hom_{M_s}(\tilt_s(\ell \varsigma_s) \otimes \irr_s(\ell \lambda), \irr_s(\ell \mu)) \cong \Hom_{M_s}(\tilt_s(\ell \varsigma_s), \irr_s(\ell \mu) \otimes \irr_s(\ell \lambda)^*).
\]
Now all the composition factors of $\irr_s(\ell \mu) \otimes \irr_s(\ell \lambda)^*$ are of the form $\irr_s(\ell \nu)$, and we have
\[
\Hom_{M_s}(\tilt_s(\ell \varsigma_s), \irr_s(\ell \nu))=0
\]
for any $\nu$ (see Lemma~\ref{lem:tilting-SL2-l}), which implies our claim.

From this claim we deduce in particular that the top of $\tilt_s(\ell \varsigma_s) \otimes \irr_s(\ell \lambda)$ is either $\irr_s(\ell \lambda + \ell \varsigma_s - \alpha_s)$ or $\irr_s(\ell \lambda + \ell \varsigma_s - \alpha_s)^{\oplus 2}$. But the latter case cannot occur, since otherwise the embedding
\[
\irr_s(\ell \lambda + \ell \varsigma_s - \alpha_s) \hookrightarrow \tilt_s(\ell \varsigma_s) \otimes \irr_s(\ell \lambda)
\] 
deduced from the embedding $\irr_s(\ell \varsigma_s - \alpha_s) \hookrightarrow \tilt_s(\ell \varsigma_s)$ would split, and then $\tilt_s(\ell \varsigma_s) \otimes \irr_s(\ell \lambda)$ would have a simple module of the form $\irr_s(\ell \mu)$ in its socle, which does not hold as we have seen.
\end{proof}

\begin{prop}\label{prop:steinberg-tilting}
Let $\lambda \in \bX$ be such that $\langle \lambda, \alpha_s^\vee \rangle \geq 0$.
\begin{enumerate}
\item
\label{it:ind-tilting}
As $P_s$-modules, we have $\Ind_B^{P_s} (\St_s \otimes \bk_B(\varsigma_s) \otimes \irr_s(\ell\lambda)) \cong \tilt_s(\ell\varsigma_s) \otimes \irr_s(\ell\lambda)$. 
\item 
\label{it:ind-surj}
For any nonzero map of $B$-modules
\[
g: \St_s \otimes \bk_B(\varsigma_s) \otimes \irr_s(\ell\lambda) \to \bk_B(\ell\varsigma_s) \otimes \irr_s(\ell\lambda),
\]
the morphism
\[
\Ind_B^{P_s}(g) : \Ind_B^{P_s}(\St_s \otimes \bk_B(\varsigma_s) \otimes \irr_s(\ell\lambda)) \to \Ind_B^{P_s} (\bk_B(\ell\varsigma_s) \otimes \irr_s(\ell\lambda))
\]
is surjective.
\item
\label{it:ind-isom}
Let $\theta$ be an endomorphism of $\Ind_B^{P_s} (\St_s \otimes \bk_B(\varsigma_s) \otimes \irr_s(\ell\lambda))$, and let 
\[
h : \Ind_B^{P_s} (\St_s \otimes \bk_B(\varsigma_s) \otimes \irr_s(\ell\lambda)) \to \Ind_B^{P_s} (\bk_B(\ell\varsigma_s) \otimes \irr_s(\ell\lambda))
\]
be a morphism.
If the composition $h \circ \theta$
is surjective, then $\theta$ is an isomorphism.
\end{enumerate}
\end{prop}

\begin{proof}
\eqref{it:ind-tilting}~By the tensor identity, we have
\begin{multline}
\label{eqn:steinberg-tilting}
\Ind_B^{P_s} (\St_s \otimes \bk_B(\varsigma_s) \otimes \irr_s(\ell\lambda)) \cong (\St_s \otimes \irr_s(\ell\lambda)) \otimes \Ind_B^{P_s} \bk_B(\varsigma_s) \\
\cong \St_s \otimes \irr_s(\ell\lambda) \otimes \irr_s(\varsigma_s).
\end{multline}
Then the claim follows from Lemma~\ref{lem:tilting-SL2-l}.

\eqref{it:ind-surj}~First consider the special case where $\lambda = 0$.  In this case, a nonzero map $\St_s \otimes \bk_B(\varsigma_s) \to \bk_B(\ell\varsigma_s)$ is clearly unique up to a scalar.  Applying the functor $\Ind_B^{P_s}$ yields a map $\tilt_s(\ell\varsigma_s) \to \coweyl_s(\ell\varsigma_s)$, which is nonzero by adjunction.  Now,
the general theory of tilting modules implies that $\Hom_{P_s}(\tilt_s(\ell\varsigma_s), \coweyl_s(\ell\varsigma_s))$ is $1$-dimensional, and that any nonzero map in this space is surjective. This implies the desired claim in the special case $\lambda=0$.

For general $\lambda$, we have
\begin{multline*}
\Hom_B(\St_s \otimes \bk_B(\varsigma_s) \otimes \irr_s(\ell\lambda), \bk_B(\ell\varsigma_s) \otimes \irr_s(\ell\lambda)) \\
\cong \Hom_B(\St_s \otimes \irr_s(\ell\lambda), \bk_B((\ell-1)\varsigma_s) \otimes \irr_s(\ell\lambda)) \\
\cong \Hom_{P_s}(\St_s \otimes \irr_s(\ell\lambda), \Ind_B^{P_s}(\bk_B((\ell-1)\varsigma_s) \otimes \irr_s(\ell\lambda))) \\
\cong \Hom_{P_s}(\St_s \otimes \irr_s(\ell\lambda), \St_s \otimes \irr_s(\ell\lambda)) \cong \bk,
\end{multline*}
where the last step holds because $\St_s \otimes \irr_s(\ell\lambda)$ is simple by Steinberg's tensor product theorem.  This calculation shows that any nonzero map $g: \St_s \otimes \bk_B(\varsigma_s) \otimes \irr_s(\ell\lambda) \to \bk_B(\ell\varsigma_s) \otimes \irr_s(\ell\lambda)$ is of the form $g_0 \otimes \id_{\irr_s(\ell\lambda)}$, where $g_0: \St_s \otimes \bk_B(\varsigma_s) \to \bk_B(\ell\varsigma_s)$ is a nonzero map.  It follows (using the tensor identity) that $\Ind_B^{P_s}(g)$ can be identified with $\Ind_B^{P_s}(g_0) \otimes \id_{\irr_s(\ell\lambda)}$, so that this map is surjective by the special case considered above.

\eqref{it:ind-isom}~If $h \circ \theta$ is surjective, then $h$ is surjective. Now we have
\[
\Ind_B^{P_s} (\bk_B(\ell\varsigma_s) \otimes \irr_s(\ell\lambda)) \cong \coweyl_s(\ell\varsigma_s) \otimes \irr_s(\ell\lambda)
\]
by the tensor identity. From this, Lemma~\ref{lem:tilting-SL2-l} and Steinberg's tensor product theorem, we deduce that there exists a surjection
\[
\Ind_B^{P_s} (\bk_B(\ell\varsigma_s) \otimes \irr_s(\ell\lambda)) \twoheadrightarrow \irr_s(\ell \varsigma_s - \alpha_s + \ell \lambda).
\]
This implies that the surjection from $\Ind_B^{P_s} (\St_s \otimes \bk_B(\varsigma_s) \otimes \irr_s(\ell\lambda)) \cong \tilt_s(\ell\varsigma_s) \otimes \irr_s(\ell\lambda)$ (see~\eqref{it:ind-tilting}) to its top $\irr_s(\ell \varsigma_s - \alpha_s + \ell \lambda)$ (see Corollary~\ref{cor:tilting-SL2}) factors through $h$. Hence from our assumption we obtain that the composition
\begin{multline*}
\Ind_B^{P_s} (\St_s \otimes \bk_B(\varsigma_s) \otimes \irr_s(\ell\lambda)) \xrightarrow{\theta} \Ind_B^{P_s} (\St_s \otimes \bk_B(\varsigma_s) \otimes \irr_s(\ell\lambda)) \\
\twoheadrightarrow \mathrm{top} \bigl( \Ind_B^{P_s} (\St_s \otimes \bk_B(\varsigma_s) \otimes \irr_s(\ell\lambda)) \bigr)
\end{multline*}
is nonzero, which implies that $\theta$ is surjective, and then an isomorphism since it is an endomorphism of a finite-dimensional module.
\end{proof}

\section{Koszul duality}
\label{sec:koszul}

In this section we fix a field $\F$, an algebraic $\F$-group scheme $H$, and a finite-dimensional $H$-module $V$. We review the construction and main properties of the \emph{Koszul duality} equivalence relating dg-modules over the exterior algebra of $V$ and dg-modules over the symmetric algebra of $V^*$. The version we use is essentially the version of~\cite{gkm}, but the construction given there has the annoying feature that it requires unnatural boundedness conditions on the dg-modules. Here we use slightly different arguments, which require introducing an extra grading, but allow us to get rid of these conditions. These arguments are very close to those of~\cite{mr2}, so we omit most proofs.

\subsection{Reminder on Koszul duality}
\label{ss:reminder-Koszul}

Let us consider the dg-algebra
\[
\bL:=\bigwedge \hspace{-3pt} {}^\bullet \, V,
\]
where $V$ is placed in degree $-1$, and the differential is trivial. We will consider the $H \times \Gm$-action on $\bL$ which is compatible with the multiplication in the obvious sense, and where $H$, resp.~$\Gm$, acts on $V$ via its natural action, resp.~in such a way that $z \in \Gm$ acts by dilation by $z^{-2}$. In this way $\bL$ can be considered as an $H \times \Gm$-equivariant dg-algebra, and we can consider the category $\bL\ldgmod_{H \times \Gm}$ of $H \times \Gm$-equivariant $\bL$-dg-modules as in~\S\ref{ss:equiv-dgmod}, the corresponding derived category $D_{H \times \Gm}(\bL)$, and the full subcategory $\Dfg_{H \times \Gm}(\bL)$. 

The $\Gm$-action on an $H \times \Gm$-equivariant $\bL$-dg-module will rather be regarded as an extra $\Z$-grading on the dg-module, which we will call the \emph{internal} grading. Using this point of view
we can consider the full subcategory $\bL\ldgmod_{H \times \Gm}^{\boxplus}$ of  $\bL\ldgmod_{H \times \Gm}$ consisting of objects whose internal grading is bounded below, and the corresponding derived category $D^{\boxplus}_{H \times \Gm}(\bL)$. (This category shouldn't be confused with the derived category $D^{+}_{H \times \Gm}(\bL)$ of equivariant $\bL$-dg-modules which are bounded below for the \emph{cohomological} grading.) The embedding of $\bL\ldgmod_{H \times \Gm}^{\boxplus}$ in $\bL\ldgmod_{H \times \Gm}$ induces a functor $D^{\boxplus}_{H \times \Gm}(\bL) \to D_{H \times \Gm}(\bL)$, which is easily seen to be fully faithful. The essential image of this functor contains $\Dfg_{H \times \Gm}(\bL)$, so that $\Dfg_{H \times \Gm}(\bL)$ can be considered as a full subcategory in $D^{\boxplus}_{H \times \Gm}(\bL)$.

We will also consider the dg-algebra
\[
\bS := \mathrm{Sym}(V^*),
\]
where $V^*$ is placed in degree $2$, and the differential is trivial. We will consider the $H \times \Gm$-action on $\bS$ which is compatible with the multiplication in the obvious sense, and where $H$, resp.~$\Gm$, acts on $V^*$ via its natural action, resp.~in such a way that $z \in \Gm$ acts by dilation by $z^{2}$. In this way $\bS$ can be considered as an $H \times \Gm$-equivariant dg-algebra, and we can consider the category $\bS\ldgmod_{H \times \Gm}$ of $H \times \Gm$-equivariant $\bS$-dg-modules as in~\S\ref{ss:equiv-dgmod}, the corresponding derived category $D_{H \times \Gm}(\bS)$, and the full subcategory $\Dfg_{H \times \Gm}(\bS)$. As above one can also consider the category $\bS\ldgmod_{H \times \Gm}^{\boxplus}$ of $H \times \Gm$-equivariant dg-modules whose internal grading is bounded below, and the corresponding derived category $D^{\boxplus}_{H \times \Gm}(\bS)$. Again it is easily checked that the natural functor $D^{\boxplus}_{H \times \Gm}(\bS) \to D_{H \times \Gm}(\bS)$ is fully faithful, and that its essential image contains the subcategory $\Dfg_{H \times \Gm}(\bS)$.

We will denote by
\[
\langle 1 \rangle : D^{\boxplus}_{H \times \Gm}(\bL) \to D^{\boxplus}_{H \times \Gm}(\bL) \quad \text{and} \quad \langle 1 \rangle : D^{\boxplus}_{H \times \Gm}(\bS) \to D^{\boxplus}_{H \times \Gm}(\bS)
\]
the functors of tensoring with the tautological $1$-dimensional $\Gm$-module.

The goal of this subsection is to outline the proof of the following result.

\begin{thm}
\label{thm:koszul}
There exists an equivalence of triangulated categories
\[
\Koszul : D^{\boxplus}_{H \times \Gm}(\bS) \simto D^{\boxplus}_{H \times \Gm}(\bL)
\]
which commutes with the functors $\langle 1 \rangle$. This equivalence restricts to an equivalence of triangulated categories
\[
\Dfg_{H \times \Gm}(\bS) \simto \Dfg_{H \times \Gm}(\bL).
\]
\end{thm}

\begin{proof}[Sketch of proof]
As in~\cite{mr2} we consider functors
\[
\mathscr{A} : \bL\ldgmod_{H \times \Gm}^{\boxplus} \to \bS\ldgmod_{H \times \Gm}^{\boxplus}, \quad \mathscr{B} : \bS\ldgmod_{H \times \Gm}^{\boxplus} \to \bL\ldgmod_{H \times \Gm}^{\boxplus}
\]
defined by
\[
\mathscr{A}(M) = \bS \otimes_\F M, \qquad \mathscr{B}(N) = \Hom_\F(\bL, N),
\]
where the $\bS$-action (respectively the $\bL$-action), the differential and the grading are defined as in~\cite[\S 2.2]{mr2}. (In each case the differential is obtained as the sum of the natural differential with a ``Koszul type'' differential.) One can check as in~\cite[Theorem~2.6(i)]{mr2} that these functors send acyclic complexes to acyclic complexes, and hence that they induce triangulated functors
\[
\overline{\mathscr{A}} : D_{H \times \Gm}^{\boxplus}(\bL) \to D_{H \times \Gm}^{\boxplus}(\bS), \quad \overline{\mathscr{B}} : D_{H \times \Gm}^{\boxplus}(\bS) \to D_{H \times \Gm}^{\boxplus}(\bL).
\]
Next, as in~\cite[Theorem~2.6(ii)]{mr2} one checks that these functors are quasi-inverse to each other, and we obtain the desired equivalence $\Koszul := \overline{\mathscr{B}}$. Finally, arguments similar to those in the proof of~\cite[Proposition~2.11]{mr2} imply that $\overline{\mathscr{A}}$, resp.~$\overline{\mathscr{B}}$, sends $\Dfg_{H \times \Gm}(\bL)$ into $\Dfg_{H \times \Gm}(\bS)$, resp.~$\Dfg_{H \times \Gm}(\bS)$ into $\Dfg_{H \times \Gm}(\bL)$. The second statement follows.
\end{proof}

\begin{rmk}
The equivalence constructed (in a much more general setting) in~\cite{mr2} differs from
the equivalence of Theorem~\ref{thm:koszul} by composition with duality. This turns out to be a crucial idea in order to obtain the general equivalence considered in~\cite{mr2}.
\end{rmk}

\subsection{Regrading and forgetting the grading}
\label{ss:regrading}

The version of Koszul duality we will use later is not exactly the one provided by Theorem~\ref{thm:koszul}. First, consider the category $\bS\lmod_{H \times \Gm}$ of $H \times \Gm$-equivariant $\bS$-modules, and the corresponding derived category $D(\bS\lmod_{H \times \Gm})$. Let also $\bS\lmod_{H \times \Gm}^{\mathrm{fg}}$ be the full subcategory of $\bS\lmod_{H \times \Gm}$ consisting of finitely generated modules. Then it is well known that the natural functor
\[
\Db(\bS\lmod_{H \times \Gm}^{\mathrm{fg}}) \to D(\bS\lmod_{H \times \Gm})
\]
is fully faithful, and that its essential image is the subcategory of complexes whose total cohomology is finitely generated.

Let $C(\bS\lmod_{H \times \Gm})$ be the category of chain complexes of objects of $\bS\lmod_{H \times \Gm}$. If $M$ is in $\bS\lmod_{H \times \Gm}$, as in~\S\ref{ss:reminder-Koszul} we will consider the $\Gm$-action on $M$ as an ``internal'' grading $M=\bigoplus_i M_i$. Then we consider the functor
\[
\xi : C(\bS\lmod_{H \times \Gm}) \to \bS\ldgmod_{H \times \Gm}
\]
sending a complex $(M^i)_{i \in \Z}$ to the dg-module $\xi(M)$ whose $n$-th term is
\[
\xi(M)^n = \bigoplus_{i+j=n} M^i_j,
\]
with the natural differential, $\bS$-action, and $H$-action, and where $\Gm$ acts on $M^i_j \subset \xi(M)^{i+j}$ with weight $j$. It is clear that $\xi$ is an equivalence of categories; therefore it induces an equivalence of triangulated categories
\[
\xi : D(\bS\lmod_{H \times \Gm}) \simto D_{H \times \Gm}(\bS)
\]
which satisfies
\[
\xi \circ \langle 1 \rangle = \langle 1 \rangle [-1] \circ \xi.
\]
It is clear also from the comments above that $\xi$ induces an equivalence of triangulated categories
\[
\Db(\bS\lmod_{H \times \Gm}^{\mathrm{fg}}) \simto \Dfg_{H \times \Gm}(\bS),
\]
which we will again denote $\xi$.

Consider now the functor
\[
\For^{H \times \Gm}_H : D_{H \times \Gm}(\bL) \to D_H(\bL)
\]
associated with the obvious embedding $H=H \times \{1\} \hookrightarrow H \times \Gm$.

\begin{lem}
\label{lem:forget-grading}
For any $M$ in $\Dfg_{H \times \Gm}(\bL)$ and any $N$ in $D^+_{H \times \Gm}(\bL)$, the functor $\For^{H \times \Gm}_H$ induces an isomorphism
\begin{equation}
\label{eqn:forget-grading}
\bigoplus_{n \in \Z} \Hom_{D_{H \times \Gm}(\bL)}(M,N \langle n \rangle) \simto \Hom_{D_{H}(\bL)}(\For^{H \times \Gm}_H M, \For^{H \times \Gm}_H N).
\end{equation}
\end{lem}

\begin{proof}
Using truncation functors and the five-lemma we can assume that $M$ is finite-dimensional and concentrated in a single degree. In the proof of Lemma~\ref{lem:K-inj-equiv} we have seen how to construct an object $N'$ which is K-injective as an $H \times \Gm$-equivariant $\bL$-dg-module and a quasi-isomorphism $N \xrightarrow{\qis} N'$. Looking at this construction, and using the fact that any injective $H \times \Gm$-module is also injective as an $H$-module (as can be deduced from~\cite[Propositions~I.3.9(c) and~I.3.10(b)]{jantzen}), one can easily check that $N'$ is also K-injective as an $H$-equivariant $\bL$-dg-module. It follows that the left-hand side in~\eqref{eqn:forget-grading} is the $0$-th cohomology of the complex
\[
\bigoplus_{n \in \Z} \Hom^\bullet_{\bL\ldgmod_{H \times \Gm}}(M,N' \langle n \rangle),
\]
while the right-hand side is the $0$-th cohomology of the complex
\[
\Hom^\bullet_{\bL\ldgmod_{H}}(M,N').
\]
The functor $\For^{H \times \Gm}_H$ clearly induces an isomorphism between these two complexes, and the claim follows.
\end{proof}

We finally set
\[
\kappa := \For^{H \times \Gm}_H \circ \Koszul \circ \xi : \Db(\bS\lmod_{H \times \Gm}^{\mathrm{fg}}) \to \Dfg_H(\bL).
\]
This functor is endowed with a canonical isomorphism
\[
\kappa \circ \langle 1 \rangle[1] \cong \kappa.
\]
Moreover, it follows from Lemma~\ref{lem:forget-grading} that, for any $M,N$ in $\Db(\bS\lmod_{H \times \Gm}^{\mathrm{fg}})$, $\kappa$ and this isomorphism induce an isomorphism
\begin{equation}
\label{eqn:forget-grading-kappa}
\bigoplus_{n \in \Z} \Hom_{\Db(\bS\lmod_{H \times \Gm}^{\mathrm{fg}})}(M,N \langle n \rangle [n]) \simto \Hom_{\Dfg_H(\bL)}(\kappa M, \kappa N).
\end{equation}

\subsection{Compatibilities}
\label{ss:compatibilities-Koszul}

Let now $V' \subset V$ be an $H$-stable subspace. Then we can consider the dg-algebras $\bL$ and $\bS$ as above, but also the similar dg-algebras
\[
\bL' := \bigwedge \hspace{-3pt} {}^\bullet \, V', \qquad \bS':=\mathrm{Sym}((V')^*)
\]
attached to $V'$, and the corresponding functor $\kappa'$. The embedding $V' \hookrightarrow V$ induces an embedding $e : \bL' \hookrightarrow \bL$ and a surjection $f : \bS \twoheadrightarrow \bS'$. Therefore we can consider the functors
\[
e^* : \Dfg_H(\bL) \to \Dfg_H(\bL'), \qquad \bS' \lotimes_{\bS} (-) : \Db(\bS\lmod_{H \times \Gm}^{\mathrm{fg}}) \to \Db(\bS'\lmod_{H \times \Gm}^{\mathrm{fg}}).
\]

\begin{prop}
\label{prop:koszul-compatibility-1}
There exists a canonical isomorphism of functors making the following diagram commutative:
\[
\xymatrix@C=1.5cm{
\Db(\bS\lmod_{H \times \Gm}^{\mathrm{fg}}) \ar[d]_-{\bS' \lotimes_{\bS} (-)} \ar[r]^-{\kappa} & \Dfg_H(\bL) \ar[d]^-{e^*} \\
\Db(\bS'\lmod_{H \times \Gm}^{\mathrm{fg}}) \ar[r]^-{\kappa'} & \Dfg_H(\bL').
}
\]
\end{prop}

\begin{proof}
Consider the functor
\[
\bS' \otimes_{\bS} (-) : \bS\ldgmod_{H \times \Gm}^{\boxplus} \to  \bS'\ldgmod_{H \times \Gm}^{\boxplus}.
\]
It is easily checked that there are enough objects in $\bS\ldgmod_{H \times \Gm}^{\boxplus}$ which are K-flat as $\bS$-dg-modules, and this implies that this functor admits a left derived functor
\[
\bS' \lotimes_{\bS} (-) : D_{H \times \Gm}^{\boxplus}(\bS) \to D_{H \times \Gm}^{\boxplus}(\bS').
\]
Then to prove the proposition it suffices to construct an isomorphism of functors making the following square commutative:
\begin{equation}
\label{eqn:koszul-comp}
\vcenter{
\xymatrix@C=2cm{
D_{H \times \Gm}^{\boxplus}(\bS) \ar[r]^-{\Koszul} \ar[d]_-{\bS' \lotimes_{\bS} (-)} & D_{H \times \Gm}^{\boxplus}(\bL) \ar[d]^-{e^*} \\
D_{H \times \Gm}^{\boxplus}(\bS') \ar[r]^-{\Koszul'} & D_{H \times \Gm}^{\boxplus}(\bL').
}
}
\end{equation}

The left vertical arrow in~\eqref{eqn:koszul-comp} is left-adjoint to the functor $f^* : D_{H \times \Gm}^{\boxplus}(\bS') \to D_{H \times \Gm}^{\boxplus}(\bS)$. And since $\bL$ is free over $\bL'$, the functor $e_*$ induces a functor $Re_* : D_{H \times \Gm}^{\boxplus}(\bL') \to D_{H \times \Gm}^{\boxplus}(\bL)$, which is right-adjoint to $e^*$. Since the horizontal arrows in~\eqref{eqn:koszul-comp} are equivalences, to prove that this diagram is commutative it suffices to prove that the following diagram is commutative:
\begin{equation}
\label{eqn:koszul-comp2}
\vcenter{
\xymatrix@C=2cm{
D_{H \times \Gm}^{\boxplus}(\bS) \ar[r]^-{\Koszul} & D_{H \times \Gm}^{\boxplus}(\bL) \\
D_{H \times \Gm}^{\boxplus}(\bS') \ar[r]^-{\Koszul'} \ar[u]^-{f^*} & D_{H \times \Gm}^{\boxplus}(\bL'). \ar[u]_-{Re_*}
}
}
\end{equation}

Now, recall the functor $\mathscr{B}$ considered in the proof of Theorem~\ref{thm:koszul}, and let $\mathscr{B}'$ be the similar functor associated with $V'$. Then by construction we have an isomorphism
\[
e_* \circ \mathscr{B}' \cong \mathscr{B} \circ f^*.
\]
Since all the functors considered here are exact, we deduce the desired commutativity of~\eqref{eqn:koszul-comp2}.
\end{proof}

Let now $K \subset H$ be a closed subgroup, and assume that $H/K$ is a projective noetherian scheme. Then we can consider the functor $\kappa$ in the $H$-equivariant setting or in the $K$-equivariant setting. 

On the $\bL$-side, we can consider the functor
\[
\For^H_K : \Dfg_H(\bL) \to \Dfg_K(\bL),
\]
and its right adjoint
\[
R\Ind_K^H : \Dfg_K(\bL) \to \Dfg_H(\bL),
\]
see~\S\ref{ss:Kinj-equivariant}. (The fact that this functor restricts to a functor between the categories of objects with finitely generated cohomology follows from the commutativity of diagram~\eqref{eqn:diag-for-ind} and~\cite[Proposition~I.5.12]{jantzen}.)

On the $\bS$-side, we can also consider the functor
\[
\For^{H \times \Gm}_{K \times \Gm} : \Db(\bS\lmod_{H \times \Gm}^{\mathrm{fg}}) \to \Db(\bS\lmod_{K \times \Gm}^{\mathrm{fg}}).
\]
The category $\bS\lmod_{K \times \Gm}$ identifies with the category $\QCoh^{K \times \Gm}(V)$ of $K \times \Gm$-equivariant quasi-coherent sheaves on $V$. From this point of view, it is well known that it admits enough injective objects, see e.g.~\cite[\S A.2]{mr:etspc}. Using the same procedure as in~\S\ref{ss:Kinj-equivariant} we see that the functor $\Ind_{K \times \Gm}^{H \times \Gm}$ induces a functor from $\bS\lmod_{K \times \Gm}$ to $\bS\lmod_{H \times \Gm}$, which we will also denote $\Ind_{K \times \Gm}^{H \times \Gm}$. Since the category $\bS\lmod_{K \times \Gm}$ has enough injective objects, this functor admits a right derived functor
\[
R\Ind_{K \times \Gm}^{H \times \Gm} : D^+(\bS\lmod_{K \times \Gm}) \to D^+(\bS\lmod_{H \times \Gm}).
\]
From the point of view of quasi-coherent sheaves, the functor $\Ind_{K \times \Gm}^{H \times \Gm}$ identifies with the composition of the ``induction equivalence''
\[
\QCoh^{K \times \Gm}(V) \cong \QCoh^{H \times \Gm}((H \times \Gm) \times^{K \times \Gm} V)
\]
with the direct image functor associated with the morphism
\[
(H \times \Gm) \times^{K \times \Gm} V \to V
\]
induced by the $H \times \Gm$-action on $V$. This morphism is projective under our assumptions, and using the compatibility between equivariant and ordinary direct image functors (see~\cite[Proposition~A.10]{mr:etspc}), we deduce that $R\Ind_{K \times \Gm}^{H \times \Gm}$ restricts to a functor
\[
R\Ind_{K \times \Gm}^{H \times \Gm} : \Db(\bS\lmod_{K \times \Gm}^{\mathrm{fg}}) \to \Db(\bS\lmod_{H \times \Gm}^{\mathrm{fg}}),
\]
which is right adjoint to the functor $\For^{H \times \Gm}_{K \times \Gm}$.

\begin{prop}
\label{prop:koszul-compatibility-2}
There exist canonical isomorphisms of functors making the following diagrams commutative:
\[
\xymatrix{
\Db(\bS\lmod_{H \times \Gm}^{\mathrm{fg}}) \ar[d]_-{\For^{H \times \Gm}_{K \times \Gm}} \ar[r]^-{\kappa} & \Dfg_H(\bL) \ar[d]^-{\For^H_K} \\
\Db(\bS\lmod_{K \times \Gm}^{\mathrm{fg}}) \ar[r]^-{\kappa} & \Dfg_K(\bL),
}
\qquad
\xymatrix{
\Db(\bS\lmod_{K \times \Gm}^{\mathrm{fg}}) \ar[r]^-{\kappa}\ar[d]_-{R\Ind_{K \times \Gm}^{H \times \Gm}} & \Dfg_K(\bL) \ar[d]^-{R\Ind_K^H} \\
\Db(\bS\lmod_{H \times \Gm}^{\mathrm{fg}}) \ar[r]^-{\kappa} & \Dfg_H(\bL).
}
\]
\end{prop}

\begin{proof}
It is enough to prove similar compatibilities for the functor $\Koszul \circ \xi$. In this setting the commutativity of the first diagram is obvious, and the commutativity of the second one follows by adjunction.
\end{proof}

\newpage

\part{Formality theorems}
\label{pt:formality}

\textbf{Overview.}
This part of the paper contains the proof of the Formality theorem (Theorem~\ref{thm:intro-formality}).  First, in Section~\ref{sec:formality-non-equiv} we prove a formality theorem for a derived category of representations of the Frobenius kernel $P_{I,1}$ of $P_I$. Then in Section~\ref{sec:equiv-formality} we upgrade this to an ``equivariant'' version, containing all of Theorem~\ref{thm:intro-formality} except the commutative diagram.  Finally, that commutative diagram is established in Section~\ref{sec:compatibility}.

\section{Formality for \texorpdfstring{$P_{I,1}$}{PI1}-modules}
\label{sec:formality-non-equiv}

In this section we fix a subset $I \subset S$.
We will denote by
\[
\Db_{\Stein}(P_{I,1})
\]
the full triangulated subcategory of the category $\Db \Repf(P_{I,1})$ generated by the object $\St_I$. The goal of this section is to describe this category in terms of differential graded modules over the exterior algebra of $\dot\fn_I$.

\subsection{A differential graded resolution of $\snn_I$}
\label{ss:RnI}

Recall the algebras $\fN_I$, $\fZ_I$, and $\snn_I$ introduced in~\S\ref{ss:notation-alg-groups}.
Let $\bL_I = \bigwedge^\bullet \fnt_I$, considered as a dg-algebra as in Section~\ref{sec:koszul}, and
consider the graded algebra
\[
\bR_I:=\bL_I \otimes \fZ_I.
\]
This algebra identifies with the (graded-)symmetric algebra of the complex $\dot\fn_I \xrightarrow{\id} \dot\fn_I$, where the first term is in degree $-1$. Therefore, it admits a natural differential which satisfies the (graded) Leibniz rule; in other words it admits a natural structure of a differential graded algebra. Moreover, the differential is $\fZ_I$-linear, and we have
\begin{equation}
\label{eqn:cohomology-R}
\coH^n(\bR_I) = \begin{cases}
\bk & \text{if $n=0$;} \\
0 & \text{otherwise.}
\end{cases}
\end{equation}
(In fact, decomposing $\dot\fn_I$ as a direct sum of $1$-dimensional vector spaces, we see that the complex $\bR_I$ is a tensor product of $\dim(\dot\fn_I)$ copies of the complex $\bk[X] \xrightarrow{X \cdot (-)} \bk[X]$ where $X$ is an indeterminate.) Hence the natural morphism of complexes of $\fZ_I$-modules $\bR_I \to \bk$ induced by the augmentation $\fZ_I \to \bk$ is a quasi-isomorphism.

A major role in our arguments will be played by the differential graded algebra
\[
\Rn_I := \bR_I \otimes_{\fZ_I} \fN_I.
\]
Since $\fN_I$ is flat (in fact, free) over $\fZ_I$, by~\eqref{eqn:cohomology-R} we have
\[
\coH^n(\Rn_I) = \begin{cases}
\snn_I & \text{if $n=0$;} \\
0 & \text{otherwise.}
\end{cases}
\]
Hence the morphism of differential graded algebras $\pi_I : \Rn_I \to \snn_I$ induced by the morphism $\bR_I \to \bk$ (where $\snn_I$ is considered as a differential graded algebra concentrated in degree $0$, with trivial differential) is a quasi-isomorphism.

The differential graded algebras $\Rn_I$ and $\snn_I$ admit natural actions of $P_I$ induced by the adjoint action of $P_I$ on $\fn_I$, and the quasi-isomorphism $\pi_I$ is $P_I$-equivariant. By restriction, we deduce actions of $M_I$ and $M_{I,1}$, such that $M_{I,1}$ acts trivially on the subalgebra $\bL_I \subset \Rn_I$.

Note that the functor
\begin{equation}
\label{eqn:RpiI}
R\pi_{I*} : D^+_{M_{I,1}}(\Rn_I) \to D^+_{M_{I,1}}(\snn_I)
\end{equation}
is well defined; see Remark~\ref{rmk:infinitesimal-groups-2}.

\begin{lem}
\label{lem:piI-equivalence}
The functor
\[
\pi_I^* : D_{M_{I,1}}(\snn_I) \to D_{M_{I,1}}(\Rn_I)
\]
is an equivalence of categories, which restricts to an equivalence
\begin{equation}
\label{eqn:equiv-piI}
\Dfg_{M_{I,1}}(\snn_I) \simto \Dfg_{M_{I,1}}(\Rn_I).
\end{equation}
The functor~\eqref{eqn:RpiI} is also an equivalence, and it restricts to a functor $\Dfg_{M_{I,1}}(\Rn_I) \to \Dfg_{M_{I,1}}(\snn_I)$ which is a quasi-inverse to~\eqref{eqn:equiv-piI}.
\end{lem}

\begin{proof}
The first claim follows from Lemma~\ref{lem:qis-equiv}.
By the same lemma,
the functor $\pi_I^*$ restricts to an equivalence $D^+_{M_{I,1}}(\snn_I) \simto D^+_{M_{I,1}}(\Rn_I)$. The right adjoint $R\pi_{I*} : D^+_{M_{I,1}}(\Rn_I) \to D^+_{M_{I,1}}(\snn_I)$ to this restriction must therefore be its quasi-inverse; in particular it must be an equivalence. Finally, for $X$ in $\Dfg_{M_{I,1}}(\Rn_I)$ we have
\[
\pi_I^* R\pi_{I*}(X) \cong X,
\]
which implies that $R \pi_{I*}(X)$ has finitely generated cohomology.
\end{proof}

\subsection{A crucial vanishing lemma}

Note that the category of $M_{I,1}$-equivariant $\Rn_I$-dg-modules is canonically equivalent to the category of modules over the semi-direct product $\Rn_I \rtimes \smm_I$, where $\smm_I$ is the restricted enveloping algebra of $\fm_I$, or equivalently the distribution algebra of $M_{I,1}$; see Remark~\ref{rmk:infinitesimal-groups-2}. The same consideration applies to $\bL_I$-modules.

Now, consider the dg-subalgebra $\fN_I \subset \Rn_I \rtimes \smm_I$. This dg-subalgebra is normal, $\Rn_I \rtimes \smm_I$ is K-flat as a right $\fN_I$-dg-module, and we have $(\Rn_I \rtimes \smm_I) \aq \fN_I \cong \bL_I \rtimes \smm_I$.
Hence we can apply the results of~\S\ref{ss:normal-subalg} in this setting, and in particular consider the object $R\Hom^\bullet_{\fN_I}(\bk,\bk)$
in $D_{M_{I,1}}(\bL_I)$. Since the dg-algebra $\bL_I$ is concentrated in nonpositive degrees, the usual truncation functors for complexes define functors on $D_{M_{I,1}}(\bL_I)$. We set
\begin{equation}
\label{eqn:RHom-fN}
R\Hom^{>0}_{\fN_I}(\bk,\bk) := \tau_{>0} \bigl( R\Hom^\bullet_{\fN_I}(\bk,\bk) \bigr).
\end{equation}
Then, considering similar constructions for the dg-subalgebra $\bL_I \subset \bL_I \rtimes \smm_I$, we can form the object
\[
R\Hom_{\bL_I}(\bk, R\Hom^{>0}_{\fN_I}(\bk,\bk))
\]
in $D^+(\smm_I) \cong D^+ \Rep(M_{I,1})$.

\begin{lem}
\label{lem:cohomology-fNI}
For any $k \in \Z$, the $P_I$-module $\Ext^{k}_{\fN_I}(\bk,\bk)$ is a subquotient of $(\bigwedge\hspace{-3pt}{}^k \, \fn_I)^*$. 
\end{lem}

\begin{proof}
We can compute $\Ext^{k}_{\fN_I}(\bk,\bk)$ using the Chevalley--Eilenberg complex, which provides a $P_I$-equivariant projective resolution of the trivial $\fN_I$-module, see e.g.~\cite[Theorem~7.7.2]{weibel}. In this way we see that this $P_I$-module is the $k$-th cohomology of a complex whose underlying graded vector space is $(\bigwedge\hspace{-3pt}{}^k \, \fn_I)^*$, and the claim follow.
\end{proof}

The main result of this subsection is the following technical result.

\begin{lem}
\label{lem:vanishing-RHom}
We have
\[
R\Hom_{M_{I,1}}(\St_I, R\Hom_{\bL_I}(\bk, R\Hom^{>0}_{\fN_I}(\bk,\bk)) \otimes \St_I)=0.
\]
\end{lem}

\begin{proof}
It follows in particular from Lemma~\ref{lem:cohomology-fNI} that the object $R\Hom^{>0}_{\fN_I}(\bk,\bk)$ has bounded cohomology. Using truncation functors, we deduce that to prove the lemma, it suffices to show that for any $k>0$ we have
\[
R\Hom_{M_{I,1}}(\St_I, R\Hom_{\bL_I}(\bk, \Ext^k_{\fN_I}(\bk,\bk)) \otimes \St_I)=0,
\]
where $\Ext^k_{\fN_I}(\bk,\bk)$ is considered as a (trivial) $\bL_I$-dg-module concentrated in degree $0$. Then, to prove this fact it is enough to prove that
\begin{equation}
\label{eqn:vanishing-RHom-1}
R\Hom_{M_{I,1}}(\St_I, R\Hom_{\bL_I}(\bk, \bk) \otimes \Ext^k_{\fN_I}(\bk,\bk) \otimes \St_I)=0 \quad \text{for any $k>0$.}
\end{equation}
And since $M_{I,1}$ acts trivially on $\bL_I$, the complex $R\Hom_{\bL_I}(\bk, \bk) \in D^+ \Rep(M_{I,1})$ is isomorphic to a direct sum of shifts of trivial modules, so that~\eqref{eqn:vanishing-RHom-1} reduces to the claim that
\begin{equation*}
R\Hom_{M_{I,1}}(\St_I, \Ext^k_{\fN_I}(\bk,\bk) \otimes \St_I)=0 \quad \text{for any $k>0$.}
\end{equation*}
Finally, since $\St_I$ is projective as an $M_{I,1}$-module (see~\S\ref{ss:Steinberg}), to prove this, we must show that
\begin{equation}
\label{eqn:vanishing-RHom-2}
\Hom_{M_{I,1}}(\St_I, \Ext^k_{\fN_I}(\bk,\bk) \otimes \St_I)=0 \quad \text{for any $k>0$.}
\end{equation}

By Lemma~\ref{lem:cohomology-fNI}, all the $T$-weights in the $M_I$-module $\Ext^k_{\fN_I}(\bk,\bk)$ are of the form
\[
-\sum_{\alpha \in F} \alpha
\]
where $F \subset \Phi^+ \setminus \Phi_I^+$ is a subset of cardinality $k$. By~\cite[Lemma~II.12.10]{jantzen}, under our assumptions that $k \neq 0$ and $\ell>h$, no such weight belongs to $\ell \bX$. Now, using
Lemma~\ref{lem:st-sll}, the tensor identity, and the fact that the induction functor $\Ind_{B_{I,1}}^{M_{I,1}}(-)$ is exact, 
we see that the $M_{I,1}$-module $\Ext^k_{\fN_I}(\bk,\bk) \otimes \St_I$ admits a (finite) filtration with subquotients of the form
\[
\Ind_{B_{I,1}}^{M_{I,1}}(\bk_{B_{I,1}}((\ell-1)\varsigma_I + \nu)),
\]
where $\nu$ is a $T$-weight of $\Ext^k_{\fN_I}(\bk,\bk)$. As explained above, no weight of the form $(\ell-1)\varsigma_I + \nu$ belongs to $W_I \bullet (\ell-1)\varsigma_I + \ell \bX = (\ell-1)\varsigma_I + \ell \bX$. By the linkage principle for $M_{I,1}$-modules (see~\cite[Corollary~II.9.12]{jantzen}), it follows that the simple module $\St_I$ is not a composition factor of any subquotient of this filtration. This proves~\eqref{eqn:vanishing-RHom-2} and finishes the proof.
\end{proof}

\subsection{From $\bL_I$-modules to $\Rn_I$-modules}
\label{ss:fromLtoRn}

Recall that $M_{I,1}$ acts trivially (in other words through the quotient $M_{I,1} \to \{1\}$) on $\bL_I$. Therefore, we can consider the functor
\[
\For^{\{1\}}_{M_{I,1}} : D(\bL_I) \to D_{M_{I,1}}(\bL_I).
\]
On the other hand, for any $V$ in $\Rep(M_{I,1})$ one can consider the functor
\[
(-) \otimes V : D_{M_{I,1}}(\bL_I) \to D_{M_{I,1}}(\bL_I).
\]
(Here, $\bL_I$ acts on $X \otimes V$ via its action on $X$, and $M_{I,1}$ acts diagonally.) In particular we can consider the object $\bk \otimes V$, where $\bk$ is the trivial dg-module; this object will simply be denoted $V$.
Using this convention, we denote by
\[
\Dfg_{\Stein}(\bL_I)
\]
the full triangulated subcategory of the category $\Dfg_{M_{I,1}}(\bL_I)$ generated by
$\St_I$. 

\begin{lem}
\label{lem:stein-ff}
The functor $\Dfg(\bL_I) \to \Dfg_{\Stein}(\bL_I)$ given by
\[
V \mapsto \For^{\{1\}}_{M_{I,1}} (V) \otimes \St_I
\]
is fully faithful.
\end{lem}

\begin{proof}
First, we observe that
the category $\Dfg(\bL_I)$ is generated, as a triangulated category, by $\bk$. Indeed, since $\bL_I$ is concentrated in nonpositive degrees, the usual truncation functors for complexes induce functors for $\bL_I$-modules. Then using these truncation functors we see that that the category $\Dfg(\bL_I)$ is generated (again as a triangulated category) by differential graded modules which are concentrated in degree $0$. Such objects are direct sums of copies of $\bk$, and the claim is proved.

Using this claim, to prove the lemma it suffices to prove that the morphism
\[
\Ext^\bullet_{\bL_I}(\bk,\bk) \to \Ext^\bullet_{D_{M_{I,1}}(\bL_I)}(\St_I,\St_I)
\]
induced by our functor
is an isomorphism. Now,
since $\St_I$ is a projective $M_{I,1}$-module with $\End_{M_{I,1}}(\St_I) \cong \bk$ (see~\S\ref{ss:Steinberg}), there are natural isomorphisms
\[
\Ext^\bullet_{\bL_I}(\bk,\bk) \cong
\Ext^\bullet_{\bL_I}(\bk,\bk) \otimes \End_{M_{I,1}}(\St_I) \cong
\Ext^\bullet_{D_{M_{I,1}}(\bL_I)}(\St_I,\St_I),
\]
and the lemma follows. (Here, in order to prove the second isomorphism, we remark that if $X \to \bk$ is a quasi-isomorphism of $\bL_I$-dg-modules with $X$ K-projective, then the induced morphism $X \otimes \St_I \to \St_I$ will be a quasi-isomorphism of $M_{I,1}$-equivariant $\bL_I$-dg-modules, with $X \otimes \St_I$ K-projective as an $M_{I,1}$-equivariant $\bL_I$-dg-module.)
\end{proof}

Consider now the morphism $\sigma_I : \Rn_I \to \Rn_I \aq \fN_I = \bL_I$.

\begin{prop}
\label{prop:sigma-fully-faithful}
The functor
\[
\sigma_I^* : \Dfg_{\Stein}(\bL_I) \to \Dfg_{M_{I,1}}(\Rn_I)
\]
is fully faithful.
\end{prop}

\begin{proof}
To prove the proposition it suffices to prove that the morphism
\[
\Hom^\bullet_{\Dfg_{M_{I,1}}(\bL_I)}(\St_I,\St_I) \to \Hom^\bullet_{\Dfg_{M_{I,1}}(\Rn_I)}(\St_I,\St_I)
\]
induced by $\sigma_I^*$
is an isomorphism. 
Using the constructions of~\S\ref{ss:normal-subalg} for the normal dg-subalgebras $\bL_I \subset \bL_I \rtimes \smm_I$ and $\fN_I \subset \Rn_I \rtimes \smm_I$, and in particular isomorphisms~\eqref{eqn:RHom-normal-subalg} and~\eqref{eqn:transitivity-direct-image} (see also~\eqref{eqn:direct-image-normal-subalg}), we have canonical isomorphisms
\begin{align*}
\Hom^\bullet_{\Dfg_{M_{I,1}}(\bL_I)}(\St_I,\St_I) &\cong \Hom^\bullet_{D^+ \Rep(M_{I,1})}(\St_I, R\Hom_{\bL_I}(\bk, \St_I)),
\\
\Hom^\bullet_{\Dfg_{M_{I,1}}(\Rn_I)}(\St_I,\St_I) &\cong \Hom^\bullet_{D^+ \Rep(M_{I,1})}(\St_I, R\Hom_{\bL_I}(\bk, R\Hom_{\fN_I}(\bk,\St_I))).
\end{align*}
Now $\fN_I$ and $\bL_I$ act trivially on $\St_I$, so that we have
\[
R\Hom_{\fN_I}(\bk,\St_I) \cong R\Hom_{\fN_I}(\bk,\bk) \otimes \St_I, \; R\Hom_{\bL_I}(\bk, \St_I) \cong R\Hom_{\bL_I}(\bk,\bk) \otimes \St_I.
\]
Hence the claim above reduces to the claim that the morphism
\begin{multline*}
\Hom^\bullet_{D^+ \Rep(M_{I,1})}(\St_I, R\Hom_{\bL_I}(\bk, \bk) \otimes \St_I) \\
\to \Hom^\bullet_{D^+ \Rep(M_{I,1})}(\St_I, R\Hom_{\bL_I}(\bk, R\Hom_{\fN_I}(\bk, \bk)) \otimes \St_I)
\end{multline*}
induced by the natural morphism
\begin{equation}
\label{eqn:sigma-fully-faithful}
\bk = \tau_{\leq 0} R\Hom_{\fN_I}(\bk,\bk) \to R\Hom_{\fN_I}(\bk,\bk)
\end{equation}
in $\Dfg_{M_{I,1}}(\bL_I)$ is an isomorphism.
The cone of~\eqref{eqn:sigma-fully-faithful} is $R\Hom^{>0}_{\fN_I}(\bk,\bk)$, so the desired claim follows from the fact that
\[
\Hom^\bullet_{D^+ \Rep(M_{I,1})}(\St_I, R\Hom_{\bL_I}(\bk, R\Hom^{>0}_{\fN_I}(\bk, \bk)) \otimes \St_I)=0,
\]
which was proved in Lemma~\ref{lem:vanishing-RHom}.
\end{proof}

\subsection{Formality theorem for $P_{I,1}$}
\label{ss:formality-PI1}

Since $N_{I,1}$-modules are the same thing as $\snn_I$-modules, see Remark~\ref{rmk:infinitesimal-groups-1}, there exists a canonical equivalence of categories
\begin{equation}
\label{eqn:equiv-PI1}
D^+_{M_{I,1}}(\snn_I) \cong D^+ \Rep(P_{I,1}).
\end{equation}
Let us consider the following composition of functors, 
which we will denote by $\varphi_I$:
\begin{multline*}
D^+(\bL_I) \xrightarrow{\For^{\{1\}}_{M_{I,1}} (-) \otimes \St_I} D^+_{M_{I,1}}(\bL_I) \xrightarrow{\sigma_I^*} D^+_{M_{I,1}}(\Rn_I) \\
\xrightarrow{R\pi_{I*}} D^+_{M_{I,1}}(\snn_I) \xrightarrow[\sim]{\eqref{eqn:equiv-PI1}}D^+ \Rep(P_{I,1}).
\end{multline*}
It is clear that this functor satisfies $\varphi_I(\bk) \cong \St_I$.
Combining Lemma~\ref{lem:piI-equivalence}, Lemma~\ref{lem:stein-ff}, and Proposition~\ref{prop:sigma-fully-faithful}, we obtain the following ``formality'' theorem.

\begin{thm}
\label{thm:formality-non-equiv}
The functor $\varphi_I$ is fully faithful on the full subcategory $\Dfg(\bL_I)$, and it induces an equivalence
\[
\Dfg(\bL_I) \simto \Db_{\Stein} (P_{I,1}).
\]
\end{thm}

\subsection{Equivariance}
\label{ss:equivariance}

In this subsection we fix a subset $J \subset I$.
The dg-algebra $\bL_I$ has a natural action of $\dot P_I$, and hence of $P_J$ via the morphism $P_J \to \dot P_I$ induced by the Frobenius morphism $\Fr_{P_I} : P_I \to \dot P_I$.
If $V \in \Rep(\dot P_J)$, as in~\S\ref{ss:fromLtoRn} we can consider $V$ (or more precisely $\For^{\dot P_J}_{P_J}(V)$) as a $P_J$-equivariant $\bL_I$-dg-module concentrated in degree $0$, with trivial $\bL_I$-action. Using the constructions of~\S\ref{ss:Hom-action}, we deduce a natural action of (the group of $\bk$-points of) $P_J$ on the vector space $\Hom^n_{D(\bL_I)}(\bk,V)$, for any $n \in \Z$.

On the other hand, consider the distribution algebra $\spp_I$ of $P_{I,1}$. Since $P_{I,1} \subset P_I$ is a normal subgroup, there exists a natural action of $P_I$, hence also of $P_J$, on this algebra. If $V \in \Rep(\dot P_J)$, we can consider $\St_I \otimes V$ as a representation of $\Fr_{P_I}^{-1}(\dot P_J)$, hence as a $P_J$-equivariant $\spp_I$-dg-module.
Using again the constructions of~\S\ref{ss:Hom-action}, we deduce a natural action of $P_J$ on the vector space
\[
\Hom^n_{D\Rep(P_{I,1})}(\St_I,\St_I \otimes V) \cong \Hom^n_{D(\spp_I)}(\St_I,\St_I \otimes V),
\]
for any $n \in \Z$.

In Section~\ref{sec:equiv-formality} we will need the following consequence of Theorem~\ref{thm:formality-non-equiv}.

\begin{prop}
\label{prop:phiI-equivariance}
For any $V \in \Rep(\dot P_J)$, considered as a $\bL_I$-dg-module concentrated in degree $0$ (with trivial $\bL_I$-action), there exists a canonical isomorphism
\[
\varphi_I(V) \cong \St_I \otimes V.
\]
Moreover, for any $n \in \Z$, the functor $\varphi_I$ induces a $P_J$-equivariant isomorphism of vector spaces
\[
\Hom^n_{D(\bL_I)}(\bk, V) \simto \Hom^n_{D \Rep(P_{I,1})}(\St_I, \St_I \otimes V).
\]
\end{prop}

\begin{proof}
To prove the isomorphism $\varphi_I(V) \cong \St_I \otimes V$, since $\spp_I$ acts trivially on $V$, it suffices to prove that the functor $\varphi_I$ commutes with tensoring with a vector space (up to natural isomorphism). However, it is clear that the functors $\For^{\{1\}}_{M_{I,1}}(-) \otimes \St_I$, $\sigma_I^*$, and the equivalence \eqref{eqn:equiv-PI1}, commute with this operation. And the functor $R\pi_{I*}$ also commutes with tensoring with a vector space, since its inverse $\pi_I^*$ (see Lemma~\ref{lem:piI-equivalence}) clearly has this property.

Now, we claim that the natural morphism
\begin{equation}
\label{eqn:Hom-bLI-tensoring}
\Hom^n_{D(\bL_I)}(\bk, \bk) \otimes V \to \Hom^n_{D(\bL_I)}(\bk, V)
\end{equation}
is an isomorphism. Indeed, consider the Koszul resolution $K_I$ for the trivial $\bL_I$-dg-module $\bk$, as considered e.g.~in~\cite[\S 2.3]{mr}. This dg-module is a K-projective resolution of $\bk$, so that we have
\[
\Hom^n_{D(\bL_I)}(\bk, V) \cong \coH^n(\Hom^\bullet_{\bL_I}(K_I,V)) \cong \coH^n( \Sym^\bullet(\dot \fn_I^*) \otimes V),
\]
where $\dot \fn_I^*$ is in degree $2$. We deduce that
\[
\Hom^n_{D(\bL_I)}(\bk, V) \cong \begin{cases}
\Sym^{n/2}(\dot \fn_I^*) \otimes V & \text{if $n \in 2\Z_{\geq 0}$;} \\
0 & \text{otherwise.}
\end{cases}
\]
We have a similar description for $\Hom^n_{D(\bL_I)}(\bk, \bk)$, and from this it is clear that~\eqref{eqn:Hom-bLI-tensoring} is an isomorphism.

Similarly, we claim that the natural morphism
\begin{equation}
\label{eqn:Hom-PI1-tensoring}
\Hom^n_{D \Rep(P_{I,1})}(\St_I, \St_I) \otimes V \to \Hom^n_{D \Rep(P_{I,1})}(\St_I, \St_I \otimes V)
\end{equation}
is an isomorphism. Indeed, if $X^\bullet$ is an injective resolution of $\St_I$ as a $P_{I,1}$-module, then using~\cite[Proposition~I.3.10(c)]{jantzen} we see that $X^\bullet \otimes V$ is an injective resolution of $\St_I \otimes V$, so that we have
\begin{multline*}
\Hom^n_{D^+ \Rep(P_{I,1})}(\St_I, \St_I \otimes V) \cong \coH^n(\Hom^\bullet_{P_{I,1}}(\St_I, X^\bullet \otimes V)) \\
\cong \coH^n(\Hom^\bullet_{P_{I,1}}(\St_I, X^\bullet) \otimes V) \cong \coH^n(\Hom^\bullet_{P_{I,1}}(\St_I, X^\bullet)) \otimes V.
\end{multline*}
This shows that~\eqref{eqn:Hom-PI1-tensoring} is indeed an isomorphism.

It is easy to check that isomorphisms~\eqref{eqn:Hom-bLI-tensoring} and~\eqref{eqn:Hom-PI1-tensoring} are $P_J$-equivariant, and compatible with the morphisms induced by $\varphi_I$ in the obvious sense. So, to conclude, it suffices to prove that $\varphi_I$ induces a $P_J$-equivariant isomorphism
\[
\Hom^n_{D(\bL_I)}(\bk, \bk) \simto \Hom^n_{D \Rep(P_{I,1})}(\St_I, \St_I).
\]
The fact that this morphism is invertible follows from Theorem~\ref{thm:formality-non-equiv}, and what remains to be proved is $P_J$-equivariance. For this we can assume (for simplicity of notation) that $J=I$.

We remark that the morphism
\begin{equation}
\label{eqn:For-P-N}
\Hom^n_{D \Rep(P_{I,1})}(\St_I, \St_I) \to \Hom^n_{D \Rep(N_{I,1})}(\St_I, \St_I)
\end{equation}
induced by the functor $\For^{P_{I,1}}_{N_{I,1}}$ associated with the embedding $N_{I,1} \hookrightarrow P_{I,1}$ is injective. Indeed, by~\eqref{eqn:RHom-normal-subalg} applied to $\snn_I \subset \spp_I$, we have a canonical isomorphism
\[
\Hom^\bullet_{P_{I,1}}(\St_I, \St_I) \cong \Hom^\bullet_{M_{I,1}}(\St_I, R\Hom_{N_{I,1}}(\bk, \St_I)).
\]
And since $\St_I$ is projective as an $M_{I,1}$-module (see~\S\ref{ss:Steinberg}), we deduce for any $n$ canonical isomorphisms
\begin{multline*}
\Hom^n_{D^+ \Rep(P_{I,1})}(\St_I, \St_I) \cong \Hom_{M_{I,1}}(\St_I, \Hom^n_{N_{I,1}}(\bk,\St_I)) \\
\cong \Hom_{M_{I,1}}(\bk, \Hom^n_{N_{I,1}}(\St_I,\St_I))
\end{multline*}
(since $N_{I,1}$ acts trivially on $\St_I$). The claim follows.

For similar reasons as above, the vector space $\Hom^n_{D \Rep(N_{I,1})}(\St_I, \St_I)$ has a natural action of $P_I$ and, by Lemma~\ref{lem:action-Hom} (applied to the inclusion $\snn_I \hookrightarrow \spp_I$),~\eqref{eqn:For-P-N} is $P_I$-equivariant. Hence to conclude it suffices to prove that the morphism
\[
\Hom^n_{D(\bL_I)}(\bk, \bk) \to \Hom^n_{D^+ \Rep(N_{I,1})}(\St_I, \St_I)
\]
induced by $\For^{P_{I,1}}_{N_{I,1}} \circ \varphi_I$ is $P_I$-equivariant. Now, applying the commutativity of diagram~\eqref{eqn:restriction-scalars-For} for the functors $\sigma_I^*$ and $\pi_I^*$, we see that we have an isomorphism of functors
\[
\For^{P_{I,1}}_{N_{I,1}} \circ \varphi_I \cong R\pi_{I*} \circ (\sigma_I^*(-) \otimes \St_I) \cong (R\pi_{I*} \circ \sigma_I^*(-)) \otimes \St_I,
\]
where $R\pi_{I*}$ is now considered as a functor from $D(\Rn_I)$ to $D(\snn_I)$ and $\sigma_I^*$ as a functor from $D(\bL_I)$ to $D(\Rn_I)$.
(We also use once again the fact that $N_{I,1}$ acts trivially on $\St_I$.) 
Since the natural morphism
\[
\Hom^n_{D \Rep(N_{I,1})}(\bk, \bk) \to \Hom^n_{D \Rep(N_{I,1})}(\St_I, \St_I)
\]
is clearly $P_I$-equivariant,
to conclude we only need to prove that the morphism
\[
\Hom^n_{D(\bL_I)}(\bk, \bk) \to \Hom^n_{D \Rep(N_{I,1})}(\bk, \bk)
\]
induced by the functor $R\pi_{I*} \circ \sigma_I^*$ is $P_I$-equivariant. However, this morphism is the composition
\[
\Hom^n_{D(\bL_I)}(\bk, \bk) \to \Hom^n_{D(\Rn_I)}(\bk, \bk) \to \Hom^n_{D(\snn_I)}(\bk, \bk)
\]
where the first morphism is induced by $\sigma_I^*$, and the second one is the inverse to the isomorphism induced by $\pi_I^*$. Hence the desired property follows from Lemma~\ref{lem:action-Hom}.
\end{proof}

\subsection{An $\Ext^2$-computation for $B$-modules}

In this subsection we fix some $s \in S$.
The following fact, whose proof uses a computation done in the course of the proof of Proposition~\ref{prop:phiI-equivariance},
will be used in~\S\ref{ss:beta-min-parabolic} below.

\begin{lem}\label{lem:ext2b}
We have $\dim_\bk \bigl( \Ext^2_B(\bk_B(\ell\varsigma_s),\bk_B(\ell\varsigma_s-\ell\alpha_s)) \bigr)=1$.
\end{lem}

\begin{proof}
We certainly have $\Ext^2_B(\bk_B(\ell\varsigma_s),\bk_B(\ell\varsigma_s-\ell\alpha_s)) \cong \Ext^2_B(\bk_B,\bk_B(-\ell\alpha_s))$.  It follows from Proposition~\ref{prop:phiI-equivariance} and its proof (in the special case $J=I=\varnothing$) that as $B$-modules we have
\[
\Ext_{B_1}^q(\bk_{B},\bk_{B}(-\ell\alpha_s)) 
\cong
\begin{cases}
\Sym^{q/2}(\fnt^*) \otimes \bk_B(-\ell\alpha_s) & \text{if $q \in 2\Z_{\ge 0}$,} \\
0 & \text{otherwise.}
\end{cases}
\]
Corollary~\ref{cor:dim-Ext} (together with~\cite[Proposition~I.9.5]{jantzen}) then tells us that
\begin{multline}\label{eqn:ext2b-ineq}
\dim_\bk \bigl( \Ext_B^2(\bk_B,\bk_B(-\ell\alpha_s)) \bigr) \\
\le \dim_\bk \bigl( \coH^0(\dot B,\Ext^2_{B_1}(\bk_{B_1},\bk_{B_1}(-\ell\alpha_s))) \bigr)
+ \dim_\bk \bigl( \coH^2(\dot B,\Hom_{B_1}(\bk_{B_1},\bk_{B_1}(-\ell\alpha_s))) \bigr) \\
=
\dim_\bk \bigl( \coH^0(\dot B, \fnt^* \otimes \bk_{\dot B}(-\alpha_s)) \bigr)
+ \dim_\bk \bigl( \coH^2(\dot B, \bk_{\dot B}(-\alpha_s)) \bigr).
\end{multline}
The weights of $\fnt^* \otimes \bk_{\dot B}(-\alpha_s)$ are of the form $\beta - \alpha_s$ with $\beta \in \Phi^+$, each with multiplicity~$1$.  In particular, it has a $1$-dimensional $0$-weight space, so
\[
\dim_{\bk} \bigl( \coH^0(\dot B, \fnt^* \otimes \bk(-\alpha_s)) \bigr) \le 1.
\]

Let us now study
\[
\coH^2(\dot B, \bk_{\dot B}(-\alpha_s)) = \Ext^2_{\dot B}(\bk_{\dot B}, \bk_{\dot B}(-\alpha_s)).
\]
Using adjunction and the fact that $R\Ind_{\dot B}^{\dot G} \bk_{\dot B}(-\alpha_s) \cong \bk_{\dot G}[-1]$ (as follows from~\cite[Corollary~II.5.5]{jantzen}), we have
\[
\Ext^2_{\dot B}(\bk_{\dot B},\bk_{\dot B}(-\alpha_s)) \cong \Hom^2_{\dot G}(\bk_{\dot G}, R\Ind_{\dot B}^{\dot G}(\bk_{\dot B}(-\alpha_s))) \cong \Ext^1_{\dot G}(\bk_{\dot G},\bk_{\dot G}) = 0.
\]
So~\eqref{eqn:ext2b-ineq} now says that $\dim \Ext_B^2(\bk_B,\bk_B(-\ell\alpha_s)) \le 1$.  We have already seen in Lemma~\ref{lem:steinberg-diff} that $\Ext_B^2(\bk_B,\bk_B(-\ell\alpha_s)) \cong \Ext_B^2(\bk_B(\ell\varsigma_s), \bk_B(\ell\varsigma_s - \ell\alpha_s)) \ne 0$, and the lemma follows.
\end{proof}

\section{\texorpdfstring{$\dot P_J$}{P.J}-equivariant formality}
\label{sec:equiv-formality}

As in Section~\ref{sec:formality-non-equiv}, we fix a subset $I \subset S$. We also fix another subset $J \subset I$.

\subsection{Statement}
\label{ss:statement-equiv-formality}

We denote by $P_J M_{I,1}$ the subgroup of $G$ generated by $P_J$ and $M_{I,1}$, or equivalently by $P_J$ and $P_{I,1}$, which is normalized by $P_J$. Note that any element of $P_J M_{I,1}$ can be written (nonuniquely) as the product of an element of $P_J$ and an element of $M_{I,1}$, which justifies our notation, but that $M_{I,1}$ is \emph{not} normalized by $P_J$. The subgroup $P_J M_{I,1} \subset P_I$ can also be characterized as the inverse image of $\dot P_J$ under the Frobenius morphism $\Fr_{P_I} : P_I \to \dot P_I$; in particular we have a natural surjective morphism $\Fr_{J,I} : P_J M_{I,1} \to \dot P_J$. We denote by
\[
\Db_{\Stein}(P_J M_{I,1}) \subset \Db \Repf(P_J M_{I,1})
\]
the full subcategory generated by objects of the form $\St_I \otimes V$ for $V \in \Repf(\dot P_J)$. (Here, what we really mean by $\St_I \otimes V$ is $\For^{P_I}_{P_J M_{I,1}} (\St_I) \otimes \For^{\dot P_J}_{P_J M_{I,1}}(V)$, where the functor $\For^{P_I}_{P_J M_{I,1}}$ is defined with respect to the embedding $P_J M_{I,1} \hookrightarrow P_I$, and the functor $\For^{\dot P_J}_{P_J M_{I,1}}$ is defined with respect to $\Fr_{J,I}$.)

The group $\dot P_I$ acts on $\dot \fn_I$, and hence on the dg-algebra $\bL_I$. By restriction, we can consider $\bL_I$ as a $\dot P_J$-equivariant dg-algebra.
We define the functor
\[
\psi_{J,I} : D^+_{\dot P_J}(\bL_I) \to D^+ \Rep(P_J M_{I,1})
\]
as the composition
\begin{multline*}
D^+_{\dot P_J}(\bL_I) \xrightarrow{\For^{\dot P_J}_{P_J M_{I,1}}} D^+_{P_J M_{I,1}}(\bL_I) \xrightarrow{- \otimes \rmZ_I} D^+_{P_J M_{I,1}}(\bL_I) \xrightarrow{\sigma_I^*} D^+_{P_J M_{I,1}}(\Rn_I) \xrightarrow{(\pi_{I}^*)^{-1}} \\
D^+_{P_J M_{I,1}}(\snn_I) \cong D^+\Rep(N_{I,1} \rtimes P_J M_{I,1}) \xrightarrow{R\Ind_{N_{I,1} \rtimes P_J M_{I,1}}^{P_J M_{I,1}}} D^+ \Rep(P_J M_{I,1}).
\end{multline*}
Here the first arrow is associated with the morphism $\Fr_{J,I}$, the equivalence on the second line is induced by the equivalence $\Rep(N_{I,1}) \cong \snn_I\lmod$ (see~\eqref{eqn:Rep-Distmod}), and the functor $R\Ind_{N_{I,1} \rtimes P_J M_{I,1}}^{P_J M_{I,1}}$ is defined with respect to the morphism $N_{I,1} \rtimes P_J M_{I,1} \to P_J M_{I,1}$ given by multiplication in $P_J M_{I,1}$. 
Since the morphism $\pi_I$ is a quasi-isomorphism (see~\S\ref{ss:RnI}), by
Lemma~\ref{lem:qis-equiv} the functor $\pi_I^* : D^+_{P_J M_{I,1}}(\snn_I) \to D^+_{P_J M_{I,1}}(\Rn_I)$ is an equivalence, so that the fourth arrow is well defined.

The main result of this section is the following.

\begin{thm}
\label{thm:formality}
The functor $\psi_{J,I}$ is fully faithful on the subcategory $\Dfg_{\dot P_J}(\bL_I)$, and it induces an equivalence of categories
\[
\psi_{J,I}: \Dfg_{\dot P_J}(\bL_I) \simto \Db_{\Stein}(P_J M_{I,1}).
\]
Moreover, for any $X \in \Dfg_{\dot P_J}(\bL_I)$ and any $V \in \Rep(\dot P_J)$, there exists a natural and functorial isomorphism
\begin{equation}
\label{eqn:isom-psi-tensoring}
\psi_{J,I}(X \otimes V) \cong \psi_{J,I}(X) \otimes \For^{\dot P_J}_{P_JM_{I,1}}(V).
\end{equation}
\end{thm}

Theorem~\ref{thm:formality} will be proved in~\S\ref{ss:proof-fomality}. For this proof we will relate the functor $\psi_{J,I}$ to the functor $\varphi_I$ of Section~\ref{sec:formality-non-equiv}. More precisely, in~\S\ref{ss:proof-psi-phi} we prove the following.

\begin{prop}
\label{prop:psi-phi}
The following diagram commutes up to an isomorphism of functors:
\[
\xymatrix@C=2cm{
D^+_{\dot P_J}(\bL_I) \ar[r]^-{\psi_{J,I}} \ar[d]_-{\For^{\dot P_J}_{\{1\}}} & D^+ \Rep(P_J M_{I,1}) \ar[d]^-{\For^{P_J M_{I,1}}_{P_{I,1}}} \\
D^+(\bL_I) \ar[r]^-{\varphi_I} & D^+ \Rep(P_{I,1}).
}
\]
\end{prop}

\subsection{Proof of Proposition~\ref{prop:psi-phi}}
\label{ss:proof-psi-phi}

Let us consider the large diagram of Figure~\ref{fig:psi-phi}. Here to save space we have omitted the identifications
\begin{align*}
D^+_{P_J M_{I,1}}(\snn_I) &\cong D^+ \Rep(N_{I,1} \rtimes P_J M_{I,1}), \\
D^+_{P_{I,1}}(\snn_I) &\cong D^+ \Rep(N_{I,1} \rtimes P_{I,1}), \\
D^+_{M_{I,1}}(\snn_I) &\cong D^+ \Rep(N_{I,1} \rtimes M_{I,1}) \cong D^+ \Rep(P_{I,1})
\end{align*}
induced by~\eqref{eqn:Rep-Distmod}, and the functor $R\Ind_{N_{I,1} \rtimes P_{I,1}}^{P_{I,1}}$ is defined with respect to the multiplication morphism $N_{I,1} \rtimes P_{I,1} \to P_{I,1}$. Note that the lower vertical arrows in the second and third columns are well defined thanks to Remark~\ref{rmk:infinitesimal-groups-3}, and that the functors $R\pi_{I*}$ on the second and third lines are well defined thanks to Remark~\ref{rmk:infinitesimal-groups-2}.

By construction, the functor $\psi_{J,I}$ is the composition of the arrows appearing on the top of this diagram, and the functor $\varphi_I$ is the composition of the arrows appearing on the bottom of this diagram. 
Hence to prove the proposition it suffices to prove that each subdiagram (a)--(g) commutes (up to isomorphism).

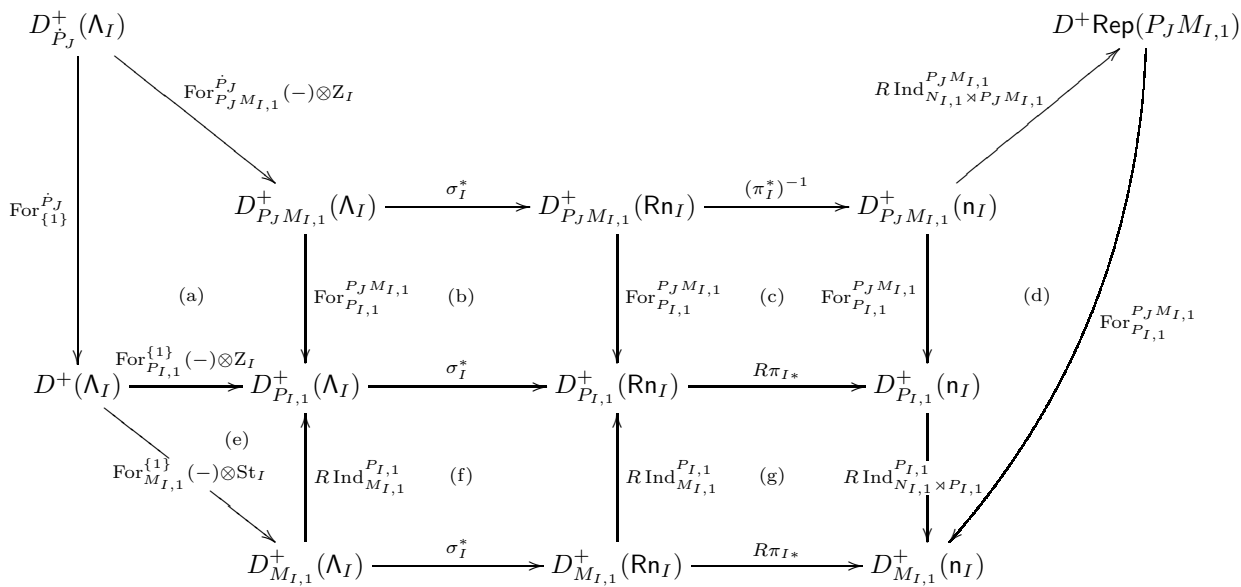
\begin{figure}
\begin{sideways}
\xymatrix@C=0.5cm@R=1.7cm{
D^+_{\dot P_J}(\bL_I) \ar[drr]^-{\For^{\dot P_J}_{P_J M_{I,1}}(-) \otimes \rmZ_I} \ar[dd]_-{\For^{\dot P_J}_{\{1\}}} 
  \ar@{}[dddrr]|{\mathrm{(a)}} &&&&&&&&& 
D^+\Rep(P_J M_{I,1}) 
  \ar@/^5ex/[dddl]^-{\For^{P_J M_{I,1}}_{P_{I,1}}} \ar@{}[dddl]|{\mathrm{(d)}}\\
&&
  D^+_{P_J M_{I,1}}(\bL_I) \ar[rrr]^-{\sigma_I^*} \ar[d]^-{\For^{P_{J} M_{I,1}}_{P_{I,1}}} 
    \ar@{}[drrr]|{\mathrm{(b)}} &&&
  D^+_{P_J M_{I,1}}(\Rn_I) \ar[rrr]^-{(\pi_{I}^*)^{-1}} \ar[d]^-{\For^{P_{J} M_{I,1}}_{P_{I,1}}} 
    \ar@{}[drrr]|{\mathrm{(c)}} &&&
D^+_{P_J M_{I,1}}(\snn_I) \ar[ur]^-{R\Ind_{N_{I,1} \rtimes P_J M_{I,1}}^{P_J M_{I,1}}}\ar[d]_-{\For^{P_J M_{I,1}}_{P_{I,1}}} & \\
D^+(\bL_I) 
\ar[rr]^-{\For^{\{1\}}_{P_{I,1}}({-}) \otimes \rmZ_I} 
\ar[drr]|-{\For^{\{1\}}_{M_{I,1}}({-}) \otimes \St_I} 
&&
  D^+_{P_{I,1}}(\bL_I) \ar[rrr]^-{\sigma_I^*}  \ar@{<-}[d]^-{R\Ind^{P_{I,1}}_{M_{I,1}}} 
    \ar@{}[drrr]|{\mathrm{(f)}} \ar@{}[dll]|(.3){\mathrm{(e)}} &&&
  D^+_{P_{I,1}}(\Rn_I) \ar[rrr]^-{R\pi_{I*}} \ar@{<-}[d]^-{R\Ind^{P_{I,1}}_{M_{I,1}}} \ar@{}[drrr]|{\mathrm{(g)}}
&&&
  D^+_{P_{I,1}}(\snn_I) \ar[d]|-{\hspace{-10pt}R\Ind_{N_{I,1} \rtimes P_{I,1}}^{P_{I,1}}} \\
  &&
  D^+_{M_{I,1}}(\bL_I) \ar[rrr]^-{\sigma_I^*} &&&
  D^+_{M_{I,1}}(\Rn_I) \ar[rrr]^-{R\pi_{I*}} &&&
  D^+_{M_{I,1}}(\snn_I) 
}
\end{sideways}
\caption{Diagram for the proof of Proposition~\ref{prop:psi-phi}}\label{fig:psi-phi}
\end{figure}

It is clear that subdiagram (a) commutes, and (b) commutes by~\eqref{eqn:restriction-scalars-For}. Consider now subdiagram (c). As in Lemma~\ref{lem:piI-equivalence}, the functor
\[
R\pi_{I*} : D^+_{P_{I,1}}(\Rn_I) \to D^+_{P_{I,1}}(\snn_I)
\]
is an equivalence of categories, with quasi-inverse $\pi_I^*$. Hence to prove the desired commutativity it suffices to prove that the following diagram commutes:
\[
\xymatrix@C=1.5cm{
D^+_{P_J M_{I,1}}(\snn_I) \ar[r]^-{\pi_I^*} \ar[d]_-{\For^{P_J M_{I,1}}_{P_{I,1}}} & D^+_{P_J M_{I,1}}(\Rn_I) \ar[d]^-{\For^{P_J M_{I,1}}_{P_{I,1}}} \\
D^+_{P_{I,1}}(\snn_I) \ar[r]^-{\pi_I^*} & D^+_{P_{I,1}}(\Rn_I).
}
\]
This again follows from~\eqref{eqn:restriction-scalars-For}.

Next, Lemma~\ref{lem:Ind-For}, applied to the multiplication morphism $N_{I,1} \rtimes P_{J} M_{I,1} \to P_{J} M_{I,1}$ and to the subgroups $N_{I,1} \rtimes P_{I,1} \subset N_{I,1} \rtimes P_{J} M_{I,1}$ and $P_{I,1} \subset P_{J} M_{I,1}$, implies that subdiagram (d) commutes.

Consider now subdiagram (e). We claim that for any bounded below $\bL_I$-dg-module $V$, the $M_{I,1}$-equivariant $\bL_I$-dg-module $\For^{\{1\}}_{M_{I,1}}(V) \otimes \St_I$ is split on the right for the functor $\Ind_{M_{I,1}}^{P_{I,1}} : \bL_I\ldgmod_{M_{I,1}} \to \bL_I\ldgmod_{P_{I,1}}$.
Indeed, to prove this claim it suffices to prove that if $\For^{\{1\}}_{M_{I,1}}(V) \otimes \St_I \to Y$ is a quasi-isomorphism of $M_{I,1}$-equivariant $\bL_I$-dg-modules such that $Y$ is K-injective, then the induced morphism $\Ind_{M_{I,1}}^{P_{I,1}}(\For^{\{1\}}_{M_{I,1}}(V) \otimes \St_I ) \to \Ind_{M_{I,1}}^{P_{I,1}}(Y)$ is a quasi-isomorphism. However $Y$ is K-injective as a complex of $M_{I,1}$-modules because $\bL_I \rtimes \smm_I$ is K-flat as a right $\smm_I$-module, and 
$\For^{\{1\}}_{M_{I,1}}(V) \otimes \St_I$ is a bounded below complex of injective $M_{I,1}$-modules by~\cite[Proposition~I.3.10(c)]{jantzen}, since $\St_I$ is an injective $M_{I,1}$-module (see~\S\ref{ss:Steinberg}). Hence this fact is clear.

Using this claim, we see that the composition $R\Ind_{M_{I,1}}^{P_{I,1}} \circ (\For^{\{1\}}_{M_{I,1}}(-) \otimes \St_I)$ appearing in subdiagram (e) is the functor on derived categories induced by the exact functor
\[
\bL_I\ldgmod^+ \to \bL_I\ldgmod^+_{P_{I,1}} : V \mapsto \Ind_{M_{I,1}}^{P_{I,1}} (\For^{\{1\}}_{M_{I,1}}(V) \otimes \St_I).
\]
Now, for any $V$ in $\bL_I\ldgmod^+$ we obviously have
\[
\Ind_{M_{I,1}}^{P_{I,1}} (\For^{\{1\}}_{M_{I,1}}(V) \otimes \St_I) \cong \For^{\{1\}}_{P_{I,1}}(V) \otimes \Ind_{M_{I,1}}^{P_{I,1}} (\St_I) \cong \For^{\{1\}}_{P_{I,1}}(V) \otimes \rmZ_I
\]
by Lemma~\ref{lem:pi-ind}, which finishes the proof of the commutativity of subdiagram (e).

Finally, subdiagram (f) commutes by Lemma~\ref{lem:base-change} (see also Remark~\ref{rmk:infinitesimal-groups-3}), and subdiagram (g) commutes by~\eqref{eqn:transitivity-direct-image}, since the following diagram commutes:
\[
\xymatrix@C=1.5cm{
\Rn_I \rtimes \smm_I \ar@{^{(}->}[r] \ar[rd]_-{\pi_I \rtimes \id} & \Rn_I \rtimes \spp_I \ar[r]^-{\pi_I \rtimes \id} & \snn_I \rtimes \spp_I \ar[d]^-{\mathrm{mult}} \\
& \snn_I \rtimes \smm_I \ar@{=}[r] & \spp_I.
}
\]

We have proved that all the pieces in the diagram of Figure~\ref{fig:psi-phi} commute. Hence the diagram as a whole commutes, and Proposition~\ref{prop:psi-phi} is proved.

\subsection{Proof of Theorem~\ref{thm:formality}}
\label{ss:proof-fomality}

We begin with some preliminary lemmas. 

\begin{lem}
\label{lem:frobenius-psi}
For any $X \in D^+_{\dot P_J}(\bL_I)$ and any $V \in \Rep(\dot P_J)$, there is a natural isomorphism $\psi_{J,I}(X \otimes V) \cong \psi_{J,I}(X) \otimes \For_{P_JM_{I,1}}^{\dot P_J}(V)$.
\end{lem}
\begin{proof}
We certainly have the following collection of natural isomorphisms (in each line, $Y$ should be understood as belonging to the appropriate category of dg-modules):
\begin{align*}
\For_{P_JM_{I,1}}^{\dot P_J}(Y \otimes V) \otimes \rmZ_I &\cong \For_{P_JM_{I,1}}^{\dot P_J}(Y) \otimes \rmZ_I \otimes \For_{P_JM_{I,1}}^{\dot P_J}(V), \\
\sigma_I^*(Y \otimes \For_{P_JM_{I,1}}^{\dot P_J}(V)) &\cong \sigma_I^*(Y) \otimes \For_{P_JM_{I,1}}^{\dot P_J}(V), \\
\pi_I^*(Y \otimes \For_{P_JM_{I,1}}^{\dot P_J}(V)) &\cong \pi_I^*(Y) \otimes \For_{P_JM_{I,1}}^{\dot P_J}(V).
\end{align*}
The tensor identity  (or rather its easy extension to our more general version of induction) tells us that
\[
R\Ind_{N_{I,1} \rtimes P_J M_{I,1}}^{P_J M_{I,1}}(Y \otimes \For_{P_JM_{I,1}}^{\dot P_J}(V))
\cong 
R\Ind_{N_{I,1} \rtimes P_J M_{I,1}}^{P_J M_{I,1}}(Y) \otimes \For_{P_JM_{I,1}}^{\dot P_J}(V).
\]
The lemma follows from the combination of these isomorphisms.
\end{proof}

\begin{lem}
\label{lem:trivial-psi}
There exists a canonical isomorphism $\psi_{J,I}(\bk) \cong \St_I$.
\end{lem}

\begin{proof}
From the definition of $\psi_{J,I}$ we see that
\[
\psi_{J,I}(\bk) \cong R\Ind_{N_{I,1} \rtimes P_J M_{I,1}}^{P_J M_{I,1}}(\rmZ_I),
\]
where the induction functor is defined with respect to the multiplication morphism, and where $N_{I,1} \rtimes P_J M_{I,1}$ acts on $\rmZ_I$ via the projection to the second component $P_J M_{I,1}$. Since $\varphi_I(\bk)=\St_I$ is concentrated in degree $0$,
using Proposition~\ref{prop:psi-phi} we see that $\psi_{J,I}(\bk)$ is also concentrated in degree $0$, so that
\[
\psi_{J,I}(\bk) \cong \Ind_{N_{I,1} \rtimes P_J M_{I,1}}^{P_J M_{I,1}}(\rmZ_I).
\]
We also deduce that $\For^{P_J M_{I,1}}_{P_{I,1}}(\Ind_{N_{I,1} \rtimes P_J M_{I,1}}^{P_J M_{I,1}}(\rmZ_I)) \cong \St_I$.

Now by adjunction we have
\[
\Hom_{P_J M_{I,1}}(\St_I, \Ind_{N_{I,1} \rtimes P_J M_{I,1}}^{P_J M_{I,1}}(\rmZ_I)) \cong \Hom_{N_{I,1} \rtimes P_J M_{I,1}}(\St_I, \rmZ_I),
\]
where $N_{I,1} \rtimes P_J M_{I,1}$ acts on $\St_I$ via the multiplication morphism to $P_{J} M_{I,1}$. But since $N_{I,1}$ acts trivially on $\St_I$, this action coincides with the action via the projection $N_{I,1} \rtimes P_J M_{I,1} \to P_J M_{I,1}$ on the second factor, and we deduce that
\[
\Hom_{N_{I,1} \rtimes P_J M_{I,1}}(\St_I, \rmZ_I) \cong \Hom_{P_J M_{I,1}}(\St_I, \rmZ_I).
\]
By Corollary~\ref{cor:StI-PiI} there exists a nonzero $P_J M_{I,1}$-equivariant morphism $\St_I \to \rmZ_I$, and by these isomorphisms we deduce a nonzero morphism of $P_{J}M_{I,1}$-modules $\St_I \to \Ind_{N_{I,1} \rtimes P_J M_{I,1}}^{P_J M_{I,1}}(\rmZ_I)$. Since $\St_I$ is simple (see~\S\ref{ss:Steinberg}), this morphism is injective. And the remarks above imply that our two modules have the same dimension, so that this morphism must be an isomorphism.

We have thus proved that there exists an isomorphism $\psi_{J,I}(\bk) \cong \St_I$. To construct a canonical isomorphism, we simply remark that the forgetful functor induces an isomorphism
\[
\Hom_{P_J M_{I,1}}(\St_I, \St_I) \simto \Hom_{P_{I,1}}(\St_I, \St_I)
\]
(since both spaces have dimension $1$), so that the canonical isomorphism $\varphi_I(\bk) \cong \St_I$ induces, via Proposition~\ref{prop:psi-phi}, a canonical isomorphism $\psi_{J,I}(\bk) \cong \St_I$.
\end{proof}

As explained in~\S\ref{ss:equivariance}, for any $V \in \Rep(\dot P_J)$ and any $n \in \Z$, the vector space $\Hom^n_{D(\bL_I)}(\bk,V)$ admits a natural action of $P_J$, which can easily be seen to factor through an action of $\dot P_J$ (see the proof of Proposition~\ref{prop:phiI-equivariance}).

\begin{lem}
\label{lem:Hom-fixed-points-bLI}
For any injective $\dot P_J$-module $V$ and any $n \in \Z$, the morphism
\[
\Hom^n_{D_{\dot P_J}(\bL_I)}(\bk,V) \to \Hom^n_{D(\bL_I)}(\bk,V)
\]
induced by the functor $\For^{\dot P_J}_{\{1\}}$
is injective, and it induces an isomorphism
\[
\Hom^n_{D_{\dot P_J}(\bL_I)}(\bk,V) \simto \mathbb{I}^{\dot P_J} \bigl( \Hom^n_{D(\bL_I)}(\bk,V) \bigr).
\]
\end{lem}

\begin{proof}
By Lemma~\ref{lem:K-inj-equiv}, there exists an object $X$ in $\bL_I\ldgmod_{\dot P_J}^+$ which is K-injective and has components which are injective $\dot P_J$-modules, and a quasi-isomorphism of $\dot P_J$-equivariant dg-modules $V \xrightarrow{\qis} X$. Then we have
\[
\Hom^n_{D_{\dot P_J}(\bL_I)}(\bk,V) \cong \coH^n(\Hom^\bullet_{\bL_I\ldgmod_{\dot P_J}}(\bk,X)).
\]
Now,
as in~\S\ref{ss:equivariance},
consider the Koszul resolution $K_I$ of the $\bL_I$-dg-module $\bk$. Then since $X$ is K-injective the quasi-isomorphism $K_I \xrightarrow{\qis} \bk$ induces a quasi-isomorphism
\[
\Hom^\bullet_{\bL_I\ldgmod_{\dot P_J}}(\bk,X) \xrightarrow{\qis} \Hom^\bullet_{\bL_I\ldgmod_{\dot P_J}}(K_I,X).
\]
Next, we remark that we have
\[
\Hom^\bullet_{\bL_I\ldgmod_{\dot P_J}}(K_I,X) = \mathbb{I}^{\dot P_J}(\Hom^\bullet_{\bL_I}(K_I,X)) = \mathbb{I}^{\dot P_J}(\Sym^\bullet(\dot \fn_I^*) \otimes X),
\]
where $\dot \fn_I^*$ is in degree $2$.
The morphism $\Sym^\bullet(\dot \fn_I^*) \otimes V \to \Sym^\bullet(\dot \fn_I^*) \otimes X$ induced by the quasi-isomor\-phism $V \xrightarrow{\qis} X$ is a quasi-isomorphism of bounded below complexes of injective $\dot P_J$-modules; therefore it induces a quasi-isomorphism
\[
\mathbb{I}^{\dot P_J}(\Sym^\bullet(\dot \fn_I^*) \otimes V) \xrightarrow{\qis} \mathbb{I}^{\dot P_J}(\Sym^\bullet(\dot \fn_I^*) \otimes X).
\]
Combining these isomorphisms, we obtain that
\begin{multline*}
\Hom^n_{D_{\dot P_J}(\bL_I)}(\bk,V) \cong \coH^n(\mathbb{I}^{\dot P_J}(\Sym^\bullet(\dot \fn_I^*) \otimes V)) \\
\cong \begin{cases}
\mathbb{I}^{\dot P_J}(\Sym^{n/2}(\dot \fn_I^*) \otimes V) & \text{if $n \in 2\Z_{\geq 0}$;} \\
0 & \text{otherwise.}
\end{cases}
\end{multline*}

Similarly we have
\[
\Hom^n_{D_{\dot P_J}(\bL_I)}(\bk,V) \cong \begin{cases}
\Sym^{n/2}(\dot \fn_I^*) \otimes V & \text{if $n \in 2\Z_{\geq 0}$;} \\
0 & \text{otherwise}
\end{cases}
\]
(see the proof of Proposition~\ref{prop:phiI-equivariance})
and the lemma follows.
\end{proof}

Similarly (see again~\S\ref{ss:equivariance}), for any $\dot P_J$-module $V$ and any $n \in \Z$, the vector space $\Hom^n_{D^+ \Rep(P_{I,1})}(\St_I, \St_I \otimes V)$ admits a natural action of $P_J$.

\begin{lem}
\label{lem:Hom-fixed-points-PI1}
For any injective $\dot P_J$-module $V$ and any $n \in \Z$, the $P_J$-action on $\Hom^n_{D^+ \Rep(P_{I,1})}(\St_I,\St_I \otimes V)$ factors through an action of $\dot P_J$. Moreover,
the morphism
\[
\Hom^n_{D^+ \Rep(P_J M_{I,1})}(\St_I,\St_I \otimes V) \to \Hom^n_{D^+ \Rep(P_{I,1})}(\St_I,\St_I \otimes V)
\]
induced by the functor $\For^{P_J M_{I,1}}_{P_{I,1}}$
is injective, and induces an isomorphism
\[
\Hom^n_{D^+ \Rep(P_J M_{I,1})}(\St_I,\St_I \otimes V) \simto \mathbb{I}^{\dot P_J} \bigl( \Hom^n_{D^+ \Rep(P_{I,1})}(\St_I,\St_I \otimes V) \bigr).
\]
\end{lem}

\begin{proof}
Let $\St_I \xrightarrow{\qis} X$ be an injective resolution in $\Rep(P_J M_{I,1})$. Then $X \otimes V$ is an injective resolution of $\St_I \otimes V$, so that we have
\[
\Hom^n_{D^+ \Rep(P_J M_{I,1})}(\St_I,\St_I \otimes V) \cong \coH^n (\Hom^{\bullet}_{P_J M_{I,1}}(\St_I, X \otimes V)).
\]
On the other hand we have
\[
\Hom^{\bullet}_{P_J M_{I,1}}(\St_I, X \otimes V) = \mathbb{I}^{\dot P_J} ( \Hom^{\bullet}_{P_{I,1}}(\St_I, X \otimes V) ) \cong \mathbb{I}^{\dot P_J} ( \Hom^{\bullet}_{P_{I,1}}(\St_I, X) \otimes V ),
\]
where the $\dot P_J$-action is induced by the $P_{J} M_{I,1}$-actions on $\St_I$, $X$ and $V$. Since $V$ is injective, the functor $\mathbb{I}^{\dot P_J}(- \otimes V)$ is exact; therefore we obtain that
\begin{multline*}
\Hom^n_{D^+ \Rep(P_J M_{I,1})}(\St_I,\St_I \otimes V) \cong \mathbb{I}^{\dot P_J} ( \coH^n (\Hom^{\bullet}_{P_{I,1}}(\St_I, X)) \otimes V ) \\
\cong \mathbb{I}^{\dot P_J} ( \coH^n (\Hom^{\bullet}_{P_{I,1}}(\St_I, X \otimes V ))).
\end{multline*}

By~\cite[Proposition~I.4.12 \& Corollary~I.5.13(b)]{jantzen}, any injective $P_{J} M_{I,1}$-module is injective as a $P_{I,1}$-module; in particular $X \otimes V$ is an injective resolution of $\St_I \otimes V$ as a $P_{I,1}$-module, and we have
\[
\coH^n (\Hom^{\bullet}_{P_{I,1}}(\St_I, X \otimes V)) \cong \Hom^n_{D^+ \Rep(P_{I,1})}(\St_I, \St_I \otimes V).
\]
This finally proves that
\[
\Hom^n_{D^+ \Rep(P_J M_{I,1})}(\St_I,\St_I \otimes V) \cong \mathbb{I}^{\dot P_J}(\Hom^n_{D^+ \Rep(P_{I,1})}(\St_I, \St_I \otimes V)).
\]

This isomorphism proves the lemma, provided we prove that the $P_J$-action deduced (via the Frobenius) from the $\dot P_J$-action considered in this proof coincides with the action constructed (in the general setting) in~\S\ref{ss:Hom-action}. For this we choose a complex of $P_J$-equivariant $\spp_I$-modules $Y$ and a $P_J$-equivariant quasi-isomorphism $Y \xrightarrow{\qis} \St_I$ which is a projective resolution over $\spp_I$. Then this morphism induces a quasi-isomorphism
\[
\Hom^\bullet_{P_{I,1}}(\St_I, X \otimes V) \xrightarrow{\qis} \Hom^\bullet_{P_{I,1}}(Y,X \otimes V)
\]
because $X \otimes V$ is a bounded below complex of injective $P_{I,1}$-modules. And the quasi-isomorphism $\St_I \otimes V \xrightarrow{\qis} X \otimes V$ induces a quasi-isomorphism
\[
\Hom^\bullet_{P_{I,1}}(Y,\St_I \otimes V) \xrightarrow{\qis} \Hom^\bullet_{P_{I,1}}(Y,X \otimes V)
\]
since $Y$ is a bounded above complex of projective $P_{I,1}$-modules. These quasi-isomorphisms are $P_J$-equivariant, so the actions do indeed coincide.
\end{proof}

\begin{cor}
\label{cor:psi-ff-injectives}
For any injective $\dot P_J$-module $V$ and any $n \in \Z$, the functor $\psi_{J,I}$ induces an isomorphism
\[
\Hom^n_{D_{\dot P_J}(\bL_I)}(\bk,V) \simto \Hom^n_{D^+ \Rep(P_J M_{I,1})}(\St_I, \St_I \otimes V).
\]
\end{cor}

\begin{proof}
By Lemmas~\ref{lem:frobenius-psi} and~\ref{lem:trivial-psi}, we have canonical isomorphisms $\psi_{J,I}(\bk) \cong \St_I$ and $\psi_{J,I}(V) \cong \St_I \otimes V$. Now by Proposition~\ref{prop:psi-phi} we have a commutative diagram
\[
\xymatrix@C=1.5cm{
\Hom^n_{D_{\dot P_J}(\bL_I)}(\bk,V) \ar[r] \ar[d] & \Hom^n_{D^+ \Rep(P_J M_{I,1})}(\St_I, \St_I \otimes V) \ar[d] \\
\Hom^n_{D(\bL_I)}(\bk,V) \ar[r] & \Hom^n_{D^+ \Rep(P_{I,1})}(\St_I, \St_I \otimes V),
}
\]
where the horizontal morphisms are induced by $\psi_{J,I}$ and $\varphi_I$ respectively, and the vertical morphisms by the appropriate forgetful functors. By Proposition~\ref{prop:phiI-equivariance} the lower line is a $P_J$-equivariant isomorphism, and by Lemmas~\ref{lem:Hom-fixed-points-bLI} and~\ref{lem:Hom-fixed-points-PI1} the vertical arrows are embeddings of the $P_J$-fixed points. Therefore the upper line is also an isomorphism.
\end{proof}

Now we deduce a similar property for finite-dimensional $\dot P_J$-modules.

\begin{prop}
\label{prop:psi-ff-V}
For any finite dimensional $\dot P_J$-module $V$, and any $n \in \Z$, the functor $\psi_{J,I}$ induces an isomorphism
\[
\Hom^n_{\Dfg_{\dot P_J}(\bL_I)}(\bk,V) \simto \Ext^n_{P_J M_{I,1}}(\St_I, \St_I \otimes V).
\]
\end{prop}

\begin{proof}
As in the proof of Corollary~\ref{cor:psi-ff-injectives},
we have canonical isomorphisms $\psi_{J,I}(\bk) \cong \St_I$ and $\psi_{J,I}(V) \cong \St_I \otimes V$.
Choose an injective resolution $V \to X^\bullet$ of $V$ as a $\dot P_J$-module and, for any $k \geq 0$, let $X_k$ be the complex
\[
\cdots \to 0 \to X^0 \to \cdots \to X^k \to 0 \to \cdots
\]
We have natural isomorphisms
\begin{align*}
\Hom^n_{\Dfg_{\dot P_J}(\bL_I)}(\bk,V) &\cong \Hom^n_{\Dfg_{\dot P_J}(\bL_I)}(\bk,X^\bullet), \\
\Ext^n_{P_J M_{I,1}}(\St_I, \St_I \otimes V) &\cong \Hom^n_{P_J M_{I,1}}(\St_I, \St_I \otimes X^\bullet).
\end{align*}
Hence the natural morphism $X^\bullet \to X_k$ induces a commutative diagram
\begin{equation}
\label{eqn:diag-psi-V}
\vcenter{
\xymatrix{
\Hom^n_{\Dfg_{\dot P_J}(\bL_I)}(\bk,V) \ar[r] \ar[d] & \Ext^n_{P_J M_{I,1}}(\St_I, \St_I \otimes V) \ar[d] \\
\Hom^n_{\Dfg_{\dot P_J}(\bL_I)}(\bk,X_k) \ar[r] & \Ext^n_{P_J M_{I,1}}(\St_I, \St_I \otimes X_k),
}
}
\end{equation}
where the horizontal arrows are induced by $\psi_{J,I}$. By Corollary~\ref{cor:psi-ff-injectives} and the $5$-lemma, the lower line is an isomorphism. On the other hand, the same arguments as in the proof of Lemma~\ref{lem:Hom-fixed-points-bLI} show that we have
\begin{equation}
\label{eqn:Ext-bLI}
\Hom^n_{\Dfg_{\dot P_J}(\bL_I)}(\bk,X^\bullet) \cong \bigoplus_{i+2j = n} \coH^i(\mathbb{I}^{\dot P_J}(\Sym^j(\dot \fn_I^*) \otimes X^\bullet)),
\end{equation}
and similarly for $X_k$. In particular, we deduce that the left-hand morphism in~\eqref{eqn:diag-psi-V} is an isomorphism for $k \gg 0$. It follows that the upper horizontal morphism is injective, and to finish the proof we only have to prove that
\begin{equation}
\label{eqn:ineq-dimensions}
\dim_\bk(\Ext^n_{P_J M_{I,1}}(\St_I, \St_I \otimes V)) \leq \dim_\bk(\Hom^n_{\Dfg_{\dot P_J}(\bL_I)}(\bk,V)).
\end{equation}

The formula~\eqref{eqn:Ext-bLI} also shows that
\begin{multline*}
\dim_\bk(\Hom^n_{\Dfg_{\dot P_J}(\bL_I)}(\bk,V)) = \sum_{i+2j = n} \dim_\bk ( \Ext^i_{\dot P_J}(\bk, \Sym^j(\dot \fn_I^*) \otimes V)) \\
= \sum_{i+k = n} \dim_\bk ( \Ext^i_{\dot P_J}(\bk, \Hom^k_{D(\bL_I)}(\bk,\bk) \otimes V) ).
\end{multline*}
On the other hand, by Corollary~\ref{cor:dim-Ext} we have
\begin{multline*}
\dim_\bk(\Ext^n_{P_J M_{I,1}}(\St_I, \St_I \otimes V)) \leq \sum_{i+k=n} \dim_\bk(\Ext^i_{\dot P_J}(\bk, \Ext^k_{P_{I,1}}(\St_I, \St_I \otimes V))) \\
= \sum_{i+k=n} \dim_\bk ( \Ext^i_{\dot P_J}(\bk, \Ext^k_{P_{I,1}}(\St_I, \St_I) \otimes V)).
\end{multline*}
By Proposition~\ref{prop:phiI-equivariance}, for any $k$ we have an isomorphism of $\dot P_J$-modules
\[
\Hom^k_{D(\bL_I)}(\bk,\bk) \cong \Ext^k_{P_{I,1}}(\St_I, \St_I),
\]
hence these formulas prove~\eqref{eqn:ineq-dimensions} and conclude the proof. (Note that all the dimensions under consideration here are finite thanks to~\cite[Proposition~II.4.10]{jantzen}.)
\end{proof}

We can finally complete the proof of Theorem~{\rm \ref{thm:formality}}.

\begin{proof}[Proof of Theorem~{\rm \ref{thm:formality}}]
The second part of the theorem has already been established in Lemma~\ref{lem:frobenius-psi}.
Since the category $\Dfg_{\dot P_J}(\bL_I)$, resp.~$\Db_{\Stein}(P_J M_{I,1})$, is generated by the objects $V$, resp.~$\St_I \otimes V$, for $V \in \Repf(\dot P_J)$ (see the proof of Lemma~\ref{lem:stein-ff} for the first case), and since $\psi_{J,I}(V) \cong \St_I \otimes V$ (see Lemmas~\ref{lem:frobenius-psi} and~\ref{lem:trivial-psi}), to prove the first part of the theorem, it suffices to show that for any $V,V \in \Repf(\dot P_J)$ and any $n \in \Z$ the morphism
\[
\Hom^n_{D_{\dot P_J}(\bL_I)}(V,V') \to \Hom^n_{\Db_{\Stein}(P_J M_{I,1})}(\St_I \otimes V, \St_I \otimes V')
\]
induced by $\psi_{J,I}$ is an isomorphism. However we have a commutative diagram
\[
\xymatrix{
\Hom^n_{D_{\dot P_J}(\bL_I)}(V,V') \ar[r] \ar[d]_-{\wr} & \Hom^n_{\Db_{\Stein}(P_J M_{I,1})}(\St_I \otimes V, \St_I \otimes V') \ar[d]^-{\wr} \\
\Hom^n_{D_{\dot P_J}(\bL_I)}(\bk,V^* \otimes V') \ar[r] & \Hom^n_{\Db_{\Stein}(P_J M_{I,1})}(\St_I, \St_I \otimes V^* \otimes V')
}
\]
where both horizontal arrows are induced by $\psi_{J,I}$ and the vertical arrows are induced by the natural adjunctions. The lower horizontal arrow is invertible by Proposition~\ref{prop:psi-ff-V}, hence so is the upper arrow, and the theorem is proved.
\end{proof}

\section{Compatibility with induction}
\label{sec:compatibility}

In this section, we show that the equivalence of Theorem~\ref{thm:formality} is compatible (in the appropriate sense) with induction of representations from one subgroup of the form $P_J M_{I,1}$ to a larger one.  A larger such subgroup can be obtained by enlarging either $J$ or $I$.  The two cases are rather different, and we will treat them separately.

\subsection{Enlarging $J$}

In this subsection we fix $J \subset J' \subset I$. Then $\dot P_J \subset \dot P_{J'}$. Using the constructions of~\S\ref{ss:Kinj-equivariant} we can consider the functor
\[
R\Ind_{\dot P_J}^{\dot P_{J'}} : D^+_{\dot P_J}(\bL_I) \to D^+_{\dot P_{J'}}(\bL_I).
\]
Using~\cite[Proposition~I.5.12]{jantzen} and the commutativity of diagram~\eqref{eqn:diag-for-ind}, we see that this functor restricts to a functor from $\Dfg_{\dot P_J}(\bL_I)$ to $\Dfg_{\dot P_{J'}}(\bL_I)$. 

\begin{lem}
\label{lem:Ind-J-J'}
For any $V \in \Rep(\dot P_J)$, there exists a canonical isomorphism
\[
R\Ind_{P_J M_{I,1}}^{P_{J'} M_{I,1}}(\St_I \otimes \For^{\dot P_{J}}_{P_{J} M_{I,1}}(V)) \cong \St_I \otimes \For^{\dot P_{J'}}_{P_{J'} M_{I,1}}(R\Ind_{\dot P_J}^{\dot P_{J'}}(V)).
\]
\end{lem}

\begin{proof}
Using the tensor identity, it suffices to prove that we have a canonical isomorphism
\[
R\Ind_{P_J M_{I,1}}^{P_{J'} M_{I,1}}(\For^{\dot P_{J}}_{P_{J} M_{I,1}}(V)) \cong \For^{\dot P_{J'}}_{P_{J'} M_{I,1}}(R\Ind_{\dot P_J}^{\dot P_{J'}}(V)).
\]
First, we remark that since $P_{I,1}$ acts trivially on $\For^{\dot P_{J}}_{P_{J} M_{I,1}}(V)$, there exists a canonical isomorphism
\[
\Ind_{P_J M_{I,1}}^{P_{J'} M_{I,1}}(\For^{\dot P_{J}}_{P_{J} M_{I,1}}(V)) \cong \For^{\dot P_{J'}}_{P_{J'} M_{I,1}}(\Ind_{\dot P_J}^{\dot P_{J'}}(V))
\]
for any $V$ in $\Rep(\dot P_J)$. Hence, as in the proof of Lemma~\ref{lem:Ind-For}, to conclude it is enough to prove that
\[
R^i \Ind_{P_J M_{I,1}}^{P_{J'} M_{I,1}}(\For^{\dot P_{J}}_{P_{J} M_{I,1}}(\cO(\dot P_J))) = 0 \quad \text{for $i>0$.}
\]
Now, again as in the proof of Lemma~\ref{lem:Ind-For}, we have
\begin{multline*}
R \Ind_{P_J M_{I,1}}^{P_{J'} M_{I,1}}(\For^{\dot P_{J}}_{P_{J} M_{I,1}}(\cO(\dot P_J))) \cong R\mathbb{I}^{P_J M_{I,1}}(\cO(P_{J'} M_{I,1}) \otimes \cO(\dot P_J)) \\
\cong R\Ind_{P_J M_{I,1}}^{\dot P_J}(\cO(P_{J'} M_{I,1})).
\end{multline*}
The functor $R\Ind_{P_J M_{I,1}}^{\dot P_J}$ is right adjoint to the functor $\For_{P_J M_{I,1}}^{\dot P_J}$; but this functor also admits as a right adjoint the right derived functor of the functor $\widetilde{\mathbb{I}}^{P_{I,1}}(-) : \Rep(P_J M_{I,1}) \to \Rep(\dot P_J)$ induced by $\mathbb{I}^{P_{I,1}}$. Hence we also have
\[
R\Ind_{P_J M_{I,1}}^{P_{J'} M_{I,1}}(\For^{\dot P_{J}}_{P_{J} M_{I,1}}(\cO(\dot P_J))) \cong R\widetilde{\mathbb{I}}^{P_{I,1}}(\cO(P_{J'} M_{I,1})).
\]
Now using~\cite[Proposition~I.4.12 \& Corollary~I.5.13(b)]{jantzen} we see that any injective $P_{J} M_{I,1}$-module is also injective over $P_{I,1}$, so that
\[
R\widetilde{\mathbb{I}}^{P_{I,1}}(\cO(P_{J'} M_{I,1})) \cong R\mathbb{I}^{P_{I,1}}(\cO(P_{J'} M_{I,1})) \cong R\Ind_{P_{I,1}}^{P_{J'} M_{I,1}}(\bk).
\]
And using again~\cite[Corollary~I.5.13(b)]{jantzen} we obtain that $R^i\Ind_{P_{I,1}}^{P_{J'} M_{I,1}}(\bk)=0$ for $i>0$, which finishes the proof.
\end{proof}

It follows in particular from Lemma~\ref{lem:Ind-J-J'} and~\cite[Proposition~I.5.12]{jantzen} that the functor $R\Ind_{P_J M_{I,1}}^{P_{J'} M_{I,1}}$ restricts to a functor from $\Db_\Stein(P_J M_{I,1})$ to $\Db_\Stein(P_{J'}M_{I,1})$.

\begin{thm}
\label{thm:parabolicJ}
The following diagram commutes up to isomorphism:
\[
\xymatrix@C=2cm{
\Dfg_{\dot P_J}(\bL_I) \ar[d]_-{R\Ind_{\dot P_J}^{\dot P_{J'}}} \ar[r]^-{\psi_{J,I}}_-{\sim} & \Db_\Stein(P_J M_{I,1}) \ar[d]^-{R\Ind_{P_J M_{I,1}}^{P_{J'} M_{I,1}}} \\
\Dfg_{\dot P_{J'}}(\bL_I) \ar[r]_-{\psi_{J',I}}^-{\sim} & \Db_\Stein(P_{J'}M_{I,1}).
}
\]
\end{thm}

\begin{proof}
The functor $R\Ind_{\dot P_J}^{\dot P_{J'}}$ is right adjoint to the functor $\For_{\dot P_J}^{\dot P_{J'}}$, and the functor $R\Ind_{P_J M_{I,1}}^{P_{J'} M_{I,1}}$ is right adjoint to the functor $\For_{P_J M_{I,1}}^{P_{J'} M_{I,1}}$. Hence to construct an isomorphism as in the statement of the theorem it suffices to construct an isomorphism which makes the following diagram commutative:
\[
\xymatrix@C=2cm{
\Dfg_{\dot P_J}(\bL_I) \ar[r]^-{\psi_{J,I}}_-{\sim} & \Db_\Stein(P_J M_{I,1}) \\
\Dfg_{\dot P_{J'}}(\bL_I) \ar[r]_-{\psi_{J',I}}^-{\sim} \ar[u]^-{\For_{\dot P_J}^{\dot P_{J'}}} & \Db_\Stein(P_{J'}M_{I,1}). \ar[u]_-{\For_{P_J M_{I,1}}^{P_{J'} M_{I,1}}}
}
\]
\begin{figure}
{\tiny
\xymatrix@C=0.8cm@R=1.3cm{
D^+_{\dot P_J}(\bL_I) \ar[rd]^-{\rmZ_I \otimes \For^{\dot P_J}_{P_J M_{I,1}}(-)} &&&& D^+ \Rep(P_J M_{I,1}) \\
& D^+_{P_{J} M_{I,1}}(\bL_I) \ar[r]^-{\sigma_I^*} & D^+_{P_{J} M_{I,1}}(\Rn_I) \ar[r]^-{(\pi_I^*)^{-1}} & D^+_{P_{J} M_{I,1}}(\snn_I) \ar[ru]^-{R\Ind_{N_{I,1} \rtimes P_J M_{I,1}}^{P_J M_{I,1}}} & \\
& D^+_{P_{J'} M_{I,1}}(\bL_{I}) \ar[r]^-{\sigma_{I}^*} \ar[u]^-{\For^{P_{J'} M_{I,1}}_{{P_{J} M_{I,1}}}} & D^+_{P_{J'} M_{I,1}}(\Rn_{I}) \ar[r]^-{(\pi_{I}^*)^{-1}} \ar[u]^-{\For^{P_{J'} M_{I,1}}_{{P_{J} M_{I,1}}}} & D^+_{P_{J'} M_{I,1}}(\snn_{I}) \ar[u]^-{\For^{P_{J'} M_{I,1}}_{{P_{J} M_{I,1}}}} \ar[rd]_-{R\Ind_{N_{I,1} \rtimes P_{J'} M_{I,1}}^{P_{J'} M_{I,1}}}
& \\
D^+_{\dot P_{J'}}(\bL_{I}) \ar[uuu]|-{\For^{\dot P_{J'}}_{\dot P_J}} \ar[ru]_-{\rmZ_{I} \otimes \For^{\dot P_{J'}}_{P_{J'} M_{I,1}}(-)} &&&& D^+ \Rep(P_{J'} M_{I,1}) \ar[uuu]|-{\For_{P_J M_{I,1}}^{P_{J'} M_{I,1}}} \\
}
}
\caption{Diagram for the proof of Theorem~\ref{thm:parabolicJ}}\label{fig:psi-parabolic-2}
\end{figure}
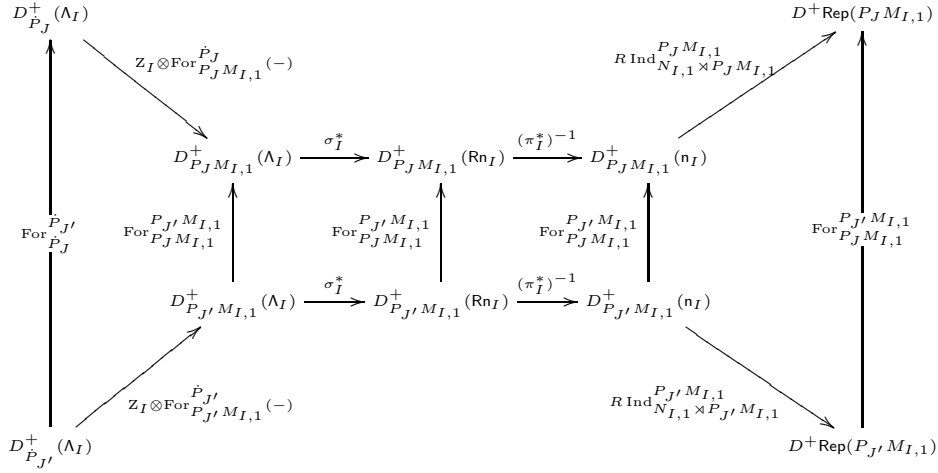%
For this we consider the large diagram of Figure~\ref{fig:psi-parabolic-2}. We will prove that all parts of this diagram are commutative; restricting to $\Dfg_{\dot P_{J'}}(\bL_I)$ will provide the desired isomorphism.

First, we remark that the left-hand trapezoid obviously commutes, and that the two central squares are special cases of diagram~\eqref{eqn:restriction-scalars-For}, so that they indeed commute. Hence to conclude the proof it suffices to prove that the right-hand trapezoid commutes.

By definition, we have
\[
\Ind_{N_{I,1} \rtimes P_{J'} M_{I,1}}^{P_{J'} M_{I,1}} = \mathbb{I}^{N_{I,1} \rtimes P_{J'} M_{I,1}}(\cO(P_{J'} M_{I,1}) \otimes -).
\]
Now the restriction morphism
$\cO(P_{J'} M_{I,1}) \to \cO(P_{J} M_{I,1})$ 
induces a morphism of functors
\begin{multline*}
\For^{P_{J'} M_{I,1}}_{P_{J} M_{I,1}} \circ
\Ind_{N_{I,1} \rtimes P_{J'} M_{I,1}}^{P_{J'} M_{I,1}} = \For^{P_{J'} M_{I,1}}_{P_{J} M_{I,1}} \circ \mathbb{I}^{N_{I,1} \rtimes P_{J'} M_{I,1}}(\cO(P_{J'} M_{I,1}) \otimes -) \\
\to \mathbb{I}^{N_{I,1} \rtimes P_{J} M_{I,1}}(\cO(P_{J} M_{I,1}) \otimes \For^{N_{I,1} \rtimes P_{J'} M_{I,1}}_{N_{I,1} \rtimes P_{J} M_{I,1}}(-)) \\
= \Ind_{N_{I,1} \rtimes P_{J} M_{I,1}}^{P_{J} M_{I,1}} \circ \For^{N_{I,1} \rtimes P_{J'} M_{I,1}}_{N_{I,1} \rtimes P_{J} M_{I,1}}.
\end{multline*}
By general properties of derived functors, this morphism induces a morphism
\begin{equation}
\label{eqn:morph-parabolicJ}
\For^{P_{J'} M_{I,1}}_{P_{J} M_{I,1}} \circ R\Ind_{N_{I,1} \rtimes P_{J'} M_{I,1}}^{P_{J'} M_{I,1}} \to R\Ind_{N_{I,1} \rtimes P_{J} M_{I,1}}^{P_{J} M_{I,1}} \circ \For^{N_{I,1} \rtimes P_{J'} M_{I,1}}_{N_{I,1} \rtimes P_{J} M_{I,1}}.
\end{equation}
By Lemma~\ref{lem:Ind-For}, we have canonical isomorphisms
\begin{align*}
\For^{P_{J'} M_{I,1}}_{P_{I,1}} \circ R\Ind_{N_{I,1} \rtimes P_{J'} M_{I,1}}^{P_{J'} M_{I,1}} &\cong R\Ind_{N_{I,1} \rtimes P_{I,1}}^{P_{I,1}} \circ \For^{N_{I,1} \rtimes P_{J'} M_{I,1}}_{N_{I,1} \rtimes P_{I,1}}, \\
\For^{P_{J} M_{I,1}}_{P_{I,1}} \circ R\Ind_{N_{I,1} \rtimes P_{J} M_{I,1}}^{P_{J} M_{I,1}} &\cong R\Ind_{N_{I,1} \rtimes P_{I,1}}^{P_{I,1}} \circ \For^{N_{I,1} \rtimes P_{J} M_{I,1}}_{N_{I,1} \rtimes P_{I,1}}.
\end{align*}
Moreover, under these identifications, the image of~\eqref{eqn:morph-parabolicJ} is the identity morphism of the functor $R\Ind_{N_{I,1} \rtimes P_{I,1}}^{P_{I,1}} \circ \For^{N_{I,1} \rtimes P_{J'} M_{I,1}}_{N_{I,1} \rtimes P_{I,1}}$; in particular it is an isomorphism. We deduce that~\eqref{eqn:morph-parabolicJ} induces an isomorphism on every object of $D^+ \Rep(N_{I,1} \rtimes P_{J'} M_{I,1})$, hence that it is an isomorphism of functors.
\end{proof}

\subsection{Enlarging $I$}

In this subsection we fix $J \subset I \subset I' \subset S$. Then $P_{I} \subset P_{I'}$ and $\dot \fn_{I'} \subset \dot \fn_I$. We deduce a $\dot P_J$-equivariant embedding of $\dot P_J$-equivariant dg-algebras $j_{I,I'} : \bL_{I'} \hookrightarrow \bL_I$.

\begin{lem}
\label{lem:induction-St}
There exists a canonical isomorphism of (complexes of) $P_J M_{I',1}$-modules
\[
R\Ind_{P_J M_{I,1}}^{P_J M_{I',1}} (\St_I \otimes \bk((\ell-1)(\varsigma_{I'}-\varsigma_I))) \cong \St_{I'}.
\]
\end{lem}

\begin{proof}
By Lemma~\ref{lem:Ind-For}, in $D^+ \Rep(P_{I,1})$ we have
\begin{multline}
\label{eqn:For-Ind-Steinberg}
\For^{P_J M_{I',1}}_{P_{I',1}} \circ R\Ind_{P_J M_{I,1}}^{P_J M_{I',1}} (\St_I \otimes \bk((\ell-1)(\varsigma_{I'}-\varsigma_I)))
\cong \\
R\Ind_{P_{I,1}}^{P_{I',1}} (\St_I \otimes \bk((\ell-1)(\varsigma_{I'}-\varsigma_I))).
\end{multline}
In particular, since the functor $\Ind_{P_{I,1}}^{P_{I',1}}$ is exact (see~\cite[Corollary~I.5.13(b)]{jantzen}), we deduce that $R\Ind_{P_J M_{I,1}}^{P_J M_{I',1}} (\St_I \otimes \bk((\ell-1)(\varsigma_{I'}-\varsigma_I)))$ is concentrated in degree $0$. Now, 
as in Lemma~\ref{lem:st-sll},
we have an isomorphism of $P_{I,1}$-modules
$\St_I \cong \Ind_{B_1}^{P_{I,1}}(\bk_{B_1}((\ell-1)\varsigma_I))$, and similarly for $I'$. It follows that
\begin{multline*}
\Ind_{P_{I,1}}^{P_{I',1}} (\St_I \otimes \bk((\ell-1)(\varsigma_{I'}-\varsigma_I))) \cong \Ind_{P_{I,1}}^{P_{I',1}} ( \Ind_{B_1}^{P_{I,1}}((\ell-1)\varsigma_I) \otimes \bk((\ell-1)(\varsigma_{I'}-\varsigma_I))) \\
\cong \Ind_{B_1}^{P_{I',1}} ((\ell-1)\varsigma_{I'}) \cong \St_{I'},
\end{multline*}
where the second isomorphism uses the tensor identity and transitivity of induction. Combining these isomorphisms with~\eqref{eqn:For-Ind-Steinberg}, we obtain an isomorphism of $P_{I',1}$-modules
\[
\Ind_{P_J M_{I,1}}^{P_J M_{I',1}} (\St_I \otimes \bk((\ell-1)(\varsigma_{I'}-\varsigma_I))) \cong \St_{I'}.
\]

By adjunction we have
\begin{multline*}
\Hom_{P_J M_{I',1}}(\St_{I'}, \Ind_{P_J M_{I,1}}^{P_J M_{I',1}} (\St_I \otimes \bk((\ell-1)(\varsigma_{I'}-\varsigma_I))) \\
\cong \Hom_{P_J M_{I,1}}(\St_{I'}, \St_I \otimes \bk((\ell-1)(\varsigma_{I'} - \varsigma_I))) \\
\cong \Hom_{P_J M_{I,1}}(\Ind_B^{P_{I'}}(\bk((\ell-1)\varsigma_{I'})), \Ind_B^{P_I} ((\ell-1) \varsigma_{I'})).
\end{multline*}
Since restriction of functions from $P_{I'}$ to $P_I$ induces a nonzero $P_I$-equivariant morphism $\Ind_B^{P_{I'}}(\bk((\ell-1)\varsigma_{I'})) \to \Ind_B^{P_I} ((\ell-1) \varsigma_{I'})$, we deduce that there exists a nonzero $P_{J} M_{I',1}$-equivariant morphism
\[
\St_{I'} \to \Ind_{P_J M_{I,1}}^{P_J M_{I',1}} (\St_I \otimes \bk((\ell-1)(\varsigma_{I'}-\varsigma_I)).
\]
Since both of these modules are isomorphic to $\St_{I'}$ as $P_{I',1}$-modules, and since $\End_{P_{I',1}}(\St_{I'})=\bk$, this morphism must be an isomorphism.
\end{proof}

Lemma~\ref{lem:induction-St} and the generalized tensor identity imply that for any $V \in \Rep(\dot P_J)$ we have a canonical isomorphism
\begin{equation}
\label{eqn:Ind-Steinberg-V}
R\Ind_{P_J M_{I,1}}^{P_J M_{I',1}} \bigl( (\St_I \otimes V) \otimes \bk((\ell-1)(\varsigma_{I'}-\varsigma_I)) \bigr) \cong \St_{I'} \otimes V.
\end{equation}
In particular, it follows that
the functor $R\Ind_{P_J M_{I,1}}^{P_J M_{I',1}} (({-}) \otimes \bk((\ell-1)(\varsigma_{I'}-\varsigma_I)))$ restricts to a functor from $\Db_\Stein(P_J M_{I,1})$ to $\Db_\Stein(P_J M_{I',1})$.

\begin{thm}
\label{thm:parabolicI}
The following diagram commutes up to isomorphism:
\[
\xymatrix@C=2.5cm{
\Dfg_{\dot P_J}(\bL_I) \ar[d]_-{j_{I,I'}^*} \ar[r]^-{\psi_{J,I}}_-{\sim} & \Db_\Stein(P_J M_{I,1}) \ar[d]|-{R\Ind_{P_J M_{I,1}}^{P_J M_{I',1}} (({-}) \otimes \bk((\ell-1)(\varsigma_{I'}-\varsigma_I)))} \\
\Dfg_{\dot P_J}(\bL_{I'}) \ar[r]_-{\psi_{J,I'}}^-{\sim} & \Db_\Stein(P_J M_{I',1}).
}
\]
\end{thm}

\begin{proof}
The embedding $\fn_{I'} \subset \fn_I$ induces embeddings of dg-algebras
\[
\widehat{\jmath}_{I,I'} : \Rn_{I'} \hookrightarrow \Rn_I, \qquad \overline{\jmath}_{I,I'} : \snn_{I'} \hookrightarrow \snn_I
\]
such that both squares in the following diagram commute:
\[
\xymatrix@C=1.5cm{
\bL_{I'} \ar[d]_-{j_{I,I'}} & \Rn_{I'} \ar[r]^-{\pi_{I'}} \ar[l]_-{\sigma_{I'}} \ar[d]^-{\widehat{\jmath}_{I,I'}} & \snn_{I'} \ar[d]^-{\overline{\jmath}_{I,I'}} \\
\bL_I & \Rn_I \ar[r]^-{\pi_{I}} \ar[l]_-{\sigma_{I}} & \snn_I.
}
\]
We also set
\[
\mu:=(\ell-1)(\varsigma_I-2\rho_I+2\rho_{I'} - \varsigma_{I'}) \quad \text{and} \quad \nu:=(\ell-1)(\varsigma_{I'}-\varsigma_I).
\]
(Note that both $\mu$ and $\nu$ define characters of $M_I$, and hence of $P_I$ and any of its subgroups.)

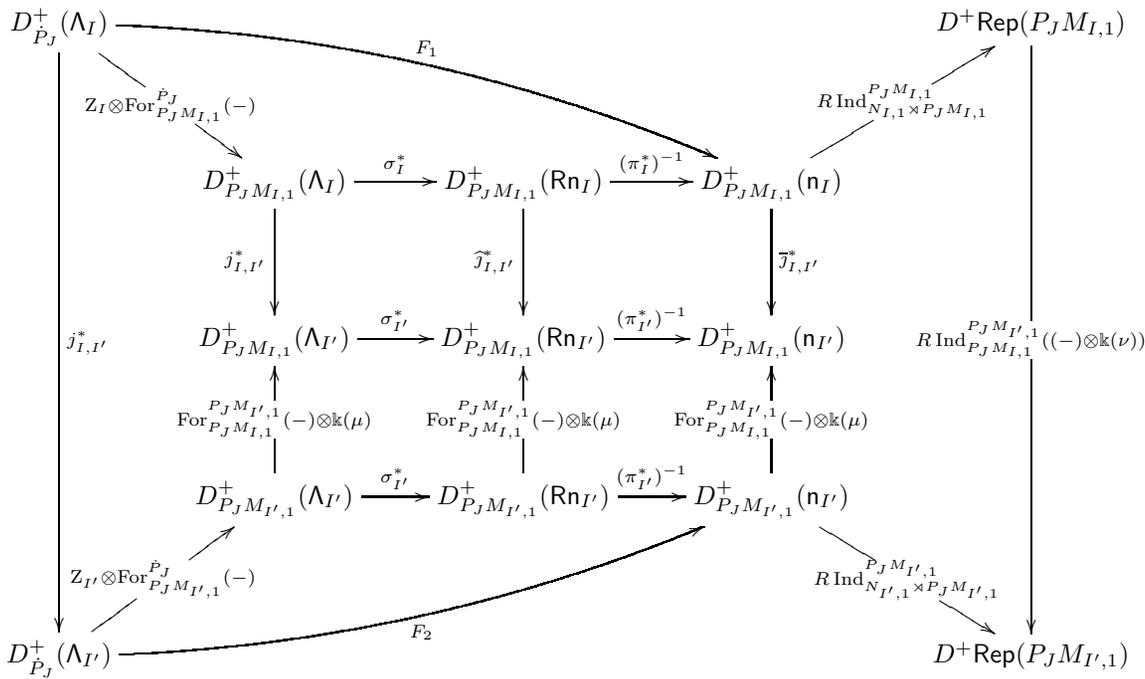
\begin{figure}
\begin{sideways}
\xymatrix@C=0.9cm@R=1.4cm{
D^+_{\dot P_J}(\bL_I) \ar[dddd]^-{j_{I,I'}^*} \ar@/^15pt/[rrrd]^-{F_1} \ar[rd]|-{\rmZ_I \otimes \For^{\dot P_J}_{P_J M_{I,1}}(-)} &&&& D^+ \Rep(P_J M_{I,1}) \ar[dddd]|-{R\Ind_{P_J M_{I,1}}^{P_J M_{I',1}} (({-}) \otimes \bk(\nu))} \\
& D^+_{P_{J} M_{I,1}}(\bL_I) \ar[r]^-{\sigma_I^*} \ar[d]_-{j_{I,I'}^*} & D^+_{P_{J} M_{I,1}}(\Rn_I) \ar[r]^-{(\pi_I^*)^{-1}} \ar[d]_-{\widehat{\jmath}_{I,I'}^*} & D^+_{P_{J} M_{I,1}}(\snn_I) \ar[d]^-{\overline{\jmath}_{I,I'}^*}
\ar[ru]|-{R\Ind_{N_{I,1} \rtimes P_J M_{I,1}}^{P_J M_{I,1}}} & \\
& D^+_{P_{J} M_{I,1}}(\bL_{I'}) \ar[r]^-{\sigma_{I'}^*} & D^+_{P_{J} M_{I,1}}(\Rn_{I'}) \ar[r]^-{(\pi_{I'}^*)^{-1}} & D^+_{P_{J} M_{I,1}}(\snn_{I'}) 
& \\
& D^+_{P_{J} M_{I',1}}(\bL_{I'}) \ar[r]^-{\sigma_{I'}^*} \ar[u]|-{\For^{P_J M_{I',1}}_{P_J M_{I,1}}(-) \otimes \bk(\mu)} & D^+_{P_{J} M_{I',1}}(\Rn_{I'}) \ar[r]^-{(\pi_{I'}^*)^{-1}} \ar[u]|-{\For^{P_J M_{I',1}}_{P_J M_{I,1}}(-) \otimes \bk(\mu)} & D^+_{P_{J} M_{I',1}}(\snn_{I'}) \ar[u]|-{\For^{P_J M_{I',1}}_{P_J M_{I,1}}(-) \otimes \bk(\mu)}
\ar[rd]|-{R\Ind_{N_{I',1} \rtimes P_J M_{I',1}}^{P_J M_{I',1}}} & \\
D^+_{\dot P_J}(\bL_{I'}) \ar[ru]|-{\rmZ_{I'} \otimes \For^{\dot P_J}_{P_J M_{I',1}}(-)} \ar@/_15pt/[rrru]_-{F_2} &&&& D^+ \Rep(P_J M_{I',1}) \\
}
\end{sideways}
\caption{Diagram for the proof of Theorem~\ref{thm:parabolicI}}\label{fig:psi-parabolic}
\end{figure}

Consider the large diagram of Figure~\ref{fig:psi-parabolic}. (Here the functors $F_1$ and $F_2$ are defined so that the corresponding triangle commutes.) It is straightforward (using in particular the commutativity of diagram~\eqref{eqn:restriction-scalars-For}) to check that the left-hand trapezoid and the four central squares in this diagram commute. In an equation, this means that
\begin{equation}
\label{eqn:isom-F1-F2}
\bigl( \For^{P_J M_{I',1}}_{P_J M_{I,1}}(-) \otimes \bk(\mu) \bigr) \circ F_2 \circ j_{I,I'}^* \cong \overline{\jmath}_{I,I'}^* \circ F_1.
\end{equation}

Now we look more closely at the right-hand trapezoid. We complete this part of the diagram as follows:
\[
\xymatrix@C=1.5cm@R=0.6cm{
D^+ \Rep(N_{I,1} \rtimes P_J M_{I,1}) \ar@<-1ex>[d]_-{\For} \ar[r]^-{R\Ind} & D^+ \Rep(P_J M_{I,1}) \ar[dd]^-{R\Ind((-) \otimes \bk(\nu))} \\
D^+ \Rep(N_{I',1} \rtimes P_J M_{I,1}) \ar@<-1ex>[d]_-{R\Ind((-) \otimes \bk(\nu))} \ar@<-1ex>[u]_-{R\Ind} & \\
D^+ \Rep(N_{I',1} \rtimes P_J M_{I',1}) \ar@<-1ex>[u]_-{\For(-) \otimes \bk(\mu)} \ar[r]^-{R\Ind} & D^+ \Rep(P_J M_{I',1}).
}
\]
(Here, for simplicity of notation we have not indicated the groups in the functors $\Ind$ or $\For$. For the vertical arrows, these functors are defined with respect to the obvious inclusions, and for the horizontal arrows they are defined with respect to the multiplication morphisms.) We claim that the pairs of functors
\begin{gather}
\label{eqn:adjunction-Ind-For-1}
(\For^{N_{I,1} \rtimes P_J M_{I,1}}_{N_{I',1} \rtimes P_J M_{I,1}}, R\Ind^{N_{I,1} \rtimes P_J M_{I,1}}_{N_{I',1} \rtimes P_J M_{I,1}}), \\
\label{eqn:adjunction-Ind-For-2}
(R\Ind_{N_{I',1} \rtimes P_J M_{I,1}}^{N_{I',1} \rtimes P_J M_{I',1}}((-) \otimes \bk(\nu)), \For_{N_{I',1} \rtimes P_J M_{I,1}}^{N_{I',1} \rtimes P_J M_{I',1}}(-) \otimes \bk(\mu))
\end{gather}
in this diagram are naturally adjoint pairs. For~\eqref{eqn:adjunction-Ind-For-1}, this follows from the general theory. 

For~\eqref{eqn:adjunction-Ind-For-2}, we first remark that the functor $\Ind_{N_{I',1} \rtimes P_J M_{I,1}}^{N_{I',1} \rtimes P_J M_{I',1}}$ is exact. In fact, using Lemma~\ref{lem:Ind-For} we see that
\[
\For^{N_{I',1} \rtimes P_J M_{I',1}}_{N_{I',1} \rtimes P_{I',1}} \circ R\Ind_{N_{I',1} \rtimes P_J M_{I,1}}^{N_{I',1} \rtimes P_J M_{I',1}} \cong R\Ind^{N_{I',1} \rtimes P_{I',1}}_{N_{I',1} \rtimes P_{I,1}} \circ \For^{N_{I',1} \rtimes P_J M_{I,1}}_{N_{I,1} \rtimes P_{I,1}},
\]
and then the claim follows from~\cite[Proposition~I.5.13(c)]{jantzen}. By~\cite[\S I.8.20]{jantzen}, the functor $\For_{N_{I',1} \rtimes P_J M_{I,1}}^{N_{I',1} \rtimes P_J M_{I',1}}$ has as left adjoint the coinduction functor $\mathrm{Coind}_{N_{I',1} \rtimes P_J M_{I,1}}^{N_{I',1} \rtimes P_J M_{I',1}}$, and moreover we have
\[
\mathrm{Coind}_{N_{I',1} \rtimes P_J M_{I,1}}^{N_{I',1} \rtimes P_J M_{I',1}} \cong \Ind_{N_{I',1} \rtimes P_J M_{I,1}}^{N_{I',1} \rtimes P_J M_{I',1}}(- \otimes \bk((\ell-1)(2\rho_{I'}-2\rho_I))).
\]
(Here, $\mathrm{Coind}_{N_{I',1} \rtimes P_J M_{I,1}}^{N_{I',1} \rtimes P_J M_{I',1}}$ stands for the functor $\mathrm{Coind}_{N_{I',1} \rtimes P_{I,1}}^{N_{I',1} \rtimes P_{I',1}}$, extended as in~\cite[Proposition~I.8.20]{jantzen}.)
Since $\nu = (\ell-1)(2\rho_{I'}-2\rho_I) - \mu$, we deduce the adjunction~\eqref{eqn:adjunction-Ind-For-2}.

Using the adjunction morphisms associated with the pairs~\eqref{eqn:adjunction-Ind-For-1} and~\eqref{eqn:adjunction-Ind-For-2} (together with the generalized tensor identity and the transitivity of induction) we construct a morphism of functors as follows:
{\footnotesize
\begin{multline*}
\label{eqn:def-eta}
R\Ind_{P_J M_{I,1}}^{P_J M_{I',1}} (({-}) \otimes \bk(\nu)) \circ \psi_{J,I} = R\Ind_{P_J M_{I,1}}^{P_J M_{I',1}} (({-}) \otimes \bk(\nu)) \circ R\Ind_{N_{I,1} \rtimes P_J M_{I,1}}^{P_J M_{I,1}} \circ F_1 \\
\xrightarrow{\sim} R\Ind_{N_{I,1} \rtimes P_J M_{I,1}}^{P_J M_{I',1}} (({-}) \otimes \bk(\nu)) \circ F_1
\\
\xrightarrow{\eqref{eqn:adjunction-Ind-For-1}} R\Ind_{N_{I,1} \rtimes P_J M_{I,1}}^{P_J M_{I',1}} (({-}) \otimes \bk(\nu)) \circ R\Ind^{N_{I,1} \rtimes P_J M_{I,1}}_{N_{I',1} \rtimes P_J M_{I,1}} \circ \For^{N_{I,1} \rtimes P_J M_{I,1}}_{N_{I',1} \rtimes P_J M_{I,1}} \circ F_1 \\
\xrightarrow[\sim]{\eqref{eqn:isom-F1-F2}} R\Ind_{N_{I,1} \rtimes P_J M_{I,1}}^{P_J M_{I',1}} (({-}) \otimes \bk(\nu)) \circ R\Ind^{N_{I,1} \rtimes P_J M_{I,1}}_{N_{I',1} \rtimes P_J M_{I,1}} \circ \bigl( \For_{N_{I',1} \rtimes P_J M_{I,1}}^{N_{I',1} \rtimes P_J M_{I',1}}(-) \otimes \bk(\mu) \bigr) \circ F_2 \circ j_{I,I'}^* \\
\xrightarrow{\sim}
R\Ind_{N_{I',1} \rtimes P_J M_{I',1}}^{P_J M_{I',1}} \circ R\Ind^{N_{I',1} \rtimes P_J M_{I',1}}_{N_{I',1} \rtimes P_J M_{I,1}} (({-}) \otimes \bk(\nu)) \circ \bigl( \For_{N_{I',1} \rtimes P_J M_{I,1}}^{N_{I',1} \rtimes P_J M_{I',1}}(-) \otimes \bk(\mu) \bigr) \circ F_2 \circ j_{I,I'}^* \\
\xrightarrow{\eqref{eqn:adjunction-Ind-For-2}} 
R\Ind_{N_{I',1} \rtimes P_J M_{I',1}}^{P_J M_{I',1}} \circ F_2 \circ j_{I,I'}^* = \psi_{J,I'} \circ \circ j_{I,I'}^*.
\end{multline*}
}%
This morphism will be denoted $\eta$.

To conclude the proof it remains to prove that $\eta$ is an isomorphism. For this it suffices to prove that $\eta_V$ is an isomorphism for any $V \in \Repf(\dot P_J)$ (since these objects generate the category $\Dfg_{\dot P_J}(\bL_I)$). And then, by compatibility of all our functors with tensoring with a finite dimensional $\dot P_J$-module, it suffices to consider the case when $V=\bk$. In this case, $\eta_\bk$ is a $P_J M_{I',1}$-equivariant endomorphism of $\St_{I'}$; hence to prove that this morphism is invertible it suffices to prove that it is nonzero. In particular, we can replace all the derived functors appearing in the equations above by their nonderived counterparts. With this replacement, the composition we have to consider looks as follows:
\begin{equation}
\label{eqn:composition-eta-k}
\St_{I'} \to \Ind_{N_{I',1} \rtimes P_J M_{I,1}}^{P_J M_{I',1}}(\rmZ_I \otimes \bk(\nu)) \to \St_{I'}.
\end{equation}

Let us consider the middle term in~\eqref{eqn:composition-eta-k}. One can check, using arguments similar to those in the final step of the proof of Theorem~\ref{thm:parabolicJ}, that, as $P_{I',1} T$-modules, we have
\begin{multline*}
\Ind_{N_{I',1} \rtimes P_J M_{I,1}}^{P_J M_{I',1}}(\rmZ_I \otimes \bk(\nu)) \cong \Ind_{N_{I',1} \rtimes P_{I,1} T}^{P_{I',1} T}(\rmZ_I \otimes \bk(\nu)) \\
\cong \Ind_{N_{I',1} \rtimes M_{I,1} T}^{P_{I',1} T}(\St_I \otimes \bk(\nu)),
\end{multline*}
where the second isomorphism uses a $T$-equivariant version of Lemma~\ref{lem:pi-ind} (see the proof of Corollary~\ref{cor:StI-PiI}). We deduce that, as $M_{I',1} T$-modules, we have
\[
\Ind_{N_{I',1} \rtimes P_J M_{I,1}}^{P_J M_{I',1}}(\rmZ_I \otimes \bk(\nu)) \cong \Ind_{M_{I,1} T}^{M_{I',1} T}(\St_I \otimes \bk(\nu)).
\]
Using Lemma~\ref{lem:Ind-Steinberg}, we see that to conclude, it suffices to prove that both morphisms appearing in~\eqref{eqn:composition-eta-k} are nonzero.

The first morphism is the image under the left exact functor $\Ind_{N_{I,1} \rtimes P_J M_{I,1}}^{P_J M_{I',1}}$ of the injective adjunction morphism $\rmZ_I \otimes \bk(\nu) \to \Ind_{N_{I',1} \rtimes P_J M_{I,1}}^{N_{I,1} \rtimes P_J M_{I,1}}(\rmZ_I \otimes \bk(\nu))$. Therefore it is injective, and in particular nonzero. 

To handle the second morphism, as above we restrict equivariance to $P_{I',1} T$. In this setting, the morphism under consideration is the image under the functor $\Ind_{N_{I',1} \rtimes P_{I',1} T}^{P_{I',1} T}$ of the morphism
\begin{equation}
\Ind_{N_{I',1} \rtimes P_{I,1} T}^{N_{I',1} \rtimes P_{I',1} T}(\rmZ_{I'} \otimes \bk((\ell-1)(2\rho_{I'}-2\rho_I))) \to \rmZ_{I'}
\end{equation}
induced by adjunction. This morphism is surjective. It is even a \emph{split} surjection. In fact, since $N_{I',1} \subset N_{I',1} \rtimes P_{I,1} T$ acts trivially on all the modules under consideration we have
\begin{multline*}
\Ind_{N_{I',1} \rtimes P_{I,1} T}^{N_{I',1} \rtimes P_{I',1} T}(\rmZ_{I'} \otimes \bk((\ell-1)(2\rho_{I'}-2\rho_I))) \cong \\
\For^{P_{I',1} T}_{N_{I',1} \rtimes P_{I',1} T} \bigl( \Ind_{P_{I,1} T}^{P_{I',1} T}(\rmZ_{I'} \otimes \bk((\ell-1)(2\rho_{I'}-2\rho_I))) \bigr)
\end{multline*}
(where the forgetful functor is defined with respect to the projection $N_{I',1} \rtimes P_{I',1} T \to P_{I',1} T$ on the second factor), and our morphism is induced by the surjective morphism
\[
 \Ind_{P_{I,1} T}^{P_{I',1} T}(\rmZ_{I'} \otimes \bk((\ell-1)(2\rho_{I'}-2\rho_I))) \to \rmZ_{I'}
\]
induced by adjunction. Since $\rmZ_{I'}$ is projective as a $P_{I',1}T$-module (see Remark~\ref{rmk:pi-ind-proj}), this surjection must be split, which finally proves that the second morphism in~\eqref{eqn:composition-eta-k} is nonzero, and concludes the proof of the theorem.
\end{proof}

\subsection{The functors $\Theta_{J,I}$ and $\Theta^{J,I}$}
\label{ss:Theta}

In the rest of the paper, we mainly consider the functors $\psi_{J,I}$ only in the special case $J=I$. In this case, we simplify the notation and set
\[
\psi_I := \psi_{I,I} : \Dfg_{\dot P_I}(\bL_I) \simto \Db_{\Stein}(P_I).
\]

Now we fix two subsets $J \subset I \subset S$. Recall the embedding $j_{J,I} : \bL_I \hookrightarrow \bL_J$. We consider the functor
\[
\Theta_{J, I}:= R\Ind_{\dot P_J}^{\dot P_I} \circ j_{J,I}^* \circ \bigl( (-) \otimes \bk_{\dot P_J}(\varsigma_J-\varsigma_I) \bigr) : \Dfg_{\dot P_J}(\bL_J) \to \Dfg_{\dot P_I}(\bL_I).
\]

\begin{prop}
\label{prop:Theta-psi}
The following diagram commutes up to isomorphism:
\[
\xymatrix@C=2cm{
 \Dfg_{\dot P_J}(\bL_J) \ar[r]^-{\psi_J} \ar[d]_-{\Theta_{J,I}} & \Db_{\Stein}(P_J)\ar[d]|-{R\Ind_{P_J}^{P_I} \bigl( (-) \otimes \bk(\varsigma_J - \varsigma_I) \bigr)} \\
 \Dfg_{\dot P_I}(\bL_I) \ar[r]^-{\psi_I}Ê& \Db_{\Stein}(P_I).
}
\]
\end{prop}

\begin{proof}
Consider the following diagram:
\[
\xymatrix@C=2cm{
 \Dfg_{\dot P_J}(\bL_J) \ar[r]^-{\psi_J} \ar[d]|-{(-) \otimes \bk_{\dot P_J}(\varsigma_J-\varsigma_I)} & \Db_{\Stein}(P_J)\ar[d]|-{(-) \otimes \bk_{P_J}(\ell(\varsigma_J - \varsigma_I))} \\
  \Dfg_{\dot P_J}(\bL_J) \ar[r]^-{\psi_J} \ar[d]_-{j_{J,I}^*} & \Db_{\Stein}(P_J)\ar[d]|-{R\Ind_{P_J}^{P_J M_{I,1}} \bigl( (-) \otimes \bk((\ell-1)(\varsigma_I - \varsigma_J)) \bigr)} \\
   \Dfg_{\dot P_J}(\bL_I) \ar[r]^-{\psi_{J,I}} \ar[d]_-{R\Ind_{\dot P_J}^{\dot P_I}} & \Db_{\Stein}(P_J M_{I,1})\ar[d]|-{R\Ind_{P_J M_{I,1}}^{P_I}} \\
 \Dfg_{\dot P_I}(\bL_I) \ar[r]^-{\psi_I}Ê& \Db_{\Stein}(P_I).
}
\]
The upper square is commutative by Lemma~\ref{lem:frobenius-psi}. The middle square commutes by Theorem~\ref{thm:parabolicI}, and the bottom square commutes by Theorem~\ref{thm:parabolicJ}. The composition on the left-hand side is $\Theta_{J,I}$, and the composition on the right-hand side is isomorphic to $R\Ind_{P_J}^{P_I}((-) \otimes \bk(\varsigma_J - \varsigma_I))$ (see~\eqref{eqn:transitivity-RInd}). Hence the proposition is proved.
\end{proof}

The algebra $\bL_J$ is free of finite rank as a right $\bL_I$-module; in particular it is K-flat as a right $\bL_I$-dg-module. Therefore the functor $\bL_J \otimes_{\bL_I} (-) : \bL_I\ldgmod_{\dot P_J} \to \bL_J\ldgmod_{\dot P_J}$ is exact, and induces a triangulated functor
\[
\bL_J \lotimes_{\bL_I} (-) : \Dfg_{\dot P_J}(\bL_I) \to \Dfg_{\dot P_J}(\bL_J).
\]
This functor is easily seen to be left adjoint to the functor $j_{J,I}^*$. Hence, if we set
\[
\Theta^{J,I} := \bigl( (-) \otimes \bk_{\dot P_J}(\varsigma_I - \varsigma_J) \bigr) \circ \bigl( \bL_J \lotimes_{\bL_I} (-) \bigr) \circ \For^{\dot P_I}_{\dot P_J} : \Dfg_{\dot P_I}(\bL_I) \to \Dfg_{\dot P_J}(\bL_J),
\]
then the functor $\Theta^{J,I}$ is left adjoint to $\Theta_{J,I}$.

\newpage

\part{Induction theorems}
\label{pt:induction}

\textbf{Overview.}
The main goal of this part is to prove the induction theorem (Theorem~\ref{thm:intro-induction}). This proof appears to be long and quite technical. For this reason, we start this part with a detailed overview explaining the basic ideas of this proof.

As explained in~\S\ref{ss:exotic-induction}, instead of considering the functor $R\Ind_{P_I}^G$ directly, we will consider the composition
\[
\Db \Coh^{\dot G \times \Gm}(\tcN_I) \xrightarrow{\varkappa_I} \Dfg_{\dot P_I}(\bL_I) \xrightarrow{\psi_I} \Db_{\Stein}(P_I) \xrightarrow{R\Ind_{P_I}^G} \Db \Rep_I(G),
\]
where $\psi_I$ is as in~\S\ref{ss:Theta}, and the functor $\varkappa_I$ is induced by the Koszul duality functor of Section~\ref{sec:koszul}. The main point of this is that we can consider some ``standard'' and ``costandard'' objects in $\Db \Coh^{\dot G \times \Gm}(\tcN_I)$ with favorable $\Hom$-vanishing properties. This construction is performed in Section~\ref{sec:exotic}. In case $I=\varnothing$ these objects are simply the standard and costandard objects in the heart of the exotic t-structure, which are well known from~\cite{bez:ctm, ar:agsr, mr:etspc}. In fact, a reader interested only in the case $I=\varnothing$ and familiar with the exotic t-structure may skip most of Section~\ref{sec:exotic}. (From this section, only \S\S\ref{ss:Springer-res}--\ref{ss:Ind-Res-Springer}, \S\ref{ss:reminder-exotic-empty}, and~\S\ref{ss:Koszul-Springer} will be used in the proof of this special case.)

We will show that the composition $R\Ind_{P_I}^G {}\circ \psi_I \circ \varkappa_I$ sends these objects to the usual standard and costandard objects in $\Rep_I(G)$ (see Proposition~\ref{prop:exotic-weyl}).
For this proof, the crucial case is when $I=\varnothing$. In this case, the claim is easy for certain standard (resp.~costandard) objects, and we will deduce the other cases from these ones using translation functors and some analogous functors $\Pi_{J,I}$ and $\Pi^{J,I}$ relating the categories $\Db \Coh^{\dot G \times \Gm}(\tcN_I)$ for different choices of $I$. The compatibility between the functors $R\Ind_{P_I}^G \circ \psi_I \circ \varkappa_I$ and translation functors is proved in Section~\ref{sec:translation}, building on the results of Section~\ref{sec:compatibility}. (More precisely, we compare the functors $\Pi_{J,I}$ and $\Pi^{J,I}$ with the functors $\Theta_{J,I}$ and $\Theta^{J,I}$ of Section~\ref{sec:compatibility} via $\varkappa_I$ in~\S\ref{ss:Koszul-Springer}, and the functors $\Theta_{J,I}$ and $\Theta^{J,I}$ with the translation functors via $R\Ind_{P_I}^G \circ \psi_I$ in~\S\ref{ss:main-translation}.)

But we will need more than the mere existence of some isomorphisms of functors: in order to prove that a certain morphism in the distinguished triangle~\eqref{eqn:exact-seq-costandard} below is nonzero, we will need to prove that one can construct certain isomorphisms of functors which are \emph{compatible with adjunctions} in an appropriate sense. This leads us to the notion of ``commutative diagram of adjoint pairs'', which is introduced and studied in Section~\ref{sec:translation}.

Once all these ingredients are introduced, the proof of the induction theorem is not difficult; see Section~\ref{sec:induction-thm}. The application to the ``graded Finkelberg--Mirkovi\'c conjecture'' is presented in the final Section~\ref{sec:fm-conjecture}.

\section{Translation functors}
\label{sec:translation}
\tikzstyle{trans}=[yscale=0.4,xscale=0.6,baseline,font=\small]
\tikzstyle{transsmall}=[scale=0.4,baseline,font=\tiny]

\newcommand{\ssm}{\setminus}

\subsection{Setting}
\label{ss:setting}

From now on we assume that the derived subgroup $\mathscr{D}(G)$ of $G$ is simply connected, and denote by $T'$ the maximal torus of $\mathscr{D}(G)$ contained in $T$. For any $\alpha \in \Sigma$, we denote by $\varpi_\alpha \in X^*(T')$ the corresponding fundamental weight, and we choose a preimage $\varsigma_\alpha$ of $\varpi_\alpha$ under the surjective morphism $\bX \to X^*(T')$. Then, for any $K \subset S$, we choose $\varsigma_K$ as
\[
\varsigma_K = \sum_{\alpha \in K} \varsigma_\alpha.
\]
With this choice, for any $J \subset I$ we have
\[
\varsigma_I-\varsigma_J = \varsigma_{I \ssm J}.
\]

We define the affine Weyl group $\Waff$ as the semi-direct product $W \ltimes \bX$. (This group is sometimes rather called the \emph{extended} affine Weyl group.) To avoid confusion, for $\lambda \in \bX$ we denote by $t_\lambda$ the element $1 \ltimes \lambda \in \Waff$. The group $\Waff$ acts on $\bX$ via the ``dot action'' defined by
\[
(v t_\lambda) \bullet \mu := v(\mu + \ell \lambda + \rho) - \rho.
\]
The subgroup $\WaffCox:=W \ltimes \Z\Phi$ of $\Waff$ has a natural Coxeter group structure (where we use the same normalization as in~\cite[\S 2.2]{mr:etspc}). Then the Bruhat order and the length function extend in a natural way to $\Waff$. We set
\[
\Waff^\circ := \{w \in \Waff \mid \ell(w)=0\};
\]
then conjugation by $\Waff^\circ$ stabilizes the set of simple reflections in $\Waff$, and we have $\Waff = \Waff^\circ \ltimes \WaffCox$.

Under our running assumption that $\ell>h$, $-\varsigma_K$ belongs to 
\[
\overline{C}_\Z := \{ \lambda \in \bX \mid \text{$0 \le \la \lambda + \rho, \alpha^\vee \ra \le \ell$ for all $\alpha \in \Phi^+$} \}.
\]
Moreover, this weight has ``singularity $K$'' in the sense that it belongs to the walls of $\overline{C}_\Z$ parametrized by the simple roots in $K$, and to no other wall.
By standard arguments (see~\cite[\S II.6.3]{jantzen}), this implies that
\begin{equation}
\label{eqn:stab-varsigmaK}
\{w \in \WaffCox \mid w \bullet (-\varsigma_K) = -\varsigma_K\} = W_K.
\end{equation}

For any $I \subset S$, we set
\[
\bX_I^+ := \{\lambda \in \bX \mid \forall \alpha \in \Phi_I^+, \, \la \lambda,\alpha^\vee \ra \geq 0\}.
\]
Then, for $\lambda \in \bX_I^+$, we denote by
\[
\weyl_I(\lambda), \quad \coweyl_I(\lambda), \quad \irr_I(\lambda)
\]
the Weyl, dual Weyl, and simple $M_I$-modules of highest weight $\lambda$, respectively.
We will also consider these $M_I$-modules as $P_I$-modules via the surjection $P_I \twoheadrightarrow M_I$. As usual, when $I=S$ we omit the subscript in this notation. (In the case $I=\{s\}$, these modules have already been encountered in~\S\ref{ss:ssrk1}.)

Now we fix $J \subset I \subset S$.  In this section, we will build on the results of~\S\ref{ss:Theta} to obtain a relationship between the adjoint functors $(\Theta^{J, I},\Theta_{J, I})$ and translation functors for $\Rep(G)$.  A summary of the categories and functors we will work with in this section appears in Figure~\ref{fig:translation}. 

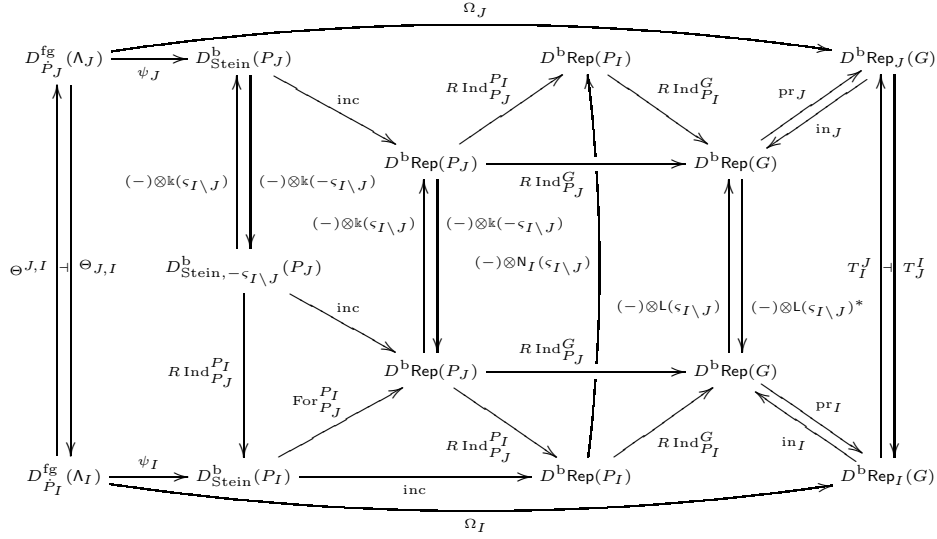
\begin{figure}
\tiny
\[
\xymatrix@C=17pt{
\Dfg_{\dot P_J}(\bL_J) \ar[r]_-{\psi_J} \ar@{}[dddd]|\dashv \ar@<1ex>[dddd]^{\Theta_{J,I}} \ar@/^5ex/[rrrrr]^{\Omega_J}&
  \Db_\Stein(P_J) \ar[dr]^{\inc} \ar@<1ex>[dd]^(.6){({-})\otimes\bk(-\varsigma_{I\ssm J})} &&
  \Db\Rep(P_I) \ar[dr]^{R\Ind_{P_I}^G} &&
  \Db\Rep_{J}(G) \ar@<.75ex>[dl]^{\incl_{J}} \ar@<1ex>[dddd]^{T_{J}^{I}} \ar@{}[dddd]|\dashv \\
&& \Db\Rep(P_J) \ar[rr]_(.4){R\Ind_{P_J}^G} \ar[ur]^{R\Ind_{P_J}^{P_I}} \ar@<1ex>[dd]^(.3){({-})\otimes\bk(-\varsigma_{I\ssm J})} 
 && \Db\Rep(G) \ar@<1ex>[dd]^(.7){({-})\otimes \irr(\varsigma_{I\ssm J})^*} \ar@<.75ex>[ur]^{\pr_{J}} \\
& \Db_{\Stein,-\varsigma_{I\ssm J}}(P_J) \ar[dd]_{R\Ind_{P_J}^{P_I}} \ar@<1ex>[uu]^(.4){({-})\otimes\bk(\varsigma_{I\ssm J})} \ar[dr]^{\inc} \\
&& \Db\Rep(P_J) \ar@<1ex>[uu]^(.7){({-})\otimes\bk(\varsigma_{I\ssm J})} \ar[rr]^(.4){R\Ind_{P_J}^G} \ar[dr]_{R\Ind_{P_J}^{P_I}}
 && \Db\Rep(G) \ar@<1ex>[uu]^(.3){({-})\otimes \irr(\varsigma_{I\ssm J})} \ar@<.75ex>[dr]^{\pr_{I}} \\
\Dfg_{\dot P_I}(\bL_I) \ar[r]^-{\psi_I} \ar@<1ex>[uuuu]^{\Theta^{J,I}} \ar@/_5ex/[rrrrr]_{\Omega_I}&
  \Db_\Stein(P_I) \ar[ur]^{\For_{P_J}^{P_I}} \ar[rr]_{\inc} &&
  \Db\Rep(P_I) \ar[ur]_{R\Ind_{P_I}^G} \ar@/_2ex/[uuuu]^-{({-})\otimes \coweyl_I(\varsigma_{I\ssm J})}|!{[ul];[ur]}\hole|!{[uuul];[uuur]}\hole &&
  \Db\Rep_{I}(G) \ar@<.75ex>[ul]^{\incl_{I}} \ar@<1ex>[uuuu]^{T_{I}^{J}}} 
\]
\caption{Diagram for the study of translation functors}\label{fig:translation}
\end{figure}

Let us explain the notation used in this figure that has not been introduced yet.
First, on the right-hand side, for $K \in \{I,J\}$ we denote by
$\Rep_K(G)$
the Serre subcategory of $\Rep(G)$ generated by the simple modules whose highest weight belongs to $\bX^+ \cap \Waff \bullet(-\varsigma_K)$. It is well known that this subcategory is a direct summand in $\Rep(G)$, and we denote by
\[
\incl_K : \Rep_K(G) \to \Rep(G), \qquad \pr_K : \Rep(G) \to \Rep_K(G)
\]
the corresponding inclusion and projection functors respectively, or the induced functors on derived categories. 
Note that in general $\Rep_K(G)$ is a direct sum of \emph{several} blocks of $\Rep(G)$, even when $K=\varnothing$; this is due to the fact that we work with $\Waff$ and not with $\WaffCox$. More precisely, for any $\omega \in \Waff^\circ$ we can consider the Serre subcategory $\Rep_{K,\omega}(G)$ of $\Rep(G)$ generated by the simple modules whose highest weight belongs to $\bX^+ \cap \WaffCox \omega \bullet(-\varsigma_K)$. Then each $\Rep_{K,\omega}(G)$ is a direct summand in $\Rep(G)$, and $\Rep_K(G)$ is the direct sum of these subcategories.

We also consider the translation functors
\begin{align*}
T_I^J &:= 
\pr_J{} \circ \bigl( ({-})\otimes \irr(\varsigma_{I\ssm J}) \bigr) \circ \incl_{I} : \Rep_I(G) \to \Rep_J(G), \\
\qquad
T_J^I &:= 
\pr_{I}{} \circ \bigl( ({-})\otimes \irr(\varsigma_{I\ssm J})^* \bigr) \circ \incl_{J} : \Rep_J(G) \to \Rep_I(G).
\end{align*}
For any $\omega \in \Waff^\circ$, the restriction of $T_I^J$ to $\Rep_{I,\omega}(G)$ is the functor denoted $T_{\omega \bullet (-\varsigma_I)}^{\omega \bullet (-\varsigma_J)}$ in~\cite[\S II.7.6]{jantzen}, and the restriction of $T_J^I$ to $\Rep_{J,\omega}(G)$ is the functor denoted $T^{\omega \bullet (-\varsigma_I)}_{\omega \bullet (-\varsigma_J)}$ in~\cite[\S II.7.6]{jantzen}

In the left-hand side of the diagram, 
$\Db_{\Stein,-\varsigma_{I\ssm J}}(P_J)$
denotes the full triangulated subcategory of $\Db\Rep(P_J)$ generated by objects of the form $V \otimes \bk(-\varsigma_{I\ssm J})$ with $V \in \Db_\Stein(P_J)$.  The functors $\inc: \Db_\Stein(P_J) \to \Db\Rep(P_J)$ and $\inc: \Db_{\Stein,-\varsigma_{I\ssm J}}(P_J) \to \Db\Rep(P_J)$ are inclusion functors.  

Finally, the functors $\Omega_J$ and $\Omega_I$ are given by
\begin{equation}
\label{eqn:omega-defn}
\Omega_J = \pr_{J} \circ R\Ind_{P_J}^G{} \circ \inc{} \circ \psi_J
\qquad\text{and}\qquad
\Omega_I = \pr_I \circ R\Ind_{P_I}^G{} \circ \inc{} \circ \psi_I.
\end{equation}

Later we will need the following easy lemma.

\begin{lem}
\label{lem:generators-SteinPK}
For any $K \subset S$,
the triangulated category
$\Db_\Stein(P_K)$ is generated 
by the objects of the form $\coweyl_K(\ell\lambda - \varsigma_K)$ with $\lambda \in \bX_K^+ + \varsigma_K$, or by the objects of the form $\weyl_K(\ell\lambda - \varsigma_K)$ with $\lambda \in \bX_K^+ + \varsigma_K$, or 
by the objects of the form $\irr_K(\ell\lambda - \varsigma_K)$ with $\lambda \in \bX_K^+ + \varsigma_K$.
\end{lem}

\begin{proof}
Note that
\[
((W_K \ltimes \bX) \bullet (-\varsigma_K)) \cap \bX_K^+ = \{\ell\lambda - \varsigma_K, \, \lambda \in \bX_K^+ + \varsigma_K\}.
\]
Using this and~\cite[II.7.3(5)]{jantzen}, we see that
the three cases are equivalent; we will prove the case of the objects $\irr_K(\ell \lambda - \varsigma_K)$.

By definition (see~\S\ref{ss:statement-equiv-formality}), $\Db_\Stein(P_K)$ is generated as a triangulated category by the objects of the form $\St_K \otimes \For^{\dot P_K}_{P_K}(\irr_K(\mu))$ with $\mu \in \bX_K^+$. Now since $\St_K$ is simple as an $M_K$-module (see~\S\ref{ss:Steinberg}), by Steinberg's tensor product theorem (see~\cite[Proposition~II.3.16]{jantzen}) we have
\[
\St_K \otimes \For^{\dot P_K}_{P_K}(\irr_K(\mu)) \cong \irr_K((\ell-1)\varsigma_K) \otimes \For^{\dot P_K}_{P_K}(\irr_K(\mu)) \cong \irr_K((\ell-1)\varsigma_K + \ell \mu),
\]
and the claim follows.
\end{proof}

\subsection{String diagrams and commutative diagrams of adjoint pairs}
\label{ssec:string}

It will be convenient to use the ``string diagram'' notation to carry out computations with natural transformations.  The string diagrams in this section should be read from top to bottom.  We follow the usual convention that if $p \dashv q$ is an adjoint pair of functors (with a fixed adjunction), then the unit $\eta: \id \to qp$ and the counit $\epsilon: pq \to \id$ are denoted by
\[
\vcenter{\hbox{\begin{tikzpicture}[trans]
\useasboundingbox (-0.5,-0.5) rectangle (2.5,1.5);
\draw (0.00,0.00) -- (0.05,0.07) -- (0.10,0.15) -- (0.15,0.22) -- (0.20,0.30) -- (0.25,0.37) -- (0.30,0.44) -- (0.35,0.50) -- (0.40,0.57) -- (0.45,0.63) -- (0.50,0.69) -- (0.55,0.74) -- (0.60,0.79) -- (0.65,0.84) -- (0.70,0.88) -- (0.75,0.91) -- (0.80,0.94) -- (0.85,0.97) -- (0.90,0.99) -- (0.95,1.00) -- (1.00,1.00);
\draw (1.00,1.00) -- (1.05,1.00) -- (1.10,0.99) -- (1.15,0.97) -- (1.20,0.94) -- (1.25,0.91) -- (1.30,0.88) -- (1.35,0.84) -- (1.40,0.79) -- (1.45,0.74) -- (1.50,0.69) -- (1.55,0.63) -- (1.60,0.57) -- (1.65,0.50) -- (1.70,0.44) -- (1.75,0.37) -- (1.80,0.30) -- (1.85,0.22) -- (1.90,0.15) -- (1.95,0.07) -- (2.00,0.00);
\draw (0.00,-0.50) node{$q$};
\draw (2.00,-0.50) node{$p$};
\end{tikzpicture}}}
\qquad\text{and}\qquad
\vcenter{\hbox{\begin{tikzpicture}[trans]
\useasboundingbox (-0.5,-0.5) rectangle (2.5,2.5);
\draw (0.00,1.00) -- (0.05,0.93) -- (0.10,0.85) -- (0.15,0.78) -- (0.20,0.70) -- (0.25,0.63) -- (0.30,0.56) -- (0.35,0.50) -- (0.40,0.43) -- (0.45,0.37) -- (0.50,0.31) -- (0.55,0.26) -- (0.60,0.21) -- (0.65,0.16) -- (0.70,0.12) -- (0.75,0.09) -- (0.80,0.06) -- (0.85,0.03) -- (0.90,0.01) -- (0.95,0.00) -- (1.00,0.00);
\draw (1.00,0.00) -- (1.05,0.00) -- (1.10,0.01) -- (1.15,0.03) -- (1.20,0.06) -- (1.25,0.09) -- (1.30,0.12) -- (1.35,0.16) -- (1.40,0.21) -- (1.45,0.26) -- (1.50,0.31) -- (1.55,0.37) -- (1.60,0.43) -- (1.65,0.50) -- (1.70,0.56) -- (1.75,0.63) -- (1.80,0.70) -- (1.85,0.78) -- (1.90,0.85) -- (1.95,0.93) -- (2.00,1.00);
\draw (0.00,1.50) node{$p$};
\draw (2.00,1.50) node{$q$};
\end{tikzpicture}}}
\]
respectively. The most important rules for doing calculations with string diagrams are those coming from the unit-counit equations
\[
\epsilon p \circ p \eta = \id_p \qquad \text{and}\qquad q \epsilon \circ \eta q = \id_q
\]
(sometimes called the ``zigzag relations''), depicted graphically as
\begin{equation}\label{eqn:unit-counit}
\vcenter{\hbox{\begin{tikzpicture}[trans]
\useasboundingbox (-0.5,-0.5) rectangle (3.5,2.5);
\draw (0.00,1.00) -- (0.05,0.92) -- (0.10,0.83) -- (0.15,0.75) -- (0.20,0.67) -- (0.25,0.59) -- (0.30,0.52) -- (0.35,0.45) -- (0.40,0.38) -- (0.45,0.31) -- (0.50,0.25) -- (0.55,0.19) -- (0.60,0.14) -- (0.65,0.10) -- (0.70,0.06) -- (0.75,0.03) -- (0.80,0.01) -- (0.85,-0.01) -- (0.90,-0.01) -- (0.95,-0.01) -- (1.00,0.00);
\draw (1.00,0.00) -- (1.02,0.01) -- (1.05,0.02) -- (1.07,0.04) -- (1.10,0.05) -- (1.12,0.07) -- (1.15,0.09) -- (1.17,0.11) -- (1.20,0.14) -- (1.22,0.16) -- (1.25,0.19) -- (1.27,0.22) -- (1.30,0.24) -- (1.32,0.27) -- (1.35,0.30) -- (1.37,0.34) -- (1.40,0.37) -- (1.42,0.40) -- (1.45,0.43) -- (1.47,0.47) -- (1.50,0.50);
\draw (1.50,0.50) -- (1.52,0.53) -- (1.55,0.57) -- (1.57,0.60) -- (1.60,0.63) -- (1.62,0.66) -- (1.65,0.70) -- (1.67,0.73) -- (1.70,0.76) -- (1.72,0.78) -- (1.75,0.81) -- (1.77,0.84) -- (1.80,0.86) -- (1.82,0.89) -- (1.85,0.91) -- (1.87,0.93) -- (1.90,0.95) -- (1.92,0.96) -- (1.95,0.98) -- (1.97,0.99) -- (2.00,1.00);
\draw (2.00,1.00) -- (2.05,1.01) -- (2.10,1.01) -- (2.15,1.01) -- (2.20,0.99) -- (2.25,0.97) -- (2.30,0.94) -- (2.35,0.90) -- (2.40,0.86) -- (2.45,0.81) -- (2.50,0.75) -- (2.55,0.69) -- (2.60,0.62) -- (2.65,0.55) -- (2.70,0.48) -- (2.75,0.41) -- (2.80,0.33) -- (2.85,0.25) -- (2.90,0.17) -- (2.95,0.08) -- (3.00,0.00);
\draw (0.00,1.50) node{$p$};
\draw (3.00,-0.50) node{$p$};
\end{tikzpicture}}}
\ =\ 
\vcenter{\hbox{\begin{tikzpicture}[trans]
\useasboundingbox (-0.5,-0.5) rectangle (0.5,2.5);
\draw (0.00,1.00) -- (0.00,0.00);
\draw (0.00,1.50) node{$p$};
\draw (0.00,-0.50) node{$p$};
\end{tikzpicture}}}
\qquad\text{and}\qquad
\vcenter{\hbox{\begin{tikzpicture}[trans]
\useasboundingbox (-0.5,-0.5) rectangle (3.5,2.5);
\draw (0.00,0.00) -- (0.05,0.08) -- (0.10,0.17) -- (0.15,0.25) -- (0.20,0.33) -- (0.25,0.41) -- (0.30,0.48) -- (0.35,0.55) -- (0.40,0.62) -- (0.45,0.69) -- (0.50,0.75) -- (0.55,0.81) -- (0.60,0.86) -- (0.65,0.90) -- (0.70,0.94) -- (0.75,0.97) -- (0.80,0.99) -- (0.85,1.01) -- (0.90,1.01) -- (0.95,1.01) -- (1.00,1.00);
\draw (1.00,1.00) -- (1.02,0.99) -- (1.05,0.98) -- (1.07,0.96) -- (1.10,0.95) -- (1.12,0.93) -- (1.15,0.91) -- (1.17,0.89) -- (1.20,0.86) -- (1.22,0.84) -- (1.25,0.81) -- (1.27,0.78) -- (1.30,0.76) -- (1.32,0.73) -- (1.35,0.70) -- (1.37,0.66) -- (1.40,0.63) -- (1.42,0.60) -- (1.45,0.57) -- (1.47,0.53) -- (1.50,0.50);
\draw (1.50,0.50) -- (1.52,0.47) -- (1.55,0.43) -- (1.57,0.40) -- (1.60,0.37) -- (1.62,0.34) -- (1.65,0.30) -- (1.67,0.27) -- (1.70,0.24) -- (1.72,0.22) -- (1.75,0.19) -- (1.77,0.16) -- (1.80,0.14) -- (1.82,0.11) -- (1.85,0.09) -- (1.87,0.07) -- (1.90,0.05) -- (1.92,0.04) -- (1.95,0.02) -- (1.97,0.01) -- (2.00,0.00);
\draw (2.00,0.00) -- (2.05,-0.01) -- (2.10,-0.01) -- (2.15,-0.01) -- (2.20,0.01) -- (2.25,0.03) -- (2.30,0.06) -- (2.35,0.10) -- (2.40,0.14) -- (2.45,0.19) -- (2.50,0.25) -- (2.55,0.31) -- (2.60,0.38) -- (2.65,0.45) -- (2.70,0.52) -- (2.75,0.59) -- (2.80,0.67) -- (2.85,0.75) -- (2.90,0.83) -- (2.95,0.92) -- (3.00,1.00);
\draw (3.00,1.50) node{$q$};
\draw (0.00,-0.50) node{$q$};
\end{tikzpicture}}}
\ =\ 
\vcenter{\hbox{\begin{tikzpicture}[trans]
\useasboundingbox (-0.5,-0.5) rectangle (0.5,2.5);
\draw (0.00,1.00) -- (0.00,0.00);
\draw (0.00,1.50) node{$q$};
\draw (0.00,-0.50) node{$q$};
\end{tikzpicture}}}
\end{equation}

Suppose now that we have four categories $\mathcal{A}, \mathcal{A}', \mathcal{B}, \mathcal{B}'$, with functors $f: \mathcal{A} \to \mathcal{B}$ and $f': \mathcal{A}' \to \mathcal{B}'$ and two adjoint pairs $p \dashv q$ and $r \dashv s$ as shown in the following diagram:
\begin{equation}\label{eqn:comm-adjoint}
\vcenter{
\xymatrix@C=2cm{
\mathcal{A} \ar[r]^{f} \ar@{}[d]|\dashv \ar@<1ex>[d]^q & \mathcal{B} \ar@{}[d]|\dashv \ar@<1ex>[d]^s \\
\mathcal{A}' \ar[r]^{f'} \ar@<1ex>[u]^p & \mathcal{B}'. \ar@<1ex>[u]^r}
}
\end{equation}
There exists a bijection
\begin{equation}\label{eqn:mateship}
\mathrm{Mor}(f'q,sf) \simto \mathrm{Mor}(rf',fp)
\end{equation}
that sends a morphism $\theta: f'q \to sf$ to the morphism $\theta^\wedge: rf' \to fp$ defined by
\ifdefined\PARTCOMPILE{
\[
FIGURE
\]
}
\else {
\[
\vcenter{\hbox{\begin{tikzpicture}[trans]
\useasboundingbox (-0.5,-0.5) rectangle (2.5,3.5);
\draw (1.00,1.00) -- (2.00,0.00);
\draw (1.00,1.00) -- (0.00,0.00);
\draw (2.00,2.00) -- (1.00,1.00);
\draw (0.00,2.00) -- (1.00,1.00);
\filldraw[fill=white] (1.00,1.00) ellipse (0.80cm and 0.50cm);
\draw (0.00,2.50) node{$r$};
\draw (2.00,2.50) node{$f'$};
\draw (1.00,1.00) node{$\theta^\wedge$};
\draw (0.00,-0.50) node{$f$};
\draw (2.00,-0.50) node{$p$};
\end{tikzpicture}}}
\ = \ 
\vcenter{\hbox{\begin{tikzpicture}[trans]
\useasboundingbox (-0.5,-0.5) rectangle (4.5,3.5);
\draw (2.00,1.00) -- (2.05,1.08) -- (2.09,1.15) -- (2.14,1.22) -- (2.18,1.30) -- (2.23,1.37) -- (2.28,1.44) -- (2.33,1.51) -- (2.37,1.57) -- (2.42,1.63) -- (2.47,1.69) -- (2.52,1.74) -- (2.57,1.79) -- (2.62,1.84) -- (2.67,1.88) -- (2.72,1.92) -- (2.78,1.95) -- (2.83,1.97) -- (2.89,1.99) -- (2.94,2.00) -- (3.00,2.00);
\draw (3.00,2.00) -- (3.06,2.00) -- (3.12,1.98) -- (3.18,1.97) -- (3.24,1.94) -- (3.30,1.91) -- (3.36,1.87) -- (3.42,1.83) -- (3.48,1.79) -- (3.53,1.73) -- (3.59,1.68) -- (3.64,1.62) -- (3.69,1.56) -- (3.74,1.50) -- (3.79,1.43) -- (3.84,1.36) -- (3.88,1.29) -- (3.91,1.22) -- (3.95,1.15) -- (3.98,1.07) -- (4.00,1.00);
\draw (4.00,1.00) -- (4.01,0.95) -- (4.03,0.90) -- (4.04,0.85) -- (4.05,0.80) -- (4.05,0.75) -- (4.06,0.70) -- (4.06,0.64) -- (4.06,0.59) -- (4.06,0.54) -- (4.06,0.49) -- (4.06,0.45) -- (4.05,0.40) -- (4.05,0.35) -- (4.04,0.30) -- (4.04,0.25) -- (4.03,0.20) -- (4.02,0.15) -- (4.02,0.10) -- (4.01,0.05) -- (4.00,0.00);
\draw (2.00,2.00) -- (2.01,1.95) -- (2.01,1.90) -- (2.02,1.85) -- (2.03,1.80) -- (2.03,1.75) -- (2.04,1.70) -- (2.04,1.65) -- (2.05,1.60) -- (2.05,1.56) -- (2.05,1.51) -- (2.06,1.46) -- (2.06,1.41) -- (2.05,1.36) -- (2.05,1.31) -- (2.05,1.25) -- (2.04,1.20) -- (2.03,1.15) -- (2.03,1.10) -- (2.01,1.05) -- (2.00,1.00);
\draw (2.00,1.00) -- (1.98,0.93) -- (1.95,0.85) -- (1.92,0.78) -- (1.88,0.71) -- (1.85,0.64) -- (1.80,0.57) -- (1.76,0.50) -- (1.71,0.44) -- (1.66,0.38) -- (1.61,0.32) -- (1.56,0.26) -- (1.50,0.21) -- (1.44,0.17) -- (1.38,0.12) -- (1.32,0.09) -- (1.26,0.06) -- (1.19,0.03) -- (1.13,0.01) -- (1.06,0.00) -- (1.00,0.00);
\draw (1.00,0.00) -- (0.94,0.00) -- (0.87,0.01) -- (0.81,0.03) -- (0.74,0.06) -- (0.68,0.09) -- (0.62,0.12) -- (0.56,0.17) -- (0.50,0.21) -- (0.44,0.26) -- (0.39,0.32) -- (0.34,0.38) -- (0.29,0.44) -- (0.24,0.50) -- (0.20,0.57) -- (0.15,0.64) -- (0.12,0.71) -- (0.08,0.78) -- (0.05,0.85) -- (0.02,0.93) -- (0.00,1.00);
\draw (0.00,1.00) -- (-0.01,1.05) -- (-0.03,1.10) -- (-0.03,1.15) -- (-0.04,1.20) -- (-0.05,1.25) -- (-0.05,1.31) -- (-0.05,1.36) -- (-0.06,1.41) -- (-0.06,1.46) -- (-0.05,1.51) -- (-0.05,1.56) -- (-0.05,1.60) -- (-0.04,1.65) -- (-0.04,1.70) -- (-0.03,1.75) -- (-0.03,1.80) -- (-0.02,1.85) -- (-0.01,1.90) -- (-0.01,1.95) -- (0.00,2.00);
\draw (2.00,1.00) -- (2.00,0.00);
\filldraw[fill=white] (2.00,1.00) ellipse (0.80cm and 0.50cm);
\draw (0.00,2.50) node{$r$};
\draw (2.00,2.50) node{$f'$};
\draw (2.00,1.00) node{$\theta$};
\draw (2.00,-0.50) node{$f$};
\draw (4.00,-0.50) node{$p$};
\end{tikzpicture}}}
\]}\fi
The inverse map of~\eqref{eqn:mateship} associates to $\phi: rf' \to fp$ the morphism $\phi^\vee : f'q \to sf$ defined by
\ifdefined\PARTCOMPILE{
\[
FIGURE
\]
}
\else {
\[
\vcenter{\hbox{\begin{tikzpicture}[trans]
\useasboundingbox (-0.5,-0.5) rectangle (2.5,3.5);
\draw (1.00,1.00) -- (2.00,0.00);
\draw (1.00,1.00) -- (0.00,0.00);
\draw (2.00,2.00) -- (1.00,1.00);
\draw (0.00,2.00) -- (1.00,1.00);
\filldraw[fill=white] (1.00,1.00) ellipse (0.80cm and 0.50cm);
\draw (0.00,2.50) node{$f'$};
\draw (2.00,2.50) node{$q$};
\draw (1.00,1.00) node{$\phi^\vee$};
\draw (0.00,-0.50) node{$s$};
\draw (2.00,-0.50) node{$f$};
\end{tikzpicture}}}
\ = \ 
\vcenter{\hbox{\begin{tikzpicture}[trans]
\useasboundingbox (-0.5,-0.5) rectangle (4.5,3.5);
\draw (0.00,0.00) -- (-0.01,0.05) -- (-0.02,0.10) -- (-0.02,0.15) -- (-0.03,0.20) -- (-0.04,0.25) -- (-0.04,0.30) -- (-0.05,0.35) -- (-0.05,0.40) -- (-0.06,0.45) -- (-0.06,0.49) -- (-0.06,0.54) -- (-0.06,0.59) -- (-0.06,0.64) -- (-0.06,0.70) -- (-0.05,0.75) -- (-0.05,0.80) -- (-0.04,0.85) -- (-0.03,0.90) -- (-0.01,0.95) -- (0.00,1.00);
\draw (0.00,1.00) -- (0.02,1.07) -- (0.05,1.15) -- (0.09,1.22) -- (0.12,1.29) -- (0.16,1.36) -- (0.21,1.43) -- (0.26,1.50) -- (0.31,1.56) -- (0.36,1.62) -- (0.41,1.68) -- (0.47,1.73) -- (0.52,1.79) -- (0.58,1.83) -- (0.64,1.87) -- (0.70,1.91) -- (0.76,1.94) -- (0.82,1.97) -- (0.88,1.98) -- (0.94,2.00) -- (1.00,2.00);
\draw (1.00,2.00) -- (1.06,2.00) -- (1.11,1.99) -- (1.17,1.97) -- (1.22,1.95) -- (1.28,1.92) -- (1.33,1.88) -- (1.38,1.84) -- (1.43,1.79) -- (1.48,1.74) -- (1.53,1.69) -- (1.58,1.63) -- (1.63,1.57) -- (1.67,1.51) -- (1.72,1.44) -- (1.77,1.37) -- (1.82,1.30) -- (1.86,1.22) -- (1.91,1.15) -- (1.95,1.08) -- (2.00,1.00);
\draw (2.00,1.00) -- (2.00,0.00);
\draw (2.00,2.00) -- (1.99,1.95) -- (1.99,1.90) -- (1.98,1.85) -- (1.97,1.80) -- (1.97,1.75) -- (1.96,1.70) -- (1.96,1.65) -- (1.95,1.60) -- (1.95,1.56) -- (1.95,1.51) -- (1.94,1.46) -- (1.94,1.41) -- (1.95,1.36) -- (1.95,1.31) -- (1.95,1.25) -- (1.96,1.20) -- (1.97,1.15) -- (1.97,1.10) -- (1.99,1.05) -- (2.00,1.00);
\draw (2.00,1.00) -- (2.02,0.93) -- (2.05,0.85) -- (2.08,0.78) -- (2.12,0.71) -- (2.15,0.64) -- (2.20,0.57) -- (2.24,0.50) -- (2.29,0.44) -- (2.34,0.38) -- (2.39,0.32) -- (2.44,0.26) -- (2.50,0.21) -- (2.56,0.17) -- (2.62,0.12) -- (2.68,0.09) -- (2.74,0.06) -- (2.81,0.03) -- (2.87,0.01) -- (2.94,0.00) -- (3.00,0.00);
\draw (3.00,0.00) -- (3.06,0.00) -- (3.13,0.01) -- (3.19,0.03) -- (3.26,0.06) -- (3.32,0.09) -- (3.38,0.12) -- (3.44,0.17) -- (3.50,0.21) -- (3.56,0.26) -- (3.61,0.32) -- (3.66,0.38) -- (3.71,0.44) -- (3.76,0.50) -- (3.80,0.57) -- (3.85,0.64) -- (3.88,0.71) -- (3.92,0.78) -- (3.95,0.85) -- (3.98,0.93) -- (4.00,1.00);
\draw (4.00,1.00) -- (4.01,1.05) -- (4.03,1.10) -- (4.03,1.15) -- (4.04,1.20) -- (4.05,1.25) -- (4.05,1.31) -- (4.05,1.36) -- (4.06,1.41) -- (4.06,1.46) -- (4.05,1.51) -- (4.05,1.56) -- (4.05,1.60) -- (4.04,1.65) -- (4.04,1.70) -- (4.03,1.75) -- (4.03,1.80) -- (4.02,1.85) -- (4.01,1.90) -- (4.01,1.95) -- (4.00,2.00);
\filldraw[fill=white] (2.00,1.00) ellipse (0.80cm and 0.50cm);
\draw (2.00,2.50) node{$f'$};
\draw (4.00,2.50) node{$q$};
\draw (2.00,1.00) node{$\phi$};
\draw (0.00,-0.50) node{$s$};
\draw (2.00,-0.50) node{$f$};
\end{tikzpicture}}}
\]}\fi
The unit-counit relations~\eqref{eqn:unit-counit} imply that the assignments $\theta \mapsto \theta^\wedge$ and $\phi \mapsto \phi^\vee$ are indeed inverse to one another.

These constructions satisfy the following property.

\begin{lem}
\label{lem:mateship-diagram}
Let $\theta \in \mathrm{Mor}(f'q,sf)$.
For any $X$ in $\mathcal{A}'$ and $Y$ in $\mathcal{A}$, the following diagram commutes:
\[
\xymatrix{
\Hom_{\mathcal{A}}(pX,Y) \ar[r]^-{f} \ar[d]_-{\mathrm{adj}}^-{\wr} & \Hom_{\mathcal{B}}(fpX,fY) \ar[d]^{(-) \circ \theta_X^\wedge} \\
\Hom_{\mathcal{A}'}(X,qY) \ar[r] & \Hom_{\mathcal{B}}(rf'X,fY),
}
\]
where the bottom map is the composition
\begin{multline*}
\Hom_{\mathcal{A}'}(X,qY) \xrightarrow{f'} \Hom_{\mathcal{B}'}(f'X,f'qY) \xrightarrow{\theta_Y \circ (-)} \Hom_{\mathcal{B}'}(f'X,sfY) \\
\xrightarrow[\sim]{\mathrm{adj}} \Hom_{\mathcal{B}}(rf'X, fY).  
\end{multline*}
\end{lem}

\begin{proof}
Consider the diagram of Figure~\ref{fig:mateship-diagram} (where we simplify the notation, and write e.g.~$\theta$ for $\theta_Y \circ (-)$). It follows from the definitions that each part of this diagram is commutative, and the exterior square in this diagram is exactly the diagram of the lemma.
\begin{figure}
\begin{sideways}
\xymatrix@R=2cm@C=0.5cm{
\Hom(pX,Y) \ar[ddd]_-{\mathrm{adj}}^-{\wr} \ar[rrr]^-{f} \ar[rd]^-{q} &&& \Hom(fpX,fY) \ar[d]^-{\epsilon} \ar@/^12ex/[ddd]^- {\theta^\wedge} \ar[ld]_-{s} \\
& \Hom(qpX,qY) \ar@/_2ex/[ldd]_-{\eta} \ar[d]^-{f'} & \Hom(sfpX,sfY) \ar[r]^-{\mathrm{adj}}_-{\sim} \ar[d]^-{\theta} & \Hom(rsfpX,fY) \ar[d]^-{\theta} \\
& \Hom(f'qpX,f'qY) \ar[d]^-{\eta} \ar[r]^-{\theta} & \Hom(f'qpX,sfY) \ar[r]^-{\mathrm{adj}}_-{\sim} \ar[d]^-{\eta} & \Hom(rf'qpX,fY) \ar[d]^-{\eta} \\
\Hom(X,qY) \ar[r]^-{f'} & \Hom(f'X,f'qY) \ar[r]^-{\theta} & \Hom(f'X,sfY) \ar[r]^-{\mathrm{adj}}_-{\sim} &\Hom(rf'X, fY).
}
\end{sideways}
\caption{Hom-spaces for Lemma~\ref{lem:mateship-diagram}}\label{fig:mateship-diagram}
\end{figure}
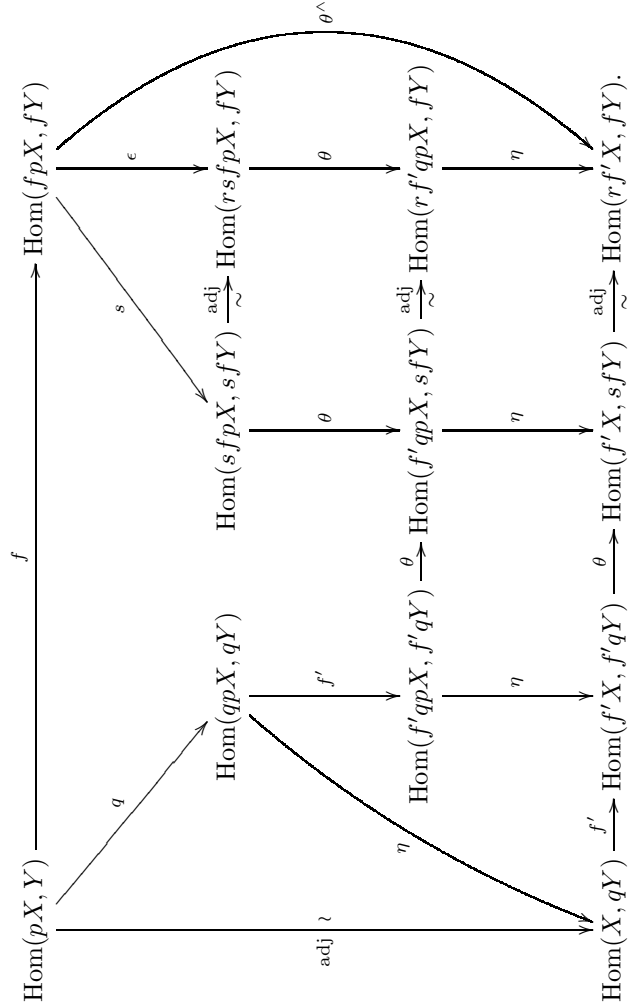
\end{proof}

\begin{defn}
\label{defn:comm-adjoint}
The diagram~\eqref{eqn:comm-adjoint} is said to be a \emph{commutative diagram of adjoint pairs} if there exists an isomorphism $\theta: f'q \simto sf$ such that $\theta^\wedge: rf' \to fp$ is also an isomorphism.
\end{defn}

Of course, the condition in Definition~\ref{defn:comm-adjoint} is equivalent to requiring that there be an isomorphism $\phi: rf' \to fp$ such that $\phi^\vee : f'q \to sf$ is also an isomorphism.  The following easy observation (which is standard and was already implicitly used in the proof of Theorem~\ref{thm:parabolicJ}) says that Definition~\ref{defn:comm-adjoint} is easy to satisfy when $f$ and $f'$ are equivalences.

\begin{lem}
\label{lem:equivalence-mate}
In diagram~\eqref{eqn:comm-adjoint}, suppose $f$ and $f'$ are equivalences of categories.  If $\theta: f'q \to sf$ is an isomorphism, then $\theta^\wedge: rf' \to fp$ is an isomorphism as well.  Similarly, if $\phi: rf' \to fp$ is an isomorphism, then so is $\phi^\vee: f'q \to sf$.
\end{lem}

\begin{proof}
This statement can be deduced from Lemma~\ref{lem:mateship-diagram} and the Yoneda lemma. Alternatively, one can argue using string diagrams as follows.
If $\theta$ is an isomorphism, then the following two natural transformations (whose construction uses the natural adjunctions $f^{-1} \dashv f$ and $f'{}^{-1} \dashv f'$) are isomorphisms as well, inverse to each other:
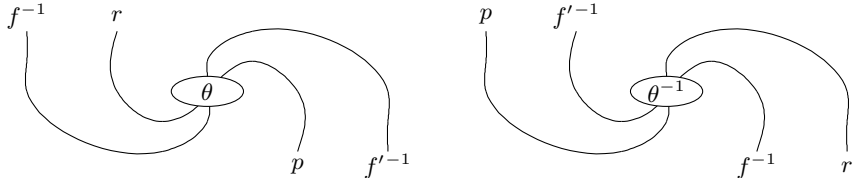
\ifdefined\PARTCOMPILE{
\[
FIGURE
\]
}
\else {
\[
\vcenter{\hbox{\begin{tikzpicture}[trans]
\useasboundingbox (-0.5,-0.5) rectangle (8.5,5.5);
\draw (2.00,4.00) -- (1.98,3.90) -- (1.96,3.80) -- (1.94,3.71) -- (1.92,3.61) -- (1.91,3.51) -- (1.89,3.41) -- (1.88,3.31) -- (1.87,3.21) -- (1.86,3.11) -- (1.85,3.01) -- (1.85,2.92) -- (1.85,2.82) -- (1.85,2.71) -- (1.86,2.61) -- (1.87,2.51) -- (1.89,2.41) -- (1.91,2.31) -- (1.93,2.21) -- (1.96,2.10) -- (2.00,2.00);
\draw (2.00,2.00) -- (2.03,1.93) -- (2.06,1.85) -- (2.10,1.78) -- (2.13,1.71) -- (2.17,1.64) -- (2.22,1.57) -- (2.26,1.50) -- (2.31,1.44) -- (2.36,1.37) -- (2.41,1.32) -- (2.46,1.26) -- (2.52,1.21) -- (2.57,1.16) -- (2.63,1.12) -- (2.69,1.09) -- (2.75,1.06) -- (2.81,1.03) -- (2.87,1.01) -- (2.94,1.00) -- (3.00,1.00);
\draw (3.00,1.00) -- (3.06,1.00) -- (3.13,1.01) -- (3.19,1.03) -- (3.26,1.06) -- (3.32,1.09) -- (3.39,1.12) -- (3.45,1.16) -- (3.51,1.21) -- (3.57,1.26) -- (3.63,1.32) -- (3.68,1.37) -- (3.73,1.44) -- (3.78,1.50) -- (3.82,1.57) -- (3.86,1.64) -- (3.90,1.71) -- (3.93,1.78) -- (3.96,1.85) -- (3.98,1.93) -- (4.00,2.00);
\draw (4.00,2.00) -- (4.01,2.05) -- (4.01,2.10) -- (4.02,2.16) -- (4.02,2.21) -- (4.02,2.26) -- (4.02,2.31) -- (4.02,2.36) -- (4.01,2.41) -- (4.01,2.46) -- (4.00,2.51) -- (4.00,2.56) -- (4.00,2.61) -- (3.99,2.66) -- (3.99,2.71) -- (3.99,2.76) -- (3.99,2.81) -- (3.99,2.86) -- (3.99,2.90) -- (3.99,2.95) -- (4.00,3.00);
\draw (4.00,3.00) -- (4.01,3.07) -- (4.03,3.13) -- (4.06,3.20) -- (4.09,3.26) -- (4.12,3.32) -- (4.16,3.38) -- (4.20,3.44) -- (4.25,3.50) -- (4.30,3.56) -- (4.35,3.61) -- (4.40,3.66) -- (4.46,3.71) -- (4.52,3.76) -- (4.59,3.80) -- (4.65,3.84) -- (4.72,3.88) -- (4.79,3.91) -- (4.86,3.95) -- (4.93,3.97) -- (5.00,4.00);
\draw (5.00,4.00) -- (5.19,4.05) -- (5.38,4.07) -- (5.57,4.08) -- (5.77,4.06) -- (5.96,4.02) -- (6.15,3.97) -- (6.34,3.90) -- (6.53,3.81) -- (6.71,3.71) -- (6.88,3.60) -- (7.05,3.47) -- (7.21,3.34) -- (7.35,3.19) -- (7.49,3.04) -- (7.61,2.87) -- (7.72,2.71) -- (7.82,2.53) -- (7.90,2.36) -- (7.96,2.18) -- (8.00,2.00);
\draw (8.00,2.00) -- (8.01,1.95) -- (8.01,1.90) -- (8.02,1.85) -- (8.02,1.80) -- (8.03,1.75) -- (8.03,1.70) -- (8.03,1.65) -- (8.03,1.60) -- (8.03,1.55) -- (8.03,1.50) -- (8.03,1.45) -- (8.02,1.40) -- (8.02,1.35) -- (8.02,1.30) -- (8.02,1.25) -- (8.01,1.20) -- (8.01,1.15) -- (8.01,1.10) -- (8.00,1.05) -- (8.00,1.00);
\draw (8.00,1.00) -- (8.00,0.95) -- (8.00,0.90) -- (7.99,0.85) -- (7.99,0.80) -- (7.99,0.75) -- (7.99,0.70) -- (7.99,0.65) -- (7.99,0.60) -- (7.99,0.55) -- (7.99,0.50) -- (7.99,0.45) -- (7.99,0.40) -- (7.99,0.35) -- (7.99,0.30) -- (7.99,0.25) -- (8.00,0.20) -- (8.00,0.15) -- (8.00,0.10) -- (8.00,0.05) -- (8.00,0.00);
\draw (0.00,4.00) -- (0.00,3.95) -- (0.00,3.90) -- (0.00,3.85) -- (0.01,3.80) -- (0.01,3.75) -- (0.01,3.70) -- (0.01,3.65) -- (0.01,3.60) -- (0.01,3.55) -- (0.01,3.50) -- (0.01,3.45) -- (0.01,3.40) -- (0.01,3.35) -- (0.01,3.30) -- (0.01,3.25) -- (0.01,3.20) -- (0.01,3.15) -- (0.00,3.10) -- (0.00,3.05) -- (0.00,3.00);
\draw (0.00,3.00) -- (-0.00,2.95) -- (-0.01,2.90) -- (-0.01,2.85) -- (-0.01,2.80) -- (-0.02,2.75) -- (-0.02,2.70) -- (-0.02,2.65) -- (-0.03,2.60) -- (-0.03,2.55) -- (-0.03,2.50) -- (-0.03,2.45) -- (-0.03,2.40) -- (-0.03,2.35) -- (-0.03,2.30) -- (-0.03,2.25) -- (-0.03,2.20) -- (-0.02,2.15) -- (-0.02,2.10) -- (-0.01,2.05) -- (0.00,2.00);
\draw (0.00,2.00) -- (0.04,1.82) -- (0.11,1.64) -- (0.19,1.47) -- (0.28,1.29) -- (0.40,1.13) -- (0.52,0.96) -- (0.66,0.81) -- (0.81,0.66) -- (0.97,0.53) -- (1.14,0.40) -- (1.31,0.29) -- (1.49,0.19) -- (1.68,0.10) -- (1.87,0.03) -- (2.06,-0.02) -- (2.25,-0.06) -- (2.44,-0.08) -- (2.63,-0.07) -- (2.82,-0.05) -- (3.00,0.00);
\draw (3.00,0.00) -- (3.07,0.02) -- (3.14,0.05) -- (3.20,0.09) -- (3.27,0.12) -- (3.33,0.16) -- (3.40,0.20) -- (3.46,0.24) -- (3.52,0.29) -- (3.57,0.34) -- (3.63,0.39) -- (3.68,0.44) -- (3.73,0.50) -- (3.77,0.56) -- (3.81,0.62) -- (3.85,0.68) -- (3.89,0.74) -- (3.92,0.80) -- (3.95,0.87) -- (3.98,0.93) -- (4.00,1.00);
\draw (4.00,1.00) -- (4.01,1.05) -- (4.02,1.10) -- (4.03,1.14) -- (4.04,1.19) -- (4.05,1.24) -- (4.05,1.29) -- (4.05,1.34) -- (4.05,1.39) -- (4.05,1.44) -- (4.05,1.49) -- (4.05,1.54) -- (4.05,1.59) -- (4.04,1.64) -- (4.04,1.69) -- (4.03,1.74) -- (4.03,1.80) -- (4.02,1.85) -- (4.01,1.90) -- (4.01,1.95) -- (4.00,2.00);
\draw (4.00,2.00) -- (4.05,2.08) -- (4.09,2.15) -- (4.14,2.22) -- (4.19,2.30) -- (4.24,2.37) -- (4.29,2.44) -- (4.33,2.51) -- (4.38,2.57) -- (4.43,2.63) -- (4.48,2.69) -- (4.53,2.74) -- (4.58,2.79) -- (4.63,2.84) -- (4.68,2.88) -- (4.73,2.92) -- (4.78,2.95) -- (4.84,2.97) -- (4.89,2.99) -- (4.94,3.00) -- (5.00,3.00);
\draw (5.00,3.00) -- (5.06,3.00) -- (5.11,2.98) -- (5.17,2.97) -- (5.23,2.94) -- (5.28,2.91) -- (5.34,2.88) -- (5.40,2.83) -- (5.45,2.79) -- (5.51,2.74) -- (5.56,2.68) -- (5.61,2.62) -- (5.67,2.56) -- (5.72,2.50) -- (5.76,2.43) -- (5.81,2.36) -- (5.85,2.29) -- (5.89,2.22) -- (5.93,2.15) -- (5.97,2.07) -- (6.00,2.00);
\draw (6.00,2.00) -- (6.04,1.90) -- (6.07,1.79) -- (6.10,1.69) -- (6.13,1.59) -- (6.14,1.49) -- (6.16,1.39) -- (6.16,1.29) -- (6.17,1.19) -- (6.17,1.09) -- (6.16,0.99) -- (6.16,0.89) -- (6.15,0.79) -- (6.13,0.69) -- (6.12,0.59) -- (6.10,0.49) -- (6.08,0.39) -- (6.06,0.29) -- (6.04,0.20) -- (6.02,0.10) -- (6.00,0.00);
\filldraw[fill=white] (4.00,2.00) ellipse (0.80cm and 0.50cm);
\draw (0.00,4.50) node{$f^{-1}$};
\draw (2.00,4.50) node{$r$};
\draw (4.00,2.00) node{$\theta$};
\draw (6.00,-0.50) node{$p$};
\draw (8.00,-0.50) node{$f^{\prime}{}^{-1}$};
\end{tikzpicture}}}
\qquad
\vcenter{\hbox{\begin{tikzpicture}[trans]
\useasboundingbox (-0.5,-0.5) rectangle (8.5,5.5);
\draw (2.00,4.00) -- (1.98,3.90) -- (1.96,3.80) -- (1.94,3.71) -- (1.92,3.61) -- (1.91,3.51) -- (1.89,3.41) -- (1.88,3.31) -- (1.87,3.21) -- (1.86,3.11) -- (1.85,3.01) -- (1.85,2.92) -- (1.85,2.82) -- (1.85,2.71) -- (1.86,2.61) -- (1.87,2.51) -- (1.89,2.41) -- (1.91,2.31) -- (1.93,2.21) -- (1.96,2.10) -- (2.00,2.00);
\draw (2.00,2.00) -- (2.03,1.93) -- (2.06,1.85) -- (2.10,1.78) -- (2.13,1.71) -- (2.17,1.64) -- (2.22,1.57) -- (2.26,1.50) -- (2.31,1.44) -- (2.36,1.37) -- (2.41,1.32) -- (2.46,1.26) -- (2.52,1.21) -- (2.57,1.16) -- (2.63,1.12) -- (2.69,1.09) -- (2.75,1.06) -- (2.81,1.03) -- (2.87,1.01) -- (2.94,1.00) -- (3.00,1.00);
\draw (3.00,1.00) -- (3.06,1.00) -- (3.13,1.01) -- (3.19,1.03) -- (3.26,1.06) -- (3.32,1.09) -- (3.39,1.12) -- (3.45,1.16) -- (3.51,1.21) -- (3.57,1.26) -- (3.63,1.32) -- (3.68,1.37) -- (3.73,1.44) -- (3.78,1.50) -- (3.82,1.57) -- (3.86,1.64) -- (3.90,1.71) -- (3.93,1.78) -- (3.96,1.85) -- (3.98,1.93) -- (4.00,2.00);
\draw (4.00,2.00) -- (4.01,2.05) -- (4.01,2.10) -- (4.02,2.16) -- (4.02,2.21) -- (4.02,2.26) -- (4.02,2.31) -- (4.02,2.36) -- (4.01,2.41) -- (4.01,2.46) -- (4.00,2.51) -- (4.00,2.56) -- (4.00,2.61) -- (3.99,2.66) -- (3.99,2.71) -- (3.99,2.76) -- (3.99,2.81) -- (3.99,2.86) -- (3.99,2.90) -- (3.99,2.95) -- (4.00,3.00);
\draw (4.00,3.00) -- (4.01,3.07) -- (4.03,3.13) -- (4.06,3.20) -- (4.09,3.26) -- (4.12,3.32) -- (4.16,3.38) -- (4.20,3.44) -- (4.25,3.50) -- (4.30,3.56) -- (4.35,3.61) -- (4.40,3.66) -- (4.46,3.71) -- (4.52,3.76) -- (4.59,3.80) -- (4.65,3.84) -- (4.72,3.88) -- (4.79,3.91) -- (4.86,3.95) -- (4.93,3.97) -- (5.00,4.00);
\draw (5.00,4.00) -- (5.19,4.05) -- (5.38,4.07) -- (5.57,4.08) -- (5.77,4.06) -- (5.96,4.02) -- (6.15,3.97) -- (6.34,3.90) -- (6.53,3.81) -- (6.71,3.71) -- (6.88,3.60) -- (7.05,3.47) -- (7.21,3.34) -- (7.35,3.19) -- (7.49,3.04) -- (7.61,2.87) -- (7.72,2.71) -- (7.82,2.53) -- (7.90,2.36) -- (7.96,2.18) -- (8.00,2.00);
\draw (8.00,2.00) -- (8.01,1.95) -- (8.01,1.90) -- (8.02,1.85) -- (8.02,1.80) -- (8.03,1.75) -- (8.03,1.70) -- (8.03,1.65) -- (8.03,1.60) -- (8.03,1.55) -- (8.03,1.50) -- (8.03,1.45) -- (8.02,1.40) -- (8.02,1.35) -- (8.02,1.30) -- (8.02,1.25) -- (8.01,1.20) -- (8.01,1.15) -- (8.01,1.10) -- (8.00,1.05) -- (8.00,1.00);
\draw (8.00,1.00) -- (8.00,0.95) -- (8.00,0.90) -- (7.99,0.85) -- (7.99,0.80) -- (7.99,0.75) -- (7.99,0.70) -- (7.99,0.65) -- (7.99,0.60) -- (7.99,0.55) -- (7.99,0.50) -- (7.99,0.45) -- (7.99,0.40) -- (7.99,0.35) -- (7.99,0.30) -- (7.99,0.25) -- (8.00,0.20) -- (8.00,0.15) -- (8.00,0.10) -- (8.00,0.05) -- (8.00,0.00);
\draw (0.00,4.00) -- (0.00,3.95) -- (0.00,3.90) -- (0.00,3.85) -- (0.01,3.80) -- (0.01,3.75) -- (0.01,3.70) -- (0.01,3.65) -- (0.01,3.60) -- (0.01,3.55) -- (0.01,3.50) -- (0.01,3.45) -- (0.01,3.40) -- (0.01,3.35) -- (0.01,3.30) -- (0.01,3.25) -- (0.01,3.20) -- (0.01,3.15) -- (0.00,3.10) -- (0.00,3.05) -- (0.00,3.00);
\draw (0.00,3.00) -- (-0.00,2.95) -- (-0.01,2.90) -- (-0.01,2.85) -- (-0.01,2.80) -- (-0.02,2.75) -- (-0.02,2.70) -- (-0.02,2.65) -- (-0.03,2.60) -- (-0.03,2.55) -- (-0.03,2.50) -- (-0.03,2.45) -- (-0.03,2.40) -- (-0.03,2.35) -- (-0.03,2.30) -- (-0.03,2.25) -- (-0.03,2.20) -- (-0.02,2.15) -- (-0.02,2.10) -- (-0.01,2.05) -- (0.00,2.00);
\draw (0.00,2.00) -- (0.04,1.82) -- (0.11,1.64) -- (0.19,1.47) -- (0.28,1.29) -- (0.40,1.13) -- (0.52,0.96) -- (0.66,0.81) -- (0.81,0.66) -- (0.97,0.53) -- (1.14,0.40) -- (1.31,0.29) -- (1.49,0.19) -- (1.68,0.10) -- (1.87,0.03) -- (2.06,-0.02) -- (2.25,-0.06) -- (2.44,-0.08) -- (2.63,-0.07) -- (2.82,-0.05) -- (3.00,0.00);
\draw (3.00,0.00) -- (3.07,0.02) -- (3.14,0.05) -- (3.20,0.09) -- (3.27,0.12) -- (3.33,0.16) -- (3.40,0.20) -- (3.46,0.24) -- (3.52,0.29) -- (3.57,0.34) -- (3.63,0.39) -- (3.68,0.44) -- (3.73,0.50) -- (3.77,0.56) -- (3.81,0.62) -- (3.85,0.68) -- (3.89,0.74) -- (3.92,0.80) -- (3.95,0.87) -- (3.98,0.93) -- (4.00,1.00);
\draw (4.00,1.00) -- (4.01,1.05) -- (4.02,1.10) -- (4.03,1.14) -- (4.04,1.19) -- (4.05,1.24) -- (4.05,1.29) -- (4.05,1.34) -- (4.05,1.39) -- (4.05,1.44) -- (4.05,1.49) -- (4.05,1.54) -- (4.05,1.59) -- (4.04,1.64) -- (4.04,1.69) -- (4.03,1.74) -- (4.03,1.80) -- (4.02,1.85) -- (4.01,1.90) -- (4.01,1.95) -- (4.00,2.00);
\draw (4.00,2.00) -- (4.05,2.08) -- (4.09,2.15) -- (4.14,2.22) -- (4.19,2.30) -- (4.24,2.37) -- (4.29,2.44) -- (4.33,2.51) -- (4.38,2.57) -- (4.43,2.63) -- (4.48,2.69) -- (4.53,2.74) -- (4.58,2.79) -- (4.63,2.84) -- (4.68,2.88) -- (4.73,2.92) -- (4.78,2.95) -- (4.84,2.97) -- (4.89,2.99) -- (4.94,3.00) -- (5.00,3.00);
\draw (5.00,3.00) -- (5.06,3.00) -- (5.11,2.98) -- (5.17,2.97) -- (5.23,2.94) -- (5.28,2.91) -- (5.34,2.88) -- (5.40,2.83) -- (5.45,2.79) -- (5.51,2.74) -- (5.56,2.68) -- (5.61,2.62) -- (5.67,2.56) -- (5.72,2.50) -- (5.76,2.43) -- (5.81,2.36) -- (5.85,2.29) -- (5.89,2.22) -- (5.93,2.15) -- (5.97,2.07) -- (6.00,2.00);
\draw (6.00,2.00) -- (6.04,1.90) -- (6.07,1.79) -- (6.10,1.69) -- (6.13,1.59) -- (6.14,1.49) -- (6.16,1.39) -- (6.16,1.29) -- (6.17,1.19) -- (6.17,1.09) -- (6.16,0.99) -- (6.16,0.89) -- (6.15,0.79) -- (6.13,0.69) -- (6.12,0.59) -- (6.10,0.49) -- (6.08,0.39) -- (6.06,0.29) -- (6.04,0.20) -- (6.02,0.10) -- (6.00,0.00);
\filldraw[fill=white] (4.00,2.00) ellipse (0.80cm and 0.50cm);
\draw (0.00,4.50) node{$p$};
\draw (2.00,4.50) node{$f^{\prime}{}^{-1}$};
\draw (4.00,2.00) node{$\theta^{-1}$};
\draw (6.00,-0.50) node{$f^{-1}$};
\draw (8.00,-0.50) node{$r$};
\end{tikzpicture}}}
\]}\fi
The former is obtained by composing $\theta^\wedge$ with the isomorphisms $\id \to f'f'{}^{-1}$ and $f^{-1}f \to \id$, so $\theta^\wedge$ is an isomorphism.  The argument for $\phi$ and $\phi^\vee$ is similar.
\end{proof}

In the following lemma, we do \emph{not} assume that $f$ and $f'$ are equivalences.

\begin{lem}\label{lem:comm-adj-counit}
Suppose that~\eqref{eqn:comm-adjoint} is a commutative diagram of adjoint pairs.  Then $f$ takes the counit for the adjoint pair $p \dashv q$ to the counit for the adjoint pair $r \dashv s$.  
More precisely, there exists an isomorphism of functors $fpq \simto rsf$ such that, for
any $X \in \mathcal{A}$, the diagram
\[
\xymatrix@C=2cm{
f(pq(X)) \ar[r]^-{f(\epsilon_X)} \ar[d]_{\wr} & f(X) \ar@{=}[d] \\
rs(f(X)) \ar[r]_-{\epsilon_{f(X)}} & f(X)}
\]
commutes.
\end{lem}

\begin{proof}
Let $\theta$ be as in Definition~\ref{defn:comm-adjoint}, and
consider the isomorphism $f(pq(X)) \simto rs(f(X))$ given by
\[
f(pq(X)) \xrightarrow{(\theta^\wedge_{q(X)})^{-1}}
rf'q(X) \xrightarrow{r(\theta_X)}
rs(f(X)).
\]
Then the lemma follows from the claim that
\ifdefined\PARTCOMPILE{
\[
FIGURE
\]
}
\else {
\[
\vcenter{\hbox{\begin{tikzpicture}[trans]
\useasboundingbox (-0.5,-0.5) rectangle (4.5,3.5);
\draw (2.00,1.00) -- (2.00,0.00);
\draw (2.00,2.00) -- (2.00,1.00);
\draw (0.00,2.00) -- (0.11,1.97) -- (0.22,1.94) -- (0.33,1.92) -- (0.43,1.89) -- (0.54,1.85) -- (0.65,1.82) -- (0.75,1.79) -- (0.86,1.75) -- (0.96,1.71) -- (1.07,1.67) -- (1.17,1.62) -- (1.27,1.57) -- (1.37,1.52) -- (1.46,1.46) -- (1.56,1.40) -- (1.65,1.33) -- (1.74,1.26) -- (1.83,1.18) -- (1.92,1.09) -- (2.00,1.00);
\draw (2.00,1.00) -- (2.05,0.94) -- (2.10,0.88) -- (2.15,0.81) -- (2.20,0.74) -- (2.25,0.68) -- (2.30,0.61) -- (2.35,0.55) -- (2.39,0.48) -- (2.44,0.42) -- (2.49,0.36) -- (2.54,0.30) -- (2.59,0.25) -- (2.63,0.20) -- (2.68,0.15) -- (2.73,0.11) -- (2.79,0.07) -- (2.84,0.05) -- (2.89,0.02) -- (2.94,0.01) -- (3.00,0.00);
\draw (3.00,0.00) -- (3.06,0.00) -- (3.12,0.01) -- (3.17,0.03) -- (3.23,0.05) -- (3.29,0.08) -- (3.35,0.11) -- (3.41,0.15) -- (3.47,0.20) -- (3.53,0.25) -- (3.58,0.31) -- (3.64,0.37) -- (3.69,0.43) -- (3.74,0.50) -- (3.79,0.56) -- (3.83,0.63) -- (3.87,0.71) -- (3.91,0.78) -- (3.95,0.85) -- (3.97,0.93) -- (4.00,1.00);
\draw (4.00,1.00) -- (4.01,1.05) -- (4.03,1.10) -- (4.04,1.15) -- (4.05,1.21) -- (4.05,1.26) -- (4.06,1.31) -- (4.06,1.36) -- (4.06,1.41) -- (4.06,1.46) -- (4.06,1.51) -- (4.06,1.56) -- (4.05,1.61) -- (4.05,1.66) -- (4.04,1.71) -- (4.04,1.75) -- (4.03,1.80) -- (4.02,1.85) -- (4.02,1.90) -- (4.01,1.95) -- (4.00,2.00);
\filldraw[fill=white] (2.00,1.00) ellipse (0.80cm and 0.50cm);
\draw (0.00,2.50) node{$r$};
\draw (2.00,2.50) node{$f'$};
\draw (4.00,2.50) node{$q$};
\draw (2.00,1.00) node{$\theta^\wedge$};
\draw (2.00,-0.50) node{$f$};
\end{tikzpicture}}}
 = 
\vcenter{\hbox{\begin{tikzpicture}[trans]
\useasboundingbox (-0.5,-0.5) rectangle (4.5,3.5);
\draw (2.00,2.00) -- (2.01,1.95) -- (2.01,1.90) -- (2.02,1.85) -- (2.03,1.80) -- (2.03,1.75) -- (2.04,1.70) -- (2.04,1.65) -- (2.05,1.60) -- (2.05,1.56) -- (2.05,1.51) -- (2.06,1.46) -- (2.06,1.41) -- (2.05,1.36) -- (2.05,1.31) -- (2.05,1.25) -- (2.04,1.20) -- (2.03,1.15) -- (2.03,1.10) -- (2.01,1.05) -- (2.00,1.00);
\draw (2.00,1.00) -- (1.98,0.93) -- (1.95,0.85) -- (1.92,0.78) -- (1.88,0.71) -- (1.85,0.64) -- (1.80,0.57) -- (1.76,0.50) -- (1.71,0.44) -- (1.66,0.38) -- (1.61,0.32) -- (1.56,0.26) -- (1.50,0.21) -- (1.44,0.17) -- (1.38,0.12) -- (1.32,0.09) -- (1.26,0.06) -- (1.19,0.03) -- (1.13,0.01) -- (1.06,0.00) -- (1.00,0.00);
\draw (1.00,0.00) -- (0.94,0.00) -- (0.87,0.01) -- (0.81,0.03) -- (0.74,0.06) -- (0.68,0.09) -- (0.62,0.12) -- (0.56,0.17) -- (0.50,0.21) -- (0.44,0.26) -- (0.39,0.32) -- (0.34,0.38) -- (0.29,0.44) -- (0.24,0.50) -- (0.20,0.57) -- (0.15,0.64) -- (0.12,0.71) -- (0.08,0.78) -- (0.05,0.85) -- (0.02,0.93) -- (0.00,1.00);
\draw (0.00,1.00) -- (-0.01,1.05) -- (-0.03,1.10) -- (-0.03,1.15) -- (-0.04,1.20) -- (-0.05,1.25) -- (-0.05,1.31) -- (-0.05,1.36) -- (-0.06,1.41) -- (-0.06,1.46) -- (-0.05,1.51) -- (-0.05,1.56) -- (-0.05,1.60) -- (-0.04,1.65) -- (-0.04,1.70) -- (-0.03,1.75) -- (-0.03,1.80) -- (-0.02,1.85) -- (-0.01,1.90) -- (-0.01,1.95) -- (0.00,2.00);
\draw (2.00,1.00) -- (2.00,0.00);
\draw (4.00,2.00) -- (2.00,1.00);
\filldraw[fill=white] (2.00,1.00) ellipse (0.80cm and 0.50cm);
\draw (0.00,2.50) node{$r$};
\draw (2.00,2.50) node{$f'$};
\draw (4.00,2.50) node{$q$};
\draw (2.00,1.00) node{$\theta$};
\draw (2.00,-0.50) node{$f$};
\end{tikzpicture}}}
\]}\fi
which follows immediately from the definition of $\theta^\wedge$ and the rules in~\eqref{eqn:unit-counit}.
\end{proof}

\subsection{More natural transformations}
\label{ss:more-natural}

We now list a number of natural transformations related to Figure~\ref{fig:translation}.  Consider first the triangle
\[
\xymatrix{
& \Db\Rep(P_J) \ar[dr]^-{R\Ind_{P_J}^{P_I}} \\
\Db_{\Stein}(P_I) \ar[ur]^-{\For_{P_J}^{P_I}} \ar[rr]^-{\inc} &&
\Db\Rep(P_I).
}
\]
The unit for the adjoint pair $\For_{P_J}^{P_I} \dashv R\Ind_{P_J}^{P_I}$ gives rise to a natural transformation
\begin{equation}
\label{eqn:RInd-For-PI-PJ}
\vcenter{\hbox{\begin{tikzpicture}[trans,scale=-1]
\useasboundingbox (-0.5,-0.5) rectangle (2.5,3.5);
\draw (1.00,1.00) -- (1.00,0.00);
\draw (0.00,2.00) -- (0.03,1.94) -- (0.06,1.87) -- (0.09,1.81) -- (0.12,1.74) -- (0.15,1.68) -- (0.19,1.62) -- (0.22,1.56) -- (0.26,1.50) -- (0.29,1.44) -- (0.33,1.39) -- (0.37,1.34) -- (0.41,1.29) -- (0.45,1.24) -- (0.49,1.20) -- (0.54,1.16) -- (0.59,1.13) -- (0.64,1.10) -- (0.69,1.07) -- (0.74,1.04) -- (0.80,1.03);
\draw (0.80,1.03) -- (0.81,1.02) -- (0.82,1.02) -- (0.83,1.02) -- (0.84,1.02) -- (0.85,1.01) -- (0.86,1.01) -- (0.87,1.01) -- (0.88,1.01) -- (0.89,1.01) -- (0.90,1.01) -- (0.91,1.01) -- (0.92,1.00) -- (0.93,1.00) -- (0.94,1.00) -- (0.95,1.00) -- (0.96,1.00) -- (0.97,1.00) -- (0.98,1.00) -- (0.99,1.00) -- (1.00,1.00);
\draw (1.00,1.00) -- (1.01,1.00) -- (1.02,1.00) -- (1.03,1.00) -- (1.04,1.00) -- (1.05,1.00) -- (1.06,1.00) -- (1.07,1.00) -- (1.08,1.00) -- (1.09,1.01) -- (1.10,1.01) -- (1.11,1.01) -- (1.12,1.01) -- (1.13,1.01) -- (1.14,1.01) -- (1.15,1.01) -- (1.16,1.02) -- (1.17,1.02) -- (1.18,1.02) -- (1.19,1.02) -- (1.20,1.03);
\draw (1.20,1.03) -- (1.26,1.04) -- (1.31,1.07) -- (1.36,1.10) -- (1.41,1.13) -- (1.46,1.16) -- (1.51,1.20) -- (1.55,1.24) -- (1.59,1.29) -- (1.63,1.34) -- (1.67,1.39) -- (1.71,1.44) -- (1.74,1.50) -- (1.78,1.56) -- (1.81,1.62) -- (1.85,1.68) -- (1.88,1.74) -- (1.91,1.81) -- (1.94,1.87) -- (1.97,1.94) -- (2.00,2.00);
\draw (0.00,2.50) node{$\For_{P_J}^{P_I}$};
\draw (2.00,2.50) node{$R\Ind_{P_J}^{P_I}$};
\draw (1.00,-0.50) node{$\inc$};
\end{tikzpicture}}}
\end{equation}
which is easily seen to be an isomorphism. Similarly, consider the triangle
\[
\xymatrix@R=0.3cm@C=2cm{
 \Db_{\Stein, -\varsigma_{I \ssm J}}(P_J) \ar[dr]^-{\inc} \ar[dd]_-{R\Ind_{P_J}^{P_I}}& \\
 & \Db\Rep(P_J).
 \\
\Db_{\Stein}(P_I) \ar[ur]^-{\For_{P_J}^{P_I}}  &
}
\]
The counit for the adjoint pair $\For_{P_J}^{P_I} \dashv R\Ind_{P_J}^{P_I}$ gives rise to a natural transformation
\begin{equation}
\label{eqn:RInd-For-PI-PJ-2}
\vcenter{\hbox{\begin{tikzpicture}[trans]
\useasboundingbox (-0.5,-0.5) rectangle (2.5,3.5);
\draw (1.00,1.00) -- (1.00,0.00);
\draw (0.00,2.00) -- (0.03,1.94) -- (0.06,1.87) -- (0.09,1.81) -- (0.12,1.74) -- (0.15,1.68) -- (0.19,1.62) -- (0.22,1.56) -- (0.26,1.50) -- (0.29,1.44) -- (0.33,1.39) -- (0.37,1.34) -- (0.41,1.29) -- (0.45,1.24) -- (0.49,1.20) -- (0.54,1.16) -- (0.59,1.13) -- (0.64,1.10) -- (0.69,1.07) -- (0.74,1.04) -- (0.80,1.03);
\draw (0.80,1.03) -- (0.81,1.02) -- (0.82,1.02) -- (0.83,1.02) -- (0.84,1.02) -- (0.85,1.01) -- (0.86,1.01) -- (0.87,1.01) -- (0.88,1.01) -- (0.89,1.01) -- (0.90,1.01) -- (0.91,1.01) -- (0.92,1.00) -- (0.93,1.00) -- (0.94,1.00) -- (0.95,1.00) -- (0.96,1.00) -- (0.97,1.00) -- (0.98,1.00) -- (0.99,1.00) -- (1.00,1.00);
\draw (1.00,1.00) -- (1.01,1.00) -- (1.02,1.00) -- (1.03,1.00) -- (1.04,1.00) -- (1.05,1.00) -- (1.06,1.00) -- (1.07,1.00) -- (1.08,1.00) -- (1.09,1.01) -- (1.10,1.01) -- (1.11,1.01) -- (1.12,1.01) -- (1.13,1.01) -- (1.14,1.01) -- (1.15,1.01) -- (1.16,1.02) -- (1.17,1.02) -- (1.18,1.02) -- (1.19,1.02) -- (1.20,1.03);
\draw (1.20,1.03) -- (1.26,1.04) -- (1.31,1.07) -- (1.36,1.10) -- (1.41,1.13) -- (1.46,1.16) -- (1.51,1.20) -- (1.55,1.24) -- (1.59,1.29) -- (1.63,1.34) -- (1.67,1.39) -- (1.71,1.44) -- (1.74,1.50) -- (1.78,1.56) -- (1.81,1.62) -- (1.85,1.68) -- (1.88,1.74) -- (1.91,1.81) -- (1.94,1.87) -- (1.97,1.94) -- (2.00,2.00);
\draw (0.00,2.50) node{$\For_{P_J}^{P_I}$};
\draw (2.00,2.50) node{$R\Ind_{P_J}^{P_I}$};
\draw (1.00,-0.50) node{$\inc$};
\end{tikzpicture}}}
\end{equation}
Pasting these two triangles, we also have a natural isomorphism $\inc \circ R\Ind_{P_J}^{P_I} \simto R\Ind_{P_J}^{P_I} \circ \inc$, which we will depict as
\begin{equation}
\label{eqn:cross-inc-Ind}
\vcenter{\hbox{\begin{tikzpicture}[trans]
\useasboundingbox (-0.5,-0.5) rectangle (2.5,3.5);
\draw (0.00,2.00) -- (0.05,1.95) -- (0.10,1.90) -- (0.15,1.85) -- (0.20,1.80) -- (0.25,1.75) -- (0.30,1.70) -- (0.35,1.65) -- (0.40,1.60) -- (0.45,1.55) -- (0.50,1.50) -- (0.55,1.45) -- (0.60,1.40) -- (0.65,1.35) -- (0.70,1.30) -- (0.75,1.25) -- (0.80,1.20) -- (0.85,1.15) -- (0.90,1.10) -- (0.95,1.05) -- (1.00,1.00);
\draw (1.00,1.00) -- (1.05,0.95) -- (1.10,0.90) -- (1.15,0.85) -- (1.20,0.80) -- (1.25,0.75) -- (1.30,0.70) -- (1.35,0.65) -- (1.40,0.60) -- (1.45,0.55) -- (1.50,0.50) -- (1.55,0.45) -- (1.60,0.40) -- (1.65,0.35) -- (1.70,0.30) -- (1.75,0.25) -- (1.80,0.20) -- (1.85,0.15) -- (1.90,0.10) -- (1.95,0.05) -- (2.00,0.00);
\draw (2.00,2.00) -- (1.95,1.95) -- (1.90,1.90) -- (1.85,1.85) -- (1.80,1.80) -- (1.75,1.75) -- (1.70,1.70) -- (1.65,1.65) -- (1.60,1.60) -- (1.55,1.55) -- (1.50,1.50) -- (1.45,1.45) -- (1.40,1.40) -- (1.35,1.35) -- (1.30,1.30) -- (1.25,1.25) -- (1.20,1.20) -- (1.15,1.15) -- (1.10,1.10) -- (1.05,1.05) -- (1.00,1.00);
\draw (1.00,1.00) -- (0.95,0.95) -- (0.90,0.90) -- (0.85,0.85) -- (0.80,0.80) -- (0.75,0.75) -- (0.70,0.70) -- (0.65,0.65) -- (0.60,0.60) -- (0.55,0.55) -- (0.50,0.50) -- (0.45,0.45) -- (0.40,0.40) -- (0.35,0.35) -- (0.30,0.30) -- (0.25,0.25) -- (0.20,0.20) -- (0.15,0.15) -- (0.10,0.10) -- (0.05,0.05) -- (0.00,0.00);
\draw (0.00,2.50) node{$\inc$};
\draw (2.00,2.50) node{$R\Ind_{P_J}^{P_I}$};
\draw (0.00,-0.50) node{$R\Ind_{P_J}^{P_I}$};
\draw (2.00,-0.50) node{$\inc$};
\end{tikzpicture}}}
\end{equation}
The following lemma follows directly from the zigzag relation for the adjunction $\For_{P_J}^{P_I} \dashv R\Ind_{P_J}^{P_I}$.

\begin{lem}
\label{lem:zigzag-Ind-For}
The composition
\[
\inc {}\circ R\Ind_{P_J}^{P_I} \xrightarrow{\eqref{eqn:RInd-For-PI-PJ}} R\Ind_{P_J}^{P_I} \circ \For^{P_I}_{P_J} \circ R\Ind_{P_J}^{P_I} \xrightarrow{\eqref{eqn:RInd-For-PI-PJ-2}} R\Ind_{P_J}^{P_I} \circ \inc
\]
coincides with the isomorphism~\eqref{eqn:cross-inc-Ind}.  In other words, we have
\ifdefined\PARTCOMPILE{
\[
FIGURE
\]
}
\else {
\[
\vcenter{\hbox{\begin{tikzpicture}[trans]
\useasboundingbox (-0.5,-0.5) rectangle (4.5,4.5);
\draw (1.00,2.00) -- (1.00,3.00);
\draw (0.00,0.00) -- (-0.01,0.05) -- (-0.01,0.10) -- (-0.02,0.15) -- (-0.02,0.20) -- (-0.03,0.25) -- (-0.03,0.30) -- (-0.04,0.35) -- (-0.04,0.40) -- (-0.04,0.45) -- (-0.05,0.50) -- (-0.05,0.55) -- (-0.05,0.60) -- (-0.05,0.65) -- (-0.04,0.70) -- (-0.04,0.75) -- (-0.03,0.80) -- (-0.03,0.85) -- (-0.02,0.90) -- (-0.01,0.95) -- (0.00,1.00);
\draw (0.00,1.00) -- (0.02,1.06) -- (0.04,1.13) -- (0.06,1.19) -- (0.08,1.26) -- (0.11,1.32) -- (0.14,1.38) -- (0.17,1.44) -- (0.21,1.50) -- (0.25,1.55) -- (0.29,1.61) -- (0.33,1.66) -- (0.37,1.71) -- (0.42,1.76) -- (0.47,1.80) -- (0.52,1.84) -- (0.57,1.87) -- (0.63,1.91) -- (0.68,1.93) -- (0.74,1.96) -- (0.80,1.97);
\draw (0.80,1.97) -- (0.81,1.98) -- (0.82,1.98) -- (0.83,1.98) -- (0.84,1.98) -- (0.85,1.98) -- (0.86,1.99) -- (0.87,1.99) -- (0.88,1.99) -- (0.89,1.99) -- (0.90,1.99) -- (0.91,1.99) -- (0.92,1.99) -- (0.93,2.00) -- (0.94,2.00) -- (0.95,2.00) -- (0.96,2.00) -- (0.97,2.00) -- (0.98,2.00) -- (0.99,2.00) -- (1.00,2.00);
\draw (1.00,2.00) -- (1.01,2.00) -- (1.02,2.00) -- (1.03,2.00) -- (1.04,2.00) -- (1.05,2.00) -- (1.06,2.00) -- (1.07,2.00) -- (1.08,2.00) -- (1.09,2.00) -- (1.10,2.00) -- (1.11,2.00) -- (1.12,2.00) -- (1.13,2.00) -- (1.14,2.00) -- (1.15,1.99) -- (1.16,1.99) -- (1.17,1.99) -- (1.18,1.99) -- (1.19,1.99) -- (1.20,1.99);
\draw (1.20,1.99) -- (1.24,1.98) -- (1.29,1.97) -- (1.33,1.95) -- (1.38,1.94) -- (1.42,1.92) -- (1.46,1.90) -- (1.50,1.88) -- (1.54,1.86) -- (1.58,1.83) -- (1.62,1.81) -- (1.66,1.78) -- (1.70,1.75) -- (1.74,1.72) -- (1.78,1.69) -- (1.81,1.66) -- (1.85,1.63) -- (1.89,1.60) -- (1.93,1.56) -- (1.96,1.53) -- (2.00,1.50);
\draw (2.00,1.50) -- (2.04,1.47) -- (2.07,1.44) -- (2.11,1.40) -- (2.15,1.37) -- (2.19,1.34) -- (2.22,1.31) -- (2.26,1.28) -- (2.30,1.25) -- (2.34,1.22) -- (2.38,1.19) -- (2.42,1.17) -- (2.46,1.14) -- (2.50,1.12) -- (2.54,1.10) -- (2.58,1.08) -- (2.62,1.06) -- (2.67,1.05) -- (2.71,1.03) -- (2.76,1.02) -- (2.80,1.01);
\draw (2.80,1.01) -- (2.81,1.01) -- (2.82,1.01) -- (2.83,1.01) -- (2.84,1.01) -- (2.85,1.01) -- (2.86,1.00) -- (2.87,1.00) -- (2.88,1.00) -- (2.89,1.00) -- (2.90,1.00) -- (2.91,1.00) -- (2.92,1.00) -- (2.93,1.00) -- (2.94,1.00) -- (2.95,1.00) -- (2.96,1.00) -- (2.97,1.00) -- (2.98,1.00) -- (2.99,1.00) -- (3.00,1.00);
\draw (3.00,1.00) -- (3.01,1.00) -- (3.02,1.00) -- (3.03,1.00) -- (3.04,1.00) -- (3.05,1.00) -- (3.06,1.00) -- (3.07,1.00) -- (3.08,1.01) -- (3.09,1.01) -- (3.10,1.01) -- (3.11,1.01) -- (3.12,1.01) -- (3.13,1.01) -- (3.14,1.01) -- (3.15,1.02) -- (3.16,1.02) -- (3.17,1.02) -- (3.18,1.02) -- (3.19,1.02) -- (3.20,1.03);
\draw (3.20,1.03) -- (3.26,1.04) -- (3.32,1.07) -- (3.37,1.09) -- (3.43,1.13) -- (3.48,1.16) -- (3.53,1.20) -- (3.58,1.24) -- (3.63,1.29) -- (3.67,1.34) -- (3.71,1.39) -- (3.75,1.45) -- (3.79,1.50) -- (3.83,1.56) -- (3.86,1.62) -- (3.89,1.68) -- (3.92,1.74) -- (3.94,1.81) -- (3.96,1.87) -- (3.98,1.94) -- (4.00,2.00);
\draw (4.00,2.00) -- (4.01,2.05) -- (4.02,2.10) -- (4.03,2.15) -- (4.03,2.20) -- (4.04,2.25) -- (4.04,2.30) -- (4.05,2.35) -- (4.05,2.40) -- (4.05,2.45) -- (4.05,2.50) -- (4.04,2.55) -- (4.04,2.60) -- (4.04,2.65) -- (4.03,2.70) -- (4.03,2.75) -- (4.02,2.80) -- (4.02,2.85) -- (4.01,2.90) -- (4.01,2.95) -- (4.00,3.00);
\draw (3.00,1.00) -- (3.00,0.00);
\draw (1.00,3.50) node{$\inc$};
\draw (4.00,3.50) node{$R\Ind_{P_J}^{P_I}$};
\draw (0.00,-0.50) node{$R\Ind_{P_J}^{P_I}$};
\draw (3.00,-0.50) node{$\inc$};
\end{tikzpicture}}}
\ =\ 
\vcenter{\hbox{\begin{tikzpicture}[trans]
\useasboundingbox (-0.5,-0.5) rectangle (2.5,3.5);
\draw (0.00,2.00) -- (0.05,1.95) -- (0.10,1.90) -- (0.15,1.85) -- (0.20,1.80) -- (0.25,1.75) -- (0.30,1.70) -- (0.35,1.65) -- (0.40,1.60) -- (0.45,1.55) -- (0.50,1.50) -- (0.55,1.45) -- (0.60,1.40) -- (0.65,1.35) -- (0.70,1.30) -- (0.75,1.25) -- (0.80,1.20) -- (0.85,1.15) -- (0.90,1.10) -- (0.95,1.05) -- (1.00,1.00);
\draw (1.00,1.00) -- (1.05,0.95) -- (1.10,0.90) -- (1.15,0.85) -- (1.20,0.80) -- (1.25,0.75) -- (1.30,0.70) -- (1.35,0.65) -- (1.40,0.60) -- (1.45,0.55) -- (1.50,0.50) -- (1.55,0.45) -- (1.60,0.40) -- (1.65,0.35) -- (1.70,0.30) -- (1.75,0.25) -- (1.80,0.20) -- (1.85,0.15) -- (1.90,0.10) -- (1.95,0.05) -- (2.00,0.00);
\draw (2.00,2.00) -- (1.95,1.95) -- (1.90,1.90) -- (1.85,1.85) -- (1.80,1.80) -- (1.75,1.75) -- (1.70,1.70) -- (1.65,1.65) -- (1.60,1.60) -- (1.55,1.55) -- (1.50,1.50) -- (1.45,1.45) -- (1.40,1.40) -- (1.35,1.35) -- (1.30,1.30) -- (1.25,1.25) -- (1.20,1.20) -- (1.15,1.15) -- (1.10,1.10) -- (1.05,1.05) -- (1.00,1.00);
\draw (1.00,1.00) -- (0.95,0.95) -- (0.90,0.90) -- (0.85,0.85) -- (0.80,0.80) -- (0.75,0.75) -- (0.70,0.70) -- (0.65,0.65) -- (0.60,0.60) -- (0.55,0.55) -- (0.50,0.50) -- (0.45,0.45) -- (0.40,0.40) -- (0.35,0.35) -- (0.30,0.30) -- (0.25,0.25) -- (0.20,0.20) -- (0.15,0.15) -- (0.10,0.10) -- (0.05,0.05) -- (0.00,0.00);
\draw (0.00,2.50) node{$\inc$};
\draw (2.00,2.50) node{$R\Ind_{P_J}^{P_I}$};
\draw (0.00,-0.50) node{$R\Ind_{P_J}^{P_I}$};
\draw (2.00,-0.50) node{$\inc$};
\end{tikzpicture}}}
\]
}
\fi
\end{lem}

Throughout this section, functors like $({-})\otimes \bk(\varsigma_{I\ssm J})$ and $({-}) \otimes \irr(\varsigma_{I\ssm J})^*$ will often be denoted simply by $\bk(\varsigma_{I\ssm J})$ and $\irr(\varsigma_{I\ssm J})^*$, respectively.  The functors $({-})\otimes \bk(\varsigma_{I\ssm J})$ and $({-})\otimes \bk(-\varsigma_{I\ssm J})$ commute with the appropriate inclusion functors.  These commutativity isomorphisms will be denoted by diagrams of the form
\ifdefined\PARTCOMPILE{
\begin{equation}\label{eqn:comm-defn}
FIGURE
\end{equation}
}
\else {
\begin{equation}\label{eqn:comm-defn}
\vcenter{\hbox{\begin{tikzpicture}[trans]
\useasboundingbox (-0.5,-0.5) rectangle (2.5,3.5);
\draw (0.00,2.00) -- (0.05,1.95) -- (0.10,1.90) -- (0.15,1.85) -- (0.20,1.80) -- (0.25,1.75) -- (0.30,1.70) -- (0.35,1.65) -- (0.40,1.60) -- (0.45,1.55) -- (0.50,1.50) -- (0.55,1.45) -- (0.60,1.40) -- (0.65,1.35) -- (0.70,1.30) -- (0.75,1.25) -- (0.80,1.20) -- (0.85,1.15) -- (0.90,1.10) -- (0.95,1.05) -- (1.00,1.00);
\draw (1.00,1.00) -- (1.05,0.95) -- (1.10,0.90) -- (1.15,0.85) -- (1.20,0.80) -- (1.25,0.75) -- (1.30,0.70) -- (1.35,0.65) -- (1.40,0.60) -- (1.45,0.55) -- (1.50,0.50) -- (1.55,0.45) -- (1.60,0.40) -- (1.65,0.35) -- (1.70,0.30) -- (1.75,0.25) -- (1.80,0.20) -- (1.85,0.15) -- (1.90,0.10) -- (1.95,0.05) -- (2.00,0.00);
\draw (2.00,2.00) -- (1.95,1.95) -- (1.90,1.90) -- (1.85,1.85) -- (1.80,1.80) -- (1.75,1.75) -- (1.70,1.70) -- (1.65,1.65) -- (1.60,1.60) -- (1.55,1.55) -- (1.50,1.50) -- (1.45,1.45) -- (1.40,1.40) -- (1.35,1.35) -- (1.30,1.30) -- (1.25,1.25) -- (1.20,1.20) -- (1.15,1.15) -- (1.10,1.10) -- (1.05,1.05) -- (1.00,1.00);
\draw (1.00,1.00) -- (0.95,0.95) -- (0.90,0.90) -- (0.85,0.85) -- (0.80,0.80) -- (0.75,0.75) -- (0.70,0.70) -- (0.65,0.65) -- (0.60,0.60) -- (0.55,0.55) -- (0.50,0.50) -- (0.45,0.45) -- (0.40,0.40) -- (0.35,0.35) -- (0.30,0.30) -- (0.25,0.25) -- (0.20,0.20) -- (0.15,0.15) -- (0.10,0.10) -- (0.05,0.05) -- (0.00,0.00);
\draw (0.00,2.50) node{$\inc$};
\draw (2.00,2.50) node{$\bk(-\varsigma_{I\ssm J})$};
\draw (0.00,-0.50) node{$\bk(-\varsigma_{I\ssm J})$};
\draw (2.00,-0.50) node{$\inc$};
\end{tikzpicture}}}
\qquad\qquad
\vcenter{\hbox{\begin{tikzpicture}[trans]
\useasboundingbox (-0.5,-0.5) rectangle (2.5,3.5);
\draw (0.00,2.00) -- (0.05,1.95) -- (0.10,1.90) -- (0.15,1.85) -- (0.20,1.80) -- (0.25,1.75) -- (0.30,1.70) -- (0.35,1.65) -- (0.40,1.60) -- (0.45,1.55) -- (0.50,1.50) -- (0.55,1.45) -- (0.60,1.40) -- (0.65,1.35) -- (0.70,1.30) -- (0.75,1.25) -- (0.80,1.20) -- (0.85,1.15) -- (0.90,1.10) -- (0.95,1.05) -- (1.00,1.00);
\draw (1.00,1.00) -- (1.05,0.95) -- (1.10,0.90) -- (1.15,0.85) -- (1.20,0.80) -- (1.25,0.75) -- (1.30,0.70) -- (1.35,0.65) -- (1.40,0.60) -- (1.45,0.55) -- (1.50,0.50) -- (1.55,0.45) -- (1.60,0.40) -- (1.65,0.35) -- (1.70,0.30) -- (1.75,0.25) -- (1.80,0.20) -- (1.85,0.15) -- (1.90,0.10) -- (1.95,0.05) -- (2.00,0.00);
\draw (2.00,2.00) -- (1.95,1.95) -- (1.90,1.90) -- (1.85,1.85) -- (1.80,1.80) -- (1.75,1.75) -- (1.70,1.70) -- (1.65,1.65) -- (1.60,1.60) -- (1.55,1.55) -- (1.50,1.50) -- (1.45,1.45) -- (1.40,1.40) -- (1.35,1.35) -- (1.30,1.30) -- (1.25,1.25) -- (1.20,1.20) -- (1.15,1.15) -- (1.10,1.10) -- (1.05,1.05) -- (1.00,1.00);
\draw (1.00,1.00) -- (0.95,0.95) -- (0.90,0.90) -- (0.85,0.85) -- (0.80,0.80) -- (0.75,0.75) -- (0.70,0.70) -- (0.65,0.65) -- (0.60,0.60) -- (0.55,0.55) -- (0.50,0.50) -- (0.45,0.45) -- (0.40,0.40) -- (0.35,0.35) -- (0.30,0.30) -- (0.25,0.25) -- (0.20,0.20) -- (0.15,0.15) -- (0.10,0.10) -- (0.05,0.05) -- (0.00,0.00);
\draw (0.00,2.50) node{$\bk(-\varsigma_{I\ssm J})$};
\draw (2.00,2.50) node{$\inc$};
\draw (0.00,-0.50) node{$\inc$};
\draw (2.00,-0.50) node{$\bk(-\varsigma_{I\ssm J})$};
\end{tikzpicture}}}
\qquad\qquad
\vcenter{\hbox{\begin{tikzpicture}[trans]
\useasboundingbox (-0.5,-0.5) rectangle (2.5,3.5);
\draw (0.00,2.00) -- (0.05,1.95) -- (0.10,1.90) -- (0.15,1.85) -- (0.20,1.80) -- (0.25,1.75) -- (0.30,1.70) -- (0.35,1.65) -- (0.40,1.60) -- (0.45,1.55) -- (0.50,1.50) -- (0.55,1.45) -- (0.60,1.40) -- (0.65,1.35) -- (0.70,1.30) -- (0.75,1.25) -- (0.80,1.20) -- (0.85,1.15) -- (0.90,1.10) -- (0.95,1.05) -- (1.00,1.00);
\draw (1.00,1.00) -- (1.05,0.95) -- (1.10,0.90) -- (1.15,0.85) -- (1.20,0.80) -- (1.25,0.75) -- (1.30,0.70) -- (1.35,0.65) -- (1.40,0.60) -- (1.45,0.55) -- (1.50,0.50) -- (1.55,0.45) -- (1.60,0.40) -- (1.65,0.35) -- (1.70,0.30) -- (1.75,0.25) -- (1.80,0.20) -- (1.85,0.15) -- (1.90,0.10) -- (1.95,0.05) -- (2.00,0.00);
\draw (2.00,2.00) -- (1.95,1.95) -- (1.90,1.90) -- (1.85,1.85) -- (1.80,1.80) -- (1.75,1.75) -- (1.70,1.70) -- (1.65,1.65) -- (1.60,1.60) -- (1.55,1.55) -- (1.50,1.50) -- (1.45,1.45) -- (1.40,1.40) -- (1.35,1.35) -- (1.30,1.30) -- (1.25,1.25) -- (1.20,1.20) -- (1.15,1.15) -- (1.10,1.10) -- (1.05,1.05) -- (1.00,1.00);
\draw (1.00,1.00) -- (0.95,0.95) -- (0.90,0.90) -- (0.85,0.85) -- (0.80,0.80) -- (0.75,0.75) -- (0.70,0.70) -- (0.65,0.65) -- (0.60,0.60) -- (0.55,0.55) -- (0.50,0.50) -- (0.45,0.45) -- (0.40,0.40) -- (0.35,0.35) -- (0.30,0.30) -- (0.25,0.25) -- (0.20,0.20) -- (0.15,0.15) -- (0.10,0.10) -- (0.05,0.05) -- (0.00,0.00);
\draw (0.00,2.50) node{$\inc$};
\draw (2.00,2.50) node{$\bk(\varsigma_{I\ssm J})$};
\draw (0.00,-0.50) node{$\bk(\varsigma_{I\ssm J})$};
\draw (2.00,-0.50) node{$\inc$};
\end{tikzpicture}}}
\qquad\qquad
\vcenter{\hbox{\begin{tikzpicture}[trans]
\useasboundingbox (-0.5,-0.5) rectangle (2.5,3.5);
\draw (0.00,2.00) -- (0.05,1.95) -- (0.10,1.90) -- (0.15,1.85) -- (0.20,1.80) -- (0.25,1.75) -- (0.30,1.70) -- (0.35,1.65) -- (0.40,1.60) -- (0.45,1.55) -- (0.50,1.50) -- (0.55,1.45) -- (0.60,1.40) -- (0.65,1.35) -- (0.70,1.30) -- (0.75,1.25) -- (0.80,1.20) -- (0.85,1.15) -- (0.90,1.10) -- (0.95,1.05) -- (1.00,1.00);
\draw (1.00,1.00) -- (1.05,0.95) -- (1.10,0.90) -- (1.15,0.85) -- (1.20,0.80) -- (1.25,0.75) -- (1.30,0.70) -- (1.35,0.65) -- (1.40,0.60) -- (1.45,0.55) -- (1.50,0.50) -- (1.55,0.45) -- (1.60,0.40) -- (1.65,0.35) -- (1.70,0.30) -- (1.75,0.25) -- (1.80,0.20) -- (1.85,0.15) -- (1.90,0.10) -- (1.95,0.05) -- (2.00,0.00);
\draw (2.00,2.00) -- (1.95,1.95) -- (1.90,1.90) -- (1.85,1.85) -- (1.80,1.80) -- (1.75,1.75) -- (1.70,1.70) -- (1.65,1.65) -- (1.60,1.60) -- (1.55,1.55) -- (1.50,1.50) -- (1.45,1.45) -- (1.40,1.40) -- (1.35,1.35) -- (1.30,1.30) -- (1.25,1.25) -- (1.20,1.20) -- (1.15,1.15) -- (1.10,1.10) -- (1.05,1.05) -- (1.00,1.00);
\draw (1.00,1.00) -- (0.95,0.95) -- (0.90,0.90) -- (0.85,0.85) -- (0.80,0.80) -- (0.75,0.75) -- (0.70,0.70) -- (0.65,0.65) -- (0.60,0.60) -- (0.55,0.55) -- (0.50,0.50) -- (0.45,0.45) -- (0.40,0.40) -- (0.35,0.35) -- (0.30,0.30) -- (0.25,0.25) -- (0.20,0.20) -- (0.15,0.15) -- (0.10,0.10) -- (0.05,0.05) -- (0.00,0.00);
\draw (0.00,2.50) node{$\bk(\varsigma_{I\ssm J})$};
\draw (2.00,2.50) node{$\inc$};
\draw (0.00,-0.50) node{$\inc$};
\draw (2.00,-0.50) node{$\bk(\varsigma_{I\ssm J})$};
\end{tikzpicture}}}
\end{equation}
}
\fi
The ``transitivity'' isomorphism $R\Ind_{P_I}^G {}\circ R\Ind_{P_J}^{P_I} \cong R\Ind_{P_J}^G$ (see~\eqref{eqn:transitivity-RInd}) will be denoted by
\ifdefined\PARTCOMPILE{
\begin{equation}\label{eqn:tr-defn}
FIGURE
\end{equation}
}
\else {
\begin{equation}\label{eqn:tr-defn}
\vcenter{\hbox{\begin{tikzpicture}[trans]
\useasboundingbox (-0.5,-0.5) rectangle (2.5,3.5);
\draw (1.00,1.00) -- (1.00,0.00);
\draw (2.00,2.00) -- (1.00,1.00);
\draw (0.00,2.00) -- (1.00,1.00);
\filldraw[fill=white] (1.00,1.00) ellipse (0.80cm and 0.50cm);
\draw (0.00,2.50) node{$R\Ind_{P_I}^G$};
\draw (2.00,2.50) node{$R\Ind_{P_J}^{P_I}$};
\draw (1.00,1.00) node{tr};
\draw (1.00,-0.50) node{$R\Ind_{P_J}^G$};
\end{tikzpicture}}}
\qquad\text{or}\qquad
\vcenter{\hbox{\begin{tikzpicture}[trans]
\useasboundingbox (-0.5,-0.5) rectangle (2.5,3.5);
\draw (1.00,1.00) -- (2.00,0.00);
\draw (1.00,1.00) -- (0.00,0.00);
\draw (1.00,2.00) -- (1.00,1.00);
\filldraw[fill=white] (1.00,1.00) ellipse (0.80cm and 0.50cm);
\draw (1.00,2.50) node{$R\Ind_{P_J}^G$};
\draw (1.00,1.00) node{tr};
\draw (0.00,-0.50) node{$R\Ind_{P_I}^G$};
\draw (2.00,-0.50) node{$R\Ind_{P_J}^{P_I}$};
\end{tikzpicture}}}
\end{equation}
}
\fi

Lastly, we have a canonical isomorphism of $P_J$-modules
\[
\bk(\varsigma_{I\ssm J})^* \cong \bk(-\varsigma_{I\ssm J}).
\]
Let us fix
a nonzero (surjective) map of $P_J$-modules
\begin{equation}\label{eqn:lk-quot}
\irr(\varsigma_{I\ssm J}) \to \bk(\varsigma_{I\ssm J});
\end{equation}
then
by duality we deduce a nonzero (injective) map
\begin{equation}\label{eqn:lk-dual}
\bk(-\varsigma_{I\ssm J}) \to \irr(\varsigma_{I\ssm J})^*.
\end{equation}

We define a natural transformation
\[
\gamma: R\Ind_{P_J}^G {}\circ {\bk(-\varsigma_{I\ssm J})} \to {\irr(\varsigma_{I\ssm J})^*} \circ R\Ind_{P_J}^G
\qquad\text{or}\qquad
\vcenter{\hbox{\begin{tikzpicture}[trans]
\useasboundingbox (-0.5,-0.5) rectangle (2.5,3.5);
\draw (1.00,1.00) -- (2.00,0.00);
\draw (1.00,1.00) -- (0.00,0.00);
\draw (2.00,2.00) -- (1.00,1.00);
\draw (0.00,2.00) -- (1.00,1.00);
\filldraw[fill=white] (1.00,1.00) ellipse (0.80cm and 0.50cm);
\draw (0.00,2.50) node{$R\Ind_{P_J}^G$};
\draw (2.00,2.50) node{$\bk(-\varsigma_{I\ssm J})$};
\draw (1.00,1.00) node{$\gamma$};
\draw (0.00,-0.50) node{$\irr(\varsigma_{I\ssm J})^*$};
\draw (2.00,-0.50) node{$R\Ind_{P_J}^G$};
\end{tikzpicture}}}
\]
by
\[
R\Ind_{P_J}^G(M \otimes \bk(-\varsigma_{I\ssm J})) \to R\Ind_{P_J}^G(M \otimes \irr(\varsigma_{I\ssm J})^*)
\simto R\Ind_{P_J}^G(M) \otimes \irr(\varsigma_{I\ssm J})^*,
\]
where the first morphism is induced by~\eqref{eqn:lk-dual} and the second one by the tensor identity.

We likewise define
\[
\delta: {\irr(\varsigma_{I\ssm J})} \circ R\Ind_{P_J}^G \to R\Ind_{P_J}^G{} \circ \bk(\varsigma_{I\ssm J})
\qquad\text{or}\qquad
\vcenter{\hbox{\begin{tikzpicture}[trans]
\useasboundingbox (-0.5,-0.5) rectangle (2.5,3.5);
\draw (1.00,1.00) -- (2.00,0.00);
\draw (1.00,1.00) -- (0.00,0.00);
\draw (2.00,2.00) -- (1.00,1.00);
\draw (0.00,2.00) -- (1.00,1.00);
\filldraw[fill=white] (1.00,1.00) ellipse (0.80cm and 0.50cm);
\draw (0.00,2.50) node{$\irr(\varsigma_{I\ssm J})$};
\draw (2.00,2.50) node{$R\Ind_{P_J}^G$};
\draw (1.00,1.00) node{$\delta$};
\draw (0.00,-0.50) node{$R\Ind_{P_J}^G$};
\draw (2.00,-0.50) node{$\bk(\varsigma_{I\ssm J})$};
\end{tikzpicture}}}
\]
by
\[
R\Ind_{P_J}^G(M) \otimes \irr(\varsigma_{I\ssm J}) \simto R\Ind_{P_J}^G(M \otimes \irr(\varsigma_{I\ssm J})) \to
R\Ind_{P_J}^G(M \otimes \bk(\varsigma_{I\ssm J})),
\]
where the second morphism is induced by~\eqref{eqn:lk-quot}.

\subsection{Natural transformations related to induction}
\label{ss:natural-transfo-induction}

In this subsection, we prove several lemmas about $\gamma$ and $\delta$.  Note that the diagram
\[
\vcenter{
\xymatrix@C=3.5cm{
\Db\Rep(P_J) \ar[r]^{R\Ind_{P_J}^G} \ar@{}[d]|-{\dashv} \ar@<1ex>[d]^-{({-}) \otimes \bk(-\varsigma_{I\ssm J})} & \Db\Rep(G) \ar@{}[d]|-{\dashv} \ar@<1ex>[d]^-{({-}) \otimes \irr(\varsigma_{I\ssm J})^*} \\
\Db\Rep(P_J) \ar[r]_{R\Ind_{P_J}^G} \ar@<1ex>[u]^-{({-}) \otimes \bk(\varsigma_{I\ssm J})} & \Db\Rep(G) \ar@<1ex>[u]^-{({-}) \otimes \irr(\varsigma_{I\ssm J})}
}
}
\]
matches the pattern of~\eqref{eqn:comm-adjoint}, so that the following lemma makes sense.

\begin{lem}
\label{lem:delta-gamma-mate}
We have $\delta = \gamma^\wedge$ and $\gamma = \delta^\vee$.  In other words,
\ifdefined\PARTCOMPILE{
\[
FIGURE
\]
}
\else {
\[
\vcenter{\hbox{\begin{tikzpicture}[trans]
\useasboundingbox (-0.5,-0.5) rectangle (2.5,3.5);
\draw (1.00,1.00) -- (2.00,0.00);
\draw (1.00,1.00) -- (0.00,0.00);
\draw (2.00,2.00) -- (1.00,1.00);
\draw (0.00,2.00) -- (1.00,1.00);
\filldraw[fill=white] (1.00,1.00) ellipse (0.80cm and 0.50cm);
\draw (0.00,2.50) node{$\irr(\varsigma_{I\ssm J})$};
\draw (2.00,2.50) node{$R\Ind_{P_J}^G$};
\draw (1.00,1.00) node{$\delta$};
\draw (0.00,-0.50) node{$R\Ind_{P_J}^G$};
\draw (2.00,-0.50) node{$\bk(\varsigma_{I\ssm J})$};
\end{tikzpicture}}}
\ =\ 
\vcenter{\hbox{\begin{tikzpicture}[trans]
\useasboundingbox (-0.5,-0.5) rectangle (4.5,3.5);
\draw (2.00,1.00) -- (2.05,1.08) -- (2.09,1.15) -- (2.14,1.22) -- (2.18,1.30) -- (2.23,1.37) -- (2.28,1.44) -- (2.33,1.51) -- (2.37,1.57) -- (2.42,1.63) -- (2.47,1.69) -- (2.52,1.74) -- (2.57,1.79) -- (2.62,1.84) -- (2.67,1.88) -- (2.72,1.92) -- (2.78,1.95) -- (2.83,1.97) -- (2.89,1.99) -- (2.94,2.00) -- (3.00,2.00);
\draw (3.00,2.00) -- (3.06,2.00) -- (3.12,1.98) -- (3.18,1.97) -- (3.24,1.94) -- (3.30,1.91) -- (3.36,1.87) -- (3.42,1.83) -- (3.48,1.79) -- (3.53,1.73) -- (3.59,1.68) -- (3.64,1.62) -- (3.69,1.56) -- (3.74,1.50) -- (3.79,1.43) -- (3.84,1.36) -- (3.88,1.29) -- (3.91,1.22) -- (3.95,1.15) -- (3.98,1.07) -- (4.00,1.00);
\draw (4.00,1.00) -- (4.01,0.95) -- (4.03,0.90) -- (4.04,0.85) -- (4.05,0.80) -- (4.05,0.75) -- (4.06,0.70) -- (4.06,0.64) -- (4.06,0.59) -- (4.06,0.54) -- (4.06,0.49) -- (4.06,0.45) -- (4.05,0.40) -- (4.05,0.35) -- (4.04,0.30) -- (4.04,0.25) -- (4.03,0.20) -- (4.02,0.15) -- (4.02,0.10) -- (4.01,0.05) -- (4.00,0.00);
\draw (2.00,2.00) -- (2.01,1.95) -- (2.01,1.90) -- (2.02,1.85) -- (2.03,1.80) -- (2.03,1.75) -- (2.04,1.70) -- (2.04,1.65) -- (2.05,1.60) -- (2.05,1.56) -- (2.05,1.51) -- (2.06,1.46) -- (2.06,1.41) -- (2.05,1.36) -- (2.05,1.31) -- (2.05,1.25) -- (2.04,1.20) -- (2.03,1.15) -- (2.03,1.10) -- (2.01,1.05) -- (2.00,1.00);
\draw (2.00,1.00) -- (1.98,0.93) -- (1.95,0.85) -- (1.92,0.78) -- (1.88,0.71) -- (1.85,0.64) -- (1.80,0.57) -- (1.76,0.50) -- (1.71,0.44) -- (1.66,0.38) -- (1.61,0.32) -- (1.56,0.26) -- (1.50,0.21) -- (1.44,0.17) -- (1.38,0.12) -- (1.32,0.09) -- (1.26,0.06) -- (1.19,0.03) -- (1.13,0.01) -- (1.06,0.00) -- (1.00,0.00);
\draw (1.00,0.00) -- (0.94,0.00) -- (0.87,0.01) -- (0.81,0.03) -- (0.74,0.06) -- (0.68,0.09) -- (0.62,0.12) -- (0.56,0.17) -- (0.50,0.21) -- (0.44,0.26) -- (0.39,0.32) -- (0.34,0.38) -- (0.29,0.44) -- (0.24,0.50) -- (0.20,0.57) -- (0.15,0.64) -- (0.12,0.71) -- (0.08,0.78) -- (0.05,0.85) -- (0.02,0.93) -- (0.00,1.00);
\draw (0.00,1.00) -- (-0.01,1.05) -- (-0.03,1.10) -- (-0.03,1.15) -- (-0.04,1.20) -- (-0.05,1.25) -- (-0.05,1.31) -- (-0.05,1.36) -- (-0.06,1.41) -- (-0.06,1.46) -- (-0.05,1.51) -- (-0.05,1.56) -- (-0.05,1.60) -- (-0.04,1.65) -- (-0.04,1.70) -- (-0.03,1.75) -- (-0.03,1.80) -- (-0.02,1.85) -- (-0.01,1.90) -- (-0.01,1.95) -- (0.00,2.00);
\draw (2.00,1.00) -- (2.00,0.00);
\filldraw[fill=white] (2.00,1.00) ellipse (0.80cm and 0.50cm);
\draw (0.00,2.50) node{$\irr(\varsigma_{I\ssm J})$};
\draw (2.00,2.50) node{$R\Ind_{P_J}^G$};
\draw (2.00,1.00) node{$\gamma$};
\draw (2.00,-0.50) node{$R\Ind_{P_J}^G$};
\draw (4.00,-0.50) node{$\bk(\varsigma_{I\ssm J})$};
\end{tikzpicture}}},
\qquad
\vcenter{\hbox{\begin{tikzpicture}[trans]
\useasboundingbox (-0.5,-0.5) rectangle (2.5,3.5);
\draw (1.00,1.00) -- (2.00,0.00);
\draw (1.00,1.00) -- (0.00,0.00);
\draw (2.00,2.00) -- (1.00,1.00);
\draw (0.00,2.00) -- (1.00,1.00);
\filldraw[fill=white] (1.00,1.00) ellipse (0.80cm and 0.50cm);
\draw (0.00,2.50) node{$R\Ind_{P_J}^G$};
\draw (2.00,2.50) node{$\bk(-\varsigma_{I\ssm J})$};
\draw (1.00,1.00) node{$\gamma$};
\draw (0.00,-0.50) node{$\irr(\varsigma_{I\ssm J})^*$};
\draw (2.00,-0.50) node{$R\Ind_{P_J}^G$};
\end{tikzpicture}}}
\ = \ 
\vcenter{\hbox{\begin{tikzpicture}[trans]
\useasboundingbox (-0.5,-0.5) rectangle (4.5,3.5);
\draw (0.00,0.00) -- (-0.01,0.05) -- (-0.02,0.10) -- (-0.02,0.15) -- (-0.03,0.20) -- (-0.04,0.25) -- (-0.04,0.30) -- (-0.05,0.35) -- (-0.05,0.40) -- (-0.06,0.45) -- (-0.06,0.49) -- (-0.06,0.54) -- (-0.06,0.59) -- (-0.06,0.64) -- (-0.06,0.70) -- (-0.05,0.75) -- (-0.05,0.80) -- (-0.04,0.85) -- (-0.03,0.90) -- (-0.01,0.95) -- (0.00,1.00);
\draw (0.00,1.00) -- (0.02,1.07) -- (0.05,1.15) -- (0.09,1.22) -- (0.12,1.29) -- (0.16,1.36) -- (0.21,1.43) -- (0.26,1.50) -- (0.31,1.56) -- (0.36,1.62) -- (0.41,1.68) -- (0.47,1.73) -- (0.52,1.79) -- (0.58,1.83) -- (0.64,1.87) -- (0.70,1.91) -- (0.76,1.94) -- (0.82,1.97) -- (0.88,1.98) -- (0.94,2.00) -- (1.00,2.00);
\draw (1.00,2.00) -- (1.06,2.00) -- (1.11,1.99) -- (1.17,1.97) -- (1.22,1.95) -- (1.28,1.92) -- (1.33,1.88) -- (1.38,1.84) -- (1.43,1.79) -- (1.48,1.74) -- (1.53,1.69) -- (1.58,1.63) -- (1.63,1.57) -- (1.67,1.51) -- (1.72,1.44) -- (1.77,1.37) -- (1.82,1.30) -- (1.86,1.22) -- (1.91,1.15) -- (1.95,1.08) -- (2.00,1.00);
\draw (2.00,1.00) -- (2.00,0.00);
\draw (2.00,2.00) -- (1.99,1.95) -- (1.99,1.90) -- (1.98,1.85) -- (1.97,1.80) -- (1.97,1.75) -- (1.96,1.70) -- (1.96,1.65) -- (1.95,1.60) -- (1.95,1.56) -- (1.95,1.51) -- (1.94,1.46) -- (1.94,1.41) -- (1.95,1.36) -- (1.95,1.31) -- (1.95,1.25) -- (1.96,1.20) -- (1.97,1.15) -- (1.97,1.10) -- (1.99,1.05) -- (2.00,1.00);
\draw (2.00,1.00) -- (2.02,0.93) -- (2.05,0.85) -- (2.08,0.78) -- (2.12,0.71) -- (2.15,0.64) -- (2.20,0.57) -- (2.24,0.50) -- (2.29,0.44) -- (2.34,0.38) -- (2.39,0.32) -- (2.44,0.26) -- (2.50,0.21) -- (2.56,0.17) -- (2.62,0.12) -- (2.68,0.09) -- (2.74,0.06) -- (2.81,0.03) -- (2.87,0.01) -- (2.94,0.00) -- (3.00,0.00);
\draw (3.00,0.00) -- (3.06,0.00) -- (3.13,0.01) -- (3.19,0.03) -- (3.26,0.06) -- (3.32,0.09) -- (3.38,0.12) -- (3.44,0.17) -- (3.50,0.21) -- (3.56,0.26) -- (3.61,0.32) -- (3.66,0.38) -- (3.71,0.44) -- (3.76,0.50) -- (3.80,0.57) -- (3.85,0.64) -- (3.88,0.71) -- (3.92,0.78) -- (3.95,0.85) -- (3.98,0.93) -- (4.00,1.00);
\draw (4.00,1.00) -- (4.01,1.05) -- (4.03,1.10) -- (4.03,1.15) -- (4.04,1.20) -- (4.05,1.25) -- (4.05,1.31) -- (4.05,1.36) -- (4.06,1.41) -- (4.06,1.46) -- (4.05,1.51) -- (4.05,1.56) -- (4.05,1.60) -- (4.04,1.65) -- (4.04,1.70) -- (4.03,1.75) -- (4.03,1.80) -- (4.02,1.85) -- (4.01,1.90) -- (4.01,1.95) -- (4.00,2.00);
\filldraw[fill=white] (2.00,1.00) ellipse (0.80cm and 0.50cm);
\draw (2.00,2.50) node{$R\Ind_{P_J}^G$};
\draw (4.00,2.50) node{$\bk(-\varsigma_{I\ssm J})$};
\draw (2.00,1.00) node{$\delta$};
\draw (0.00,-0.50) node{$\irr(\varsigma_{I\ssm J})^*$};
\draw (2.00,-0.50) node{$R\Ind_{P_J}^G$};
\end{tikzpicture}}}
\]
}
\fi
\end{lem}
\begin{proof}
Unwinding the definition of $\gamma^\wedge$, we encounter the composition
\[
\irr(\varsigma_{I\ssm J}) \to \bk(\varsigma_{I\ssm J}) \otimes \bk(-\varsigma_{I\ssm J}) \otimes \irr(\varsigma_{I\ssm J})
\to \bk(\varsigma_{I\ssm J}) \otimes \irr(\varsigma_{I\ssm J})^* \otimes \irr(\varsigma_{I\ssm J}) \to \bk(\varsigma_{I\ssm J}),
\]
where the first and last maps come from adjunction, and the second one is induced by~\eqref{eqn:lk-dual}.  It is easy to see that this composition is equal to the map in~\eqref{eqn:lk-quot}.  It follows that $\gamma^\wedge = \delta$. The second equality follows, since $(-)^\vee$ is inverse to $(-)^\wedge$.
\end{proof}

\begin{lem}
\label{lem:induce-block}
For any $M \in \Db_\Stein(P_I)$, the natural adjunction maps
\begin{gather*}
R\Ind_{P_I}^G(\inc M) \to \incl_I\pr_I R\Ind_{P_I}^G(\inc M), \\
\incl_I\pr_I R\Ind_{P_I}^G(\inc M) \to R\Ind_{P_I}^G(\inc M)
\end{gather*}
are isomorphisms.  
\end{lem}

\begin{proof}
This statement is equivalent to saying that for any $M \in \Db_\Stein(P_I)$, the object $R\Ind_{P_I}^G(\inc M)$ belongs to $\Db \Rep_I(G)$, or equivalently that its cohomology objects belong to $\Rep_I(G)$.
Using
Lemma~\ref{lem:generators-SteinPK},
it suffices to prove this claim for the objects $\coweyl_I(w \bullet (- \varsigma_I))$ with $w \in W_I \ltimes \bX$ and $w \bullet (- \varsigma_I) \in \bX_I^+$. In this case, using Kempf's vanishing theorem (see~\cite[Proposition~II.4.5]{jantzen}) and~\eqref{eqn:transitivity-RInd} we have
\[
R\Ind_{P_I}^G(\inc M) \cong R\Ind_{P_I}^G(R\Ind_B^{P_I} (\bk(w \bullet( - \varsigma_I)))) \cong R\Ind_B^G (\bk(w \bullet ( - \varsigma_I))).
\]
Then~\cite[II.7.3(5)]{jantzen} implies that this object indeed belongs to $\Db \Rep_I(G)$, and the claim is proved.
\end{proof}

\begin{lem}
\label{lem:induce-translate}
The natural transformation
\[
\pr_I \gamma \inc : \pr_I {}\circ R\Ind_{P_J}^G {}\circ {\bk(-\varsigma_{I\ssm J})} \circ \inc \to \pr_I {}\circ {\irr(\varsigma_{I\ssm J})^*} \circ R\Ind_{P_J}^G {}\circ \inc
\]
of functors from $\Db_\Stein(P_J)$ to $\Db\Rep_{I}(G)$ is an isomorphism.
\end{lem}

\begin{proof}
Using again Lemma~\ref{lem:generators-SteinPK}, it suffices to prove that this morphism is an isomorphism when applied to any object $\coweyl_J(\ell \lambda - \varsigma_J)$ with $\lambda \in \bX_J^+ + \varsigma_J$. In this case,
the argument is closely modeled on the proof of~\cite[Proposition~II.7.11]{jantzen}.  Let $Q$ be the cokernel of the map~\eqref{eqn:lk-dual}.  Then there is a distinguished triangle
\begin{multline*}
R\Ind_{P_J}^G(\coweyl_J(\ell\lambda-\varsigma_J) \otimes \bk(-\varsigma_{I\ssm J}))
\xrightarrow{\gamma_{\coweyl_J(\ell\lambda-\varsigma_J)}}
R\Ind_{P_J}^G(\coweyl_J(\ell\lambda-\varsigma_J)) \otimes \irr(\varsigma_{I\ssm J})^* \\
\to
R\Ind_{P_J}^G(\coweyl_J(\ell\lambda-\varsigma_J) \otimes Q) \xrightarrow{[1]},
\end{multline*}
so that to conclude we only have to
show that $\pr_I R\Ind_{P_J}^G(\coweyl_J(\ell\lambda-\varsigma_J)  \otimes Q) = 0$.  Since (as in the proof of Lemma~\ref{lem:induce-block}), $R\Ind_{P_J}^G(\coweyl_J(\ell\lambda-\varsigma_J)  \otimes Q) \cong R\Ind_B^G(\bk(\ell\lambda -\varsigma_J) \otimes Q)$, we have reduced the problem to showing that
\begin{equation}\label{eqn:induce-translate}
\pr_I R\Ind_B^G(\bk(\ell\lambda-\varsigma_J)  \otimes Q) = 0.
\end{equation}

Let $\nu$ be a weight of $\irr(\varsigma_{I\ssm J})^*$, and assume that $-\varsigma_J + \nu \in \Waff \bullet (-\varsigma_I)$. Then we must have $-\varsigma_J + \nu \in \WaffCox \bullet (-\varsigma_I)$. Indeed, write
\[
-\varsigma_J + \nu = (wt_\mu) \bullet (-\varsigma_I) = w(\ell \mu - \varsigma_I + \rho) - \rho
\]
where $w \in W$ and $\mu \in \bX$. Then we have
\[
\ell w(\mu) = -\varsigma_J + \nu + w(\varsigma_I) - w(\rho) + \rho.
\]
Here it is easily checked that the right-hand side belongs to $\Z\Phi$; so $\ell w(\mu)$ belongs to $\Z\Phi \cap \ell \bX = \ell \Z\Phi$. (Here the equality follows from the fact that $\bX/\Z\Phi$ has no $\ell$-torsion since $\ell>h$.) This implies that $w(\mu) \in \Z\Phi$, hence that $\mu \in \Z\Phi$, and finally that $wt_\mu \in \WaffCox$, as claimed.

According to~\cite[Lemma~II.7.7]{jantzen}, we must have $\nu = -w\varsigma_{I\ssm J}$ for some $w \in W$, and $-\varsigma_J + \nu = w' \bullet (-\varsigma_I)$ for some $w' \in \WaffCox$ such that $w' \bullet (-\varsigma_J) = -\varsigma_J$. By~\eqref{eqn:stab-varsigmaK}, the latter implies that $w' \in W_J$, so $w' \bullet (-\varsigma_I) \in -\varsigma_I + \Z\Phi_J$.  To summarize, we have that
\begin{equation}\label{eqn:induce-translate2}
-w\varsigma_{I\ssm J} \in -\varsigma_I + \varsigma_J + \Z\Phi_J = -\varsigma_{I\ssm J} + \Z\Phi_J.
\end{equation}
Assume that $w$ was chosen to have minimal length, and choose a reduced expression $w = s_1\cdots s_r$.  Since $-\varsigma_{I\ssm J}$ is antidominant, we have
\[
-\varsigma_{I\ssm J} \prec -s_r\varsigma_{I\ssm J} \prec -s_{r-1}s_r\varsigma_{I\ssm J} \prec \cdots \prec -w\varsigma_{I\ssm J},
\]
where $\prec$ is the standard order on $\bX$ associated with our choice of positive roots (see~\S\ref{ss:Hom-calculations} below).
Write $-w\varsigma_{I\ssm J} + \varsigma_{I\ssm J}$ as $\sum_{s \in S} n_s\alpha_s$.  Here each $n_s$ is a nonnegative integer; it is strictly positive if $s$ occurs at least once in the product $s_1 \cdots s_r$.  If $w \ne 1$, then at least one simple reflection not in $J$ must occur, since $W_J$ stabilizes $-\varsigma_{I\ssm J}$ for the standard action.  So if $w \ne 1$, we have $-w\varsigma_{I\ssm J} + \varsigma_{I\ssm J} \notin \Z\Phi_J$, contradicting~\eqref{eqn:induce-translate2}.

We conclude that $w = 1$, i.e., that the only weight $\nu$ of $\irr(\varsigma_{I\ssm J})^*$ such that $-\varsigma_J + \nu \in \Waff \bullet (-\varsigma_I)$ is $\nu = -\varsigma_{I\ssm J}$.  In other words, if $\nu$ is any weight of $Q$, then $-\varsigma_J + \nu \notin \Waff \bullet (-\varsigma_I)$, and hence
\[
\ell\lambda - \varsigma_J + \nu \notin \Waff \bullet (-\varsigma_I).
\]
Then~\eqref{eqn:induce-translate} follows from this by~\cite[II.7.3(5)]{jantzen}.
\end{proof}

\begin{lem}\label{lem:induce-p-translate}
The natural transformation 
\[
\pr_J \delta \For_{P_J}^{P_I}: \pr_J {}\circ \irr(\varsigma_{I\ssm J}) \circ R\Ind_{P_J}^G {}\circ \For_{P_J}^{P_I} \to \pr_J {}\circ R\Ind_{P_J}^G {}\circ \bk(\varsigma_{I\ssm J}) \circ \For_{P_J}^{P_I}
\]
of functors from $\Db_\Stein(P_I)$ to $\Db\Rep_{J}(G)$ is an isomorphism.
\end{lem}

\begin{proof}
By adjunction, and since $\coweyl(\varsigma_{I \ssm J}) = \Ind_{P_I}^G(\coweyl_I(\varsigma_{I \ssm J}))$, there exists a canonical morphism $\For^G_{P_I} (\coweyl(\varsigma_{I \ssm J})) \to \coweyl_I(\varsigma_{I \ssm J})$. Moreover this morphism is surjective (see e.g.~\cite[Theorem~3.1.1]{brionkumar} for a much more general statement). Composing with the embedding $\irr(\varsigma_{I \ssm J}) \hookrightarrow \coweyl(\varsigma_{I \ssm J})$ and with a morphism of $P_J$-modules $ \coweyl_I(\varsigma_{I \ssm J}) \to \bk(\varsigma_{I \ssm J})$, we see that~\eqref{eqn:lk-quot} factors as a composition
\begin{equation}
\label{eqn:lk-quot-s1}
\irr(\varsigma_{I \ssm J}) \to \coweyl_I(\varsigma_{I \ssm J}) \to \bk(\varsigma_{I \ssm J}).
\end{equation}

Now, consider the functor
\[
({-}) \otimes \coweyl_I(\varsigma_{I\ssm J}): \Db\Rep(P_I) \to \Db\Rep(P_I).
\]
Using the morphisms in~\eqref{eqn:lk-quot-s1}
in place of~\eqref{eqn:lk-quot}, we can define two natural transformations
\begin{gather*}
\delta': \coweyl_I(\varsigma_{I\ssm J}) \circ R\Ind_{P_J}^{P_I} \to R\Ind_{P_J}^{P_I} {}\circ \bk(\varsigma_{I\ssm J}), \\
\delta'': \irr(\varsigma_{I\ssm J}) \circ R\Ind_{P_I}^G \to R\Ind_{P_I}^G {}\circ \coweyl_I(\varsigma_{I\ssm J})
\end{gather*}
that are analogous to $\delta$.  These transformations are related to $\delta$ by
\ifdefined\PARTCOMPILE{
\[
FIGURE
\]
}
\else {
\[
\vcenter{\hbox{\begin{tikzpicture}[trans]
\useasboundingbox (-0.5,-0.5) rectangle (2.5,3.5);
\draw (1.00,1.00) -- (2.00,0.00);
\draw (1.00,1.00) -- (0.00,0.00);
\draw (2.00,2.00) -- (1.00,1.00);
\draw (0.00,2.00) -- (1.00,1.00);
\filldraw[fill=white] (1.00,1.00) ellipse (0.80cm and 0.50cm);
\draw (0.00,2.50) node{$\irr(\varsigma_{I\ssm J})$};
\draw (2.00,2.50) node{$R\Ind_{P_J}^G$};
\draw (1.00,1.00) node{$\delta$};
\draw (0.00,-0.50) node{$R\Ind_{P_J}^G$};
\draw (2.00,-0.50) node{$\bk(\varsigma_{I\ssm J})$};
\end{tikzpicture}}}
\ =\ \ 
\vcenter{\hbox{\begin{tikzpicture}[trans]
\useasboundingbox (-0.5,-0.5) rectangle (2.5,6.5);
\draw (2.00,4.00) -- (1.86,3.89) -- (1.73,3.78) -- (1.59,3.68) -- (1.46,3.58) -- (1.33,3.48) -- (1.20,3.38) -- (1.08,3.29) -- (0.96,3.21) -- (0.84,3.13) -- (0.73,3.06) -- (0.62,3.00) -- (0.52,2.95) -- (0.43,2.91) -- (0.34,2.88) -- (0.26,2.87) -- (0.19,2.87) -- (0.13,2.88) -- (0.08,2.90) -- (0.03,2.94) -- (0.00,3.00);
\draw (0.00,3.00) -- (-0.01,3.03) -- (-0.02,3.07) -- (-0.03,3.10) -- (-0.03,3.14) -- (-0.04,3.19) -- (-0.04,3.23) -- (-0.04,3.28) -- (-0.04,3.33) -- (-0.04,3.38) -- (-0.04,3.43) -- (-0.04,3.49) -- (-0.03,3.54) -- (-0.03,3.60) -- (-0.03,3.66) -- (-0.02,3.71) -- (-0.02,3.77) -- (-0.01,3.83) -- (-0.01,3.89) -- (-0.00,3.94) -- (0.00,4.00);
\draw (0.00,4.00) -- (0.00,4.06) -- (0.01,4.11) -- (0.01,4.16) -- (0.01,4.22) -- (0.01,4.27) -- (0.01,4.32) -- (0.01,4.37) -- (0.01,4.42) -- (0.01,4.47) -- (0.01,4.52) -- (0.01,4.57) -- (0.01,4.62) -- (0.01,4.67) -- (0.01,4.72) -- (0.01,4.76) -- (0.01,4.81) -- (0.01,4.86) -- (0.00,4.91) -- (0.00,4.95) -- (0.00,5.00);
\draw (2.00,4.00) -- (2.00,3.95) -- (2.00,3.91) -- (1.99,3.86) -- (1.99,3.81) -- (1.99,3.76) -- (1.99,3.72) -- (1.99,3.67) -- (1.99,3.62) -- (1.99,3.57) -- (1.99,3.52) -- (1.99,3.47) -- (1.99,3.42) -- (1.99,3.37) -- (1.99,3.32) -- (1.99,3.27) -- (1.99,3.22) -- (1.99,3.16) -- (1.99,3.11) -- (2.00,3.06) -- (2.00,3.00);
\draw (2.00,3.00) -- (2.00,2.94) -- (2.01,2.89) -- (2.01,2.83) -- (2.02,2.77) -- (2.02,2.71) -- (2.03,2.66) -- (2.03,2.60) -- (2.03,2.54) -- (2.04,2.49) -- (2.04,2.43) -- (2.04,2.38) -- (2.04,2.33) -- (2.04,2.28) -- (2.04,2.23) -- (2.04,2.19) -- (2.03,2.14) -- (2.03,2.10) -- (2.02,2.07) -- (2.01,2.03) -- (2.00,2.00);
\draw (2.00,2.00) -- (1.97,1.94) -- (1.92,1.90) -- (1.87,1.88) -- (1.81,1.87) -- (1.74,1.87) -- (1.66,1.88) -- (1.57,1.91) -- (1.48,1.95) -- (1.38,2.00) -- (1.27,2.06) -- (1.16,2.13) -- (1.04,2.21) -- (0.92,2.29) -- (0.80,2.38) -- (0.67,2.48) -- (0.54,2.58) -- (0.41,2.68) -- (0.27,2.78) -- (0.14,2.89) -- (0.00,3.00);
\draw (2.00,5.00) -- (2.00,4.00);
\draw (0.00,3.00) -- (0.00,2.95) -- (0.00,2.91) -- (0.01,2.86) -- (0.01,2.81) -- (0.01,2.76) -- (0.01,2.72) -- (0.01,2.67) -- (0.01,2.62) -- (0.01,2.57) -- (0.01,2.52) -- (0.01,2.47) -- (0.01,2.42) -- (0.01,2.37) -- (0.01,2.32) -- (0.01,2.27) -- (0.01,2.22) -- (0.01,2.16) -- (0.01,2.11) -- (0.00,2.06) -- (0.00,2.00);
\draw (0.00,2.00) -- (-0.00,1.94) -- (-0.01,1.89) -- (-0.01,1.83) -- (-0.02,1.77) -- (-0.02,1.71) -- (-0.03,1.66) -- (-0.03,1.60) -- (-0.03,1.54) -- (-0.04,1.49) -- (-0.04,1.43) -- (-0.04,1.38) -- (-0.04,1.33) -- (-0.04,1.28) -- (-0.04,1.23) -- (-0.04,1.19) -- (-0.03,1.14) -- (-0.03,1.10) -- (-0.02,1.07) -- (-0.01,1.03) -- (0.00,1.00);
\draw (0.00,1.00) -- (0.03,0.94) -- (0.08,0.90) -- (0.13,0.88) -- (0.19,0.87) -- (0.26,0.87) -- (0.34,0.88) -- (0.43,0.91) -- (0.52,0.95) -- (0.62,1.00) -- (0.73,1.06) -- (0.84,1.13) -- (0.96,1.21) -- (1.08,1.29) -- (1.20,1.38) -- (1.33,1.48) -- (1.46,1.58) -- (1.59,1.68) -- (1.73,1.78) -- (1.86,1.89) -- (2.00,2.00);
\draw (2.00,2.00) -- (2.00,1.95) -- (2.00,1.90) -- (2.00,1.85) -- (2.00,1.80) -- (2.00,1.75) -- (2.00,1.70) -- (2.00,1.65) -- (2.00,1.60) -- (2.00,1.55) -- (2.00,1.50) -- (2.00,1.45) -- (2.00,1.40) -- (2.00,1.35) -- (2.00,1.30) -- (2.00,1.25) -- (2.00,1.20) -- (2.00,1.15) -- (2.00,1.10) -- (2.00,1.05) -- (2.00,1.00);
\draw (2.00,1.00) -- (2.00,0.95) -- (2.00,0.90) -- (2.00,0.85) -- (2.00,0.80) -- (2.00,0.75) -- (2.00,0.70) -- (2.00,0.65) -- (2.00,0.60) -- (2.00,0.55) -- (2.00,0.50) -- (2.00,0.45) -- (2.00,0.40) -- (2.00,0.35) -- (2.00,0.30) -- (2.00,0.25) -- (2.00,0.20) -- (2.00,0.15) -- (2.00,0.10) -- (2.00,0.05) -- (2.00,0.00);
\draw (0.00,1.00) -- (0.00,0.00);
\filldraw[fill=white] (2.00,4.00) ellipse (0.80cm and 0.50cm);
\filldraw[fill=white] (0.00,3.00) ellipse (0.80cm and 0.50cm);
\filldraw[fill=white] (2.00,2.00) ellipse (0.80cm and 0.50cm);
\filldraw[fill=white] (0.00,1.00) ellipse (0.80cm and 0.50cm);
\draw (0.00,5.50) node{$\irr(\varsigma_{I\ssm J})$};
\draw (2.00,5.50) node{$R\Ind_{P_J}^G$};
\draw (2.00,4.00) node{tr};
\draw (0.00,3.00) node{$\delta^{\prime\prime}$};
\draw (2.00,2.00) node{$\delta^{\prime}$};
\draw (0.00,1.00) node{tr};
\draw (0.00,-0.50) node{$R\Ind_{P_J}^G$};
\draw (2.00,-0.50) node{$\bk(\varsigma_{I\ssm J})$};
\end{tikzpicture}}}
\]}\fi
Thus, the lemma will follow if we can show that the following two natural transformations are isomorphisms:
\begin{equation}
\delta' \For_{P_J}^{P_I} :\coweyl_I(\varsigma_{I\ssm J}) \circ R\Ind_{P_J}^{P_I} {}\circ \For_{P_J}^{P_I} \to R\Ind_{P_J}^{P_I} {}\circ \bk(\varsigma_{I\ssm J}) \circ \For_{P_J}^{P_I}, \label{eqn:induce-deltap}
\end{equation}
\begin{multline}
\pr_J \delta'' R\Ind_{P_J}^{P_I} \For_{P_J}^{P_I} : \pr_J {}\circ \irr(\varsigma_{I\ssm J}) \circ R\Ind_{P_I}^G {}\circ R\Ind_{P_J}^{P_I}  {}\circ \For_{P_J}^{P_I} \label{eqn:induce-deltapp}\\
\to \pr_J {}\circ R\Ind_{P_I}^G {}\circ \coweyl_I(\varsigma_{I\ssm J}) \circ R\Ind_{P_J}^{P_I} {}\circ \For_{P_J}^{P_I}.
\end{multline}

The fact that~\eqref{eqn:induce-deltap} is an isomorphism follows the observation that
\[
R\Ind_{P_J}^{P_I}(\For_{P_J}^{P_I}(V)) \otimes \coweyl_I(\varsigma_{I\ssm J}) \cong V \otimes R\Ind_{P_J}^{P_I} \bk(\varsigma_{I\ssm J})
\cong R\Ind_{P_J}^{P_I}(\For_{P_J}^{P_I}(V) \otimes \bk(\varsigma_{I\ssm J}))
\]
by the tensor identity.
On the other hand, because the morphism
\[
R\Ind_{P_J}^{P_I} {}\circ \For_{P_J}^{P_I} \to \inc
\]
induced by the counit
is an isomorphism, we see that~\eqref{eqn:induce-deltapp} is an isomorphism if and only if
\begin{equation}\label{eqn:induce-deltapp-inc}
\pr_J \delta'' \inc : \pr_J {}\circ \irr(\varsigma_{I\ssm J}) \circ R\Ind_{P_I}^G {}\circ \inc \to \pr_J {}\circ R\Ind_{P_I}^G {}\circ \coweyl_I(\varsigma_{I\ssm J}) \circ \inc
\end{equation}
is an isomorphism.  We will prove this by an argument similar to that in the proof of Lemma~\ref{lem:induce-translate}.  Let $C$ be the cone of our morphism $\irr(\varsigma_{I\ssm J}) \to \coweyl_I(\varsigma_{I\ssm J})$.  For $V \in \Db\Rep(P_I)$, we have a distinguished triangle
\[
R\Ind_{P_I}^G(V) \otimes \irr(\varsigma_{I\ssm J})
\xrightarrow{\delta''_V}
R\Ind_{P_I}^G(V \otimes \coweyl_I(\varsigma_{I\ssm J})) \to R\Ind_{P_I}^G(V \otimes C) \xrightarrow{[1]}.
\]
Hence,
for a given $V$, $\pr_J \delta''_V$ is an isomorphism if and only if $\pr_J R\Ind_{P_I}^G(V \otimes C) = 0$.  
Using Lemma~\ref{lem:generators-SteinPK}, we see that
to prove that~\eqref{eqn:induce-deltapp-inc} is an isomorphism, it suffices to show that
\[
\pr_J R\Ind_{P_I}^G(\Ind_B^{P_I}(\bk(\ell\lambda - \varsigma_I)) \otimes C)
\cong \pr_J R\Ind_B^G(\bk(\ell\lambda - \varsigma_I) \otimes C)
\]
vanishes for any $\lambda \in \bX_I^+ + \varsigma_I$. 

We have $\coH^i(C)=0$ unless $i \in \{-1,0\}$, and moreover any weight $\nu$ of $\coH^{-1}(C)$ or $\coH^0(C)$ lies in the $W$-orbit of a dominant weight $\nu^+$ which satisfies $\nu^+ \preceq \varsigma_{I\ssm J}$ (since $\nu$ is a weight of $\coweyl(\varsigma_{I\ssm J})$). Hence, by~\cite[Lemma~II.7.7]{jantzen} and the same arguments as in the proof of Lemma~\ref{lem:induce-translate}, if $\nu$ is such a weight 
and if $-\varsigma_I + \nu \in \Waff \bullet (-\varsigma_J)$, then we have $\nu = w\varsigma_{I\ssm J}$ for some $w \in W$, and $-\varsigma_I + \nu = w' \bullet (-\varsigma_J)$ for some $w' \in \WaffCox$ such that $w' \bullet (-\varsigma_I) = -\varsigma_I$. By~\eqref{eqn:stab-varsigmaK} we have $w' \in W_I$; then, in analogy with~\eqref{eqn:induce-translate2}, we deduce that
\[
w\varsigma_{I\ssm J} \in \varsigma_{I\ssm J} + \Z\Phi_I.
\]
Reasoning similar to which followed~\eqref{eqn:induce-translate2} now shows that $w$ must lie in $W_I$.
However, our morphism $\irr(\varsigma_{I\ssm J}) \to \coweyl_I(\varsigma_{I\ssm J})$ is an isomorphism on the weight space of weight $\varsigma_{I\ssm J}$, and hence also on any weight space whose weight is in $W_I(\varsigma_{I\ssm J})$, so no such weight can appear in $\coH^{-1}(C)$ or $\coH^0(C)$.

To summarize, if $\nu$ is any weight of a cohomology object of $C$, then $-\varsigma_I + \nu \notin \Waff \bullet(-\varsigma_J)$, and hence $\ell\lambda - \varsigma_I + \nu \notin \Waff \bullet (-\varsigma_J)$.  By~\cite[II.7.3(5)]{jantzen}, we conclude that $\pr_J \Ind_B^G(\bk(\ell\lambda - \varsigma_{I\ssm J}) \otimes C) = 0$, as desired.
\end{proof}

\subsection{Natural transformations related to the formality theorem}
\label{ss:natural-formality}

According to Proposition~\ref{prop:Theta-psi}, there exists a natural isomorphism
\[
\alpha: \psi_I \circ \Theta_{J,I} \simto R\Ind_{P_J}^{P_I} {}\circ {\bk(-\varsigma_{I\ssm J})} \circ \psi_J,
\]
which we will depict with the following diagram:
\ifdefined\PARTCOMPILE{
\begin{equation}\label{eqn:alpha-defn}
FIGURE
\end{equation}
}
\else {
\begin{equation}\label{eqn:alpha-defn}
\vcenter{\hbox{\begin{tikzpicture}[trans]
\useasboundingbox (-0.5,-0.5) rectangle (4.5,3.5);
\draw (2.00,1.00) -- (4.00,0.00);
\draw (2.00,1.00) -- (2.00,0.00);
\draw (2.00,1.00) -- (0.00,0.00);
\draw (3.00,2.00) -- (2.00,1.00);
\draw (1.00,2.00) -- (2.00,1.00);
\filldraw[fill=white] (2.00,1.00) ellipse (0.80cm and 0.50cm);
\draw (1.00,2.50) node{$\psi_I$};
\draw (3.00,2.50) node{$\Theta_{J,I}$};
\draw (2.00,1.00) node{$\alpha$};
\draw (0.00,-0.50) node{$R\Ind_{P_J}^{P_I}$};
\draw (2.00,-0.50) node{$\bk(-\varsigma_{I\ssm J})$};
\draw (4.00,-0.50) node{$\psi_J$};
\end{tikzpicture}}}
\end{equation}}\fi
Consider the two functors $\Dfg_{\dot P_I}(\bL_I) \to \Db\Rep(P_J)$ given by $M \mapsto \For_{P_J}^{P_I}(\psi_I(M)) \otimes \bk(\varsigma_{I\ssm J})$ and $M \mapsto \inc(\psi_J(\Theta^{J,I}(M)))$.  We define a natural transformation
\[
\beta: {\bk(\varsigma_{I\ssm J})} \circ \For_{P_J}^{P_I} {}\circ \psi_I \to \inc {}\circ \psi_J \circ \Theta^{J,I}
\]
by
\begin{multline*}
\For_{P_J}^{P_I}(\psi_I(M)) \otimes \bk(\varsigma_{I\ssm J}) \to
\For_{P_J}^{P_I}(\psi_I(\Theta_{J,I}\Theta^{J,I}(M))) \otimes \bk(\varsigma_{I\ssm J}) \\ 
\xrightarrow[\sim]{\alpha}
\For_{P_J}^{P_I}(R\Ind_{P_J}^{P_I}(\psi_J(\Theta^{J,I}(M)) \otimes \bk(-\varsigma_{I\ssm J}))) \otimes \bk(\varsigma_{I\ssm J}) \\
\to
\psi_J(\Theta^{J,I}(M)) \otimes \bk(-\varsigma_{I\ssm J}) \otimes \bk(\varsigma_{I\ssm J})
\simto \psi_J(\Theta^{J,I}(M)),
\end{multline*}
where the first, third and fourth morphisms are induced by adjunction.
Graphically, this means that
\ifdefined\PARTCOMPILE{
\[
FIGURE
\]
}
\else {
\[
\vcenter{\hbox{\begin{tikzpicture}[trans]
\useasboundingbox (-0.5,-0.5) rectangle (4.5,3.5);
\draw (2.00,1.00) -- (4.00,0.00);
\draw (2.00,1.00) -- (2.00,0.00);
\draw (2.00,1.00) -- (0.00,0.00);
\draw (4.00,2.00) -- (2.00,1.00);
\draw (2.00,2.00) -- (2.00,1.00);
\draw (0.00,2.00) -- (2.00,1.00);
\filldraw[fill=white] (2.00,1.00) ellipse (0.80cm and 0.50cm);
\draw (0.00,2.50) node{$\bk(\varsigma_{I\ssm J})$};
\draw (2.00,2.50) node{$\For_{P_J}^{P_I}$};
\draw (4.00,2.50) node{$\psi_I$};
\draw (2.00,1.00) node{$\beta$};
\draw (0.00,-0.50) node{$\inc$};
\draw (2.00,-0.50) node{$\psi_J$};
\draw (4.00,-0.50) node{$\Theta^{J,I}$};
\end{tikzpicture}}}
= 
\vcenter{\hbox{\begin{tikzpicture}[trans]
\useasboundingbox (-0.5,-0.5) rectangle (7.5,5.5);
\draw (7.00,0.00) -- (7.00,0.05) -- (7.00,0.10) -- (6.99,0.15) -- (6.99,0.20) -- (6.99,0.25) -- (6.99,0.30) -- (6.99,0.35) -- (6.99,0.40) -- (6.99,0.45) -- (6.99,0.50) -- (6.99,0.55) -- (6.99,0.60) -- (6.99,0.65) -- (6.99,0.70) -- (6.99,0.75) -- (6.99,0.80) -- (6.99,0.85) -- (6.99,0.90) -- (7.00,0.95) -- (7.00,1.00);
\draw (7.00,1.00) -- (7.01,1.10) -- (7.02,1.20) -- (7.03,1.30) -- (7.04,1.40) -- (7.05,1.49) -- (7.06,1.59) -- (7.07,1.69) -- (7.09,1.79) -- (7.09,1.89) -- (7.10,1.99) -- (7.11,2.09) -- (7.11,2.19) -- (7.11,2.29) -- (7.11,2.39) -- (7.10,2.49) -- (7.09,2.59) -- (7.08,2.69) -- (7.06,2.79) -- (7.03,2.90) -- (7.00,3.00);
\draw (7.00,3.00) -- (6.97,3.07) -- (6.95,3.15) -- (6.91,3.22) -- (6.88,3.29) -- (6.84,3.36) -- (6.80,3.43) -- (6.76,3.50) -- (6.71,3.56) -- (6.67,3.63) -- (6.62,3.68) -- (6.56,3.74) -- (6.51,3.79) -- (6.45,3.84) -- (6.39,3.88) -- (6.33,3.91) -- (6.27,3.94) -- (6.21,3.97) -- (6.14,3.99) -- (6.07,4.00) -- (6.00,4.00);
\draw (6.00,4.00) -- (5.89,3.99) -- (5.77,3.96) -- (5.65,3.92) -- (5.53,3.87) -- (5.40,3.80) -- (5.28,3.73) -- (5.16,3.65) -- (5.04,3.57) -- (4.92,3.48) -- (4.81,3.39) -- (4.70,3.31) -- (4.59,3.23) -- (4.49,3.16) -- (4.39,3.09) -- (4.31,3.04) -- (4.23,3.00) -- (4.15,2.97) -- (4.09,2.96) -- (4.04,2.97) -- (4.00,3.00);
\draw (4.00,3.00) -- (3.99,3.02) -- (3.97,3.05) -- (3.96,3.08) -- (3.95,3.11) -- (3.95,3.15) -- (3.94,3.19) -- (3.94,3.23) -- (3.94,3.28) -- (3.94,3.33) -- (3.94,3.38) -- (3.94,3.44) -- (3.95,3.50) -- (3.95,3.55) -- (3.96,3.61) -- (3.96,3.68) -- (3.97,3.74) -- (3.98,3.80) -- (3.98,3.87) -- (3.99,3.93) -- (4.00,4.00);
\draw (2.00,4.00) -- (1.99,3.95) -- (1.98,3.91) -- (1.98,3.86) -- (1.97,3.81) -- (1.96,3.77) -- (1.95,3.72) -- (1.95,3.67) -- (1.94,3.63) -- (1.94,3.58) -- (1.94,3.53) -- (1.93,3.48) -- (1.93,3.43) -- (1.94,3.38) -- (1.94,3.33) -- (1.94,3.27) -- (1.95,3.22) -- (1.96,3.17) -- (1.97,3.11) -- (1.98,3.06) -- (2.00,3.00);
\draw (2.00,3.00) -- (2.03,2.92) -- (2.06,2.83) -- (2.09,2.75) -- (2.13,2.66) -- (2.18,2.58) -- (2.22,2.50) -- (2.27,2.42) -- (2.32,2.34) -- (2.38,2.27) -- (2.43,2.21) -- (2.49,2.15) -- (2.55,2.09) -- (2.60,2.05) -- (2.66,2.01) -- (2.72,1.98) -- (2.78,1.96) -- (2.84,1.95) -- (2.89,1.96) -- (2.95,1.97) -- (3.00,2.00);
\draw (3.00,2.00) -- (3.01,2.01) -- (3.01,2.01) -- (3.02,2.02) -- (3.03,2.02) -- (3.04,2.03) -- (3.04,2.03) -- (3.05,2.04) -- (3.06,2.05) -- (3.06,2.06) -- (3.07,2.06) -- (3.08,2.07) -- (3.08,2.08) -- (3.09,2.09) -- (3.10,2.09) -- (3.11,2.10) -- (3.11,2.11) -- (3.12,2.12) -- (3.13,2.12) -- (3.13,2.13) -- (3.14,2.14);
\draw (3.14,2.14) -- (3.19,2.19) -- (3.23,2.24) -- (3.28,2.30) -- (3.33,2.35) -- (3.37,2.41) -- (3.42,2.46) -- (3.47,2.51) -- (3.52,2.56) -- (3.57,2.61) -- (3.61,2.66) -- (3.66,2.71) -- (3.71,2.76) -- (3.75,2.80) -- (3.79,2.84) -- (3.83,2.88) -- (3.87,2.91) -- (3.91,2.94) -- (3.94,2.96) -- (3.97,2.98) -- (4.00,3.00);
\draw (4.00,3.00) -- (4.06,3.02) -- (4.10,2.99) -- (4.11,2.92) -- (4.10,2.82) -- (4.06,2.69) -- (4.00,2.54) -- (3.93,2.36) -- (3.84,2.17) -- (3.73,1.96) -- (3.61,1.74) -- (3.47,1.52) -- (3.33,1.30) -- (3.17,1.08) -- (3.01,0.87) -- (2.85,0.67) -- (2.68,0.49) -- (2.51,0.33) -- (2.34,0.19) -- (2.17,0.08) -- (2.00,0.00);
\draw (2.00,0.00) -- (1.84,-0.04) -- (1.68,-0.05) -- (1.53,-0.02) -- (1.38,0.03) -- (1.24,0.11) -- (1.11,0.22) -- (0.98,0.35) -- (0.86,0.50) -- (0.74,0.67) -- (0.63,0.85) -- (0.53,1.05) -- (0.44,1.25) -- (0.35,1.47) -- (0.28,1.69) -- (0.21,1.91) -- (0.15,2.14) -- (0.10,2.36) -- (0.06,2.58) -- (0.02,2.79) -- (0.00,3.00);
\draw (0.00,3.00) -- (-0.00,3.06) -- (-0.01,3.11) -- (-0.01,3.16) -- (-0.01,3.22) -- (-0.02,3.27) -- (-0.02,3.32) -- (-0.02,3.37) -- (-0.02,3.42) -- (-0.02,3.47) -- (-0.02,3.52) -- (-0.02,3.57) -- (-0.02,3.62) -- (-0.02,3.67) -- (-0.01,3.72) -- (-0.01,3.76) -- (-0.01,3.81) -- (-0.01,3.86) -- (-0.01,3.91) -- (-0.00,3.95) -- (0.00,4.00);
\draw (4.00,3.00) -- (4.07,2.96) -- (4.13,2.91) -- (4.20,2.87) -- (4.26,2.82) -- (4.32,2.78) -- (4.38,2.73) -- (4.44,2.69) -- (4.50,2.64) -- (4.56,2.60) -- (4.62,2.55) -- (4.67,2.50) -- (4.72,2.45) -- (4.77,2.40) -- (4.81,2.35) -- (4.85,2.29) -- (4.89,2.24) -- (4.92,2.18) -- (4.95,2.12) -- (4.98,2.06) -- (5.00,2.00);
\draw (5.00,2.00) -- (5.01,1.96) -- (5.02,1.91) -- (5.03,1.86) -- (5.04,1.82) -- (5.04,1.77) -- (5.04,1.72) -- (5.05,1.67) -- (5.05,1.62) -- (5.05,1.57) -- (5.04,1.52) -- (5.04,1.47) -- (5.04,1.42) -- (5.03,1.36) -- (5.03,1.31) -- (5.02,1.26) -- (5.02,1.21) -- (5.01,1.16) -- (5.01,1.10) -- (5.00,1.05) -- (5.00,1.00);
\draw (5.00,1.00) -- (5.00,0.95) -- (4.99,0.90) -- (4.99,0.85) -- (4.99,0.80) -- (4.99,0.74) -- (4.99,0.69) -- (4.99,0.64) -- (4.99,0.59) -- (4.99,0.54) -- (4.99,0.49) -- (4.99,0.44) -- (4.99,0.39) -- (4.99,0.35) -- (4.99,0.30) -- (4.99,0.25) -- (4.99,0.20) -- (4.99,0.15) -- (5.00,0.10) -- (5.00,0.05) -- (5.00,0.00);
\draw (3.00,2.00) -- (3.07,1.95) -- (3.15,1.90) -- (3.22,1.85) -- (3.30,1.80) -- (3.37,1.75) -- (3.44,1.70) -- (3.50,1.65) -- (3.57,1.60) -- (3.63,1.55) -- (3.69,1.50) -- (3.74,1.45) -- (3.79,1.40) -- (3.84,1.35) -- (3.88,1.30) -- (3.91,1.25) -- (3.94,1.20) -- (3.97,1.15) -- (3.99,1.10) -- (4.00,1.05) -- (4.00,1.00);
\draw (4.00,1.00) -- (4.00,0.95) -- (3.99,0.90) -- (3.97,0.85) -- (3.94,0.80) -- (3.91,0.75) -- (3.88,0.70) -- (3.84,0.65) -- (3.79,0.60) -- (3.74,0.55) -- (3.69,0.50) -- (3.63,0.45) -- (3.57,0.40) -- (3.50,0.35) -- (3.44,0.30) -- (3.37,0.25) -- (3.30,0.20) -- (3.22,0.15) -- (3.15,0.10) -- (3.07,0.05) -- (3.00,0.00);
\filldraw[fill=white] (4.00,3.00) ellipse (0.80cm and 0.50cm);
\draw (0.00,4.50) node{$\bk(\varsigma_{I\ssm J})$};
\draw (2.00,4.50) node{$\For_{P_J}^{P_I}$};
\draw (4.00,4.50) node{$\psi_I$};
\draw (4.00,3.00) node{$\alpha$};
\draw (3.00,-0.50) node{$\inc$};
\draw (5.00,-0.50) node{$\psi_J$};
\draw (7.00,-0.50) node{$\Theta^{J,I}$};
\end{tikzpicture}}}
\]
}
\fi

\subsection{Study of $\beta$ for a minimal parabolic}
\label{ss:beta-min-parabolic}

In this subsection, we assume that $J = \varnothing$ and $I = \{s\}$.  Our goal is to prove the following statement.

\begin{prop}
\label{prop:indbeta}
Assume that $J = \varnothing$ and that $I = \{s\}$ for some $s \in S$.  The natural transformation
\[
R\Ind_B^{P_s}\beta: R\Ind_B^{P_s} {}\circ {\bk(\varsigma_s)} \circ \For_B^{P_s} {}\circ \psi_{\{s\}} \to R\Ind_B^{P_s} {}\circ \inc {} \circ \psi_{\varnothing} \circ \Theta^{\varnothing,\{s\}}
\]
is an isomorphism.
\end{prop}

For the proof of this proposition
we will use the following simplified notation:
\[
\psi:=\psi_\varnothing, \quad \psi_s := \psi_{\{s\}}, \quad \bL:=\bL_\varnothing, \quad \bL_s:=\bL_{\{s\}}, \quad \Theta^s:=\Theta^{\varnothing, \{s\}}, \quad \Theta_s:=\Theta_{\varnothing, \{s\}}.
\]

We will need some preliminary lemmas concerning the object
\[
Y_s = \Theta^s(\bk) \cong (\bL \aq \bL_s) \otimes \bk_{\dot B}(\varsigma_s) \quad \in \Dfg_{\dot B}(\bL).
\]
(See~\S\ref{ss:normal-subalg} for the definition of the quotient $\bL \aq \bL_s$.)
It is easy to see from the definition of $\Theta^s$ that for any $V \in \Rep(\dot P_s)$ (regarded as a $\dot P_s$-equivariant $\bL_s$-module, as in~\S\ref{ss:fromLtoRn}), there is a canonical isomorphism
\[
\Theta^s(V) \cong Y_s \otimes V.
\]
(Here and below, we omit the functor $\For^{\dot P_s}_{\dot B}$.)
Note that $\bL \aq \bL_s$ is isomorphic to the exterior algebra on the $1$-dimensional space $\dot\fn/\dot\fn_s \cong \bk(-\alpha_s)$.  We therefore have
\[
\coH^i(Y_s \otimes V) \cong
\begin{cases}
\bk_{\dot B}(\varsigma_s) \otimes V& \text{if $i = 0$;} \\
\bk_{\dot B}(\varsigma_s - \alpha_s) \otimes V& \text{if $i = -1$;} \\
0 & \text{otherwise.}
\end{cases}
\]
In particular, we have a truncation homomorphism
\[
\tau: Y_s \otimes V \to \bk_{\dot B}(\varsigma_s) \otimes V.
\]

\begin{lem}
\label{lem:phi-ys}
The object $\psi(Y_s \otimes V) \in \Db_\Stein(B)$ is isomorphic to the following chain complex concentrated in degrees $-1$ and $0$, where $f$ is the map defined in Lemma~{\rm \ref{lem:steinberg-diff}} (and were we omit the functor $\For^{\dot P_s}_{B}$):
\begin{equation}
\label{eqn:phi-ys}
\cdots \to 0 \to \St_s \otimes \bk_B(\varsigma_s - \alpha_s)  \otimes V \xrightarrow{f \otimes \id_V} \St_s \otimes \bk_B(\varsigma_s) \otimes V \to 0 \cdots
\end{equation}
\end{lem}

\begin{proof}
Recall from~\eqref{eqn:isom-psi-tensoring} that we have $\psi(Y_s \otimes V) \cong \psi(Y_s) \otimes V$.  Therefore, it suffices to prove the lemma in the special case where $V$ is the trivial $\dot P_s$-module.  Consider the truncation distinguished triangle
\[
Y_s \xrightarrow{\tau} \bk_{\dot B}(\varsigma_s) \xrightarrow{\zeta} \bk_{\dot B}(\varsigma_s - \alpha_s)[2] \to.
\]
The object $Y_s$ is certainly indecomposable (because it is indecomposable as a $\bL$-module), so the connecting morphism $\zeta$ is nonzero.  Therefore, $\psi(\zeta)$ is a nonzero element of $\Ext^2_B(\bk_B(\ell\varsigma_s), \bk_B(\ell\varsigma_s - \ell\alpha_s))$.  By Lemma~\ref{lem:ext2b}, $\psi(\zeta)$ must be a nonzero scalar multiple of the element $\theta \in \Ext^2_B(\bk_B(\ell\varsigma_s), \bk_B(\ell\varsigma_s - \ell\alpha_s))$ constructed in Lemma~\ref{lem:steinberg-diff}.  It follows that the cone of $\theta$ is isomorphic to the cone of $\psi(\zeta)$.  The cone of $\theta$ is given by the chain complex~\eqref{eqn:phi-ys} (with $V=\bk$), while the cone of $\psi(\zeta)$ is $\psi(Y_s)$.
\end{proof}

\begin{lem}
\label{lem:ys-isom}
For any simple module $V \in \Rep(\dot P_s)$, the composition
\begin{equation}\label{eqn:ys-isom}
V \xrightarrow{\eta} \Theta_s \Theta^s(V) \xrightarrow{\Theta_s(\tau)} \Theta_s(\bk_{\dot B}(\varsigma_s) \otimes V)
\end{equation}
(where $\eta$ is the adjunction morphism)
is an isomorphism.
\end{lem}
\begin{proof}
It is easy to see from the definition that $\Theta_s(\bk_{\dot B}(\varsigma_s) \otimes V) \cong V$. Since $V$ is simple by assumption, we need only show that $\Theta_s(\tau) \circ \eta$ is nonzero.  
But this morphism is the image of $\tau$ under the isomorphism
\[
\Hom_{\Dfg_{\dot B}(\bL)}(\Theta^s(V), \bk_{\dot B}(\varsigma_s) \otimes V) \cong
\Hom_{\Dfg_{\dot P_s}(\bL_s)}(V, \Theta_s(\bk_{\dot B}(\varsigma_s) \otimes V))
\]
induced by adjunction; hence it is indeed nonzero.
\end{proof}

We are now ready to prove Proposition~\ref{prop:indbeta}.

\begin{proof}[Proof of Proposition~{\rm \ref{prop:indbeta}}]
By the same arguments as in the proof of Lemma~\ref{lem:stein-ff},
the category $\Dfg_{\dot P_s}(\bL_s)$ is generated by the simple $\dot P_s$-modu\-les $V$, regarded as $\dot P_s$-equivariant $\bL_s$-dg-modules with trivial $\bL_s$-action. Hence we can fix such a $V$, and it suffices to show that $R\Ind_B^{P_s} \beta_V$ is an isomorphism. Applying $\psi_s$ to the maps in~\eqref{eqn:ys-isom}, and using the natural transformation $\alpha$, we obtain the commutative diagram
\begin{equation}\label{eqn:indbeta2}
\vcenter{
\xymatrix@C=2pt{
\psi_s(V) \ar[r] &
  \psi_s\Theta_s\Theta^s(V) \ar[rr] \ar[d]_\alpha^{\wr} && \psi_s\Theta_s(\bk_{\dot B}(\varsigma_s) \otimes V) \ar[d]_\alpha^{\wr} \\
& R\Ind_B^{P_s}(\psi\Theta^s(V) \otimes \bk(-\varsigma_s)) \ar[rr] &&
  R\Ind_B^{P_s}(\psi(\bk_{\dot B}(\varsigma_s) \otimes V) \otimes \bk(-\varsigma_s)).
}
}
\end{equation}
For brevity, we introduce the notation
\[
Q_V := \psi\Theta^s(V) \otimes \bk_B(-\varsigma_s).
\]
According to Lemma~\ref{lem:phi-ys}, $Q_V$ can be identified with a chain complex
\[
\St_s \otimes \bk(-\alpha_s) \otimes V \to \St_s \otimes V
\]
concentrated in degrees $-1$ and $0$.  

We also have $\psi(\bk_{\dot B}(\varsigma_s) \otimes V) \cong \bk_B(\ell\varsigma_s) \otimes V$, so from~\eqref{eqn:indbeta2} we obtain the maps
\[
\psi_s(V) \to R\Ind_B^{P_s}
(Q_V)
\to
R\Ind_B^{P_s}(\bk_B((\ell-1)\varsigma_s) \otimes V).
\]
By Lemma~\ref{lem:ys-isom}, the composition of these two maps is an isomorphism.  Next, applying $\For_B^{P_s}$ and using the counit $\For_B^{P_s} R\Ind_B^{P_s} \to \id$, we obtain the commutative diagram
\begin{equation}\label{eqn:indbeta35}
\vcenter{
\xymatrix@C=10pt{
\For_B^{P_s}\psi_s(V) \ar[r] &
  \For_B^{P_s} R\Ind_B^{P_s} (Q_V)
   \ar[r] \ar[d] & 
  \For_B^{P_s}R\Ind_B^{P_s}(\bk_B((\ell-1)\varsigma_s) \otimes V) \ar[d] \\
& Q_V \ar[r] &
  \bk_B((\ell-1)\varsigma_s) \otimes V.
}
}
\end{equation}
Note that $R\Ind_B^{P_s}(\bk_B((\ell-1)\varsigma_s) \otimes V) \cong \St_s \otimes V$ by the tensor identity and Kempf's vanishing theorem. Hence the right-hand vertical arrow identifies with a surjective map $\For_B^{P_s} (\St_s \otimes V) \twoheadrightarrow \bk_B((\ell-1)\varsigma_s) \otimes V$. 

Let $Q'_V = Q_V \otimes \bk(\varsigma_s)$.  
Tensoring~\eqref{eqn:indbeta35} with $\bk(\varsigma_s)$, we obtain a sequence of maps
\begin{equation}\label{eqn:indbeta4}
\For_B^{P_s} \psi_s(V) \otimes \bk(\varsigma_s) 
\xrightarrow{\beta_V}
Q'_V
\to
\bk_B(\ell\varsigma_s) \otimes V,
\end{equation}
where the first map is induced by the natural transformation $\beta$.  The composition of these two maps is again surjective.  Now apply $R\Ind_B^{P_s}$ to obtain the diagram
\begin{equation}\label{eqn:indbeta-isom}
R\Ind_B^{P_s}(\For_B^{P_s} \psi_s(V) \otimes \bk(\varsigma_s))
\xrightarrow{R\Ind_B^{P_s}\beta_V}
R\Ind_B^{P_s}Q'_V 
\to
R\Ind_B^{P_s}(\bk_B(\ell\varsigma_s) \otimes V).
\end{equation}
Recall that $\psi_s(V) \cong \St_s \otimes V$, so the first term above is isomorphic to $R\Ind_B^{P_s} (\St_s \otimes \bk(\varsigma_s) \otimes V)$.  Next, $Q'_V$ is given by a chain complex of the form
\[
\cdots \to 0 \to \St_s \otimes \bk(\varsigma_s-\alpha_s) \otimes V \to \St_s \otimes \bk(\varsigma_s) \otimes V \to 0 \to \cdots,
\]
with nonzero terms in degrees $-1$ and $0$. Since $R\Ind_B^{P_s} (\St_s \otimes \bk(\varsigma_s - \alpha_s) \otimes V) \cong \St_s \otimes R\Ind_B^{P_s} \bk(\varsigma_s - \alpha_s) \otimes V = 0$, we can identify the second term in~\eqref{eqn:indbeta-isom} with $R\Ind_B^{P_s} (\St_s \otimes \bk(\varsigma_s) \otimes V)$ as well.  By Proposition~\ref{prop:steinberg-tilting}\eqref{it:ind-surj} and the surjectivity of the composition in~\eqref{eqn:indbeta4}, the composition of the two maps in~\eqref{eqn:indbeta-isom} is surjective.  Then Proposition~\ref{prop:steinberg-tilting}\eqref{it:ind-isom} tells us that the first map must be an isomorphism, as desired.
\end{proof}

\subsection{Main result}
\label{ss:main-translation}

Recall the definition of the functors $\Omega_I$ and $\Omega_J$ in~\eqref{eqn:omega-defn}. We
define natural transformations
\[
\theta: \Omega_I \circ \Theta_{J,I} \to T_J^I \circ \Omega_J \quad \text{and} \quad \phi: T_I^J \circ \Omega_I \to \Omega_J \circ \Theta^{J,I}
\]
by the diagrams in Figure~\ref{fig:thm-trans}. (The dotted boxes in that figure have no significance for the definition of $\theta$ and $\phi$, but they appear in the proof of the next lemma.)

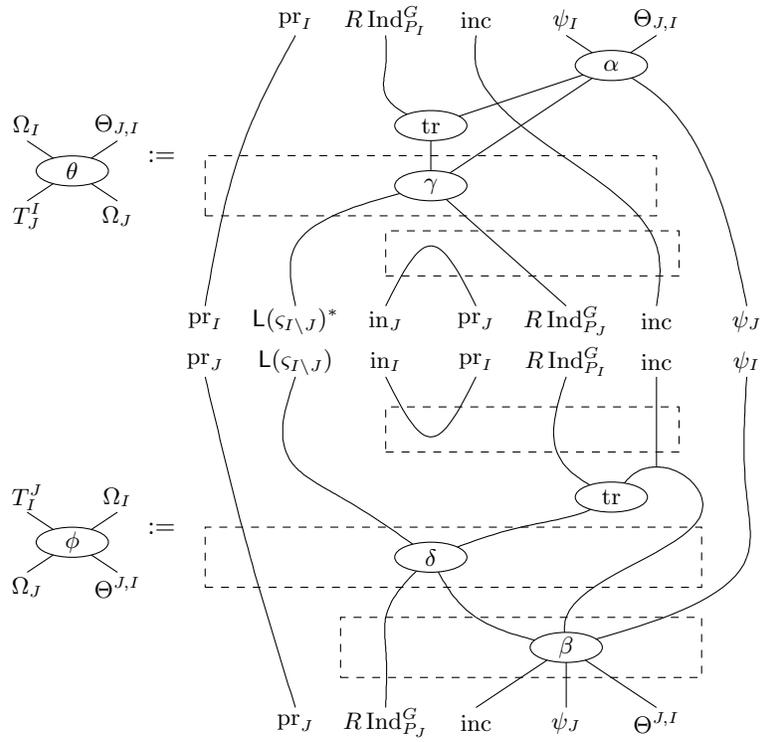
\begin{figure}
\ifdefined\PARTCOMPILE{
\[
FIGURE
\]
}
\else {
\begin{align*}
\vcenter{\hbox{\begin{tikzpicture}[trans]
\useasboundingbox (-0.5,-0.5) rectangle (2.5,3.5);
\draw (1.00,1.00) -- (2.00,0.00);
\draw (1.00,1.00) -- (0.00,0.00);
\draw (2.00,2.00) -- (1.00,1.00);
\draw (0.00,2.00) -- (1.00,1.00);
\filldraw[fill=white] (1.00,1.00) ellipse (0.80cm and 0.50cm);
\draw (0.00,2.50) node{$\Omega_I$};
\draw (2.00,2.50) node{$\Theta_{J,I}$};
\draw (1.00,1.00) node{$\theta$};
\draw (0.00,-0.50) node{$T_J^I$};
\draw (2.00,-0.50) node{$\Omega_J$};
\end{tikzpicture}}}
&:= 
\vcenter{\hbox{\begin{tikzpicture}[trans]
\useasboundingbox (-0.5,-0.5) rectangle (12.5,10.5);
\draw (4.00,9.00) -- (4.00,8.95) -- (4.00,8.90) -- (4.01,8.85) -- (4.01,8.80) -- (4.01,8.75) -- (4.01,8.70) -- (4.01,8.65) -- (4.01,8.61) -- (4.01,8.56) -- (4.01,8.51) -- (4.01,8.46) -- (4.01,8.41) -- (4.01,8.36) -- (4.01,8.31) -- (4.01,8.26) -- (4.01,8.20) -- (4.01,8.15) -- (4.01,8.10) -- (4.00,8.05) -- (4.00,8.00);
\draw (4.00,8.00) -- (4.00,7.95) -- (3.99,7.90) -- (3.99,7.84) -- (3.98,7.79) -- (3.98,7.74) -- (3.97,7.69) -- (3.97,7.64) -- (3.96,7.58) -- (3.96,7.53) -- (3.96,7.48) -- (3.95,7.43) -- (3.95,7.38) -- (3.95,7.33) -- (3.96,7.28) -- (3.96,7.23) -- (3.96,7.18) -- (3.97,7.14) -- (3.98,7.09) -- (3.99,7.04) -- (4.00,7.00);
\draw (4.00,7.00) -- (4.02,6.94) -- (4.05,6.88) -- (4.08,6.82) -- (4.11,6.76) -- (4.15,6.71) -- (4.19,6.65) -- (4.23,6.60) -- (4.28,6.55) -- (4.33,6.50) -- (4.38,6.45) -- (4.44,6.40) -- (4.50,6.36) -- (4.56,6.31) -- (4.62,6.27) -- (4.68,6.22) -- (4.74,6.18) -- (4.80,6.13) -- (4.87,6.09) -- (4.93,6.04) -- (5.00,6.00);
\draw (10.00,0.00) -- (10.01,0.10) -- (10.02,0.21) -- (10.03,0.31) -- (10.04,0.41) -- (10.05,0.52) -- (10.05,0.62) -- (10.06,0.72) -- (10.07,0.82) -- (10.07,0.93) -- (10.07,1.03) -- (10.08,1.13) -- (10.08,1.23) -- (10.07,1.33) -- (10.07,1.43) -- (10.06,1.52) -- (10.06,1.62) -- (10.05,1.72) -- (10.03,1.81) -- (10.02,1.91) -- (10.00,2.00);
\draw (10.00,2.00) -- (9.94,2.23) -- (9.87,2.45) -- (9.77,2.67) -- (9.67,2.88) -- (9.55,3.09) -- (9.42,3.29) -- (9.27,3.49) -- (9.12,3.68) -- (8.96,3.88) -- (8.79,4.07) -- (8.62,4.26) -- (8.44,4.45) -- (8.26,4.64) -- (8.08,4.83) -- (7.89,5.02) -- (7.71,5.21) -- (7.53,5.40) -- (7.35,5.60) -- (7.17,5.80) -- (7.00,6.00);
\draw (7.00,6.00) -- (6.93,6.09) -- (6.85,6.18) -- (6.78,6.28) -- (6.71,6.37) -- (6.65,6.47) -- (6.58,6.56) -- (6.52,6.66) -- (6.46,6.76) -- (6.40,6.86) -- (6.35,6.96) -- (6.30,7.06) -- (6.25,7.16) -- (6.20,7.26) -- (6.16,7.36) -- (6.13,7.47) -- (6.09,7.57) -- (6.06,7.68) -- (6.04,7.78) -- (6.02,7.89) -- (6.00,8.00);
\draw (6.00,8.00) -- (5.99,8.05) -- (5.99,8.10) -- (5.99,8.15) -- (5.98,8.20) -- (5.98,8.25) -- (5.98,8.30) -- (5.98,8.35) -- (5.98,8.40) -- (5.98,8.45) -- (5.98,8.50) -- (5.98,8.55) -- (5.98,8.60) -- (5.98,8.65) -- (5.98,8.70) -- (5.99,8.75) -- (5.99,8.80) -- (5.99,8.85) -- (5.99,8.90) -- (6.00,8.95) -- (6.00,9.00);
\draw (9.00,8.00) -- (9.05,7.95) -- (9.11,7.90) -- (9.16,7.86) -- (9.21,7.81) -- (9.26,7.76) -- (9.32,7.71) -- (9.37,7.67) -- (9.42,7.62) -- (9.47,7.57) -- (9.52,7.52) -- (9.57,7.47) -- (9.62,7.42) -- (9.67,7.37) -- (9.72,7.32) -- (9.77,7.27) -- (9.82,7.21) -- (9.87,7.16) -- (9.91,7.11) -- (9.96,7.05) -- (10.00,7.00);
\draw (10.00,7.00) -- (10.07,6.91) -- (10.13,6.82) -- (10.20,6.73) -- (10.26,6.64) -- (10.32,6.54) -- (10.37,6.45) -- (10.43,6.35) -- (10.48,6.25) -- (10.53,6.15) -- (10.58,6.05) -- (10.63,5.95) -- (10.68,5.85) -- (10.72,5.74) -- (10.77,5.64) -- (10.81,5.53) -- (10.85,5.43) -- (10.89,5.32) -- (10.93,5.21) -- (10.96,5.11) -- (11.00,5.00);
\draw (11.00,5.00) -- (11.08,4.75) -- (11.15,4.51) -- (11.23,4.26) -- (11.29,4.02) -- (11.35,3.77) -- (11.41,3.52) -- (11.47,3.27) -- (11.52,3.02) -- (11.57,2.77) -- (11.62,2.52) -- (11.66,2.27) -- (11.71,2.02) -- (11.75,1.77) -- (11.79,1.51) -- (11.82,1.26) -- (11.86,1.01) -- (11.90,0.76) -- (11.93,0.51) -- (11.97,0.25) -- (12.00,0.00);
\draw (9.00,8.00) -- (8.90,7.90) -- (8.80,7.80) -- (8.70,7.70) -- (8.60,7.60) -- (8.50,7.50) -- (8.40,7.40) -- (8.30,7.30) -- (8.20,7.20) -- (8.10,7.10) -- (8.00,7.00) -- (7.90,6.90) -- (7.80,6.80) -- (7.70,6.70) -- (7.60,6.60) -- (7.50,6.50) -- (7.40,6.40) -- (7.30,6.30) -- (7.20,6.20) -- (7.10,6.10) -- (7.00,6.00);
\draw (7.00,6.00) -- (6.90,5.90) -- (6.80,5.80) -- (6.70,5.70) -- (6.60,5.60) -- (6.50,5.50) -- (6.40,5.40) -- (6.30,5.30) -- (6.20,5.20) -- (6.10,5.10) -- (6.00,5.00) -- (5.90,4.90) -- (5.80,4.80) -- (5.70,4.70) -- (5.60,4.60) -- (5.50,4.50) -- (5.40,4.40) -- (5.30,4.30) -- (5.20,4.20) -- (5.10,4.10) -- (5.00,4.00);
\draw (9.00,8.00) -- (8.90,7.95) -- (8.80,7.90) -- (8.70,7.85) -- (8.60,7.80) -- (8.50,7.75) -- (8.40,7.70) -- (8.30,7.65) -- (8.20,7.60) -- (8.10,7.55) -- (8.00,7.50) -- (7.90,7.45) -- (7.80,7.40) -- (7.70,7.35) -- (7.60,7.30) -- (7.50,7.25) -- (7.40,7.20) -- (7.30,7.15) -- (7.20,7.10) -- (7.10,7.05) -- (7.00,7.00);
\draw (7.00,7.00) -- (6.90,6.95) -- (6.80,6.90) -- (6.70,6.85) -- (6.60,6.80) -- (6.50,6.75) -- (6.40,6.70) -- (6.30,6.65) -- (6.20,6.60) -- (6.10,6.55) -- (6.00,6.50) -- (5.90,6.45) -- (5.80,6.40) -- (5.70,6.35) -- (5.60,6.30) -- (5.50,6.25) -- (5.40,6.20) -- (5.30,6.15) -- (5.20,6.10) -- (5.10,6.05) -- (5.00,6.00);
\draw (10.00,9.00) -- (9.00,8.00);
\draw (8.00,9.00) -- (9.00,8.00);
\draw (2.00,9.00) -- (1.95,8.85) -- (1.89,8.70) -- (1.84,8.55) -- (1.78,8.40) -- (1.73,8.25) -- (1.68,8.11) -- (1.62,7.96) -- (1.57,7.81) -- (1.52,7.66) -- (1.47,7.51) -- (1.42,7.36) -- (1.37,7.21) -- (1.32,7.06) -- (1.27,6.91) -- (1.22,6.76) -- (1.18,6.61) -- (1.13,6.45) -- (1.09,6.30) -- (1.04,6.15) -- (1.00,6.00);
\draw (1.00,6.00) -- (0.92,5.71) -- (0.85,5.41) -- (0.78,5.12) -- (0.71,4.82) -- (0.65,4.52) -- (0.59,4.23) -- (0.54,3.93) -- (0.48,3.63) -- (0.43,3.33) -- (0.39,3.03) -- (0.34,2.73) -- (0.30,2.43) -- (0.26,2.12) -- (0.22,1.82) -- (0.18,1.52) -- (0.14,1.21) -- (0.11,0.91) -- (0.07,0.61) -- (0.03,0.30) -- (0.00,0.00);
\draw (5.00,6.00) -- (5.00,5.95) -- (5.00,5.90) -- (5.00,5.85) -- (5.00,5.80) -- (5.00,5.75) -- (5.00,5.70) -- (5.00,5.65) -- (5.00,5.60) -- (5.00,5.55) -- (5.00,5.50) -- (5.00,5.45) -- (5.00,5.40) -- (5.00,5.35) -- (5.00,5.30) -- (5.00,5.25) -- (5.00,5.20) -- (5.00,5.15) -- (5.00,5.10) -- (5.00,5.05) -- (5.00,5.00);
\draw (5.00,5.00) -- (5.00,4.95) -- (5.00,4.90) -- (5.00,4.85) -- (5.00,4.80) -- (5.00,4.75) -- (5.00,4.70) -- (5.00,4.65) -- (5.00,4.60) -- (5.00,4.55) -- (5.00,4.50) -- (5.00,4.45) -- (5.00,4.40) -- (5.00,4.35) -- (5.00,4.30) -- (5.00,4.25) -- (5.00,4.20) -- (5.00,4.15) -- (5.00,4.10) -- (5.00,4.05) -- (5.00,4.00);
\draw (5.00,4.00) -- (8.00,0.00);
\draw (5.00,4.00) -- (4.80,3.93) -- (4.60,3.85) -- (4.41,3.78) -- (4.21,3.70) -- (4.02,3.62) -- (3.84,3.54) -- (3.65,3.46) -- (3.48,3.37) -- (3.30,3.29) -- (3.14,3.19) -- (2.98,3.10) -- (2.83,3.00) -- (2.69,2.89) -- (2.56,2.78) -- (2.43,2.67) -- (2.32,2.55) -- (2.22,2.42) -- (2.14,2.29) -- (2.06,2.15) -- (2.00,2.00);
\draw (2.00,2.00) -- (1.97,1.91) -- (1.95,1.83) -- (1.93,1.74) -- (1.91,1.65) -- (1.90,1.55) -- (1.89,1.46) -- (1.89,1.36) -- (1.89,1.26) -- (1.89,1.16) -- (1.89,1.06) -- (1.89,0.96) -- (1.90,0.85) -- (1.91,0.75) -- (1.92,0.64) -- (1.93,0.54) -- (1.94,0.43) -- (1.96,0.32) -- (1.97,0.22) -- (1.99,0.11) -- (2.00,0.00);
\draw (4.00,0.00) -- (4.05,0.15) -- (4.10,0.30) -- (4.15,0.45) -- (4.20,0.59) -- (4.25,0.73) -- (4.30,0.87) -- (4.35,1.01) -- (4.40,1.14) -- (4.45,1.26) -- (4.50,1.38) -- (4.55,1.48) -- (4.60,1.58) -- (4.65,1.68) -- (4.70,1.76) -- (4.75,1.83) -- (4.80,1.89) -- (4.85,1.94) -- (4.90,1.97) -- (4.95,1.99) -- (5.00,2.00);
\draw (5.00,2.00) -- (5.05,1.99) -- (5.10,1.97) -- (5.15,1.94) -- (5.20,1.89) -- (5.25,1.83) -- (5.30,1.76) -- (5.35,1.68) -- (5.40,1.58) -- (5.45,1.48) -- (5.50,1.37) -- (5.55,1.26) -- (5.60,1.14) -- (5.65,1.01) -- (5.70,0.87) -- (5.75,0.73) -- (5.80,0.59) -- (5.85,0.45) -- (5.90,0.30) -- (5.95,0.15) -- (6.00,0.00);
\filldraw[fill=white] (9.00,8.00) ellipse (0.80cm and 0.50cm);
\filldraw[fill=white] (5.00,6.00) ellipse (0.80cm and 0.50cm);
\draw[dashed,] (0.00,5.00) rectangle (10.00,3.00);
\filldraw[fill=white] (5.00,4.00) ellipse (0.80cm and 0.50cm);
\draw[dashed,] (4.00,1.00) rectangle (10.50,2.50);
\draw (2.00,9.50) node{$\pr_I$};
\draw (4.00,9.50) node{$R\Ind_{P_I}^G$};
\draw (6.00,9.50) node{$\inc$};
\draw (8.00,9.50) node{$\psi_I$};
\draw (10.00,9.50) node{$\Theta_{J,I}$};
\draw (9.00,8.00) node{$\alpha$};
\draw (5.00,6.00) node{tr};
\draw (5.00,4.00) node{$\gamma$};
\draw (0.00,-0.50) node{$\pr_I$};
\draw (2.00,-0.50) node{$\irr(\varsigma_{I\ssm J})^*$};
\draw (4.00,-0.50) node{$\incl_J$};
\draw (6.00,-0.50) node{$\pr_J$};
\draw (8.00,-0.50) node{$R\Ind_{P_J}^G$};
\draw (10.00,-0.50) node{$\inc$};
\draw (12.00,-0.50) node{$\psi_J$};
\end{tikzpicture}}}
\\
\vcenter{\hbox{\begin{tikzpicture}[trans]
\useasboundingbox (-0.5,-0.5) rectangle (2.5,3.5);
\draw (1.00,1.00) -- (2.00,0.00);
\draw (1.00,1.00) -- (0.00,0.00);
\draw (2.00,2.00) -- (1.00,1.00);
\draw (0.00,2.00) -- (1.00,1.00);
\filldraw[fill=white] (1.00,1.00) ellipse (0.80cm and 0.50cm);
\draw (0.00,2.50) node{$T_I^J$};
\draw (2.00,2.50) node{$\Omega_I$};
\draw (1.00,1.00) node{$\phi$};
\draw (0.00,-0.50) node{$\Omega_J$};
\draw (2.00,-0.50) node{$\Theta^{J,I}$};
\end{tikzpicture}}}
 &:= 
\vcenter{\hbox{\begin{tikzpicture}[trans]
\useasboundingbox (-0.5,-0.5) rectangle (12.5,12.5);
\draw (4.00,11.00) -- (4.05,10.85) -- (4.10,10.70) -- (4.15,10.55) -- (4.20,10.41) -- (4.25,10.27) -- (4.30,10.13) -- (4.35,9.99) -- (4.40,9.86) -- (4.45,9.74) -- (4.50,9.62) -- (4.55,9.52) -- (4.60,9.42) -- (4.65,9.32) -- (4.70,9.24) -- (4.75,9.17) -- (4.80,9.11) -- (4.85,9.06) -- (4.90,9.03) -- (4.95,9.01) -- (5.00,9.00);
\draw (5.00,9.00) -- (5.05,9.01) -- (5.10,9.03) -- (5.15,9.06) -- (5.20,9.11) -- (5.25,9.17) -- (5.30,9.24) -- (5.35,9.32) -- (5.40,9.42) -- (5.45,9.52) -- (5.50,9.63) -- (5.55,9.74) -- (5.60,9.86) -- (5.65,9.99) -- (5.70,10.13) -- (5.75,10.27) -- (5.80,10.41) -- (5.85,10.55) -- (5.90,10.70) -- (5.95,10.85) -- (6.00,11.00);
\draw (2.00,11.00) -- (1.96,10.83) -- (1.93,10.67) -- (1.89,10.51) -- (1.86,10.34) -- (1.83,10.18) -- (1.80,10.02) -- (1.77,9.86) -- (1.75,9.70) -- (1.73,9.54) -- (1.72,9.38) -- (1.72,9.23) -- (1.72,9.08) -- (1.72,8.93) -- (1.74,8.79) -- (1.76,8.65) -- (1.79,8.51) -- (1.83,8.38) -- (1.87,8.25) -- (1.93,8.12) -- (2.00,8.00);
\draw (2.00,8.00) -- (2.04,7.94) -- (2.07,7.89) -- (2.12,7.84) -- (2.16,7.78) -- (2.20,7.73) -- (2.25,7.68) -- (2.30,7.63) -- (2.35,7.58) -- (2.40,7.53) -- (2.45,7.48) -- (2.50,7.43) -- (2.56,7.38) -- (2.61,7.33) -- (2.66,7.29) -- (2.72,7.24) -- (2.78,7.19) -- (2.83,7.14) -- (2.89,7.10) -- (2.94,7.05) -- (3.00,7.00);
\draw (3.00,7.00) -- (3.11,6.90) -- (3.22,6.81) -- (3.33,6.71) -- (3.43,6.61) -- (3.54,6.51) -- (3.64,6.42) -- (3.74,6.32) -- (3.84,6.22) -- (3.94,6.12) -- (4.04,6.02) -- (4.14,5.92) -- (4.24,5.82) -- (4.33,5.71) -- (4.43,5.61) -- (4.53,5.51) -- (4.62,5.41) -- (4.72,5.31) -- (4.81,5.20) -- (4.91,5.10) -- (5.00,5.00);
\draw (8.00,11.00) -- (7.96,10.84) -- (7.93,10.67) -- (7.89,10.51) -- (7.86,10.34) -- (7.83,10.18) -- (7.80,10.02) -- (7.78,9.86) -- (7.76,9.70) -- (7.74,9.54) -- (7.73,9.39) -- (7.72,9.24) -- (7.72,9.09) -- (7.73,8.94) -- (7.74,8.79) -- (7.76,8.65) -- (7.79,8.51) -- (7.83,8.38) -- (7.88,8.25) -- (7.93,8.12) -- (8.00,8.00);
\draw (8.00,8.00) -- (8.04,7.94) -- (8.07,7.89) -- (8.11,7.83) -- (8.15,7.78) -- (8.20,7.73) -- (8.24,7.68) -- (8.29,7.63) -- (8.34,7.57) -- (8.39,7.52) -- (8.44,7.48) -- (8.49,7.43) -- (8.55,7.38) -- (8.60,7.33) -- (8.66,7.28) -- (8.71,7.23) -- (8.77,7.19) -- (8.83,7.14) -- (8.88,7.09) -- (8.94,7.05) -- (9.00,7.00);
\draw (10.00,8.00) -- (10.00,11.00);
\draw (5.00,5.00) -- (5.10,5.06) -- (5.19,5.11) -- (5.29,5.17) -- (5.38,5.22) -- (5.48,5.28) -- (5.58,5.33) -- (5.67,5.39) -- (5.77,5.44) -- (5.87,5.49) -- (5.97,5.55) -- (6.07,5.60) -- (6.17,5.65) -- (6.27,5.70) -- (6.37,5.74) -- (6.47,5.79) -- (6.57,5.84) -- (6.68,5.88) -- (6.78,5.92) -- (6.89,5.96) -- (7.00,6.00);
\draw (7.00,6.00) -- (7.11,6.04) -- (7.22,6.07) -- (7.33,6.11) -- (7.44,6.14) -- (7.55,6.17) -- (7.66,6.21) -- (7.78,6.24) -- (7.88,6.28) -- (7.99,6.32) -- (8.10,6.36) -- (8.20,6.40) -- (8.31,6.45) -- (8.41,6.50) -- (8.50,6.56) -- (8.59,6.62) -- (8.68,6.68) -- (8.77,6.75) -- (8.85,6.83) -- (8.93,6.91) -- (9.00,7.00);
\draw (9.00,7.00) -- (9.04,7.05) -- (9.08,7.11) -- (9.11,7.17) -- (9.15,7.23) -- (9.18,7.28) -- (9.21,7.34) -- (9.25,7.40) -- (9.28,7.46) -- (9.32,7.52) -- (9.35,7.58) -- (9.39,7.63) -- (9.42,7.68) -- (9.46,7.73) -- (9.50,7.78) -- (9.55,7.82) -- (9.59,7.86) -- (9.64,7.90) -- (9.69,7.93) -- (9.74,7.95) -- (9.80,7.97);
\draw (9.80,7.97) -- (9.81,7.98) -- (9.82,7.98) -- (9.83,7.98) -- (9.84,7.98) -- (9.85,7.99) -- (9.86,7.99) -- (9.87,7.99) -- (9.88,7.99) -- (9.89,7.99) -- (9.90,7.99) -- (9.91,8.00) -- (9.92,8.00) -- (9.93,8.00) -- (9.94,8.00) -- (9.95,8.00) -- (9.96,8.00) -- (9.97,8.00) -- (9.98,8.00) -- (9.99,8.00) -- (10.00,8.00);
\draw (10.00,8.00) -- (10.01,8.00) -- (10.02,8.00) -- (10.03,8.00) -- (10.04,8.00) -- (10.05,8.00) -- (10.06,8.00) -- (10.07,8.00) -- (10.08,8.00) -- (10.09,7.99) -- (10.10,7.99) -- (10.11,7.99) -- (10.12,7.99) -- (10.13,7.99) -- (10.14,7.99) -- (10.15,7.99) -- (10.16,7.98) -- (10.17,7.98) -- (10.18,7.98) -- (10.19,7.98) -- (10.20,7.97);
\draw (10.20,7.97) -- (10.26,7.96) -- (10.32,7.93) -- (10.37,7.90) -- (10.43,7.87) -- (10.48,7.83) -- (10.53,7.79) -- (10.58,7.75) -- (10.63,7.70) -- (10.67,7.65) -- (10.72,7.60) -- (10.76,7.54) -- (10.79,7.49) -- (10.83,7.43) -- (10.86,7.37) -- (10.89,7.31) -- (10.92,7.25) -- (10.94,7.19) -- (10.97,7.12) -- (10.98,7.06) -- (11.00,7.00);
\draw (11.00,7.00) -- (11.03,6.76) -- (11.01,6.54) -- (10.96,6.32) -- (10.86,6.11) -- (10.73,5.91) -- (10.58,5.72) -- (10.40,5.53) -- (10.20,5.35) -- (9.99,5.16) -- (9.77,4.98) -- (9.54,4.80) -- (9.31,4.62) -- (9.09,4.44) -- (8.87,4.25) -- (8.67,4.06) -- (8.48,3.86) -- (8.32,3.66) -- (8.18,3.45) -- (8.07,3.23) -- (8.00,3.00);
\draw (8.00,3.00) -- (7.99,2.95) -- (7.98,2.91) -- (7.97,2.86) -- (7.97,2.81) -- (7.97,2.76) -- (7.96,2.71) -- (7.96,2.66) -- (7.96,2.61) -- (7.96,2.56) -- (7.96,2.51) -- (7.96,2.46) -- (7.96,2.41) -- (7.97,2.36) -- (7.97,2.31) -- (7.98,2.26) -- (7.98,2.21) -- (7.98,2.15) -- (7.99,2.10) -- (7.99,2.05) -- (8.00,2.00);
\draw (12.00,11.00) -- (11.99,10.80) -- (11.98,10.61) -- (11.98,10.41) -- (11.97,10.21) -- (11.96,10.02) -- (11.96,9.82) -- (11.95,9.62) -- (11.95,9.43) -- (11.94,9.23) -- (11.94,9.03) -- (11.94,8.83) -- (11.94,8.63) -- (11.94,8.43) -- (11.94,8.23) -- (11.95,8.03) -- (11.96,7.82) -- (11.96,7.62) -- (11.97,7.41) -- (11.99,7.21) -- (12.00,7.00);
\draw (12.00,7.00) -- (12.01,6.90) -- (12.02,6.79) -- (12.03,6.69) -- (12.03,6.58) -- (12.04,6.48) -- (12.05,6.38) -- (12.06,6.27) -- (12.06,6.17) -- (12.07,6.07) -- (12.07,5.96) -- (12.08,5.86) -- (12.08,5.76) -- (12.08,5.66) -- (12.07,5.56) -- (12.07,5.47) -- (12.06,5.37) -- (12.05,5.28) -- (12.04,5.18) -- (12.02,5.09) -- (12.00,5.00);
\draw (12.00,5.00) -- (11.94,4.78) -- (11.85,4.57) -- (11.75,4.38) -- (11.62,4.19) -- (11.48,4.01) -- (11.32,3.84) -- (11.15,3.68) -- (10.96,3.52) -- (10.76,3.37) -- (10.55,3.23) -- (10.33,3.09) -- (10.09,2.95) -- (9.85,2.83) -- (9.60,2.70) -- (9.34,2.58) -- (9.08,2.46) -- (8.81,2.34) -- (8.54,2.23) -- (8.27,2.11) -- (8.00,2.00);
\draw (0.00,11.00) -- (0.05,10.70) -- (0.09,10.40) -- (0.14,10.10) -- (0.19,9.80) -- (0.24,9.50) -- (0.29,9.20) -- (0.33,8.90) -- (0.38,8.60) -- (0.43,8.30) -- (0.48,8.00) -- (0.53,7.70) -- (0.58,7.40) -- (0.63,7.10) -- (0.68,6.80) -- (0.73,6.50) -- (0.78,6.20) -- (0.84,5.90) -- (0.89,5.60) -- (0.95,5.30) -- (1.00,5.00);
\draw (1.00,5.00) -- (1.05,4.75) -- (1.09,4.50) -- (1.14,4.25) -- (1.19,4.00) -- (1.24,3.75) -- (1.29,3.50) -- (1.34,3.25) -- (1.39,3.00) -- (1.44,2.75) -- (1.49,2.50) -- (1.54,2.25) -- (1.59,2.00) -- (1.64,1.75) -- (1.69,1.50) -- (1.74,1.25) -- (1.79,1.00) -- (1.84,0.75) -- (1.90,0.50) -- (1.95,0.25) -- (2.00,0.00);
\draw (5.00,5.00) -- (5.04,4.89) -- (5.08,4.78) -- (5.11,4.66) -- (5.15,4.55) -- (5.19,4.44) -- (5.23,4.33) -- (5.27,4.22) -- (5.32,4.12) -- (5.36,4.01) -- (5.41,3.91) -- (5.45,3.80) -- (5.50,3.70) -- (5.56,3.61) -- (5.61,3.51) -- (5.67,3.42) -- (5.73,3.33) -- (5.79,3.24) -- (5.86,3.16) -- (5.93,3.08) -- (6.00,3.00);
\draw (6.00,3.00) -- (6.08,2.93) -- (6.16,2.86) -- (6.24,2.79) -- (6.33,2.73) -- (6.42,2.67) -- (6.51,2.61) -- (6.61,2.56) -- (6.70,2.50) -- (6.80,2.45) -- (6.91,2.41) -- (7.01,2.36) -- (7.12,2.32) -- (7.22,2.27) -- (7.33,2.23) -- (7.44,2.19) -- (7.55,2.15) -- (7.66,2.11) -- (7.78,2.08) -- (7.89,2.04) -- (8.00,2.00);
\draw (5.00,5.00) -- (4.93,4.90) -- (4.86,4.81) -- (4.80,4.71) -- (4.73,4.62) -- (4.67,4.52) -- (4.60,4.42) -- (4.54,4.33) -- (4.48,4.23) -- (4.42,4.13) -- (4.37,4.03) -- (4.31,3.93) -- (4.26,3.83) -- (4.22,3.73) -- (4.17,3.63) -- (4.13,3.53) -- (4.10,3.42) -- (4.07,3.32) -- (4.04,3.21) -- (4.02,3.11) -- (4.00,3.00);
\draw (4.00,3.00) -- (3.99,2.95) -- (3.99,2.90) -- (3.99,2.85) -- (3.98,2.80) -- (3.98,2.75) -- (3.98,2.71) -- (3.98,2.66) -- (3.98,2.61) -- (3.98,2.56) -- (3.98,2.51) -- (3.98,2.45) -- (3.98,2.40) -- (3.98,2.35) -- (3.98,2.30) -- (3.99,2.25) -- (3.99,2.20) -- (3.99,2.15) -- (3.99,2.10) -- (4.00,2.05) -- (4.00,2.00);
\draw (4.00,2.00) -- (4.00,1.90) -- (4.01,1.80) -- (4.01,1.70) -- (4.01,1.60) -- (4.02,1.50) -- (4.02,1.40) -- (4.02,1.30) -- (4.02,1.20) -- (4.02,1.10) -- (4.02,1.00) -- (4.02,0.90) -- (4.02,0.80) -- (4.01,0.70) -- (4.01,0.60) -- (4.01,0.50) -- (4.01,0.40) -- (4.01,0.30) -- (4.00,0.20) -- (4.00,0.10) -- (4.00,0.00);
\draw (8.00,2.00) -- (10.00,0.00);
\draw (8.00,2.00) -- (8.00,0.00);
\draw (8.00,2.00) -- (6.00,0.00);
\draw[dashed,] (4.00,10.00) rectangle (10.50,8.50);
\filldraw[fill=white] (9.00,7.00) ellipse (0.80cm and 0.50cm);
\draw[dashed,] (0.00,6.00) rectangle (11.00,4.00);
\filldraw[fill=white] (5.00,5.00) ellipse (0.80cm and 0.50cm);
\filldraw[fill=white] (8.00,2.00) ellipse (0.80cm and 0.50cm);
\draw[dashed,] (3.00,1.00) rectangle (11.00,3.00);
\draw (0.00,11.50) node{$\pr_J$};
\draw (2.00,11.50) node{$\irr(\varsigma_{I\ssm J})$};
\draw (4.00,11.50) node{$\incl_I$};
\draw (6.00,11.50) node{$\pr_I$};
\draw (8.00,11.50) node{$R\Ind_{P_I}^G$};
\draw (10.00,11.50) node{$\inc$};
\draw (12.00,11.50) node{$\psi_I$};
\draw (9.00,7.00) node{tr};
\draw (5.00,5.00) node{$\delta$};
\draw (8.00,2.00) node{$\beta$};
\draw (2.00,-0.50) node{$\pr_J$};
\draw (4.00,-0.50) node{$R\Ind_{P_J}^G$};
\draw (6.00,-0.50) node{$\inc$};
\draw (8.00,-0.50) node{$\psi_J$};
\draw (10.00,-0.50) node{$\Theta^{J,I}$};
\end{tikzpicture}}}
\end{align*}
}
\fi
\caption{Natural transformations for Theorem~\ref{thm:translation}}\label{fig:thm-trans}
\end{figure}

\begin{lem}
\phantomsection
\label{lem:theta-phi-iso}
\begin{enumerate}
\item The natural transformation $\theta$ is an isomorphism.
\item If $J = \varnothing$ and $I = \{s\}$, then $\phi$ is an isomorphism.
\end{enumerate}
\end{lem}

\begin{proof}
The large diagrams in Figure~\ref{fig:thm-trans} are mostly assembled from constituents that are already known to be isomorphisms, such as those from~\eqref{eqn:RInd-For-PI-PJ}, \eqref{eqn:comm-defn}, \eqref{eqn:tr-defn}, and~\eqref{eqn:alpha-defn}.  To complete the proof, we must check that each region enclosed in dotted lines is an isomorphism (under the appropriate assumptions).

In the definition of $\theta$, the two such regions are isomorphisms by Lemmas~\ref{lem:induce-block} and~\ref{lem:induce-translate}.  In the definition of $\phi$, the two upper regions are isomorphisms by Lemmas~\ref{lem:induce-block} and~\ref{lem:induce-p-translate}.  For the lower one, we must add the assumption that $J = \varnothing$ and $I = \{s\}$, and then invoke Proposition~\ref{prop:indbeta}.
\end{proof}

Recall (see~\S\ref{ss:Theta}) that the functor $\Theta^{J,I}$ is naturally left adjoint to $\Theta_{J,I}$. On the other hand, since the functor $T_I^J$ and $T_J^I$ are built from functors which are naturally (bi)adjoint, $T_I^J$ is naturally left adjoint to $T_J^I$. Therefore, the following lemma makes sense.

\begin{lem}
\label{lem:theta-phi-adj}
We have $\phi = \theta^\wedge$ and $\theta = \phi^\vee$.
\end{lem}

\begin{proof}
Since the operations $(-)^\wedge$ and $(-)^\vee$ are inverse to each other, the two equalities are equivalent; so we need only prove the first one.
Unpacking the definitions, this equality is equivalent to 
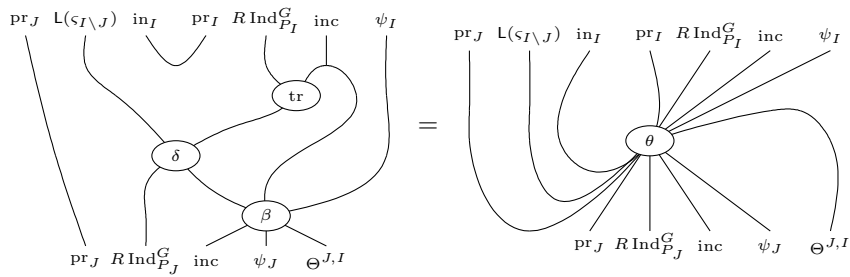
\ifdefined\PARTCOMPILE{
\[
FIGURE
\]
}
\else {
\[
\vcenter{\hbox{\begin{tikzpicture}[transsmall]
\useasboundingbox (-0.5,-0.5) rectangle (12.5,8.5);
\draw (2.00,7.00) -- (1.99,6.95) -- (1.98,6.89) -- (1.98,6.84) -- (1.97,6.79) -- (1.96,6.74) -- (1.96,6.68) -- (1.95,6.63) -- (1.95,6.58) -- (1.94,6.53) -- (1.94,6.48) -- (1.94,6.43) -- (1.94,6.38) -- (1.94,6.33) -- (1.95,6.28) -- (1.95,6.23) -- (1.96,6.18) -- (1.96,6.13) -- (1.97,6.09) -- (1.99,6.04) -- (2.00,6.00);
\draw (2.00,6.00) -- (2.02,5.94) -- (2.05,5.88) -- (2.08,5.82) -- (2.12,5.77) -- (2.16,5.71) -- (2.20,5.66) -- (2.25,5.61) -- (2.30,5.56) -- (2.35,5.51) -- (2.40,5.46) -- (2.46,5.41) -- (2.52,5.37) -- (2.57,5.32) -- (2.63,5.27) -- (2.69,5.23) -- (2.76,5.18) -- (2.82,5.14) -- (2.88,5.09) -- (2.94,5.05) -- (3.00,5.00);
\draw (3.00,5.00) -- (3.12,4.91) -- (3.24,4.81) -- (3.35,4.72) -- (3.46,4.62) -- (3.57,4.53) -- (3.67,4.43) -- (3.78,4.33) -- (3.88,4.23) -- (3.98,4.13) -- (4.08,4.03) -- (4.17,3.93) -- (4.27,3.83) -- (4.36,3.73) -- (4.46,3.62) -- (4.55,3.52) -- (4.64,3.42) -- (4.73,3.31) -- (4.82,3.21) -- (4.91,3.10) -- (5.00,3.00);
\draw (4.00,7.00) -- (4.05,6.93) -- (4.10,6.85) -- (4.15,6.78) -- (4.20,6.70) -- (4.25,6.63) -- (4.30,6.56) -- (4.35,6.50) -- (4.40,6.43) -- (4.45,6.37) -- (4.50,6.31) -- (4.55,6.26) -- (4.60,6.21) -- (4.65,6.16) -- (4.70,6.12) -- (4.75,6.09) -- (4.80,6.06) -- (4.85,6.03) -- (4.90,6.01) -- (4.95,6.00) -- (5.00,6.00);
\draw (5.00,6.00) -- (5.05,6.00) -- (5.10,6.01) -- (5.15,6.03) -- (5.20,6.06) -- (5.25,6.09) -- (5.30,6.12) -- (5.35,6.16) -- (5.40,6.21) -- (5.45,6.26) -- (5.50,6.31) -- (5.55,6.37) -- (5.60,6.43) -- (5.65,6.50) -- (5.70,6.56) -- (5.75,6.63) -- (5.80,6.70) -- (5.85,6.78) -- (5.90,6.85) -- (5.95,6.93) -- (6.00,7.00);
\draw (8.00,7.00) -- (7.99,6.95) -- (7.99,6.89) -- (7.98,6.84) -- (7.97,6.79) -- (7.97,6.74) -- (7.96,6.68) -- (7.96,6.63) -- (7.95,6.58) -- (7.95,6.53) -- (7.95,6.48) -- (7.94,6.43) -- (7.94,6.38) -- (7.95,6.33) -- (7.95,6.28) -- (7.95,6.23) -- (7.96,6.18) -- (7.97,6.14) -- (7.97,6.09) -- (7.99,6.04) -- (8.00,6.00);
\draw (8.00,6.00) -- (8.02,5.94) -- (8.05,5.88) -- (8.08,5.82) -- (8.12,5.77) -- (8.15,5.71) -- (8.20,5.66) -- (8.24,5.60) -- (8.29,5.55) -- (8.34,5.50) -- (8.39,5.45) -- (8.44,5.41) -- (8.50,5.36) -- (8.56,5.31) -- (8.62,5.27) -- (8.68,5.22) -- (8.74,5.18) -- (8.81,5.13) -- (8.87,5.09) -- (8.94,5.04) -- (9.00,5.00);
\draw (10.00,6.00) -- (10.00,7.00);
\draw (5.00,3.00) -- (5.10,3.06) -- (5.19,3.11) -- (5.29,3.17) -- (5.38,3.22) -- (5.48,3.28) -- (5.58,3.33) -- (5.67,3.39) -- (5.77,3.44) -- (5.87,3.49) -- (5.97,3.55) -- (6.07,3.60) -- (6.17,3.65) -- (6.27,3.70) -- (6.37,3.74) -- (6.47,3.79) -- (6.57,3.84) -- (6.68,3.88) -- (6.78,3.92) -- (6.89,3.96) -- (7.00,4.00);
\draw (7.00,4.00) -- (7.11,4.04) -- (7.22,4.07) -- (7.33,4.11) -- (7.44,4.14) -- (7.55,4.17) -- (7.66,4.21) -- (7.78,4.24) -- (7.88,4.28) -- (7.99,4.32) -- (8.10,4.36) -- (8.20,4.40) -- (8.31,4.45) -- (8.41,4.50) -- (8.50,4.56) -- (8.59,4.62) -- (8.68,4.68) -- (8.77,4.75) -- (8.85,4.83) -- (8.93,4.91) -- (9.00,5.00);
\draw (9.00,5.00) -- (9.04,5.05) -- (9.08,5.11) -- (9.11,5.17) -- (9.15,5.23) -- (9.18,5.28) -- (9.21,5.34) -- (9.25,5.40) -- (9.28,5.46) -- (9.32,5.52) -- (9.35,5.58) -- (9.39,5.63) -- (9.42,5.68) -- (9.46,5.73) -- (9.50,5.78) -- (9.55,5.82) -- (9.59,5.86) -- (9.64,5.90) -- (9.69,5.93) -- (9.74,5.95) -- (9.80,5.97);
\draw (9.80,5.97) -- (9.81,5.98) -- (9.82,5.98) -- (9.83,5.98) -- (9.84,5.98) -- (9.85,5.99) -- (9.86,5.99) -- (9.87,5.99) -- (9.88,5.99) -- (9.89,5.99) -- (9.90,5.99) -- (9.91,6.00) -- (9.92,6.00) -- (9.93,6.00) -- (9.94,6.00) -- (9.95,6.00) -- (9.96,6.00) -- (9.97,6.00) -- (9.98,6.00) -- (9.99,6.00) -- (10.00,6.00);
\draw (10.00,6.00) -- (10.01,6.00) -- (10.02,6.00) -- (10.03,6.00) -- (10.04,6.00) -- (10.05,6.00) -- (10.06,6.00) -- (10.07,6.00) -- (10.08,6.00) -- (10.09,5.99) -- (10.10,5.99) -- (10.11,5.99) -- (10.12,5.99) -- (10.13,5.99) -- (10.14,5.99) -- (10.15,5.99) -- (10.16,5.98) -- (10.17,5.98) -- (10.18,5.98) -- (10.19,5.98) -- (10.20,5.97);
\draw (10.20,5.97) -- (10.26,5.95) -- (10.32,5.93) -- (10.38,5.90) -- (10.43,5.87) -- (10.49,5.83) -- (10.54,5.79) -- (10.59,5.74) -- (10.64,5.70) -- (10.68,5.65) -- (10.73,5.59) -- (10.77,5.54) -- (10.81,5.48) -- (10.84,5.42) -- (10.88,5.36) -- (10.91,5.30) -- (10.93,5.24) -- (10.95,5.18) -- (10.97,5.12) -- (10.99,5.06) -- (11.00,5.00);
\draw (11.00,5.00) -- (11.01,4.81) -- (10.98,4.63) -- (10.91,4.46) -- (10.81,4.30) -- (10.67,4.16) -- (10.52,4.01) -- (10.34,3.88) -- (10.15,3.74) -- (9.94,3.61) -- (9.72,3.48) -- (9.50,3.36) -- (9.28,3.23) -- (9.07,3.10) -- (8.86,2.96) -- (8.66,2.82) -- (8.48,2.68) -- (8.32,2.52) -- (8.18,2.36) -- (8.08,2.19) -- (8.00,2.00);
\draw (8.00,2.00) -- (7.99,1.95) -- (7.98,1.91) -- (7.97,1.86) -- (7.96,1.81) -- (7.96,1.77) -- (7.95,1.72) -- (7.95,1.67) -- (7.95,1.62) -- (7.95,1.57) -- (7.95,1.52) -- (7.95,1.47) -- (7.95,1.42) -- (7.96,1.36) -- (7.96,1.31) -- (7.97,1.26) -- (7.97,1.21) -- (7.98,1.16) -- (7.99,1.10) -- (7.99,1.05) -- (8.00,1.00);
\draw (12.00,7.00) -- (12.00,6.90) -- (11.99,6.80) -- (11.99,6.71) -- (11.99,6.61) -- (11.98,6.51) -- (11.98,6.41) -- (11.98,6.31) -- (11.98,6.21) -- (11.97,6.12) -- (11.97,6.02) -- (11.97,5.92) -- (11.97,5.82) -- (11.97,5.72) -- (11.97,5.62) -- (11.98,5.51) -- (11.98,5.41) -- (11.98,5.31) -- (11.99,5.21) -- (11.99,5.10) -- (12.00,5.00);
\draw (12.00,5.00) -- (12.01,4.90) -- (12.02,4.79) -- (12.03,4.68) -- (12.03,4.58) -- (12.04,4.47) -- (12.05,4.37) -- (12.06,4.26) -- (12.07,4.16) -- (12.08,4.05) -- (12.08,3.95) -- (12.08,3.85) -- (12.09,3.75) -- (12.09,3.65) -- (12.08,3.55) -- (12.08,3.45) -- (12.07,3.36) -- (12.06,3.26) -- (12.04,3.17) -- (12.02,3.09) -- (12.00,3.00);
\draw (12.00,3.00) -- (11.93,2.82) -- (11.85,2.65) -- (11.74,2.49) -- (11.61,2.34) -- (11.47,2.21) -- (11.31,2.08) -- (11.14,1.97) -- (10.95,1.86) -- (10.75,1.76) -- (10.54,1.67) -- (10.32,1.58) -- (10.08,1.50) -- (9.84,1.43) -- (9.59,1.36) -- (9.34,1.29) -- (9.08,1.23) -- (8.81,1.17) -- (8.54,1.11) -- (8.27,1.06) -- (8.00,1.00);
\draw (0.00,7.00) -- (0.05,6.80) -- (0.09,6.60) -- (0.14,6.40) -- (0.18,6.20) -- (0.23,5.99) -- (0.28,5.79) -- (0.32,5.59) -- (0.37,5.39) -- (0.42,5.19) -- (0.47,4.99) -- (0.52,4.79) -- (0.57,4.59) -- (0.62,4.39) -- (0.67,4.19) -- (0.72,3.99) -- (0.78,3.79) -- (0.83,3.59) -- (0.89,3.40) -- (0.94,3.20) -- (1.00,3.00);
\draw (1.00,3.00) -- (1.05,2.85) -- (1.09,2.70) -- (1.14,2.55) -- (1.19,2.40) -- (1.23,2.25) -- (1.28,2.09) -- (1.33,1.94) -- (1.38,1.79) -- (1.43,1.64) -- (1.48,1.49) -- (1.53,1.34) -- (1.58,1.20) -- (1.63,1.05) -- (1.69,0.90) -- (1.74,0.75) -- (1.79,0.60) -- (1.84,0.45) -- (1.89,0.30) -- (1.95,0.15) -- (2.00,0.00);
\draw (5.00,3.00) -- (5.05,2.95) -- (5.09,2.89) -- (5.14,2.84) -- (5.19,2.79) -- (5.24,2.73) -- (5.29,2.68) -- (5.33,2.63) -- (5.38,2.58) -- (5.43,2.52) -- (5.48,2.47) -- (5.53,2.42) -- (5.58,2.37) -- (5.63,2.32) -- (5.68,2.27) -- (5.73,2.23) -- (5.79,2.18) -- (5.84,2.13) -- (5.89,2.09) -- (5.95,2.04) -- (6.00,2.00);
\draw (6.00,2.00) -- (6.09,1.93) -- (6.18,1.87) -- (6.27,1.81) -- (6.36,1.75) -- (6.46,1.69) -- (6.55,1.64) -- (6.65,1.58) -- (6.75,1.53) -- (6.85,1.48) -- (6.95,1.43) -- (7.05,1.39) -- (7.16,1.34) -- (7.26,1.30) -- (7.36,1.25) -- (7.47,1.21) -- (7.58,1.17) -- (7.68,1.12) -- (7.79,1.08) -- (7.89,1.04) -- (8.00,1.00);
\draw (5.00,3.00) -- (4.93,2.96) -- (4.87,2.91) -- (4.80,2.87) -- (4.74,2.82) -- (4.68,2.78) -- (4.62,2.73) -- (4.56,2.69) -- (4.50,2.64) -- (4.44,2.60) -- (4.38,2.55) -- (4.33,2.50) -- (4.28,2.45) -- (4.23,2.40) -- (4.19,2.35) -- (4.15,2.29) -- (4.11,2.24) -- (4.08,2.18) -- (4.05,2.12) -- (4.02,2.06) -- (4.00,2.00);
\draw (4.00,2.00) -- (3.99,1.96) -- (3.98,1.91) -- (3.97,1.86) -- (3.96,1.82) -- (3.96,1.77) -- (3.96,1.72) -- (3.95,1.67) -- (3.95,1.62) -- (3.95,1.57) -- (3.96,1.52) -- (3.96,1.47) -- (3.96,1.42) -- (3.97,1.36) -- (3.97,1.31) -- (3.98,1.26) -- (3.98,1.21) -- (3.99,1.16) -- (3.99,1.10) -- (4.00,1.05) -- (4.00,1.00);
\draw (4.00,1.00) -- (4.00,0.95) -- (4.01,0.90) -- (4.01,0.85) -- (4.01,0.80) -- (4.01,0.74) -- (4.01,0.69) -- (4.01,0.64) -- (4.01,0.59) -- (4.01,0.54) -- (4.01,0.49) -- (4.01,0.44) -- (4.01,0.39) -- (4.01,0.35) -- (4.01,0.30) -- (4.01,0.25) -- (4.01,0.20) -- (4.01,0.15) -- (4.00,0.10) -- (4.00,0.05) -- (4.00,0.00);
\draw (8.00,1.00) -- (10.00,0.00);
\draw (8.00,1.00) -- (8.00,0.00);
\draw (8.00,1.00) -- (6.00,0.00);
\filldraw[fill=white] (9.00,5.00) ellipse (0.80cm and 0.50cm);
\filldraw[fill=white] (5.00,3.00) ellipse (0.80cm and 0.50cm);
\filldraw[fill=white] (8.00,1.00) ellipse (0.80cm and 0.50cm);
\draw (0.00,7.50) node{$\pr_J$};
\draw (2.00,7.50) node{$\irr(\varsigma_{I\ssm J})$};
\draw (4.00,7.50) node{$\incl_I$};
\draw (6.00,7.50) node{$\pr_I$};
\draw (8.00,7.50) node{$R\Ind_{P_I}^G$};
\draw (10.00,7.50) node{$\inc$};
\draw (12.00,7.50) node{$\psi_I$};
\draw (9.00,5.00) node{tr};
\draw (5.00,3.00) node{$\delta$};
\draw (8.00,1.00) node{$\beta$};
\draw (2.00,-0.50) node{$\pr_J$};
\draw (4.00,-0.50) node{$R\Ind_{P_J}^G$};
\draw (6.00,-0.50) node{$\inc$};
\draw (8.00,-0.50) node{$\psi_J$};
\draw (10.00,-0.50) node{$\Theta^{J,I}$};
\end{tikzpicture}}}
\ =\ 
\vcenter{\hbox{\begin{tikzpicture}[transsmall]
\useasboundingbox (-0.5,-0.5) rectangle (12.5,7.5);
\draw (0.00,6.00) -- (0.00,5.90) -- (0.00,5.80) -- (0.00,5.70) -- (0.01,5.60) -- (0.01,5.49) -- (0.01,5.39) -- (0.01,5.29) -- (0.01,5.19) -- (0.01,5.09) -- (0.01,4.99) -- (0.01,4.89) -- (0.01,4.79) -- (0.01,4.69) -- (0.01,4.59) -- (0.01,4.49) -- (0.01,4.39) -- (0.01,4.29) -- (0.01,4.20) -- (0.00,4.10) -- (0.00,4.00);
\draw (0.00,4.00) -- (-0.00,3.95) -- (-0.00,3.90) -- (-0.00,3.85) -- (-0.01,3.80) -- (-0.01,3.76) -- (-0.01,3.71) -- (-0.01,3.66) -- (-0.01,3.61) -- (-0.01,3.56) -- (-0.01,3.51) -- (-0.01,3.46) -- (-0.01,3.41) -- (-0.01,3.36) -- (-0.01,3.31) -- (-0.01,3.26) -- (-0.01,3.21) -- (-0.01,3.16) -- (-0.01,3.11) -- (-0.00,3.05) -- (0.00,3.00);
\draw (0.00,3.00) -- (0.02,2.81) -- (0.05,2.61) -- (0.09,2.40) -- (0.13,2.20) -- (0.19,1.99) -- (0.25,1.79) -- (0.32,1.59) -- (0.40,1.39) -- (0.49,1.20) -- (0.59,1.02) -- (0.69,0.85) -- (0.80,0.69) -- (0.93,0.54) -- (1.06,0.41) -- (1.19,0.29) -- (1.34,0.19) -- (1.49,0.11) -- (1.65,0.05) -- (1.82,0.01) -- (2.00,0.00);
\draw (2.00,0.00) -- (2.26,0.02) -- (2.53,0.09) -- (2.80,0.20) -- (3.09,0.34) -- (3.37,0.50) -- (3.66,0.70) -- (3.94,0.91) -- (4.22,1.13) -- (4.49,1.36) -- (4.75,1.60) -- (4.99,1.83) -- (5.21,2.06) -- (5.41,2.28) -- (5.59,2.47) -- (5.74,2.65) -- (5.87,2.79) -- (5.96,2.91) -- (6.01,2.98) -- (6.03,3.02) -- (6.00,3.00);
\draw (6.00,3.00) -- (5.95,2.96) -- (5.89,2.89) -- (5.80,2.80) -- (5.69,2.69) -- (5.57,2.57) -- (5.44,2.43) -- (5.29,2.29) -- (5.13,2.14) -- (4.96,1.99) -- (4.79,1.83) -- (4.60,1.68) -- (4.42,1.54) -- (4.23,1.41) -- (4.04,1.28) -- (3.86,1.18) -- (3.67,1.09) -- (3.49,1.03) -- (3.32,0.99) -- (3.15,0.98) -- (3.00,1.00);
\draw (3.00,1.00) -- (2.91,1.03) -- (2.82,1.07) -- (2.74,1.13) -- (2.67,1.19) -- (2.59,1.26) -- (2.52,1.35) -- (2.46,1.44) -- (2.40,1.54) -- (2.35,1.64) -- (2.29,1.75) -- (2.25,1.87) -- (2.20,1.99) -- (2.16,2.11) -- (2.13,2.24) -- (2.10,2.37) -- (2.07,2.50) -- (2.05,2.62) -- (2.03,2.75) -- (2.01,2.88) -- (2.00,3.00);
\draw (2.00,3.00) -- (2.00,3.05) -- (1.99,3.11) -- (1.99,3.16) -- (1.99,3.21) -- (1.99,3.26) -- (1.98,3.31) -- (1.98,3.36) -- (1.98,3.42) -- (1.98,3.46) -- (1.98,3.51) -- (1.99,3.56) -- (1.99,3.61) -- (1.99,3.66) -- (1.99,3.71) -- (1.99,3.76) -- (1.99,3.81) -- (1.99,3.85) -- (2.00,3.90) -- (2.00,3.95) -- (2.00,4.00);
\draw (2.00,4.00) -- (2.00,4.10) -- (2.01,4.19) -- (2.01,4.29) -- (2.01,4.39) -- (2.01,4.49) -- (2.01,4.59) -- (2.01,4.69) -- (2.01,4.79) -- (2.01,4.89) -- (2.01,4.99) -- (2.01,5.09) -- (2.01,5.19) -- (2.01,5.29) -- (2.01,5.39) -- (2.01,5.49) -- (2.01,5.59) -- (2.00,5.70) -- (2.00,5.80) -- (2.00,5.90) -- (2.00,6.00);
\draw (4.00,6.00) -- (3.94,5.90) -- (3.88,5.81) -- (3.81,5.71) -- (3.75,5.62) -- (3.69,5.52) -- (3.63,5.43) -- (3.57,5.33) -- (3.52,5.23) -- (3.46,5.14) -- (3.41,5.04) -- (3.36,4.94) -- (3.31,4.84) -- (3.26,4.74) -- (3.21,4.64) -- (3.17,4.53) -- (3.13,4.43) -- (3.09,4.32) -- (3.06,4.22) -- (3.03,4.11) -- (3.00,4.00);
\draw (3.00,4.00) -- (2.99,3.95) -- (2.98,3.90) -- (2.97,3.85) -- (2.96,3.80) -- (2.95,3.75) -- (2.95,3.70) -- (2.94,3.65) -- (2.94,3.60) -- (2.94,3.55) -- (2.93,3.50) -- (2.93,3.45) -- (2.93,3.40) -- (2.94,3.35) -- (2.94,3.30) -- (2.95,3.25) -- (2.95,3.20) -- (2.96,3.15) -- (2.97,3.10) -- (2.99,3.05) -- (3.00,3.00);
\draw (3.00,3.00) -- (3.02,2.93) -- (3.05,2.86) -- (3.08,2.80) -- (3.12,2.73) -- (3.15,2.67) -- (3.19,2.60) -- (3.24,2.54) -- (3.28,2.48) -- (3.33,2.43) -- (3.38,2.37) -- (3.44,2.32) -- (3.49,2.27) -- (3.55,2.23) -- (3.61,2.18) -- (3.67,2.14) -- (3.73,2.11) -- (3.80,2.07) -- (3.86,2.05) -- (3.93,2.02) -- (4.00,2.00);
\draw (4.00,2.00) -- (4.11,1.98) -- (4.22,1.96) -- (4.33,1.96) -- (4.45,1.97) -- (4.56,1.98) -- (4.67,2.01) -- (4.79,2.04) -- (4.90,2.08) -- (5.01,2.13) -- (5.12,2.18) -- (5.22,2.24) -- (5.33,2.31) -- (5.43,2.38) -- (5.52,2.46) -- (5.62,2.54) -- (5.70,2.63) -- (5.79,2.72) -- (5.86,2.81) -- (5.94,2.90) -- (6.00,3.00);
\draw (6.00,3.00) -- (6.08,3.13) -- (6.14,3.27) -- (6.19,3.40) -- (6.24,3.54) -- (6.27,3.69) -- (6.29,3.83) -- (6.31,3.98) -- (6.32,4.13) -- (6.32,4.28) -- (6.31,4.43) -- (6.30,4.58) -- (6.28,4.74) -- (6.25,4.89) -- (6.23,5.05) -- (6.19,5.21) -- (6.16,5.36) -- (6.12,5.52) -- (6.08,5.68) -- (6.04,5.84) -- (6.00,6.00);
\draw (6.00,3.00) -- (6.19,3.06) -- (6.39,3.12) -- (6.58,3.17) -- (6.78,3.23) -- (6.97,3.29) -- (7.17,3.35) -- (7.36,3.40) -- (7.56,3.46) -- (7.76,3.51) -- (7.96,3.56) -- (8.16,3.61) -- (8.36,3.66) -- (8.56,3.71) -- (8.76,3.76) -- (8.96,3.81) -- (9.17,3.85) -- (9.37,3.89) -- (9.58,3.93) -- (9.79,3.97) -- (10.00,4.00);
\draw (10.00,4.00) -- (10.05,4.01) -- (10.10,4.02) -- (10.15,4.02) -- (10.21,4.03) -- (10.26,4.04) -- (10.31,4.04) -- (10.36,4.05) -- (10.41,4.05) -- (10.46,4.05) -- (10.51,4.06) -- (10.56,4.06) -- (10.61,4.06) -- (10.66,4.06) -- (10.71,4.05) -- (10.76,4.05) -- (10.81,4.04) -- (10.86,4.03) -- (10.91,4.03) -- (10.95,4.01) -- (11.00,4.00);
\draw (11.00,4.00) -- (11.06,3.98) -- (11.13,3.95) -- (11.19,3.92) -- (11.25,3.89) -- (11.31,3.85) -- (11.37,3.81) -- (11.43,3.76) -- (11.48,3.72) -- (11.53,3.67) -- (11.59,3.61) -- (11.64,3.56) -- (11.68,3.50) -- (11.73,3.44) -- (11.77,3.38) -- (11.82,3.32) -- (11.86,3.26) -- (11.90,3.20) -- (11.93,3.13) -- (11.97,3.07) -- (12.00,3.00);
\draw (12.00,3.00) -- (12.06,2.86) -- (12.11,2.72) -- (12.16,2.57) -- (12.19,2.43) -- (12.22,2.28) -- (12.24,2.13) -- (12.25,1.99) -- (12.25,1.84) -- (12.25,1.69) -- (12.25,1.54) -- (12.24,1.38) -- (12.22,1.23) -- (12.20,1.08) -- (12.18,0.93) -- (12.16,0.77) -- (12.13,0.62) -- (12.10,0.46) -- (12.07,0.31) -- (12.03,0.15) -- (12.00,0.00);
\draw (6.00,3.00) -- (10.00,0.00);
\draw (6.00,3.00) -- (8.00,0.00);
\draw (6.00,3.00) -- (6.00,0.00);
\draw (6.00,3.00) -- (4.00,0.00);
\draw (12.00,6.00) -- (6.00,3.00);
\draw (10.00,6.00) -- (6.00,3.00);
\draw (8.00,6.00) -- (6.00,3.00);
\filldraw[fill=white] (6.00,3.00) ellipse (0.80cm and 0.50cm);
\draw (0.00,6.50) node{$\pr_J$};
\draw (2.00,6.50) node{$\irr(\varsigma_{I\ssm J})$};
\draw (4.00,6.50) node{$\incl_I$};
\draw (6.00,6.50) node{$\pr_I$};
\draw (8.00,6.50) node{$R\Ind_{P_I}^G$};
\draw (10.00,6.50) node{$\inc$};
\draw (12.00,6.50) node{$\psi_I$};
\draw (6.00,3.00) node{$\theta$};
\draw (4.00,-0.50) node{$\pr_J$};
\draw (6.00,-0.50) node{$R\Ind_{P_J}^G$};
\draw (8.00,-0.50) node{$\inc$};
\draw (10.00,-0.50) node{$\psi_J$};
\draw (12.00,-0.50) node{$\Theta^{J,I}$};
\end{tikzpicture}}}
\]}\fi
Now this equality is a straightforward consequence of the definitions, Lemma~\ref{lem:zigzag-Ind-For}, Lemma~\ref{lem:delta-gamma-mate}, and the usual rules for manipulating string diagrams.
\end{proof}

Combining Lemma~\ref{lem:theta-phi-iso} and Lemma~\ref{lem:theta-phi-adj} in the special case where $J = \varnothing$ and $I = \{s\}$, we obtain the following statement, which is the main result of this section. 

\begin{thm}
\label{thm:translation}
The following diagram is a commutative diagram of adjoint pairs:
\[
\xymatrix@C=2cm{
\Dfg_{\dot B}(\bL) \ar[r]^-{\Omega_{\varnothing}} \ar@{}[d]|-{\dashv} \ar@<1ex>[d]^-{\Theta_{\varnothing,\{s\}}} 
  & \Db\Rep_{\varnothing}(G) \ar@{}[d]|-{\dashv} \ar@<1ex>[d]^-{T_{\varnothing}^{\{s\}}} \\
\Dfg_{\dot P_s}(\bL_s) \ar[r]_-{\Omega_{\{s\}}} \ar@<1ex>[u]^-{\Theta^{\varnothing,\{s\}}} 
  & \Db\Rep_{\{s\}}(G). \ar@<1ex>[u]^-{T_{\{s\}}^\varnothing}}
\]
\end{thm}

\begin{rmk}
\label{rmk:thm-translation}
It will follow from Theorem~\ref{thm:induction-thm} below that the functors $\Omega_K$ are equivalences of categories. Once this is known, the general case of Theorem~\ref{thm:translation} (for any pair $J \subset I$) will follow from Lemma~\ref{lem:equivalence-mate}.
\end{rmk}

Applying Lemma~\ref{lem:comm-adj-counit} in this special case we deduce the following corollary, which is the result we will use later in the paper.

\begin{cor}
\label{cor:isom-adjunctions-translation}
There exists an isomorphism of functors
\[
\Omega_{\varnothing} \circ \Theta^{\varnothing,\{s\}} \circ \Theta_{\varnothing,\{s\}} \simto T^{\varnothing}_{\{s\}} \circ T_{\varnothing}^{\{s\}} \circ \Omega_{\varnothing}
\]
such that for any $X$ in $\Dfg_{\dot B}(\bL)$ the following diagram commutes, where the vertical arrow is induced by our isomorphism of functors and the other arrows are induced by adjunction:
\[
\xymatrix@C=1.5cm{
\Omega_{\varnothing} \circ \Theta^{\varnothing,\{s\}} \circ \Theta_{\varnothing,\{s\}}(X) \ar[rd] \ar[d]_-{\wr} & \\
T^{\varnothing}_{\{s\}} \circ T_{\varnothing}^{\{s\}} \circ \Omega_{\varnothing}(X) \ar[r] & \Omega_{\varnothing}(X).
}
\]
\end{cor}

\begin{rmk}\label{rmk:trans-biadj}
The vertical arrows in Theorem~\ref{thm:translation} are actually \emph{biadjoint} pairs: there are also adjunctions $\Theta^{\varnothing,\{s\}} \vdash \Theta_{\varnothing,\{s\}}$ and $T_{\{s\}}^\varnothing \vdash T_\varnothing^{\{s\}}$.  This raises two questions:
\begin{enumerate}
\item \emph{Is the diagram in Theorem~{\rm \ref{thm:translation}} a commutative diagram of adjoint pairs for the adjunctions $\Theta^{\varnothing,\{s\}} \vdash \Theta_{\varnothing,\{s\}}$ and $T_{\{s\}}^\varnothing \vdash T^{\{s\}}_\varnothing$?} Concretely, consider the isomorphism $\theta^{-1}: T^{\{s\}}_\varnothing \circ \Omega_{\varnothing} \to \Omega_{\{s\}} \circ \Theta_{\varnothing,\{s\}}$.  This question asks whether the morphism
\[
(\theta^{-1})^\vee: \Omega_\varnothing \circ \Theta^{\varnothing,\{s\}} \to T_{\{s\}}^\varnothing \circ \Omega_{\{s\}}
\]
is an isomorphism.  It is difficult to answer this question with explicit string diagram calculations, mainly because it is difficult to draw a string diagram for $\theta^{-1}$.  (The problem is that the definition of $\theta$ involves morphisms, such as $\gamma$, that are not isomorphisms.)  However, we will see later that $\Omega_\varnothing$ and $\Omega_{\{s\}}$ are equivalences of categories. Lemma~\ref{lem:equivalence-mate} will then tell us that $(\theta^{-1})^\vee$ is indeed an isomorphism.
\item
\emph{Is it true that $(\theta^{-1})^\vee = \phi^{-1}$?}  Starting from Theorem~\ref{thm:translation}, there are in fact two ways to make a commutative diagram of adjoint pairs for $\Theta^{\varnothing,\{s\}} \vdash \Theta_{\varnothing,\{s\}}$ and $T_{\{s\}}^\varnothing \vdash T^{\{s\}}_\varnothing$: we can either look at $\theta^{-1}$ and $(\theta^{-1})^\vee$ as above, or at $\phi^{-1}$ and $(\phi^{-1})^\wedge$.  These are a~priori different; if they happen to coincide, then a version of Lemma~\ref{lem:comm-adj-counit} would show that there is a commutative diagram
\[
\xymatrix{
\Omega_\varnothing(M) \ar[r]^-{\eta} \ar@{=}[d] &
  \Omega_\varnothing(\Theta^{\varnothing,\{s\}}\Theta_{\varnothing,\{s\}}(M)) \ar[r]^-{\epsilon} \ar[d]_{\wr} &
  \Omega_\varnothing(M) \ar@{=}[d] \\
\Omega_\varnothing(M) \ar[r]_-{\eta} &
  T_{\{s\}}^\varnothing T^{\{s\}}_\varnothing(\Omega_\varnothing(M)) \ar[r]_-{\epsilon} &
  \Omega_\varnothing(M).}
\]
We do not know the answer to this question.  
\end{enumerate}
\end{rmk}

\section{Cotangent bundles of partial flag varieties}
\label{sec:exotic}

\subsection{Springer resolutions}
\label{ss:Springer-res}

For any $I \subset S$, we
set
\[
\tcN_I := \dot G \times^{\dot P_I} \fnt_I.
\]
This variety is endowed with a natural $\dot G$-action, and is naturally isomorphic to the cotangent bundle to $\dot G/\dot P_I$. When $I=\varnothing$ we simplify the notation to $\tcN$; in this case the variety is nothing but the usual Springer resolution of the nilpotent cone.

\begin{rmk}
If one replaces $\fnt_I$ by $(\dot \fg/\dot \fp_I)^*$ in the definition of $\tcN_I$, then
the results of the present section hold for any reductive group $\dot G$ with simply connected derived subgroup in any characteristic. (Under our assumptions, it is well known that the Killing form induces an isomorphism of $\dot P_I$-modules $\fnt_I \cong (\dot \fg/\dot \fp_I)^*$.)
\end{rmk}

We let $\Gm$ act on $\fnt_I$ by $z \cdot x = z^{-2}x$.  This induces an action on $\tcN_I$ that commutes with the left multiplication action of $\dot G$, so one can consider the category $\Coh^{\dot G \times \Gm}(\tcN_I)$.  
As in~\S\ref{ss:reminder-Koszul}, we will denote by
\[
\langle 1 \rangle : \Coh^{\dot G \times \Gm}(\tcN_I) \simto \Coh^{\dot G \times \Gm}(\tcN_I)
\]
the functor of tensoring with the tautological $\Gm$-module of dimension $1$. We will use a similar convention for all varieties endowed with a $\Gm$-action to be encountered below.

\begin{rmk}
The convention for the definition of $\langle 1 \rangle$ used in the present paper is the same as in~\cite{mr:etspc,mr:etsps}, but is opposite to the convention used in~\cite{ar:agsr}.
\end{rmk}

Throughout this section, to simplify notation we set
\[
d_I := \dim (\dot G/ \dot P_I) = \dim_{\bk}(\dot \fn_I) = |\Phi^+| - |\Phi_I^+|, \qquad n_I := |\Phi_I^+|=\dim(\dot P_I/\dot B).
\]

For any $\dot P_I$-module $V$, we denote by $\cL_{\dot G/\dot P_I}(V)$ the associated $\dot G$-equivariant vector bundle on $\dot G/\dot P_I$ (see~\cite[\S I.5.8]{jantzen}). We also denote by $\cL_{\tcN_I}(V)$ the pullback of $\cL_{\dot G/\dot P_I}(V)$ under the natural projection $\tcN_I \to \dot G/\dot P_I$. This coherent sheaf has a natural $\dot G \times \Gm$-equivariant structure. When $V=\bk_{\dot P_I}(\lambda)$ for some $\lambda \in \bX$ 
which induces a character of $\dot P_I$, we write $\cO_{\tcN_I}(\lambda)$ instead of $\cL_{\tcN_I}(\bk_{\dot P_I}(\lambda))$.

For $\lambda \in \bX_I^+ \subset \bX$ 
we denote by
\[
\dot \weyl_I(\lambda), \quad \dot \coweyl_I(\lambda), \quad \dot \irr_I(\lambda)
\]
the Weyl, dual Weyl, and simple $\dot M_I$-modules of highest weight $\lambda$, respectively. We will also consider these $\dot M_I$-modules as $\dot P_I$-modules via the surjection $\dot P_I \twoheadrightarrow \dot M_I$. Using these modules we can consider the $\dot G \times \Gm$-equivariant coherent sheaves
\[
\cL_{\tcN_I}(\dot \weyl_I(\lambda)), \quad \cL_{\tcN_I}(\dot \coweyl_I(\lambda)), \quad \cL_{\tcN_I}(\dot \irr_I(\lambda))
\]
on $\tcN_I$.

Below we will use the following lemma, whose proof can be easily adapted from the proof of~\cite[Corollary~5.9]{achar}. (Of course, in this statement $\dot \coweyl_I(\lambda)$ could have been replaced by $\dot \weyl_I(\lambda)$ or by $\dot \irr_I(\lambda)$.)

\begin{lem}
\label{lem:generators-tcNI}
The category $\Db \Coh^{\dot G \times \Gm}(\tcN_I)$ is generated, as a triangulated category, by the objects $ \cL_{\tcN_I}(\dot \coweyl_I(\lambda)) \langle i \rangle$ for $\lambda \in \bX_I^+$ and $i \in \Z$.
\end{lem}

\subsection{Induction and restriction functors}
\label{ss:Ind-Res-Springer}

If $J \subset I \subset S$, we set
\[
\tcN_{J,I} := \dot G \times^{\dot P_J} \fnt_I.
\]
For any $\dot P_J$-module $M$, as above we can consider the vector bundle $\cL_{\tcN_{J,I}}(M)$ obtained by pulling back the vector bundle $\cL_{\dot G / \dot P_J}(M)$ under the projection $\tcN_{J,I} \to \dot G / \dot P_J$. We use the same convention as above for the notation $\cO_{\tcN_{J,I}}(\lambda)$.

The inclusion map $e_{J,I}: \fnt_I \hookrightarrow \fnt_J$ induces an inclusion map
\[
\se_{J,I}: \tcN_{J,I} \hookrightarrow \tcN_J.
\]
On the other hand, there is a smooth, proper map
\[
\mu_{J,I}: \tcN_{J,I} \to \tcN_I
\]
whose fibers are isomorphic to $\dot P_I/\dot P_J$.  Define a pair of functors
\begin{align*}
\Pi_{J,I} &: \Db\Coh^{\dot G \times \Gm}(\tcN_J) \to \Db\Coh^{\dot G \times \Gm}(\tcN_I), \\
\Pi^{J,I} &: \Db\Coh^{\dot G \times \Gm}(\tcN_I) \to \Db\Coh^{\dot G \times \Gm}(\tcN_J),
\end{align*}
by
\begin{align*}
\Pi_{J,I}(\cF) & =\mu_{J,I*}\se_{J,I}^*(\cF \otimes \cO_{\tcN_J}(-\varsigma_{I\ssm J})), \\
\Pi^{J,I}(\cF) & =\se_{J,I*}\mu_{J,I}^*(\cF) \otimes \cO_{\tcN_J}(\varsigma_{I\ssm J} - 2\rho_I + 2\rho_J)\la d_I - d_J \ra.
\end{align*}
In the special case where $J = \varnothing$, we denote these functors simply by $\Pi_I$ and $\Pi^I$. When $I=\{s\}$ for some $s \in S$, we further simplify $\Pi_{\{s\}}$ and $\Pi^{\{s\}}$ to $\Pi_s$ and $\Pi^s$.

For $\lambda, \mu \in \bX_I^+$, we have
\begin{gather}
\Pi_I(\cO_{\tcN}(\mu + \varsigma_I)) \cong \cL_{\tcN_I} ( \dot\coweyl_I(\mu)), \label{eqn:pi_i} \\
\Pi^I(\cL_{\tcN_I}( \dot\coweyl_I(\lambda))) \cong \se_{\varnothing, I*}\cL_{\tcN_{\varnothing,I}}(\dot\coweyl_I(\lambda) \otimes \bk_{\dot B}(\varsigma_I - 2\rho_I) ) \la -n_I \ra.\label{eqn:pi^i}
\end{gather}
(Here~\eqref{eqn:pi^i} follows directly from the definitions, and~\eqref{eqn:pi_i} can be deduced from~\cite[I.5.18(5)]{jantzen}.)
On the other hand, if $\mu \in -\bX_I^+$, then from~\cite[II.4.2(10)]{jantzen} one can deduce that
\begin{equation}\label{eqn:pi_i-dual}
\Pi_I(\cO_\tcN(\mu + \varsigma_I - 2\rho_I)) \cong \cL_{\tcN_I} (\dot\weyl_I(w_I\mu))[-n_I].
\end{equation}

\begin{lem}
\label{lem:pi-adjoint}
The functor $\Pi_{J,I}$ has a left adjoint given by $\Pi^{J,I}\la d_I - d_J\ra[d_I - d_J]$ and a right adjoint given by $\Pi^{J,I}\la d_J - d_I\ra[d_J - d_I]$.  
\end{lem}

\begin{proof}
In this proof,
for brevity we set $r = d_I - d_J$.  The canonical bundle of $\dot P_I/\dot P_J$ is isomorphic to the line bundle corresponding to the $\dot P_J$-representation $\bigwedge^{\text{top}} (\dot\fp_I/\dot\fp_J)^* \cong \bk_{\dot P_J}(2\rho_J - 2\rho_I)$.  Since $\mu_{J,I}$ is a smooth morphism with fibers isomorphic to $\dot P_I/\dot P_J$, we have
\[
\mu_{J,I}^!(-) \cong \mu_{J,I}^*(-) \otimes_{\cO_{\tcN_{J,I}}} \cO_{\tcN_{J,I}}(2\rho_J-2\rho_I)[-r].
\]
Next, the canonical bundle of $\fnt_J$ is isomorphic to $\cO_{\fnt_J} \otimes \bk_{\dot P_J}(2\rho - 2\rho_J)\la 2d_J \ra$, and likewise for $\fnt_I$.  The map $e_{J,I}: \fnt_I \to \fnt_J$ is an inclusion of one smooth variety in another, and it follows that $e_{J,I}^!(-) \cong e_{J,I}^*(-) \otimes_{\cO_{\fnt_I}} \bigl( \cO_{\fnt_I}\otimes \bk_{\dot P_J}(2\rho_J-2\rho_I)\la 2r\ra[r] \bigr)$.  We deduce that
\begin{multline}
\label{eqn:eJI!}
\se_{J,I}^!(-) \cong \se_{J,I}^*(-) \otimes_{\cO_{\tcN_{J,I}}} \cO_{\tcN_{J,I}}(2\rho_J-2\rho_I)\la 2r\ra[r] \\
\cong \se_{J,I}^*(- \otimes_{\cO_{\tcN_J}} \cO_{\tcN_J}(2\rho_J-2\rho_I))\la 2r\ra[r].
\end{multline}

Now, the right adjoint to $\Pi_{J,I}$ is given by
\begin{multline*}
\cF \mapsto \se_{J,I*}\mu_{J,I}^!(\cF) \otimes_{\cO_{\tcN_J}} \cO_{\tcN_J}(\varsigma_{I\ssm J}) \\
\cong \se_{J,I*}\mu_{J,I}^*(\cF) \otimes_{\cO_{\tcN_J}} \cO_{\tcN_J}(\varsigma_{I\ssm J} -2\rho_I + 2\rho_J)[-r]
\cong \Pi^{J,I}(\cF)\la -r\ra[-r].
\end{multline*}
On the other hand, if we rewrite $\Pi_{J,I}$ as
\[
\Pi_{J,I}(\cF) \cong \mu_{J,I*} \se_{J,I}^!(\cF \otimes \cO_{\tcN_J}(-\varsigma_{I\ssm J} + 2\rho_I-2\rho_J))\la -2r\ra[-r],
\]
we see that its left adjoint is given by
\[
\cF \mapsto \se_{J,I*}\mu_{J,I}^*(\cF) \otimes \cO_{\tcN_J}(\varsigma_{I\ssm J} - 2\rho_I+2\rho_J)\la 2r\ra [r] \cong \Pi^{J,I}(\cF)\la r\ra[r],
\]
as desired.
\end{proof}

\begin{rmk}
\label{rmk:adjunctions-PiI}
Below we will mainly consider the case when $J=\varnothing$. In this case we have $d_\varnothing - d_I=n_I$, hence we obtain adjoint pairs $(\Pi^I \langle -n_I \rangle [-n_I], \Pi_I)$ and $(\Pi_I, \Pi^I \langle n_I \rangle [n_I])$.
\end{rmk}

\begin{lem}
\label{lem:pi-transitive}
Assume that $K \subset J \subset I$.  Then there exist natural isomorphisms
\[
\Pi_{K,I} \cong \Pi_{J,I} \circ \Pi_{K,J}
\qquad\text{and}\qquad
\Pi^{K,I} \cong \Pi^{K,J} \circ \Pi^{J,I}.
\]
\end{lem}

\begin{proof}
Let $\se': \tcN_{K,I} \to \tcN_{K,J}$ be the inclusion map induced by $e_{J,I}: \fnt_I \to \fnt_J$, and let $\mu': \tcN_{K,I} \to \tcN_{J,I}$ be the obvious map.  Consider the diagram
\[
\xymatrix{
\tcN_{K,I} \ar[r]_{\se'} \ar[d]^{\mu'} \ar@/^3ex/[rr]^{\se_{K,I}} \ar@/_3ex/[dd]_{\mu_{K,I}} & \tcN_{K,J} \ar[r]_{\se_{K,J}} \ar[d]^{\mu_{K,J}} & \tcN_K.\\
\tcN_{J,I} \ar[r]_{\se_{J,I}} \ar[d]^{\mu_{J,I}} & \tcN_J \\
\tcN_I}
\]
The square in the upper-left part of this diagram is cartesian, and the vertical maps are smooth, so there is a natural isomorphism $\se_{J,I}^* \mu_{K,J*} \cong \mu'_* (\se')^*$ (see~\cite[Proposition~A.15(3)]{mr:etspc}).  Therefore,
\begin{multline*}
\Pi_{J,I}(\Pi_{K,J}(\cF)) = \mu_{J,I*} \se_{J,I}^*(\mu_{K,J*} \se_{K,J}^*(\cF \otimes \cO_{\tcN_K}(-\varsigma_{J\ssm K})) \otimes \cO_{\tcN_J}(-\varsigma_{I \ssm J})) \\
\cong \mu_{J,I*} \se_{J,I}^*\mu_{K,J*} \se_{K,J}^*(\cF \otimes \cO_{\tcN_K}(-\varsigma_{J\ssm K}-\varsigma_{I\ssm J})) \\
\cong \mu_{J,I*} \mu'_* (\se')^* \se_{K,J}^*(\cF \otimes \cO_{\tcN_K}(-\varsigma_{I\ssm K})) \\
\cong \mu_{K,I*} \se_{K,I}^*(\cF \otimes \cO_{\tcN_K}(-\varsigma_{I\ssm K})) \cong \Pi_{K,I}(\cF).
\end{multline*}
The proof that $\Pi^{K,I} \cong \Pi^{K,J} \circ \Pi^{J,I}$ is similar.
\end{proof}

\subsection{Hom-group calculations}
\label{ss:Hom-calculations}

In this subsection we fix a subset $I \subset S$.

In the next lemma we use the standard order on $\bX$ defined by
\[
\lambda \preceq \mu \quad \Leftrightarrow \quad \mu - \lambda \in \Z_{\geq 0} \Phi^+.
\]

\begin{lem}
\phantomsection
\label{lem:tcn-pre-excep}
\begin{enumerate}
\item
\label{it:tcn-pre-excep1}
Let $\lambda, \mu \in \bX$.  If $\lambda \not\preceq \mu + 2\rho_I$, then for all $n,k \in \Z$, we have
\[
\Hom_{\Db\Coh^{\dot G \times \Gm}(\tcN)} \left( \se_{\varnothing, I*}\cO_{\tcN_{\varnothing,I}}(\mu), \cO_\tcN(\lambda)\la n\ra[k] \right) = 0.
\]
\item
\label{it:tcn-pre-excep2}
Let $\lambda \in \bX$.  We have
\begin{multline*}
\Hom_{\Db\Coh^{\dot G \times \Gm}(\tcN)} \left( \se_{\varnothing, I*}\cO_{\tcN_{\varnothing,I}}(\lambda - 2\rho_I), \cO_\tcN(\lambda)\la n\ra[k] \right) \cong 
\\
\begin{cases}
\bk & \text{if $n = 2n_I$ and $k = n_I$;} \\
0 & \text{otherwise.}
\end{cases}
\end{multline*}
\end{enumerate}
\end{lem}

\begin{proof}
In the special case where $I = \varnothing$, both of these statements are proved in~\cite[Lemma~7.10]{ar:agsr} or~\cite[Lemma~2.6]{mr:etspc}. In the general case, the coherent sheaf $\se_{\varnothing, I*}\cO_{\tcN_{\varnothing,I}}(\mu)$ admits a (Koszul) resolution by locally free coherent sheaves
\begin{equation}\label{eqn:tcni-resolution}
0 \to \cF_{n_I} \to \cF_{n_I-1} \to \cdots \to \cF_0 \to 0
\end{equation}
where
\[
\cF_i \cong \cL_{\tcN} \left( \bk_{\dot B}(\mu) \otimes \bigwedge^i (\fnt / \fnt_I)^* \right)\la 2i\ra
\]
for any $i$.
In particular, each $\cF_i$ admits a filtration whose subquotients are line bundles $\cO_\tcN(\nu)\la 2i\ra$ with $\mu \preceq \nu \preceq \mu + 2\rho_I$.

Thus, if $\lambda \not\preceq \mu + 2\rho_I$, then $\lambda \not\preceq \nu$ for all weights $\nu$ as above.  The special case $I = \varnothing$ then implies that $\Hom(\cF_i, \cO_\tcN(\lambda)\la n\ra[k]) = 0$ for all $i$, and part~\eqref{it:tcn-pre-excep1} of the lemma follows.

Suppose now that $\lambda = \mu + 2\rho_I$.  The reasoning in the previous paragraph still shows that $\Hom(\cF_i, \cO_\tcN(\lambda)\la n\ra[k]) = 0$ for $0 \le i < n_I$, and hence that there is a natural isomorphism
\[
\Hom(\cF_{n_I}[n_I],\cO_\tcN(\lambda)\la n\ra[k]) \simto \Hom(\se_{\varnothing, I*}\cO_{\tcN_{\varnothing,I}}(\lambda - 2\rho_I), \cO_\tcN(\lambda)\la n\ra[k]).
\]
Since $\cF_{n_I} \cong \cO_\tcN(\mu+2\rho_I)\la 2n_I\ra \cong \cO_\tcN(\lambda)\la 2n_I\ra$, part~\eqref{it:tcn-pre-excep2} also follows from the special case $I=\varnothing$ described above.
\end{proof}

\begin{lem}
\phantomsection\label{lem:tcn-excep}
\begin{enumerate}
\item
\label{it:tcn-excep1}
Let $\lambda, \mu \in \bX_I^+$.  If $\lambda \not\preceq \mu$, then for all $n, k \in \Z$, we have
\[
\Hom_{\Db\Coh^{\dot G \times \Gm}(\tcN_I)} \left( \cL_{\tcN_I} ( \dot\coweyl_I(\mu)), \cL_{\tcN_I} ( \dot\coweyl_I(\lambda))\la n\ra[k] \right) = 0.
\]
\item
\label{it:tcn-excep2}
Let $\lambda \in \bX_I^+$.  We have
\[
\Hom_{\Db\Coh^{\dot G \times \Gm}(\tcN_I)} \left( \cL_{\tcN_I} (\dot\coweyl_I(\lambda)), \cL_{\tcN_I} (\dot\coweyl_I(\lambda))\la n\ra[k] \right)
\cong
\begin{cases}
\bk & \text{if $n = k = 0$,} \\
0 & \text{otherwise.}
\end{cases}
\]
\end{enumerate}
\end{lem}
\begin{proof}
In the special case where $I = \varnothing$, this lemma reduces to Lemma~\ref{lem:tcn-pre-excep}, which, as we noted above, was proved in~\cite{ar:agsr,mr:etspc}.

For general $I$, using~\eqref{eqn:pi_i}, \eqref{eqn:pi^i}, and adjunction (see Remark~\ref{rmk:adjunctions-PiI}), we find that
\begin{multline*}
\Hom(\cL_{\tcN_I} ( \dot\coweyl_I(\mu)), \cL_{\tcN_I} (\dot\coweyl_I(\lambda))\la n\ra[k]) \cong \Hom(\cL_{\tcN_I} ( \dot\coweyl_I(\mu)), \Pi_I(\cO_{\tcN} (\lambda + \varsigma_I)\la n\ra[k])) \\
\cong \Hom(\se_{\varnothing, I*}\cL_{\tcN_{\varnothing,I}} (\dot\coweyl_I(\mu) \otimes \bk_{\dot B}(\varsigma_I - 2\rho_I))\la -2n_I\ra[-n_I], \cO_\tcN(\lambda + \varsigma_I)\la n\ra[k]) \\
\cong \Hom(\se_{\varnothing, I*}\cL_{\tcN_{\varnothing,I}} (\dot\coweyl_I(\mu) \otimes \bk_{\dot B}(- 2\rho_I))\la -2n_I\ra[-n_I], \cO_\tcN(\lambda)\la n\ra[k]).
\end{multline*}
The sheaf $\se_{\varnothing, I*}\cL_{\tcN_{\varnothing,I}} (\dot\coweyl_I(\mu) \otimes \bk_{\dot B}(- 2\rho_I))$ admits a filtration whose subquotients have the form $\se_{\varnothing, I*}\cO_{\tcN_{\varnothing,I}}(\nu)$ with $\nu \preceq \mu -2\rho_I$.
Thus, if $\lambda \not\preceq \mu$, then $\lambda \not\preceq \nu+2\rho_I$ for all such $\nu$.  Lemma~\ref{lem:tcn-pre-excep} then implies that $\Hom(\se_{\varnothing, I*}\cL_{\tcN_{\varnothing,I}} (\dot\coweyl_I(\mu) \otimes \bk_{\dot B}(- 2\rho_I))\la 2d_I\ra[-d_I], \cO_\tcN(\lambda)\la n\ra[k]) = 0$, so part~\eqref{it:tcn-excep1} is proved.

Suppose now that $\lambda = \mu$, and consider the surjective map 
\[
\se_{\varnothing, I*}\cL_{\tcN_{\varnothing,I}} (\dot\coweyl_I(\lambda) \otimes \bk_{\dot B}(- 2\rho_I)) \twoheadrightarrow \se_{\varnothing, I*}\cO_{\tcN_{\varnothing,I}}(\lambda - 2\rho_I).
\]
Its kernel is filtered by sheaves of the form $\se_{\varnothing, I*}\cO_{\tcN_{\varnothing,I}}(\nu)$ with $\nu \prec \lambda - 2\rho_I$, so Lemma~\ref{lem:tcn-pre-excep} implies that the induced map
\begin{multline*}
\Hom(\se_{\varnothing, I*}\cO_{\tcN_{\varnothing,I}}(\lambda - 2\rho_I)\la -2n_I\ra[-n_I], \cO_\tcN(\lambda)\la n\ra[k]) \\
\to
\Hom(\se_{\varnothing, I*}\cL_{\tcN_{\varnothing,I}} (\dot\coweyl_I(\mu) \otimes \bk_{\dot B}(- 2\rho_I))\la -2n_I\ra[-n_I], \cO_\tcN(\lambda)\la n\ra[k])
\end{multline*}
is an isomorphism.  The left-hand side is described by Lemma~\ref{lem:tcn-pre-excep}, and then part~\eqref{it:tcn-excep2} of the present lemma follows.
\end{proof}

The same arguments as in the proofs of Lemma~\ref{lem:tcn-pre-excep} and Lemma~\ref{lem:tcn-excep} allow us to deduce the following claim from~\cite[Lemma~7.10]{ar:agsr} or~\cite[Lemma~2.6]{mr:etspc}.

\begin{lem}
\label{lem:graded-finite-type}
For any $\lambda, \mu \in \bX_I^+$, the $\bk$-vector space
\[
\bigoplus_{k,n \in \Z} \Hom_{\Db\Coh^{\dot G \times \Gm}(\tcN_I)} \left( \cL_{\tcN_I} ( \dot\coweyl_I(\mu)), \cL_{\tcN_I} ( \dot\coweyl_I(\lambda))\la n\ra[k] \right)
\]
is finite-dimensional.
\end{lem}

From Lemma~\ref{lem:generators-tcNI} and Lemma~\ref{lem:graded-finite-type} we deduce in particular that the category $\Db\Coh^{\dot G \times \Gm}(\tcN_I)$ is of graded finite type in the sense of~\cite[\S 2.1.5]{bez:ctm}.

\subsection{Some orders on $\bX$}
\label{ss:orders-X}

If $\lambda \in \bX$, we denote by $w_\lambda$ the shortest element in $W t_\lambda \subset \Waff$.
Then we can define a new partial order on $\bX$ by declaring that $\lambda \le \mu$ iff $w_\lambda$ precedes $w_\mu$ in the Bruhat order on $\Waff$. The goal of this subsection is to prove some properties of this order, and explain a construction of some refinements. (These properties are well known, but we could not find any proof in the literature.)

Given $\lambda \in \bX$ and $I \subset S$, we denote by $\dom_I(\lambda)$ the unique $W_I$-translate of $\lambda$ which belongs to $\bX_I^+$. (When $I=S$, we write $\dom$ instead of $\dom_S$.) Given $w \in W$, we denote by $\min(wW_I)$, resp.~$\max(wW_I)$, the minimal, resp.~maximal, element in $wW_I$. Then we define a ``Bruhat order'' on $W/W_I$ by declaring that
\[
vW_I \leq wW_I \quad \Leftrightarrow \quad \min(vW_I) \leq \min(wW_I) \quad \Leftrightarrow \quad \max(vW_I) \leq \max(wW_I).
\]
(The equivalence between the two properties follows from~\cite[Lemma~2.2]{douglass}.) 

For $\mu \in \bX$, we denote by $\conv(\mu)$ the intersection of the convex hull of $W \mu \subset \mathbb{R} \otimes_\Z \bX$ with $\mu + \Z\Phi$, and set $\convo(\mu):=\conv(\mu) \ssm W\mu$. (This definition agrees with that in~\cite{mr:etspc}, but differs slightly from~\cite{bez:ctm}, because we take an intersection with a coset of the root lattice, rather than with the weight lattice.)  With this notation introduced, it is well known that for $\lambda, \mu \in \bX$, we have
\begin{equation}
\label{eqn:conv-order}
\lambda \in \conv(\mu) \ \Leftrightarrow \ \dom(\lambda) \preceq \dom(\mu).
\end{equation}

The first property we will need is the following.

\begin{lem}
\label{lem:order-mu-smu}
Let $\mu \in \bX$ and $s \in S$, and assume that $\mu \prec s\mu$. Then $\mu < s\mu$.
\end{lem}

\begin{proof}
Let $\nu=\dom(\mu)$, and let $I=\{t \in S \mid t(\nu)=\nu\}$. Let also $v \in W$ be the unique element such that $v=\min(v W_I)$ and $\mu=v(\nu)$. Then by~\cite[Lemmas~2.2 \& 2.4]{mr:etspc}, we have $w_\mu = t_\nu v^{-1}$, and $\ell(w_\mu) = \ell(t_\nu) - \ell(v)$. The fact that $s(\mu) \succ \mu$ implies that $\langle \nu, v^{-1}(\alpha_s^\vee) \rangle <0$, hence that $sv<v$. By a remark in~\cite[p.~86]{soergel}, this implies that $sv=\min(sv W_I)$. Using again~\cite[Lemmas~2.2 \& 2.4]{mr:etspc}, we deduce that $w_{s\mu} = t_\nu v^{-1} s = w_\mu s$ and that $\ell(w_{s\mu}) > \ell(w_\mu)$, so that indeed $s\mu > \mu$.
\end{proof}

\begin{cor}
\label{cor:order-dom-max}
Let $I \subset S$, and $\lambda, \mu \in \bX$ be such that $W_I \lambda = W_I \mu$.
\begin{enumerate}
\item
\label{it:order-dom-max}
If $\lambda \in \bX_I^+$, then $\mu \leq \lambda$.
\item
\label{it:order-antidom-min}
If $\lambda \in -\bX_I^+$, then $\mu \geq \lambda$.
\end{enumerate}
\end{cor}

\begin{proof}
We prove~\eqref{it:order-dom-max}; the proof of~\eqref{it:order-antidom-min} is completely analogous.
Let $w \in W_I$ be of minimal length such that $\mu=w\lambda$. If $w=s_1 \cdots s_r$ is a reduced decomposition, then we have
\[
\lambda \succ s_r \lambda \succ s_{r-1} s_r \lambda \succ \cdots \succ w\lambda=\mu.
\]
Hence the claim follows by a repeated application of Lemma~\ref{lem:order-mu-smu}.
\end{proof}

The following lemma can probably be proved by combinatorial arguments, but instead we rely on geometry of affine Grassmannians; for this reason we defer the proof to~\S\ref{ss:mixed-der}, where the necessary geometric background will be introduced. 

\begin{lem}
\phantomsection
\label{lem:properties-order}
\begin{enumerate}
\item
\label{it:orders-dominant}
If $\lambda, \mu \in \bX^+$, then $\lambda \leq \mu$ iff $\lambda \preceq \mu$.
\item
\label{it:orders-orbit}
Let $\lambda \in \bX^+$, and let $I=\{s \in S \mid s\lambda=\lambda\}$. Then, under the bijection
\[
\left\{
\begin{array}{ccc}
W/W_I & \simto & W\lambda \\
wW_I & \mapsto & w(\lambda)
\end{array}
\right. ,
\]
the restriction of $\leq$ to $W\lambda$ corresponds to the inverse of the Bruhat order on $W/W_I$.
\item
\label{it:order-conv}
If $\lambda \leq \mu$, then $\lambda \in \conv(\mu)$.
\end{enumerate}
\end{lem}

\begin{rmk}
It is asserted without proof in~\cite[p.~340]{bez:ctm} (and then subsequently in~\cite{mr:etspc}) that the orders $\leq$ and $\preceq$ coincide on each $W$-orbit in $\bX$. However, comparing Lemma~\ref{lem:properties-order}\eqref{it:orders-orbit} with~\cite[Theorem~1.1]{deodhar}, we see that this claim is false in general.
\end{rmk}

From these properties, we deduce in particular the following fact.

\begin{lem}
\label{lem:order-domI}
Let $\lambda, \mu \in \bX$ and $I \subset S$.
If $\mu \in W\lambda$ and $\mu \leq \lambda$, then $\dom_I(\mu) \leq \dom_I(\lambda)$.
\end{lem}

\begin{proof}
Let $\nu=\dom(\lambda)=\dom(\mu)$, and let $K:=\{s \in S \mid s(\nu)=\nu\}$. Then as in Lemma~\ref{lem:properties-order}\eqref{it:orders-orbit} we have a natural bijection $W/W_K \simto W\nu$. Write $\lambda=v_1 (\nu)$ and $\mu = v_2 (\nu)$, where $v_1=\min(v_1 W_K)$ and $v_2 = \min(v_2 W_K)$. Then, by Lemma~\ref{lem:properties-order}\eqref{it:orders-orbit}, the fact that $\mu \leq \lambda$ translates into the fact that $v_1 \leq v_2$. 

Now, let $v_1'$ be the minimal element in the double coset $W_I v_1 W_K$. Then $v_1'(\nu) \in W_I \lambda$. Now for any $s \in S$ we have $sv_1' > v_1'$, which implies that $\langle v_1'(\nu), \alpha_s^\vee \rangle \geq 0$.
Since this holds for any $s \in I$, this proves that $v_1'(\nu) \in \bX_I^+$, and finally that $\dom_I(\lambda) = v_1'(\nu)$. Moreover we clearly have $v_1'=\min(v_1' W_K)$. Similarly we have $\dom_I(\mu)=v_2'(\nu)$, where $v_2'$ is the minimal element in $W_I v_2 W_K$, and $v_2'=\min(v_2' W_K)$.

We can finally conclude. Since $v_1 \leq v_2$, by~\cite[Lemma~2.2]{douglass} we have $v_1' \leq v_2'$. By Lemma~\ref{lem:properties-order}\eqref{it:orders-orbit}, this implies that $v_2'(\nu) \leq v_1'(\nu)$, hence that $\dom_I(\mu) \leq \dom_I(\lambda)$, as stated.
\end{proof}

Below we will consider refinements $\leq'$ of the order $\leq$. We will usually require that these refinements satisfy the following property:
\begin{equation}
\label{eqn:condition-order}
\lambda \in \convo(\mu) \quad \Rightarrow \quad \lambda \leq' \mu.
\end{equation}

In the rest of this subsection we explain how one can construct explicitly a refinement of $\leq$ satisfying~\eqref{eqn:condition-order} and some extra useful properties related to a choice of a subset $I \subset S$. 
More precisely, let us choose
\begin{itemize}
\item a total order $\le_1$ on $\bX^+$ that refines the order $\preceq$ (or equivalently the order $\leq$, see Lemma~\ref{lem:properties-order}\eqref{it:orders-dominant}) and makes $(\bX^+,\leq_1)$ isomorphic to $(\Z_{\geq 0}, \leq)$;
\item for each $W$-orbit of weights $W\lambda$, a total order $\le_2$ on the set $W\lambda \cap \bX_I^+$ that refines the partial order induced by $\leq$; and
\item for each $\lambda \in \bX_I^+$, a total order $\le_3$ on $W_I\lambda$ that refines the partial order $\leq$.
\end{itemize}
Then we define a total order $\le'$ on $\bX$ by setting
\[
\lambda \le' \mu
\qquad\text{iff}\qquad
\begin{cases}
\text{$\dom(\lambda) <_1 \dom(\mu)$, or} \\
\text{$\dom(\lambda) = \dom(\mu)$ and $\dom_I(\lambda) <_2 \dom_I(\mu)$, or} \\
\text{$W_I\lambda = W_I\mu$ and $\lambda \le_3 \mu$.}
\end{cases}
\]
Clearly, the ordered set $(\bX, \leq')$ is isomorphic to $(\Z_{\geq 0}, \leq)$. A fortiori, the same property holds for $(\bX_I^+, \leq')$.

\begin{lem}
\label{lem:order-leq'}
The order $\leq'$ refines $\leq$ and satisfies~\eqref{eqn:condition-order}.
\end{lem}

\begin{proof}
First,~\eqref{eqn:condition-order} is satisfied because if $\lambda \in \convo(\mu)$ then $\dom(\lambda) \in \conv(\dom(\mu)) \setminus \{\dom(\mu)\}$, so that $\dom(\lambda) <_1 \dom(\mu)$ by~\eqref{eqn:conv-order} and our choice of order $\leq_1$, and then $\lambda \leq' \mu$ by construction of $\leq'$.

Now assume that $\lambda \leq \mu$. Then by Lemma~\ref{lem:properties-order}\eqref{it:order-conv} we have $\lambda \in \conv(\mu)$. If $\lambda \in \convo(\mu)$ then as seen above $\lambda \leq' \mu$. Otherwise we have $\lambda \in W\mu$. 
By Lemma~\ref{lem:order-domI}, since $\lambda \leq \mu$ we have $\dom_I(\lambda) \leq \dom_I(\mu)$. If $\dom_I(\lambda) < \dom_I(\mu)$ then $\dom_I(\lambda) <_2 \dom_I(\mu)$, hence $\lambda \leq' \mu$. Otherwise we have $\dom_I(\lambda)=\dom_I(\mu)$, hence $\lambda \leq_3 \mu$ and again $\lambda \leq' \mu$.
\end{proof}

It is clear that this order also satisfies the following properties:
\begin{gather}
\label{eqn:order-leq'-domI}
\mu \leq' \lambda \quad \Rightarrow \quad \dom_I(\mu) \leq' \dom_I(\lambda); \\
\label{eqn:order-leq'-WI}
\text{if $\mu <' \lambda$ and $W_I \lambda \neq W_I \mu$, then $v_1\mu <' v_2 \lambda$ for all $v_1,v_2 \in W_I$}.
\end{gather}

\subsection{Standard and costandard exotic sheaves}
\label{ss:standard-costandard-exotic}

In this subsection again we fix a subset $I \subset S$, and we
let $\bXpp_I \subset \bX_I^+$ be the set of regular dominant weights for $M_I$:
\[
\bXpp_I = \{ \lambda \in \bX \mid \text{$\la \alpha_s^\vee, \lambda\ra > 0$ for all $s \in I$} \}.
\]
We clearly have
\[
\bXpp_I = \bX_I^+ + \varsigma_I.
\]

For $\lambda \in \bXpp_I$, we define
\[
\Db\Coh^{\dot G \times \Gm}(\tcN_I)_{\le \lambda}
\]
to be the full triangulated subcategory of $\Db\Coh^{\dot G \times \Gm}(\tcN_I)$ generated by objects of the form $\cL_{\tcN_I} (\dot\coweyl_I(\mu-\varsigma_I))\la n\ra$ with $\mu \in \bXpp_I$, $\mu \le \lambda$ and $n \in \Z$.  The subcategory
\[
\Db\Coh^{\dot G \times \Gm}(\tcN_I)_{< \lambda}
\]
is defined similarly.
If $\le'$ is a partial order refining $\le$, we can likewise define the subcategories $\Db\Coh^{\dot G \times \Gm}(\tcN_I)_{\le' \lambda}$ and $\Db\Coh^{\dot G \times \Gm}(\tcN_I)_{<' \lambda}$.

In the next statement we denote by $\delta_\lambda$ the minimal length of an element $v \in W$ such that $v(\lambda)$ is dominant.

\begin{prop}
\label{prop:exotic-defn}
Choose a total order $\le'$ on $\bX$ that refines $\le$, makes $(\bX,\le')$ isomorphic to $(\Z_{\geq 0}, \leq)$, and which satisfies~\eqref{eqn:condition-order}.

For each $\lambda \in \bXpp_I$, there exist objects $\nabla_I(\lambda)$, $\Delta_I(\lambda) \in \Db\Coh^{\dot G \times \Gm}(\tcN_I)$ that are uniquely determined (up to isomorphism) by the following two properties:
\begin{enumerate}
\item
there exist distinguished triangles
\begin{gather}
\cF \to \cL_{\tcN_I} (\dot\coweyl_I(\lambda-\varsigma_I)) \la -\delta_\lambda -n_I\ra[-n_I] \to \nabla_I(\lambda) \xrightarrow{[1]}, \label{eqn:exotic-nabla-defn}\\
\Delta_I(\lambda) \to \cL_{\tcN_I} ( \dot\coweyl_I(\lambda-\varsigma_I))\la -\delta_\lambda -n_I\ra[-n_I] \to \cF' \xrightarrow{[1]} \label{eqn:exotic-delta-defn}
\end{gather}
with $\cF,\cF' \in \Db\Coh^{\dot G \times \Gm}(\tcN_I)_{<' \lambda}$;
\item
we have
\[
\Hom(\cG, \nabla_I(\lambda)) = \Hom(\Delta_I(\lambda),\cG) = 0
\qquad\text{for all}\qquad
\cG \in \Db\Coh^{\dot G \times \Gm}(\tcN_I)_{<' \lambda}.
\]
\end{enumerate}
\end{prop}

\begin{proof}
Lemma~\ref{lem:tcn-excep} guarantees that the objects $\cL_{\tcN_I} (\dot\coweyl_I(\lambda-\varsigma_I)) \la -\delta_\lambda -n_I\ra[-n_I]$ for $\lambda \in \bXpp_I$ form a graded exceptional sequence with respect to the partial order $\preceq$, in the sense of~\cite[\S 2.1.5]{bez:ctm} (see also~\cite[\S 8.1]{ar:agsr} or~\cite[\S 2.3]{mr:etspc}). The objects $\nabla_I(\lambda)$ are obtained by taking the $\leq'$-mutation of this exceptional sequence, as in~\cite[Lemma~3]{bez:ctm}, and the objects $\Delta_I(\lambda)$ form the dual graded exceptional sequence, as in~\cite[Proposition~3]{bez:ctm}.
\end{proof}

\begin{rmk}
\phantomsection\label{rmk:exotic-other-defn}
\begin{enumerate}
\item
The assumption that $\leq'$ satisfies~\eqref{eqn:condition-order} is not necessary in Proposition~\ref{prop:exotic-defn}. However this property is used in the proof of certain properties of the objects $\nabla_I(\lambda)$ and $\Delta_I(\lambda)$ considered below.
\item
\label{it:rmk-def-nabla-weyl}
Let $\lambda \in \bXpp_I$, and let $\nabla_I'(\lambda)$ be an object such that there exists a distinguished triangle
\[
\cG \to \cL_{\tcN_I} (\dot\weyl_I(\lambda-\varsigma_I)) \la -\delta_\lambda -n_I\ra[-n_I] \xrightarrow{f} \nabla'_I(\lambda) 
\]
with $\cG \in \Db\Coh^{\dot G \times \Gm}(\tcN_I)_{<' \lambda}$ and such that
\begin{equation}
\label{eqn:nabla'-orthogonal}
\Hom(\mathcal{H}, \nabla'_I(\lambda)) = 0
\qquad\text{for all}\qquad
\mathcal{H} \in \Db\Coh^{\dot G \times \Gm}(\tcN_I)_{<' \lambda}.
\end{equation}
Then there exists an isomorphism $\nabla_I'(\lambda) \cong \nabla_I(\lambda)$. Indeed, since the cone of the natural morphism $\cL_{\tcN_I} (\dot\weyl_I(\lambda-\varsigma_I)) \to \cL_{\tcN_I} (\dot\coweyl_I(\lambda-\varsigma_I))$ belongs to $\Db\Coh^{\dot G \times \Gm}(\tcN_I)_{<' \lambda}$ (see property~\eqref{eqn:condition-order}),~\eqref{eqn:nabla'-orthogonal} implies that the morphism $f$ factors through a morphism $g : \cL_{\tcN_I} (\dot\coweyl_I(\lambda-\varsigma_I)) \to \nabla'_I(\lambda)$. And
an easy argument with the octahedral axiom shows that the cone of $g$ belongs to  $\Db\Coh^{\dot G \times \Gm}(\tcN_I)_{<' \lambda}$, so that $\nabla'_I(\lambda)$ satisfies the properties which characterize $\nabla_I(\lambda)$.
Of course, similar comments apply to the objects $\Delta_I(\lambda)$.
\end{enumerate}
\end{rmk}

The following important property follows from the general theory of (graded) exceptional sequences (see~\cite[\S 2.1.5]{bez:ctm}).

\begin{cor}\label{cor:exotic-orth}
For any order $\le'$ as in Proposition~{\rm \ref{prop:exotic-defn}}, we have
\[
\Hom(\Delta_I(\mu),\nabla_I(\lambda)\la n\ra [k]) \cong
\begin{cases}
\bk & \text{if $\mu = \lambda$ and $n = k = 0$;} \\
0 & \text{otherwise.}
\end{cases}
\]
\end{cor}

\subsection{Study of the case $I=\varnothing$}
\label{ss:reminder-exotic-empty}

In the special case where $I = \varnothing$, we omit the subscripts and simply write
\[
\nabla(\lambda) = \nabla_\varnothing(\lambda),
\qquad
\Delta(\lambda) = \Delta_\varnothing(\lambda).
\]
In this case, these objects have been studied extensively in~\cite{bez:ctm, ar:agsr, mr:etspc}. (Our normalization of these objects follows the conventions in~\cite{ar:agsr, mr:etspc} but is slightly different from those of~\cite{bez:ctm}, where the shift $\langle -\delta_\lambda \rangle$ is omitted.) The proposition below summarizes the main properties we will need.  This statement mentions 
the category
\[
\Db\Coh^{\dot G \times \Gm}(\tcN)_{\convo(\lambda)},
\]
defined as the full triangulated subcategory of $\Db\Coh^{\dot G \times \Gm}(\tcN)$ generated by the objects $\cO_\tcN(\mu)\la n\ra$ with $\mu \in \convo(\lambda)$.

\begin{prop}
\label{prop:exotic-recall}
Let $\lambda \in \bX$, and let $s \in S$.
\begin{enumerate}
\item 
\label{it:exotic-recall-1}
The objects $\nabla(\lambda)$ and $\Delta(\lambda)$ are independent of the choice of order $\le'$ as in Proposition~{\rm \ref{prop:exotic-defn}}.
\item
\label{it:nabla-convo}
In the distinguished triangle
\[
\cF \to \cO_\tcN(\lambda)\la -\delta_\lambda\ra \to \nabla(\lambda) \xrightarrow{[1]},
\]
we have $\cF \in \Db\Coh^{\dot G \times \Gm}(\tcN)_{\convo(\lambda)}$.
\item
\label{it:pi-s-null}
If $s\lambda = \lambda$, then $\Pi_s(\nabla(\lambda)) = \Pi_s(\Delta(\lambda)) = 0$.
\item
\label{it:pi-s-std}
If $s\lambda \prec \lambda$, there exist distinguished triangles
\begin{gather*}
\nabla(s\lambda)\la -1\ra[-1] \to \Pi^s\Pi_s(\nabla(\lambda))\la -1\ra[-1] \xrightarrow{\epsilon} \nabla(\lambda) \xrightarrow{[1]}, \\
\Delta(\lambda)\la -1\ra[-1] \to \Pi^s\Pi_s(\Delta(s\lambda))\la -1\ra[-1] \xrightarrow{\epsilon} \Delta(s\lambda) \xrightarrow{[1]},
\end{gather*}
where the second morphism in both triangles is the counit for the adjunction $\Pi^s\la -1\ra[-1] \dashv \Pi_s$.
\item
\label{it:pi-s-push}
If $s\lambda \prec \lambda$, then there exist isomorphisms
\[
\Pi_s(\nabla(s\lambda)) \cong \Pi_s(\nabla(\lambda))\la -1\ra[-1] \\
\qquad\text{and}\qquad
\Pi_s(\Delta(s\lambda)) \cong \Pi_s(\Delta(\lambda))\la 1\ra [1].
\]
\end{enumerate}
\end{prop}

\begin{proof}
Part~\eqref{it:exotic-recall-1} is proved in~\cite[Proposition~8.5(1)]{ar:agsr} or~\cite[Remark~3.5]{mr:etspc}, and part~\eqref{it:nabla-convo} follows from~\cite[Lemma~3.1(3)--(4)]{mr:etspc} and the proof of~\cite[Proposition~3.8]{mr:etspc}.

Let us now prove part~\eqref{it:pi-s-null}. Standard arguments (involving in particular the base change theorem) show that the functor $\Pi^s \circ \Pi_s$ is isomorphic to the Fourier--Mukai transform associated with the kernel $\cO_{Y_s}(-\varsigma_s, \varsigma_s - \alpha_s) \langle -1 \rangle$, where $Y_s$ is the subvariety of $\tcN \times \tcN$ considered in~\cite[\S 3.1]{mr:etspc} (and where we follow the notational conventions of~\cite{mr:etspc}). Hence, using~\cite[Proposition~3.3(2)]{mr:etspc} and the exact sequence~\cite[(3.2)]{mr:etspc} (in which $\rho$ can be replaced by $\varsigma_s$; see~\cite[Lemma~1.5.1]{riche-RT}), we obtain that if $s\lambda=\lambda$ there exists a distinguished triangle
\begin{equation}
\label{eqn:triangle-exotic-recall}
\nabla(\lambda) \langle 1 \rangle \to \nabla(\lambda) \langle 1 \rangle \to \Pi^s \Pi_s(\nabla(\lambda)) \xrightarrow{[1]}.
\end{equation}
We have $\Hom_{\Db \Coh^{\dot G \times \Gm}(\tcN)}(\nabla(\lambda), \nabla(\lambda))=\bk$, hence the first morphism in this triangle is either $0$ or an isomorphism. If it is zero, then $\Pi^s \Pi_s(\nabla(\lambda))$ is isomorphic to $\nabla(\lambda) \langle 1 \rangle \oplus \nabla(\lambda) \langle 1 \rangle[1]$. This is absurd since $\nabla(\lambda)$ has a nontrivial restriction to the inverse image of the regular orbit in the nilpotent cone (as follows e.g.~from the proof of~\cite[Proposition~3.8]{mr:etspc}), while $\Pi^s(\cF)$ has a trivial restriction to this open subvariety for any $\cF$ in $\Db \Coh^{\dot G \times \Gm}(\tcN_s)$.

We have proved that the first arrow in~\eqref{eqn:triangle-exotic-recall} is an isomorphism. Hence we have $\Pi^s \Pi_s(\nabla(\lambda))=0$. But the functor $\Pi^s$ does not kill any nonzero object, since it is a composition of a smooth pullback with a pushforward under a closed embedding. Hence indeed we have $\Pi_s(\nabla(\lambda))=0$. The proof of the fact that $\Pi_s(\Delta(\lambda))=0$ is similar, using~\cite[Proposition~3.6(1)]{mr:etspc} as the starting point.

We now consider part~\eqref{it:pi-s-std}.
As above, from~\cite[Proposition~3.3(3)]{mr:etspc} and the exact sequence~\cite[(3.2)]{mr:etspc} we deduce that there exists a distinguished triangle
\begin{equation}
\label{eqn:triangle-exotic-recall-3}
\nabla(s\lambda) \langle -1 \rangle [-1] \to \Pi^s \Pi_s(\nabla(\lambda)) \langle -1 \rangle [-1] \to \nabla(\lambda) \xrightarrow{[1]}.
\end{equation}
The second arrow in this triangle is nonzero, since otherwise $\nabla(s \lambda)$ would be decomposable, which would contradict the fact that $\Hom_{\Db \Coh^{\dot G \times \Gm}(\tcN)}(\nabla(s\lambda), \nabla(s\lambda))=\bk$. Hence to conclude the proof in this case, we just need to prove that
\begin{equation}
\label{eqn:dim-Hom}
\dim_\bk \bigl( \Hom_{\Db \Coh^{\dot G \times \Gm}(\tcN)}(\Pi^s \Pi_s(\nabla(\lambda)) \langle -1 \rangle [-1], \nabla(\lambda)) \bigr)=1.
\end{equation}
(Indeed, this will also prove that $\Pi_s(\nabla(\lambda))$ is nonzero, and hence that the morphism induced by adjunction forms a basis of this $1$-dimensional vector space.)

Since $s\lambda \prec \lambda$, by Lemma~\ref{lem:order-mu-smu} we have $s\lambda < \lambda$, and hence
\[
\Hom_{\Db \Coh^{\dot G \times \Gm}(\tcN)}(\nabla(s\lambda), \nabla(\lambda)\la n \ra [k])=0
\]
for any $k,n \in \Z$, by definition of a (graded) exceptional sequence. Thus, using the long exact sequence obtained by applying the functor $\Hom_{\Db \Coh^{\dot G \times \Gm}(\tcN)}(-, \nabla(\lambda))$ to the triangle~\eqref{eqn:triangle-exotic-recall-3} we obtain an isomorphism
\[
\Hom_{\Db \Coh^{\dot G \times \Gm}(\tcN)}(\nabla(\lambda), \nabla(\lambda)) \cong \Hom_{\Db \Coh^{\dot G \times \Gm}(\tcN)}(\Pi^s \Pi_s(\nabla(\lambda)) \langle -1 \rangle [-1], \nabla(\lambda)),
\]
which implies~\eqref{eqn:dim-Hom} and finishes the proof in this case.

The case of the objects $\Delta(\lambda)$ and $\Delta(s\lambda)$ is very similar (using~\cite[Proposition~3.6(1)]{mr:etspc}), and left to the reader.

Finally, we consider part~\eqref{it:pi-s-push}. By~\eqref{it:exotic-recall-1}, we can assume that the order $\leq'$ has been chosen as in~\S\ref{ss:orders-X}, in terms of the subset $I=\{s\}$. Under this assumption, we will also consider the objects $\nabla_{\{s\}}(\lambda)$ and $\Delta_{\{s\}}(\lambda)$ (constructed from the same order), and we will prove more precisely that
\begin{gather}
\label{eqn:Pis-nabla}
\Pi_s(\nabla(s\lambda)) \cong \Pi_s(\nabla(\lambda))\la -1\ra[-1] \cong \nabla_{\{s\}}(\lambda), \\
\label{eqn:Pis-Delta}
\Pi_s(\Delta(s\lambda)) \cong \Pi_s(\Delta(\lambda))\la 1\ra [1] \cong \Delta_{\{s\}}(\lambda).
\end{gather}

First we prove~\eqref{eqn:Pis-nabla}.
For $\mu \in \bX$, using~\eqref{eqn:pi_i} and~\eqref{eqn:pi_i-dual} we see that
\[
\Pi_s(\cO_{\tcN}(\mu)) = \begin{cases}
0 & \text{if $s\mu=\mu$;} \\
\cL_{\tcN_s}(\dot\coweyl_s(\mu-\varsigma_s)) & \text{if $s\mu \prec \mu$;} \\
\cL_{\tcN_s}(\dot\weyl_s(s\mu - \varsigma_s))[-1] & \text{if $s\mu \succ \mu$.}
\end{cases}
\]
Note that if $\mu \in \convo(\lambda)=\convo(s\lambda)$, then $\mu <' \lambda$ and $s\mu <' \lambda$ (see~\eqref{eqn:condition-order}).
Hence, using these isomorphisms, we see that applying $\Pi_s$ to the distinguished triangle in~\eqref{it:nabla-convo} for both $\lambda$ and $s\lambda$, we
obtain distinguished triangles
\begin{gather*}
\cG \to \cL_{\tcN_s}(\dot\coweyl_s(\lambda-\varsigma_s)) \langle -\delta_\lambda - 1 \rangle [-1] \to \Pi_s(\nabla(\lambda)) \langle -1 \rangle [-1], \\
\cG' \to \cL_{\tcN_s}(\dot\weyl_s(\lambda-\varsigma_s)) \langle -\delta_\lambda - 1 \rangle [-1] \to \Pi_s(\nabla(s\lambda)),
\end{gather*}
where $\cG$ and $\cG'$ belong to $\Db\Coh^{\dot G \times \Gm}(\tcN_I)_{<'\lambda}$. Using also Remark~\ref{rmk:exotic-other-defn}\eqref{it:rmk-def-nabla-weyl}, we see that to conclude the proof of the isomorphisms in this case, it suffices to prove that
\begin{multline*}
\Hom(\cL_{\tcN_s}(\dot\coweyl_s(\mu-\varsigma_s)), \Pi_s(\nabla(\lambda))\langle n \rangle [k])\\
=\Hom(\cL_{\tcN_s}(\dot\coweyl_s(\mu-\varsigma_s)), \Pi_s(\nabla(s\lambda))\langle n \rangle [k])=0
\end{multline*}
for any $\mu \in \bXpp_s$ such that $\mu <' \lambda$. And in turn, since $s\lambda <' \lambda$ (see Lemma~\ref{lem:order-mu-smu}), using adjunction (see Lemma~\ref{lem:pi-adjoint}), to prove this it suffices to prove that
\begin{equation}
\Pi^s(\cL_{\tcN_s}(\dot\coweyl_s(\mu-\varsigma_s))) \in \Db\Coh^{\dot G \times \Gm}(\tcN)_{<'s\lambda}.
\end{equation}
Now, as in the proof of Lemma~\ref{lem:tcn-excep}, the object $\Pi^s(\cL_{\tcN_s}(\dot\coweyl_s(\mu-\varsigma_s)))$ admits a filtration with subquotients of the form $\se_{\varnothing, \{s\}*} \cO_{\tcN_{\varnothing, \{s\}}}(\nu) \langle 1 \rangle$ with $\nu \in \{\mu-\alpha_s, \cdots, s\mu\}$. And as in the proof of Lemma~\ref{lem:tcn-pre-excep}, for any such $\nu$ there exists an exact sequence
\[
\cO_{\tcN}(\nu + \alpha_s) \langle 2 \rangle \hookrightarrow \cO_{\tcN}(\nu) \twoheadrightarrow \se_{\varnothing, \{s\}*} \cO_{\tcN_{\varnothing, \{s\}}}(\nu).
\]
Hence to conclude it suffices to prove that any for weight $\eta$ in $\{\mu, \mu - \alpha_s, \cdots, s\mu\}$ we have $\eta <' s\lambda$. However, these weights satisfy $\eta \leq' \mu$, and since $\mu \notin \{\lambda, s\lambda\}$,~\eqref{eqn:order-leq'-WI} ensures that $\mu <' s\lambda$, so that indeed $\eta <' s\lambda$. This finishes the proof of~\eqref{eqn:Pis-nabla}.

Finally we deduce~\eqref{eqn:Pis-Delta}. For this we note, using~\eqref{eqn:Pis-nabla} and Lemma~\ref{lem:pi-adjoint}, that for any $\mu \in \bXpp_s$ and $n,k \in \Z$ we have
\[
\Hom(\Pi_s(\Delta(\lambda)), \nabla_s(\mu) \langle n \rangle [k]) \cong \Hom(\Delta(\lambda), \Pi^s \Pi_s(\nabla(\mu)) \langle n \rangle [k]).
\]
Then, using~\eqref{it:pi-s-std} we deduce that this vector space vanishes unless $\mu=\lambda$ and $n=k=-1$. Using~\cite[Lemma~2.5]{mr:etspc}, this proves that $\Pi_s(\Delta(\lambda)) \cong \Delta_{\{s\}}(\lambda) \langle -1 \rangle [-1]$. One can prove by similar arguments that $\Pi_s(\Delta(s\lambda)) \cong \Delta_{\{s\}}(\lambda)$, and the proof of~\eqref{eqn:Pis-Delta} is then complete.
\end{proof}

\begin{rmk}
The analogue of Proposition~\ref{prop:exotic-recall}\eqref{it:nabla-convo} for the objects $\Delta(\lambda)$ does \emph{not} hold: the cone of the morphism $\Delta(\lambda) \to \cO_{\tcN}(\lambda) \langle -\delta_\lambda \rangle$ does not belong to the subcategory $\Db\Coh^{\dot G \times \Gm}(\tcN)_{\convo(\lambda)}$ in general. This is one of the subtle differences between the objects $\Delta(\lambda)$ and the objects $\nabla(\lambda)$.
\end{rmk}

Now we return to the case of a general subset $I \subset S$.
From Proposition~\ref{prop:exotic-recall} we deduce the following fact.

\begin{cor}
\label{cor:PiI-Delta-vanishing}
Let $\lambda \in \bX$, and assume that $\lambda \notin W_I \bXpp_I$ (i.e.~that the stabilizer of $\lambda$ in $W_I$ is nontrivial). Then
\[
\Pi_I(\Delta(\lambda)) = \Pi_I(\nabla(\lambda))=0.
\]
\end{cor}

\begin{proof}
We prove that $\Pi_I(\Delta(\lambda)) =0$; the case of $\nabla(\lambda)$ is similar.

First, let us assume that $\lambda \in \bX_I^+$. Then since $\lambda \notin \bXpp_I$, there exists $s \in I$ such that $s\lambda=\lambda$. Using Lemma~\ref{lem:pi-transitive}, we obtain that
\[
\Pi_I(\Delta(\lambda)) = \Pi_{\{s\},I} \circ \Pi_s(\Delta(\lambda))=0
\]
by Proposition~\ref{prop:exotic-recall}\eqref{it:pi-s-null}.

Now we consider the general case. Let $\mu = \dom_I(\lambda)$, and let $v \in W_I$ be the element of minimal length such that $\lambda = v\mu$. Let $v=s_1 \cdots s_r$ be a reduced decomposition of $v$. Then we have
\[
\mu \succ s_r \mu \succ s_{r-1} s_r \mu \succ \cdots \succ s_1 \lambda \succ \lambda.
\]
Decomposing $\Pi_I$ as $\Pi_{\{s_k\},I} \circ \Pi_{s_k}$ for $k \in \{1, \cdots, r\}$ and using Proposition~\ref{prop:exotic-recall}\eqref{it:pi-s-push} repeatedly, we obtain that
\[
\Pi_I(\Delta(\mu)) \cong \Pi_I(\Delta(s_r \mu)) \langle -1 \rangle [-1] \cong \cdots \cong \Pi_I(\Delta(\lambda)) \langle -r \rangle [-r].
\]
By the case of $I$-dominant weights considered above we have $\Pi_I(\Delta(\mu))=0$, and hence $\Pi_I(\Delta(\lambda))=0$ as well.
\end{proof}

\subsection{Standard and costandard exotic sheaves and induction/restriction functors}
\label{ss:standard-costandard-functors}

In this subsection we fix a subset $I \subset S$ with $I \neq \varnothing$, and we assume that the objects $\nabla_I(\lambda)$ and $\Delta_I(\lambda)$ are defined with respect to an order $\leq'$ constructed as in~\S\ref{ss:orders-X} (which is authorized by Lemma~\ref{lem:order-leq'}).

Below we will need the following lemma on weights. Here, for any $X \subset \bX$, we denote by $\mathsf{Conv}(X)$ the convex hull of $X$ (in $\mathbb{R} \otimes_\Z \bX$).

\begin{lem}
\label{lem:weights}
Let $\lambda \in \bXpp_I$, and let $Y \subset \Phi_I^+$. Then the weight $\lambda - \sum_{\alpha \in Y} \alpha$ belongs to $\mathsf{Conv}(W_I \lambda) \cap (\lambda + \Z\Phi_I)$.
\end{lem}

\begin{proof}
If $\preceq_I$ is the order on $\bX$ defined by $\lambda \preceq_I \mu$ iff $\mu - \lambda \in \Z_{\geq 0} \Phi_I^+$, then
it is well known that a weight $\mu \in \bX$ belongs to $\mathsf{Conv}(W_I \lambda)  \cap (\lambda + \Z\Phi_I)$ iff $w(\mu) \preceq_I \lambda$ for any $w \in W_I$.
Hence it suffices to prove that our weight $\lambda - \sum_{\alpha \in Y} \alpha$ satisfies this condition. For this we will work in $\frac{1}{2}\bX$; we extend the order $\preceq_I$ to this lattice by using the same rule as above.

For any $w \in W_I$ we have
\[
w \left( \lambda - \sum_{\alpha \in Y} \alpha \right) = w(\lambda - \rho_I) + w \left( \rho_I - \sum_{\alpha \in Y} \alpha \right).
\]
Since $\langle \lambda - \rho_I, \alpha^\vee \rangle \in \Z_{\geq 0}$ for any $\alpha \in \Phi_I^+$, we have $w(\lambda- \rho_I) \preceq_I \lambda - \rho_I$. Hence to prove the lemma it suffices to prove that
\[
w \left( \rho_I - \sum_{\alpha \in Y} \alpha \right) \preceq_I \rho_I.
\]
However we have
\[
\rho_I - \sum_{\alpha \in Y} \alpha = \frac{1}{2} \sum_{\alpha \in (\Phi_I^+ \setminus Y) \sqcup (-Y)} \alpha.
\]
The subset $(\Phi_I^+ \setminus Y) \sqcup (-Y)$ contains one representative for each pair of opposite roots in $\Phi_I$. Hence the same property holds for its image under $w$. In other words,
there exists $Z \subset \Phi_I^+$ such that
\[
w \left( \rho_I - \sum_{\alpha \in Y} \alpha \right) = \frac{1}{2} \sum_{\alpha \in (\Phi_I^+ \setminus Z) \sqcup (-Z)} \alpha = \rho_I - \sum_{\alpha \in Z} \alpha \preceq_I \rho_I,
\]
which finishes the proof.
\end{proof}

\begin{lem}
\label{lem:pi-i-pull}
Let $\lambda \in \bXpp_I$.  We have
\[
\Pi^I(\Db\Coh^{\dot G \times \Gm}(\tcN_I)_{\le' \lambda}) \subset \Db\Coh^{\dot G \times \Gm}(\tcN)_{\le' \lambda}.
\]
\end{lem}

\begin{proof}
It suffices to prove that for any $\lambda \in \bXpp_I$ the object $\cG := \Pi^I(\cL_{\tcN_I} (\dot\coweyl_I(\lambda-\varsigma_I)))$ belongs to $\Db\Coh^{\dot G \times \Gm}(\tcN)_{\le' \lambda}$.  By~\eqref{eqn:pi^i}, $\cG$ has a filtration whose subquotients are of the form $\se_{\varnothing, I*}\cO_{\tcN_{\varnothing,I}}(\varsigma_I-2\rho_I + \nu)\la n\ra$ with $\nu$ a weight of $\dot\coweyl_I(\lambda-\varsigma_I)$.  Next, the resolution~\eqref{eqn:tcni-resolution} shows that $\se_{\varnothing,I*}\cO_{\tcN_{\varnothing,I}}(\varsigma_I-2\rho_I+\nu)\la n\ra$ lies in the full triangulated subcategory of $\Db\Coh^{\dot G \times \Gm}(\tcN)$ generated by the objects $\cO_\tcN(\sigma)\la k\ra$ with $k \in \Z$ and $\sigma$ of the form
\begin{equation}
\label{eqn:weights-PiI-subcategories}
\sigma = \varsigma_I -2\rho_I + \nu + \sum_{\alpha \in Y} \alpha = \varsigma_I + \nu - \sum_{\alpha \in \Phi_I^+ \setminus Y} \alpha,
\end{equation}
where $Y \subset \Phi_I^+$ is a subset.

It is well known that if $\nu$ is a weight of $\dot\coweyl_I(\lambda-\varsigma_I)$, then $\nu$ belongs to $\mathsf{Conv}(W_I(\lambda-\varsigma_I))$. Hence the weights $\sigma$ as in~\eqref{eqn:weights-PiI-subcategories} belong to
\begin{multline}
\label{eqn:union-weights}
\bigcup_{Z \subset \Phi_I^+} \left(\mathsf{Conv}(W_I(\lambda-\varsigma_I)) + \varsigma_I - \sum_{\alpha \in Z} \alpha\right) \\
= \bigcup_{Z \subset \Phi_I^+} \mathsf{Conv}\left(W_I(\lambda-\varsigma_I) + \varsigma_I - \sum_{\alpha \in Z} \alpha\right).
\end{multline}
Now for any $w \in W_I$ we have
\[
w(\lambda-\varsigma_I)+\varsigma_I - \sum_{\alpha \in Z} \alpha = w(\lambda) + \sum_{\substack{\beta \in \Phi_I^+ \setminus Z \\ w^{-1}(\beta)<0}} \beta - \sum_{\substack{\alpha \in Z \\ w^{-1}(\alpha)>0}} \alpha.
\]
In particular,
\[
w^{-1} \left( w(\lambda-\varsigma_I)+\varsigma_I - \sum_{\alpha \in Z} \alpha \right) = \lambda + \sum_{\substack{\gamma \in -\Phi_I^+ \\ w(\gamma) \in \Phi_I^+ \setminus Z}} \gamma + \sum_{\substack{\delta \in -\Phi_I^+ \\ w(\delta) \in -Z}} \delta.
\]
This weight is of the form considered in Lemma~\ref{lem:weights}, so it belongs to $\mathsf{Conv}(W_I \lambda)$. This analysis shows that the subset of $\mathbb{R} \otimes_\Z \bX$ considered in~\eqref{eqn:union-weights} is contained in $\mathsf{Conv}(W_I \lambda)$. Hence any weight $\sigma$ as in~\eqref{eqn:weights-PiI-subcategories} belongs to $\convo(\lambda) \cup W_I \lambda$. By condition~\eqref{eqn:condition-order} and Corollary~\ref{cor:order-dom-max}\eqref{it:order-dom-max}, we then have $\sigma \leq' \lambda$, and we finally deduce that
$\cG$ belongs to $\Db\Coh^{\dot G \times \Gm}(\tcN)_{\le' \lambda}$, as desired.
\end{proof}

\begin{prop}
\label{prop:pi-i-push}
Let $\lambda \in \bX$.
\begin{enumerate}
\item
\label{it:pi-i-push-nabla}
Assume that
$\lambda \in W_I\bXpp_I$, 
and let $w \in W_I$ be the unique element such that $w\lambda \in \bXpp_I$. Then we have
\[
\Pi_I(\nabla(\lambda)) \cong \nabla_I(w\lambda)\la -\ell(w)+n_I\ra [-\ell(w) + n_I].
\]
\item
\label{it:pi-i-push-delta}
Assume that
$\lambda \in W_I\bXpp_I$, 
and
let $w \in W_I$ be the unique element such that $w\lambda \in \bXpp_I$. Then we have
\[
\Pi_I(\Delta(\lambda)) \cong \Delta_I(w\lambda)\la \ell(w)-n_I \rangle [\ell(w)-n_I].
\]
\item
\label{it:pi-i-push-cat}
If $\mu \notin \bXpp_I$ for all $\mu \in \bX$ such that $\mu \le' \dom_I(\lambda)$, then
\[
\Pi_I \left(\Db\Coh^{\dot G \times \Gm}(\tcN)_{\le'\lambda} \right) = 0.
\]
Otherwise, let $a(\lambda)$ be the largest weight (with respect to $\le'$) such that $a(\lambda) \in \bXpp_I$ and $a(\lambda) \le' \dom_I(\lambda)$.  Then
\[
\Pi_I \left( \Db\Coh^{\dot G \times \Gm}(\tcN)_{\le'\lambda} \right) \subset \Db\Coh^{\dot G \times \Gm}(\tcN_I)_{\le' a(\lambda)}.
\]
\end{enumerate}
\end{prop}

\begin{proof}
We begin with the claim that if part~\eqref{it:pi-i-push-nabla} holds for all $\mu \in \bX$ such that $\mu \le' \lambda$, then part~\eqref{it:pi-i-push-cat} holds for $\lambda$.  Indeed, $\Db\Coh^{\dot G \times \Gm}(\tcN)_{\le'\lambda}$ is generated by the objects $\nabla(\mu)\la n\ra$ with $\mu \le'\lambda$ (see~\cite[Lemma~3]{bez:ctm}), so to prove the claim, we must check that 
\begin{equation}\label{eqn:pi-i-push-reduce}
\Pi_I(\nabla(\mu)) \in \Db\Coh^{\dot G \times \Gm}(\tcN_I)_{\le' a(\lambda)}
\end{equation}
for all such $\mu$ (where by convention the subcategory is $\{0\}$ if $a(\lambda)$ is not defined).
Part~\eqref{it:pi-i-push-nabla} and Corollary~\ref{cor:PiI-Delta-vanishing} tell us that the left-hand side either vanishes or is of the form
\[
\nabla_I(\dom_I(\mu))\la n\ra[k]
\]
with $\dom_I(\mu) \in \bXpp_I$. By~\eqref{eqn:order-leq'-domI} we have $\dom_I(\mu) \le' \dom_I(\lambda)$, so $\dom_I(\mu) \le' a(\lambda)$, so~\eqref{eqn:pi-i-push-reduce} holds.

Let us now prove part~\eqref{it:pi-i-push-nabla}.  We proceed by induction on $\dom_I(\lambda)$ (for the order $\leq'$) and, within a $W_I$-orbit, by induction on the length of the element $w \in W_I$ such that $w(\lambda) = \dom_I(\lambda)$.

So, let us fix some $\lambda \in \bX$ such that $W_I\lambda \cap \bXpp_I \neq \varnothing$. 
We first consider the case when
$\lambda \in \bXpp_I$.  Form the distinguished triangle
\begin{equation}\label{eqn:pi-i-push-dt}
\cF \to \cO_\tcN(\lambda)\la -\delta_\lambda\ra \to \nabla(\lambda) \xrightarrow{[1]}
\end{equation}
of Proposition~\ref{prop:exotic-defn}. By Proposition~\ref{prop:exotic-recall}\eqref{it:nabla-convo}, $\cF$ belongs to $\Db\Coh^{\dot G \times \Gm}(\tcN)_{\convo(\lambda)}$.
Now, if $\nu \in \convo(\lambda)$, then $\nu \leq' \lambda$ by~\eqref{eqn:condition-order}. Hence if $\eta \leq' \nu$, then  $\eta \leq' \lambda$, so $\dom_I(\eta) \leq' \lambda$ by~\eqref{eqn:order-leq'-domI}. Moreover $\eta \notin W\lambda$ (because otherwise $\nu \in \convo(\eta)$, which contradicts the fact that $\eta \leq' \nu$), so that these weights even satisfy $\dom_I(\eta) <' \lambda$.
By induction and the claim in the first paragraph, we deduce that part~\eqref{it:pi-i-push-cat} of the lemma  holds for such $\nu$: $\Pi_I(\Db\Coh^{\dot G \times \Gm}(\tcN)_{\le'\nu})$ is either $\{0\}$ or contained in the subcategory $\Db\Coh^{\dot G \times \Gm}(\tcN_I)_{\le' a(\nu)}$.  In the latter case, we have $a(\nu) \leq' \dom_I(\nu) <' \lambda$.  In all cases, we deduce that
\begin{equation}\label{eqn:pi-i-push-bound}
\Pi_I(\cF) \in \Db\Coh^{\dot G \times \Gm}(\tcN_I)_{<' \lambda}.
\end{equation}
Let us now apply the functor $\Pi_I\la -n_I\ra[-n_I]$ to~\eqref{eqn:pi-i-push-dt}. By~\eqref{eqn:pi_i}, we obtain a distinguished triangle
\begin{multline}\label{eqn:pi-i-push-pidt}
\Pi_I(\cF)\la -n_I\ra[-n_I] \to \cL_{\tcN_I} (\dot\coweyl_I(\lambda - \varsigma_I)) \la - \delta_\lambda - n_I\ra[-n_I] \\
\to \Pi_I(\nabla(\lambda))\la -n_I\ra[-n_I] \xrightarrow{[1]}.
\end{multline}
If $\cG \in \Db\Coh^{\dot G \times \Gm}(\tcN_I)_{<'\lambda}$, then using Lemma~\ref{lem:pi-adjoint} we have
\begin{equation}\label{eqn:pi-i-push-vanish}
\Hom(\cG,\Pi_I(\nabla(\lambda))\la -n_I\ra[-n_I])
\cong \Hom(\Pi^I(\cG), \nabla(\lambda)) = 0,
\end{equation}
where the last equality holds because, by Lemma~\ref{lem:pi-i-pull}, $\Pi^I(\cG)$ lies in the subcategory $\Db\Coh^{\dot G \times \Gm}(\tcN)_{<'\lambda}$.

From~\eqref{eqn:pi-i-push-bound}, \eqref{eqn:pi-i-push-pidt}, and~\eqref{eqn:pi-i-push-vanish}, we see that $\Pi_I(\nabla(\lambda))\la -n_I\ra[-n_I]$ satisfies the properties that uniquely characterize $\nabla_I(\lambda)$, so
\[
\Pi_I(\nabla(\lambda)) \cong \nabla_I(\lambda)\la n_I\ra[n_I],
\]
as desired.

Finally, suppose that $w\lambda \in \bXpp_I$ for some nontrivial $w \in W_I$.  Choose a simple reflection $s \in I$ such that 
$ws < w$.  By induction, we already know that $\Pi_I(\nabla(s\lambda)) \cong \nabla_I(w\lambda)\la -\ell(ws) + n_I\ra [-\ell(ws) + n_I]$.  But since $s\lambda \succ \lambda$, Lemma~\ref{lem:pi-transitive} and Proposition~\ref{prop:exotic-recall}\eqref{it:pi-s-push} imply that
\[
\Pi_I(\nabla(\lambda)) \cong \Pi_I(\nabla(s\lambda))\la -1\ra [-1] \cong \nabla_I(w\lambda)\la -\ell(w) + n_I\ra [-\ell(w) + n_I],
\]
as desired.  Part~\eqref{it:pi-i-push-nabla} of the lemma is now proved.  By the claim in the first paragraph, part~\eqref{it:pi-i-push-cat} is proved as well.

We now turn to part~\eqref{it:pi-i-push-delta}. 
This time we proceed by downward induction on the length of $w \in W_I$ such that $w\lambda \in \bXpp_I$, beginning with the case where $w = w_I$ (so $\ell(w) = n_I$).  Applying $\Pi_I$ to the distinguished triangle
\[
\Delta(\lambda) \to \cO_\tcN(\lambda)\la -\delta_\lambda\ra \to \cF' \xrightarrow{[1]}
\]
(where $\cF' \in \Db\Coh^{\dot G \times \Gm}(\tcN)_{<'\lambda}$) and using~\eqref{eqn:pi_i-dual}, we obtain a distinguished triangle
\[
\Pi_I(\Delta(\lambda)) \to \cL_{\tcN_I} (\dot\weyl_I(w_I\lambda - \varsigma_I)) \la - \delta_\lambda\ra[-n_I] \to \Pi_I(\cF') \xrightarrow{[1]}.
\]
If $\nu <' \lambda$, then $\dom_I(\nu) \leq' \dom_I(\lambda)=w_I \lambda$ by~\eqref{eqn:order-leq'-domI}. In fact, in this case we even have $\dom_I(\nu) <' w_I \lambda$ since $\nu \notin W_I \lambda$ by Corollary~\ref{cor:order-dom-max}\eqref{it:order-antidom-min}. Hence
$a(\nu) <'  w_I\lambda$ if $a(\nu)$ is defined.  Therefore, by part~\eqref{it:pi-i-push-cat} of the lemma, $\Pi_I(\cF')$ lies in the subcategory $\Db\Coh^{\dot G \times \Gm}(\tcN_I)_{<'w_I\lambda}$.  

If $\mu \in \bXpp_I$ and $\mu <' w_I\lambda$, then by~\eqref{eqn:order-leq'-WI} we have $\mu <' \lambda$.
Lemma~\ref{lem:pi-i-pull} and this remark
imply that
\[
\Pi^I(\Db\Coh^{\dot G \times \Gm}(\tcN_I)_{<'w_I\lambda}) \subset \Db\Coh^{\dot G \times \Gm}(\tcN)_{<'\lambda}.
\]
Then, 
an adjunction argument similar to that in~\eqref{eqn:pi-i-push-vanish} shows that
\[
\Hom(\Pi_I(\Delta(\lambda)), \cG') = 0
\]
for all $\cG' \in \Db\Coh^{\dot G \times \Gm}(\tcN_I)_{<'w_I\lambda}$.  Using Remark~\ref{rmk:exotic-other-defn}\eqref{it:rmk-def-nabla-weyl} and the fact that $\delta_\lambda = \delta_{w_I\lambda} + n_I$, we see that $\Pi_I(\Delta(\lambda))$ satisfies the properties that uniquely characterize $\Delta_I(w_I\lambda)$, so
\[
\Pi_I(\Delta(\lambda)) \cong \Delta_I(w_I\lambda),
\]
as desired.  

Finally, if $\lambda$ is a weight such that $w\lambda \in \bXpp_I$ for some $w \in W_I$, $w \ne w_I$, an induction argument using Proposition~\ref{prop:exotic-recall}\eqref{it:pi-s-push} shows that
\[
\Pi_I(\Delta(\lambda)) \cong \Delta_I(w\lambda)\la \ell(w)-n_I \rangle [\ell(w)-n_I]
\]
as desired.
\end{proof}

\subsection{Koszul duality}
\label{ss:Koszul-Springer}

For any subset $I \subset S$, we consider the algebras
\[
\bS_I := \mathrm{Sym}(\fnt_I^*), \quad \bL_I := \bigwedge \hspace{-3pt} {}^\bullet \, \fnt_I
\]
defined as in Section~\ref{sec:koszul} (with respect to the natural $\dot P_I$-module structure on $\fnt_I$). Here $\bS_I$ will be mainly considered as a $\dot P_I \times \Gm$-equivariant algebra, and $\bL_I$ will be mainly considered as $\dot P_I$-equivariant dg-algebra. Then we have the functor
\[
\kappa_I : \Db(\bS_I\lmod_{\dot P_I \times \Gm}^{\mathrm{fg}}) \to \Dfg_{\dot P_I}(\bL_I)
\]
as in~\S\ref{ss:regrading}. 
If $J \subset I$, we can also restrict the $\dot P_I$-action to $\dot P_J$, and obtain a functor
\[
\kappa_{J,I} : \Db(\bS_I\lmod_{\dot P_J \times \Gm}^{\mathrm{fg}}) \to \Dfg_{\dot P_J}(\bL_I)
\]

Let $\gamma_I: \fnt_I \to \tcN_I$ be the inclusion map $x \mapsto [1_{\dot G} : x]$.  Then coherent pullback along $\gamma_I$ gives rise to an equivalence of categories
\[
\gamma_I^* : \Coh^{\dot G \times \Gm}(\tcN_I) \simto \Coh^{\dot P_I \times \Gm}(\fnt_I) = \bS_I\lmod_{\dot P_I \times \Gm}^{\mathrm{fg}},
\]
sometimes called the ``induction equivalence,'' see e.g.~\cite[Lemma~2]{brion}.  We define $\gamma_{J,I}: \fnt_I \to \tcN_{J,I}$ similarly; it induces an equivalence
\[
\gamma_{J,I}^*: \Coh^{\dot G \times \Gm}(\tcN_{J,I}) \simto \Coh^{\dot P_J \times \Gm}(\fnt_I) =\bS_I\lmod^{\mathrm{fg}}_{\dot P_J \times \Gm}.
\]
Then we set
\[
\varkappa_I := \kappa_I \circ \Db(\gamma_I^*) : \Db\Coh^{\dot G \times \Gm}(\tcN_I) \to \Dfg_{\dot P_I}(\bL_I).
\]
As for $\kappa_I$, there exists a natural isomorphism of functors
\begin{equation}
\label{eqn:varkappa-shift}
\varkappa_I \circ \langle 1 \rangle [1] \cong \varkappa_I.
\end{equation}
And
it follows from the isomorphism in~\eqref{eqn:forget-grading-kappa} that for any $\cF,\cG \in \Db\Coh^{\dot G \times \Gm}(\tcN_I)$, the functor $\varkappa_I$ and the isomorphism~\eqref{eqn:varkappa-shift} induce an isomorphism
\begin{equation}
\label{eqn:varkappa-degr}
\bigoplus_{n \in \Z} \Hom_{\Db\Coh^{\dot G \times \Gm}(\tcN_I)}(\cF,\cG\la n\ra[n]) \to \Hom_{\Dfg_{\dot P_I}(\bL_I)}(\varkappa_I(\cF), \varkappa_I(\cG)).
\end{equation}

The functors $\gamma_I^*$ and $\gamma_{J,I}^*$ allow us to convert the study of the functors $\Pi_{J,I}$ into the language of $\bS_J$- and $\bS_I$-modules, as shown in the left part of the diagram of Figure~\ref{fig:coh-bl}.  The right part of the diagram comes from the discussion of Koszul duality in Section~\ref{sec:koszul}. It follows from the definitions that the left part of the diagram is commutative, and from Propositions~\ref{prop:koszul-compatibility-1} and~\ref{prop:koszul-compatibility-2} that the right part is commutative.

\begin{figure}
\[
\xymatrix@C=1.5cm{
\Db\Coh^{\dot G \times \Gm}(\tcN_J) \ar[r]^{\Db(\gamma_J^*)}_{\sim} \ar@/^5ex/[rr]^{\varkappa_J} 
     \ar[d]^-{({-}) \otimes \cO_\tcN(-\varsigma_{I \ssm J})} \ar@/_11ex/[ddd]_{\Pi_{J,I}} &
  \Db(\bS_J\lmod^{\mathrm{fg}}_{\dot P_J \times \Gm}) \ar[r]^-{\kappa_J} \ar[d]_-{({-}) \otimes \bk(-\varsigma_{I\ssm J})} &
  \Dfg_{\dot P_J}(\bL_J) \ar[d]_{({-}) \otimes \bk(-\varsigma_{I\ssm J})} \ar@/^6ex/[ddd]^{\Theta_{J,I}} \\
\Db\Coh^{\dot G \times \Gm}(\tcN_J) \ar[r]^{\Db(\gamma_J^*)}_{\sim} \ar[d]_{\se_{J,I}^*} &
  \Db(\bS_J\lmod^{\mathrm{fg}}_{\dot P_J \times \Gm}) \ar[r]^-{\kappa_J} \ar[d]_{\bS_I \lotimes_{\bS_J} ({-})} &
  \Dfg_{\dot P_J}(\bL_J) \ar[d]_{j_{J,I}^*} \\
\Db\Coh^{\dot G \times \Gm}(\tcN_{J,I}) \ar[r]^{\Db(\gamma_{J,I}^*)}_{\sim} \ar[d]_{\mu_{J,I*}} &
  \Db(\bS_I\lmod^{\mathrm{fg}}_{\dot P_J \times \Gm}) \ar[r]^-{\kappa_{J,I}} \ar[d]_{R\Ind_{\dot P_J}^{\dot P_I}} &
  \Dfg_{\dot P_J}(\bL_I) \ar[d]_{R\Ind_{\dot P_J}^{\dot P_I}} \\
\Db\Coh^{\dot P_I \times \Gm}(\tcN_I) \ar[r]^{\Db(\gamma_I^*)}_{\sim} \ar@/_5ex/[rr]_{\varkappa_I} &
  \Db(\bS_I\lmod^{\mathrm{fg}}_{\dot P_I \times \Gm}) \ar[r]^-{\kappa_I} &
  \Dfg_{\dot P_I}(\bL_I) }
\]
\caption{From $\tcN_I$ to $\bL_I$}\label{fig:coh-bl}
\end{figure}

\begin{prop}
\label{prop:adj-varkappa}
The diagram below is a commutative diagram of adjoint pairs:
\[
\xymatrix@C=2cm{
\Db\Coh^{\dot G \times \Gm}(\tcN_J) \ar[r]^-{\varkappa_J} \ar@{}[d]|\dashv \ar@<1ex>[d]^{\Pi_{J,I}}
  & \Dfg_{\dot P_J}(\bL_J) \ar@{}[d]|\dashv \ar@<1ex>[d]^{\Theta_{J,I}} \\
\Db\Coh^{\dot G \times \Gm}(\tcN_I) \ar[r]_-{\varkappa_I} \ar@<1ex>[u]^{\Pi^{J,I}\la d_I-d_J\ra[d_I-d_J]} 
  & \Dfg_{\dot P_I}(\bL_I) \ar@<1ex>[u]^{\Theta^{J,I}}}
\]
\end{prop}

\begin{proof}
This proposition is ``almost'' an application of Lemma~\ref{lem:equivalence-mate}, because $\varkappa_J$ and $\varkappa_I$ are close to being equivalences. More precisely we argue as follows.
For brevity, let us put $\bar\Pi^{J,I} := \Pi^{J,I}\la d_I - d_J\ra [d_I - d_J]$.  The commutativity of the diagram in Figure~\ref{fig:coh-bl} gives us an isomorphism
\[
\zeta: \varkappa_I \circ \Pi_{J,I} \simto \Theta_{J,I} \circ \varkappa_J.
\]
Let $\zeta^\wedge: \Theta^{J,I} \circ \varkappa_I \to \varkappa_J \circ \bar \Pi^{J,I}$ be the morphism constructed from $\zeta$ as in \S\ref{ssec:string}.  We must show that $\zeta^\wedge$ is an isomorphism.

We begin with a weaker claim: that for any $\cF \in \Db\Coh^{\dot G \times \Gm}(\tcN_I)$ and $\cG \in \Db\Coh^{\dot G \times \Gm}(\tcN_J)$, the map
\begin{equation}\label{eqn:zeta-test1}
({-}) \circ \zeta^\wedge_{\cF} :
\Hom(\varkappa_J \bar\Pi^{J,I}(\cF),\varkappa_J(\cG)) \to \Hom(\Theta^{J,I}\varkappa_I(\cF),\varkappa_J(\cG))
\end{equation}
is an isomorphism.  To prove this claim, we apply Lemma~\ref{lem:mateship-diagram} to obtain the following commutative diagram:
\[
\xymatrix@C=1.5cm{
\Hom(\bar\Pi^{J,I}(\cF), \cG) \ar[r]^-{\varkappa_J}  \ar[d]^{\wr}_{\mathrm{adj}} &
\Hom(\varkappa_J \bar\Pi^{J,I}(\cF),\varkappa_J(\cG)) \ar[d]^-{(-) \circ \zeta_{\cF}^\wedge} \\
  \Hom(\cF, \Pi_{J,I}(\cG)) \ar[r] & \Hom(\Theta^{J,I}\varkappa_I(\cF),\varkappa_J(\cG)).
}
\]
In the left-hand column, let us replace $\cG$ by $\cG\la n\ra[n]$ and then sum over all $n \in \Z$:
\begin{equation}\label{eqn:zeta-test2}
\vcenter{\xymatrix@C=1.5cm{
  \bigoplus_{n\in\Z} \Hom(\bar\Pi^{J,I}(\cF), \cG\la n\ra[n]) \ar[r]^-{\varkappa_J}  \ar[d]^{\wr}_{\mathrm{adj}} &
\Hom(\varkappa_J \bar\Pi^{J,I}(\cF),\varkappa_J(\cG)) \ar[d]^-{(-) \circ \zeta_{\cF}^\wedge} \\
  \bigoplus_{n\in\Z} \Hom(\cF, \Pi_{J,I}(\cG\la n\ra[n])) \ar[r] & \Hom(\Theta^{J,I}\varkappa_I(\cF),\varkappa_J(\cG)).
}}
\end{equation}
In this diagram, the top horizontal arrow is an isomorphism by~\eqref{eqn:varkappa-degr}.  The bottom horizontal arrow is defined to be the composition
\begin{multline*}
\textstyle \bigoplus_{n\in\Z} \Hom(\cF, \Pi_{J,I}(\cG\la n\ra[n]))
\xrightarrow{\varkappa_I}
\Hom(\varkappa_I(\cF), \varkappa_I\Pi_{J,I}(\cG)) \\
\xrightarrow[\sim]{\zeta_\cG \circ ({-})}
\Hom(\varkappa_I(\cF), \Theta_{J,I}\varkappa_J(\cG))
\xrightarrow[\sim]{\mathrm{adj}}
\Hom(\Theta^{J,I}\varkappa_I(\cF),\varkappa_J(\cG))
\end{multline*}
so it too is an isomorphism.  We conclude that the left-hand vertical arrow in~\eqref{eqn:zeta-test2}, i.e., the map in~\eqref{eqn:zeta-test1}, is an isomorphism as well.

For any $V \in \Rep(\dot P_J)$, we have $\varkappa_J(\cO_{\tcN_J} \otimes V) \cong \bk_{\bL_J} \otimes V$, so objects of the form $\varkappa_J(\cG)$ generate $\Dfg_{\dot P_J}(\bL_J)$ as a triangulated category.  Hence~\eqref{eqn:zeta-test1} and the five-lemma actually imply that
\[
(-) \circ \zeta_{\cF}^\wedge : \Hom(\varkappa_J \bar\Pi^{J,I}(\cF),\cG') \to \Hom(\Theta^{J,I}\varkappa_I(\cF),\cG')
\]
is an isomorphism for all $\cG' \in \Dfg_{\dot P_J}(\bL_J)$.  By Yoneda's lemma, this shows that $\zeta^\wedge_{\cF}: \Theta^{J,I} \bar\varkappa_I(\cF) \to \varkappa_J\Pi^{J,I}(\cF)$ is an isomorphism, as desired.
\end{proof}

\begin{rmk}
\begin{enumerate}
\item
Later we will use this proposition only in the case $J=\varnothing$. We treat the general case since it is not more difficult that this special case.
\item
One can also prove Proposition~\ref{prop:adj-varkappa} by showing that each small square in Figure~\ref{fig:coh-bl} is a commutative diagram of adjoint pairs.  (For the middle row of squares, one can use~\eqref{eqn:eJI!} to describe the left adjoint of $\se_{J,I}^*$; similar descriptions are possible for the other functors in that row.)
\item
As noticed (in a special case) in Remark~\ref{rmk:trans-biadj}, the functor $\Theta^{J,I}$ is also right adjoint to $\Theta_{J,I}$.  There is also a commutative diagram of adjoint pairs involving this adjunction:
\[
\xymatrix@C=3cm{
\Db\Coh^{\dot G \times \Gm}(\tcN_I) \ar[r]^-{\varkappa_I} \ar@{}[d]|\dashv \ar@<1ex>[d]^{\Pi^{J,I} \langle d_J-d_I \rangle [d_J - d_I]}
  & \Dfg_{\dot P_I}(\bL_I) \ar@{}[d]|\dashv \ar@<1ex>[d]^{\Theta^{J,I}} \\
\Db\Coh^{\dot G \times \Gm}(\tcN_J) \ar[r]_-{\varkappa_J} \ar@<1ex>[u]^{\Pi_{J,I}} 
  & \Dfg_{\dot P_J}(\bL_J). \ar@<1ex>[u]^{\Theta_{J,I}}}
\]
However, this version will not be useful to us: unlike the diagram in Proposition~\ref{prop:adj-varkappa}, this version cannot be combined with Theorem~\ref{thm:translation}.
\end{enumerate}
\end{rmk}

Applying Lemma~\ref{lem:comm-adj-counit} we deduce from Proposition~\ref{prop:adj-varkappa} the following corollary, which is the result we will use in Section~\ref{sec:induction-thm}.

\begin{cor}
\label{cor:isom-adjunctions-Pi}
There exists an isomorphism of functors
\[
\varkappa_J \circ (\Pi^{J,I} \langle d_I-d_J \rangle [d_I-d_J]) \circ \Pi_{J,I}
\simto
\Theta^{J,I} \circ \Theta_{J,I} \circ \varkappa_J
\]
such that for any $\cF$ in $\Db\Coh^{\dot G \times \Gm}(\tcN_J)$ the following diagram commutes, where the vertical arrows are induced by our isomorphism of functors and the other arrows are induced by adjunction:
\[
\xymatrix@C=1.5cm{
\varkappa_J \circ (\Pi^{J,I} \langle d_I-d_J \rangle [d_I-d_J]) \circ \Pi_{J,I}(\cF) \ar[rd] \ar[d]_-{\wr} & \\
\Theta^{J,I} \circ \Theta_{J,I} \circ \varkappa_J(\cF) \ar[r] & \varkappa_J(\cF).
}
\]
\end{cor}

\section{The induction theorem}
\label{sec:induction-thm}

\subsection{Combinatorics of weights}
\label{ss:combinatorics-weights}

Let $\Waffmin \subset \Waff$ be the subset consisting of the elements $w$ which are minimal in $Ww$. Then
it is well known that the assignment $w \mapsto w \bullet 0$ induces a bijection
\[
\Waffmin \simto (\Waff \bullet 0) \cap \bX^+.
\]
On the other hand, we also have bijections $\Waffmin \simto W \backslash \Waff \cong \bX$. Recall (see~\S\ref{ss:orders-X}) that for $\lambda \in \bX$, the inverse image of $\lambda$ under this bijection is denoted $w_\lambda$. This element is described explicitly in~\cite[Lemma~2.4]{mr:etspc}: if $v_\lambda \in W$ is the element of minimal length such that $v_\lambda (\lambda) \in \bX^+$, then $w_\lambda = v_\lambda \cdot t_\lambda$. Combining these bijections we obtain a bijection
\begin{equation}
\label{eqn:bijection-weights}
\bX \simto (\Waff \bullet 0) \cap \bX^+ : \lambda \mapsto w_\lambda \bullet 0 = v_\lambda \bullet 0 + \ell \cdot v_\lambda(\lambda).
\end{equation}

Now, consider the order
$\uparrow$ on $\bX$ as defined in~\cite[\S 6.4]{jantzen}. 

\begin{lem}
\label{lem:orders-dominant-weights}
For $\lambda,\mu \in \bX$, we have
\[
w_\lambda \bullet 0 \uparrow w_\mu \bullet 0 \quad \Leftrightarrow \quad \lambda \leq \mu.
\]
\end{lem}

\begin{proof}
In case $G$ is semisimple, this statement is equivalent to the main results of~\cite{ye, wang} (see also~\cite[\S 5]{humphreys} for a discussion of this result in English). Since here we work with a reductive group, we have to be slightly more careful.

First, let us assume that $\lambda, \mu \in \Z\Phi$. Then $w_\lambda, w_\mu \in \WaffCox$. Let us consider $V=\Z\Phi \otimes_\Z \mathbb{R}$, and denote by $\mathscr{A}_0$ the intersection of the fundamental alcove for $G$ (as defined in~\cite[II.6.2(6)]{jantzen}) with $V$. In other words, $\mathscr{A}_0$ is the fundamental alcove for the group $G/Z(G)$. The restriction of the order $\uparrow$ to $\Z\Phi$ is clearly the order $\uparrow$ for the group $G/Z(G)$. We deduce that, if we consider the order $\uparrow$ on alcoves of $G/Z(G)$ defined as in~\cite[\S II.6.5]{jantzen}, then by~\cite[II.6.5(1)]{jantzen} we have
\[
w_\lambda \bullet 0 \uparrow w_\mu \bullet 0 \quad \Leftrightarrow \quad w_\lambda \bullet \mathscr{A}_0 \uparrow w_\mu \bullet \mathscr{A}_0.
\]
Then by~\cite{ye,wang} this condition is equivalent to $w_\lambda \leq w_\mu$, hence by definition to $\lambda \leq \mu$.

Now we treat the general case. If $\lambda - \mu \notin \Z\Phi$, then neither of the conditions in the statement hold, so the equivalence is guaranteed. So, let us assume that $\lambda-\mu \in \Z\Phi$. Then there exists a unique $\omega \in \Waff$ with $\ell(\omega)=0$ and $w_\lambda \omega^{-1} \in \WaffCox$ and $w_\mu \omega^{-1} \in \WaffCox$. Since these elements belong to $\Waffmin$, there exist $\lambda',\mu' \in \Z\Phi$ such that
\[
w_\lambda = w_{\lambda'} \omega, \qquad w_\mu = w_{\mu'} \omega.
\]
By definition of the Bruhat order on $\Waff$, we have $w_\lambda \leq w_\mu$ iff $w_{\lambda'} \leq w_{\mu'}$, hence $\lambda \leq \mu$ iff $\lambda' \leq \mu'$. By the case already treated, this condition is equivalent to $w_{\lambda'} \bullet 0 \uparrow w_{\mu'} \bullet 0$. And since $0$ and $\omega \bullet 0$ both belong to the fundamental alcove (for $G$), using~\cite[II.6.5(1)]{jantzen} we see that this condition is equivalent to $w_\lambda \bullet 0 \uparrow w_{\mu} \bullet 0$, and the proof is complete.
\end{proof}

It follows in particular from Lemma~\ref{lem:orders-dominant-weights} that the order on $(\Waff \bullet 0) \cap \bX^+$ induced by any order $\leq'$ as in~\S\ref{ss:orders-X} via the bijection~\eqref{eqn:bijection-weights} refines the order $\uparrow$.

If $I \subset S$, then we define
\[
\WaffminI := \{ w \in \Waff \mid w \text{ is maximal in $wW_I$ and $wv \in \Waffmin$ for all $v \in W_I$}\}.
\]
(In fact, using the same trick from~\cite[p.~86]{soergel} as in the proof of Lemma~\ref{lem:order-mu-smu}, one can check that if $w$ is maximal in $wW_I$ and $w \in \Waffmin$, then $wv \in \Waffmin$ for all $v \in W_I$.)

\begin{lem}
\label{lem:WaffminI-domI}
Let $\lambda \in \bX$. Then $\lambda$ belongs to $\bXpp_I$ iff 
$w_\lambda \in \WaffminI$. 
\end{lem}

\begin{proof}
First, let us assume that $\lambda \in \bXpp_I$. Then for any $v \in W_I$, and any reduced expression $v=s_1 \cdots s_r$, we have
\[
\lambda \succ s_r (\lambda) \succ s_{r-1} s_r (\lambda) \succ \cdots \succ v(\lambda).
\]
As explained in the proof of Lemma~\ref{lem:order-mu-smu}, this implies that
\[
w_\lambda > w_\lambda s_r > w_\lambda s_r s_{r-1} > \cdots > w_\lambda v,
\]
and that all these elements belong to $\Waffmin$. Hence $w_\lambda \in \WaffminI$.

On the other hand, assume that $\lambda \notin \bXpp_I$. Then there exists $s \in I$ such that $s\lambda \succeq \lambda$. If $s\lambda \succ \lambda$, then as above by Lemma~\ref{lem:order-mu-smu} and its proof we have $w_\lambda < w_\lambda s$, and hence $w_\lambda \notin \WaffminI$. And if $s\lambda=\lambda$ we have
\[
w_\lambda s = v_\lambda t_\lambda s = v_\lambda s t_\lambda = (v_\lambda s v_\lambda^{-1}) w_\lambda > w_\lambda,
\]
and so again $w_\lambda \notin \WaffminI$.
\end{proof}

From Lemma~\ref{lem:WaffminI-domI} we obtain a bijection
\[
\bXpp_I \simto \WaffminI : \lambda \mapsto w_\lambda.
\]
On the other hand, it is clear that the assignment $w \mapsto w \bullet (-\varsigma_I)$ defines a bijection $\WaffminI \simto (\Waff \bullet (-\varsigma_I)) \cap \bX^+$; combining these bijections we obtain a bijection
\begin{equation}
\label{eqn:bijection-weights-I}
\bXpp_I \simto (\Waff \bullet (-\varsigma_I)) \cap \bX^+ : \lambda \mapsto w_\lambda \bullet (-\varsigma_I) = v_\lambda \bullet (-\varsigma_I) + \ell \cdot v_\lambda(\lambda).
\end{equation}

\subsection{Images of standard and costandard objects}

From now on, for any subset $I \subset S$ with $I \neq \varnothing$, we assume that the objects $\Delta_I(\lambda)$ and $\nabla_I(\lambda)$ are defined with respect to the an order constructed as in~\S\ref{ss:orders-X}. (In particular, this order depends on $I$.) In the case $I=\varnothing$, the objects $\Delta(\lambda)$ and $\nabla(\lambda)$ are independent of the choice of order satisfying~\eqref{eqn:condition-order}, by Proposition~\ref{prop:exotic-recall}\eqref{it:exotic-recall-1}.

\begin{prop}
\label{prop:exotic-weyl}
For any $\lambda \in \bXpp_I$, we have isomorphisms
\[
\Omega_I(\varkappa_I(\nabla_I(\lambda)) \cong \coweyl(w_\lambda \bullet (-\varsigma_I)), \qquad
\Omega_I(\varkappa_I(\Delta_I(\lambda)) \cong \weyl(w_\lambda \bullet (-\varsigma_I)).
\]
\end{prop}

\begin{proof}
We begin with the first isomorphism.
Suppose first that $I = \varnothing$. In this case, we will prove the isomorphism by induction on $\delta_\lambda$.  If $\delta_\lambda = 0$, i.e.~if $\lambda$ is dominant, then $\nabla(\lambda) \cong \cO_\tcN(\lambda)$ (see~\cite[Corollary~3.4]{mr:etspc}), so
\[
\Omega_\varnothing(\varkappa_\varnothing(\nabla(\lambda))) \cong \Omega_\varnothing(\bk_{\dot B}(\lambda)) \cong R\Ind_B^G(\ell\lambda) \cong \coweyl(\ell\lambda)
\]
by Kempf's vanishing theorem.
This proves the claim since $w_\lambda=t_\lambda$.

Otherwise, we have $\ell(v_\lambda)>0$. Let $s \in S$ be such that $\ell(v_\lambda s)<\ell(v_\lambda)$. Then $s\lambda \succ \lambda$, $\delta_{s\lambda} = \delta_\lambda-1$, and $w_{s\lambda} = w_\lambda s$ with $\ell(w_\lambda)=\ell(w_{s\lambda})-1$ (see Lemma~\ref{lem:order-mu-smu} and its proof). Consider the first distinguished triangle in Proposition~\ref{prop:exotic-recall}\eqref{it:pi-s-std}:
\[
\nabla(\lambda)\la -1\ra[-1] \to \Pi^s\Pi_s(\nabla(s\lambda))\la -1\ra[-1] \xrightarrow{\epsilon} \nabla(s\lambda) \xrightarrow{[1]}.
\]
Applying $\Omega_{\varnothing} \circ \varkappa_{\varnothing}$ to this triangle, and using induction and Corollaries~\ref{cor:isom-adjunctions-translation} and~\ref{cor:isom-adjunctions-Pi}, we obtain a distinguished triangle
\begin{equation}
\label{eqn:exact-seq-costandard}
\Omega_{\varnothing}(\varkappa_{\varnothing}(\nabla(\lambda))) \to T_{\{s\}}^\varnothing T_\varnothing^{\{s\}}(\coweyl(w_\lambda s \bullet 0)) \xrightarrow{\epsilon} \coweyl(w_\lambda s \bullet 0) \xrightarrow{[1]},
\end{equation}
in which the second arrow is induced by adjunction.
By~\cite[Proposition~II.7.19(a) and II.7.21(8)]{jantzen}, this distinguished triangle is actually a short exact sequence in $\Rep(G)$ whose first term is isomorphic to $\coweyl(w_\lambda \bullet 0)$, as desired.

We now turn to the case of general $I$.  Let $\lambda \in \bXpp_I$.
Using Proposition~\ref{prop:pi-i-push}\eqref{it:pi-i-push-nabla}, we have
\[
\Omega_I(\varkappa_I(\nabla_I(\lambda)) \cong \Omega_I\varkappa_I\Pi_I(\nabla(\lambda)\la -n_I\ra[-n_I]).
\]
Then, using Proposition~\ref{prop:adj-varkappa} and Lemma~\ref{lem:theta-phi-iso} we obtain isomorphisms
\[
\Omega_I(\varkappa_I(\nabla_I(\lambda))
\cong \Omega_I \Theta_{\varnothing,I}\varkappa_{\varnothing}(\nabla(\lambda)\la -n_I\ra[-n_I])
\cong T_{\varnothing}^{I}\Omega_\varnothing \varkappa_\varnothing (\nabla(\lambda)\la -n_I\ra[-n_I]).
\]
Next, using~\eqref{eqn:varkappa-shift} and the case $I=\varnothing$, we obtain an isomorphism
\[
\Omega_I(\varkappa_I(\nabla_I(\lambda)) \cong T_{\varnothing}^{I} \coweyl(w_\lambda \bullet 0).
\]
Finally,
by~\cite[Proposition~II.7.11]{jantzen} we have $T_\varnothing^I \coweyl(w_\lambda \bullet 0) \cong \coweyl(w_\lambda \bullet (-\varsigma_I))$, and the proof is complete.

Now we consider the case of $\Delta_I(\lambda)$, first in the case when $I=\varnothing$ and $\lambda$ is antidominant. In this case we have $\Delta(\lambda) \cong \cO_{\tcN}(\lambda) \langle -\delta_\lambda \rangle$ by~\cite[Proposition~3.6(2)]{mr:etspc}. As in the case of the objects $\nabla_I(\lambda)$, we deduce that
\[
\Omega_\varnothing \circ \varkappa_\varnothing(\Delta(\lambda)) \cong R\Ind_B^G(\ell \lambda)[\delta_\lambda].
\]
Now, since $\lambda$ is antidominant, its stabilizer in $W$ is $W_K$, where $K:=\{s \in S \mid s(\lambda)=\lambda\}$. It follows in particular that $v_\lambda = w_S w_K$ 
and $\delta_\lambda = d_K$. We also deduce that $R\Ind_B^G(\ell \lambda) \cong R\Ind_{P_K}^G(\ell \lambda)$. Now using~\cite[II.4.2(8)]{jantzen}, for any $i \in \Z$ we have
\begin{multline*}
R^i \Ind_{P_K}^G(\ell \lambda) \cong (R^{d_K-i} \Ind_{P_K}^G(-\ell \lambda - (2\rho-2\rho_K))^* \\
\cong (R^{d_K-i} \Ind_{B}^G(-\ell \lambda - (2\rho-2\rho_K))^*.
\end{multline*}
The weight $-\ell \lambda - (2\rho-2\rho_K)$ is dominant, so by Kempf's vanishing theorem the third term vanishes unless $i=d_K$, and we finally obtain that
\[
R\Ind_B^G(\ell \lambda)[\delta_\lambda] \cong (\Ind_B^G(-\ell \lambda - 2\rho+2\rho_K))^* \cong \weyl(w_S(\ell \lambda +2\rho-2\rho_K)).
\]
Since $w_\lambda \bullet 0 = w_S w_K(\ell \lambda + \rho)-\rho = w_S(\ell \lambda +2\rho-2\rho_K)$, this proves the desired isomorphism in this case.

We continue to assume that $I=\varnothing$, and prove the isomorphism by downward induction on $\delta_\lambda$ within a given $W$-orbit. The case when $\delta_\lambda$ is maximal is the case when $\lambda$ is antidominant, which was treated above. If $\lambda$ is not antidominant, there exists $s \in S$ such that $s\lambda \prec \lambda$, so that $\delta_{s\lambda} = \delta_\lambda +1$ and $w_{s\lambda} = w_\lambda s$ with $\ell(w_{s\lambda})=\ell(w_\lambda)-1$ (see again Lemma~\ref{lem:order-mu-smu} and its proof). Consider the second distinguished triangle in Proposition~\ref{prop:exotic-recall}\eqref{it:pi-s-std}:
\[
\Delta(\lambda)\la -1\ra[-1] \to \Pi^s\Pi_s(\Delta(s\lambda))\la -1\ra[-1] \xrightarrow{\epsilon} \Delta(s\lambda) \xrightarrow{[1]}.
\]
As above, applying the functor $\Omega_{\varnothing} \circ \varkappa_{\varnothing}$ and using induction and Corollaries~\ref{cor:isom-adjunctions-translation} and~\ref{cor:isom-adjunctions-Pi} (together with~\eqref{eqn:varkappa-shift}), we obtain a distinguished triangle
\[
\Omega_{\varnothing} \circ \varkappa_{\varnothing}(\Delta(\lambda)) \to T_{\{s\}}^\varnothing T_\varnothing^{\{s\}}(\weyl(w_\lambda s \bullet 0)) \xrightarrow{\epsilon} \weyl(w_\lambda s \bullet 0) \xrightarrow{[1]}
\]
where the second morphism is induced by adjunction. This implies that the first term is isomorphic to $\weyl(w_\lambda \bullet 0)$, and finishes the proof in this case.

Finally, as in the case of the objects $\nabla_I(\lambda)$, the case of a general subset $I$ follows from the case $I=\varnothing$ using Proposition~\ref{prop:pi-i-push}\eqref{it:pi-i-push-delta}.
\end{proof}

\begin{lem}
\label{lem:weyl-can}
For any $\lambda \in \bXpp_I$,
the image under $\Omega_I \circ \varkappa_I$ of any nonzero map $\Delta_I(\lambda) \to \nabla_I(\lambda)$ is nonzero.
\end{lem}

\begin{proof}
First, let us consider the case $I=\varnothing$. We still denote by $\leq'$ the order on $(\Waff \bullet 0) \cap \bX^+$ induced by the order $\leq'$ on $\bX$ via the bijection~\eqref{eqn:bijection-weights}. As explained 
after Lemma~\ref{lem:orders-dominant-weights},
this order is a refinement of the order $\uparrow$; in particular, $\Rep_{\varnothing}(G)$ is a highest weight category for this order, with standard objects $\weyl(\lambda)$ and costandard objects $\coweyl(\lambda)$ (for $\lambda \in (\Waff \bullet 0) \cap \bX^+$). For $\mu \in (\Waff \bullet 0) \cap \bX^+$, we denote by $\Db \Rep_{\varnothing}(G)_{<'\mu}$ the triangulated subcategory of $\Db \Rep_{\varnothing}(G)$ generated by the objects $\coweyl(\nu)$ with $\nu <' \mu$, or equivalently by the objects $\weyl(\nu)$ with $\nu <' \mu$. With this notation, Proposition~\ref{prop:exotic-weyl} implies that for any $\mu \in \bX$ we have
\begin{equation}
\label{eqn:Omega-kappa-graded}
\Omega_\varnothing \circ \varkappa_\varnothing(\Db \Coh^{\dot G \times \Gm}(\tcN)_{<'\mu}) \subset \Db \Rep_{\varnothing}(G)_{<' w_\mu \bullet 0}.
\end{equation}

Now, let us fix $\lambda \in \bX$.
There exists only one (up to scalar) nonzero morphism $f : \Delta(\lambda) \to \nabla(\lambda)$; let $C$ be its cone. Then $C$ belongs to $\Db \Coh^{\dot G \times \Gm}(\tcN)_{<'\lambda}$. The cone of $\Omega_\varnothing \circ \varkappa_\varnothing(f)$ is $\Omega_\varnothing \circ \varkappa_\varnothing(C)$, and by~\eqref{eqn:Omega-kappa-graded} it belongs to $\Db \Rep_{\varnothing}(G)_{<' w_\lambda \bullet 0}$. Now we have
\[
\Omega_\varnothing \circ \varkappa_\varnothing(\Delta(\lambda)) \cong \weyl(w_\lambda \bullet 0), \qquad \Omega_\varnothing \circ \varkappa_\varnothing(\nabla(\lambda)) \cong \coweyl(w_\lambda \bullet 0),
\]
so $\Omega_\varnothing \circ \varkappa_\varnothing(f)$ is a morphism from $\weyl(w_\lambda \bullet 0)$ to $\coweyl(w_\lambda \bullet 0)$. The fact that its cone belongs to $\Db \Rep_{\varnothing}(G)_{<' w_\lambda \bullet 0}$ forces this morphism to be nonzero, and the claim is proved in this case.

Now let $I$ be arbitrary, and let $\lambda \in \bXpp_I$. Consider a nonzero morphism $f : \Delta(w_I \lambda) \to \nabla(w_I \lambda)$. By Proposition~\ref{prop:pi-i-push} we have
\[
\Pi_I(\Delta(w_I \lambda)) \cong \Delta_I(\lambda), \qquad \Pi_I(\nabla(w_I \lambda)) \cong \nabla_I(\lambda),
\]
so $\Pi_I(f)$ is a morphism from $\Delta_I(\lambda)$ to $\nabla_I(\lambda)$. By the case treated above, the morphism $\Omega_\varnothing \circ \varkappa_\varnothing(f)$ is a nonzero morphism from $\weyl(w_\lambda w_I \bullet 0)$ to $\coweyl(w_\lambda w_I \bullet 0)$. (Here we use that $w_{w_I \lambda} = w_\lambda w_I$.) Now since $w_\lambda w_I$ is minimal in $w_\lambda w_I W_I = w_\lambda W_I$, by~\cite[Proposition~II.7.15]{jantzen} we have
\[
T_\varnothing^I(\irr(w_\lambda w_I \bullet 0)) \cong \irr(w_\lambda w_I \bullet (-\varsigma_I)) = \irr(w_\lambda \bullet (-\varsigma_I)).
\]
This implies that the image under $T_\varnothing^I$ of any nonzero morphism from $\weyl(w_\lambda w_I \bullet 0)$ to $\coweyl(w_\lambda w_I \bullet 0)$, in particular of $\Omega_\varnothing \circ \varkappa_\varnothing(f)$, is nonzero. But as in the proof of Proposition~\ref{prop:exotic-weyl} we have an isomorphism of functors
\[
T_\varnothing^I \circ \Omega_\varnothing \circ \varkappa_\varnothing \cong \Omega_I \circ \varkappa_I \circ \Pi_I;
\]
hence $\Omega_I \circ \varkappa_I \circ \Pi_I(f)$ is nonzero. This implies that $\Pi_I(f)$ is nonzero. In other words, it forms a basis of $\Hom(\Delta_I(\lambda), \nabla_I(\lambda))$, and the desired claim is proved.
\end{proof}

\begin{rmk}
We have seen in the course of the proof of Lemma~\ref{lem:weyl-can} that, if $\lambda \in \bXpp_I$, the image under $\Pi_I$ of any nonzero morphism from $\Delta(w_I \lambda)$ to $\nabla(w_I \lambda)$ is nonzero. This property can also be deduced directly from Proposition~\ref{prop:pi-i-push}.
\end{rmk}

\subsection{The parabolic induction theorem}

\begin{prop}\label{prop:coh-degr}
For any $\cF$, $\cG$ in $\Db\Coh^{\dot G \times \Gm}(\tcN_I)$,
the functor
\[
\Omega_I \circ \varkappa_I: \Db\Coh^{\dot G \times \Gm}(\tcN_I) \to \Db\Rep_{I}(G)
\]
and the isomorphism~\eqref{eqn:varkappa-shift}
induce an isomorphism
\[
\bigoplus_{n \in \Z} \Hom_{\Db\Coh^{\dot G \times \Gm}(\tcN_I)}(\cF, \cG\la n\ra[n]) \simto
\Hom_{\Db\Rep_{I}(G)}(\Omega_I(\varkappa_I(\cF)), \Omega_I(\varkappa_I(\cG))).
\]
\end{prop}

\begin{proof}
It suffices to check this property in the case when $\cF = \Delta_I(\lambda)$ and $\cG = \nabla_I(\mu)[k]$ for some $\lambda, \mu \in \bXpp_I$ and $k \in \Z$, since these objects (together with their grading shifts) generate $\Db\Coh^{\dot G \times \Gm}(\tcN_I)$ as a triangulated category (see Lemma~\ref{lem:generators-tcNI} and~\cite[Lemma~3]{bez:ctm}).  If $\lambda \ne \mu$, or if $\lambda = \mu$ but $k \ne 0$, then the left-hand side vanishes by Corollary~\ref{cor:exotic-orth}, and the right-hand side vanishes by Proposition~\ref{prop:exotic-weyl} and~\cite[Proposition~II.4.13]{jantzen} (see also the bijection~\eqref{eqn:bijection-weights-I}).

Suppose now that $\lambda = \mu$ and that $k = 0$.  Then Corollary~\ref{cor:exotic-orth} tells us that there is only one nonzero summand in the left-hand side, corresponding to $n = 0$, and that that term is $1$-dimensional.  The right-hand side is also $1$-dimensional, and Lemma~\ref{lem:weyl-can} tells us that the induced map in this case is an isomorphism.
\end{proof}

Recall from Lemma~\ref{lem:induce-block} that for $M \in \Db_\Stein(P_I)$, the object $R\Ind_{P_I}^G(M) \in \Db\Rep(G)$ actually lies in the subcategory $\Db\Rep_{I}(G)$.  In a minor abuse of notation, we henceforth denote the composition
\[
\Db_\Stein(P_I) \xrightarrow{\inc} \Db\Rep(P_I) \xrightarrow{R\Ind_{P_I}^G} \Db\Rep(G) \xrightarrow{\pr_I} \Db\Rep_{I}(G)
\]
simply by $R\Ind_{P_I}^G: \Db_\Stein(P_I) \to \Db\Rep_{I}(G)$.

\begin{thm}[Induction theorem]
\label{thm:induction-thm}
The functor
\[
R\Ind_{P_I}^G: \Db_\Stein(P_I) \to \Db\Rep_{I}(G)
\]
is an equivalence of triangulated categories.
\end{thm}

\begin{proof}
Let $\cF, \cG \in \Db\Coh^{\dot G \times \Gm}(\tcN_I)$, and consider the commutative diagram
\[
\xymatrix@C=0.2cm{
\bigoplus_{n \in \Z} \Hom_{\Db\Coh^{\dot G \times \Gm}(\tcN_I)}(\cF, \cG\la n\ra[n])
\ar[r]^-{\varkappa_I} \ar@/_5ex/[ddr]_-{\Omega_I \circ \varkappa_I} &
  \Hom_{\Dfg_{\dot P_I}(\bL_I)}(\varkappa_I(\cF), \varkappa_I(\cG)) \ar[d]^-{\psi_I} \\
& \Hom_{\Db_\Stein(P_I)}(\psi_I\varkappa_I(\cF), \psi_I\varkappa_I(\cG)) \ar[d]^-{R\Ind_{P_I}^G} \\
& \Hom_{\Db\Rep_{I}(G)}(\Omega_I(\varkappa_I(\cF)), \Omega_I(\varkappa_I(\cG)))}
\]
By Theorem~\ref{thm:formality},~\eqref{eqn:varkappa-degr}, and Proposition~\ref{prop:coh-degr}, the arrows labelled $\psi_I$, $\varkappa_I$, and $\Omega_I \circ \varkappa_I$ are isomorphisms, so the remaining arrow is an isomorphism as well.  Recall that if $\cF = \cO_{\tcN_I} \otimes V$ with $V \in \Rep(\dot P_I)$, then $\psi_I\varkappa_I(\cF) \cong \St_I \otimes \For^{\dot P_I}_{P_I} (V)$. As observed in the proof of Lemma~\ref{lem:generators-SteinPK}, such $P_I$-modules generate $\Db_\Stein(P_I)$ as a triangulated category. So we deduce that the map
\[
\Hom_{\Db_\Stein(P_I)}(M,N) \to \Hom_{\Db\Rep_{I}(G)}(R\Ind_{P_I}^G(M), R\Ind_{P_I}^G(N))
\]
induced by $R\Ind_{P_I}^G$
is an isomorphism for all $M,N \in \Db_{\Stein}(P_I)$.  In other words,
\[
R\Ind_{P_I}^G: \Db_\Stein(P_I) \to \Db\Rep_{I}(G)
\]
is fully faithful.  The category $\Db\Rep_{I}(G)$ is generated by the Weyl modules (or dual Weyl modules) appearing in Proposition~\ref{prop:exotic-weyl}, so our functor is essentially surjective as well, and hence an equivalence.
\end{proof}

\section{The graded Finkelberg--Mirkovi\'c conjecture}
\label{sec:fm-conjecture}

\subsection{Mixed derived category and mixed perverse sheaves on affine Grassmannians}
\label{ss:mixed-der}

Let $\dot T^\vee$ be the complex torus which is Langlands dual to $\dot T$ (i.e.~whose weight lattice is dual to the weight lattice of $\dot T$),
and let
$\dot G^\vee$ be the unique (up to isomorphism) connected complex reductive group with maximal torus $\dot T^\vee$ such that the root datum of $(\dot G^\vee, \dot T^\vee)$ is dual to that of $(\dot G, \dot T)$. Let also $\dot B^\vee_+ \subset \dot G^\vee$, resp.~$\dot B^\vee \subset \dot G^\vee$, be the Borel subgroup whose set of roots is $\Phi^\vee_+$, resp.~$- \Phi^\vee_+$. (Recall that we have identified characters of $\dot T$ with characters of $T$; in this way $\Phi$ is also the root system of~$(\dot G, \dot T)$.)

Let $\scK:=\C( \hspace{-1pt} (z) \hspace{-1pt})$, and $\scO:=\C[\hspace{-1pt}[ z ] \hspace{-1pt} ]$, and 
consider the loop group $\dot G^\vee(\scK)$ and its subgroup $\dot G^\vee(\scO)$.  Recall that the \emph{affine Grassmannian} for $\dot G^\vee$ is a complex ind-variety $\Gr$ whose set of $\C$-points identifies in a natural way with the quotient $\dot G^\vee(\scK) / \dot G^\vee(\scO)$. We let $\Iw \subset \dot G^\vee(\scO)$ be the Iwahori subgroup associated with $\dot B^\vee_+$, i.e.~the inverse image of $\dot B^\vee_+$ under the natural morphism $\dot G^\vee(\scO) \to \dot G^\vee$. To any $\lambda \in \bX$ (considered as a cocharacter of $\dot T^\vee$) one can associate in a natural way an element $z^\lambda \in \dot T^\vee(\scK)$, hence a point $L_\lambda = z^\lambda \dot G^\vee(\scO) \in \Gr$, and if we set $\Gr_\lambda := \Iw \cdot L_\lambda$, then each $\Gr_\lambda$ is isomorphic to an affine space and we have the Bruhat decomposition
\[
\Gr = \bigsqcup_{\lambda \in \bX} \Gr_\lambda.
\]

Following~\cite{modrap2}, we define the mixed derived category $\Dmix_{(\Iw)}(\Gr, \bk)$ of $\Iw$-construc\-tible $\bk$-sheaves on $\Gr$ as $\Kb(\Par_{(\Iw)}(\Gr, \bk))$, the bounded homotopy category of the additive category of $\Iw$-constructible parity complexes on $\Gr$ (in the sense of~\cite{jmw}). As explained in~\cite[\S 3.1]{modrap2}, this category admits a natural t-structure, called the \emph{perverse t-structure}, and whose heart will be denoted $\Perv^\mix_{(\Iw)}(\Gr, \bk)$. It also admits a ``Tate twist'' autoequivalence $\langle 1 \rangle$ which is t-exact, see~\cite[\S 2.2]{modrap2}. With respect to this autoequivalence, the category $\Perv^\mix_{(\Iw)}(\Gr, \bk)$ has a natural structure of a graded quasi-hereditary category with poset $\bX$ (for the order induced by inclusion of closures of orbits $\Gr_\lambda$); see~\cite[\S 3.2]{modrap2}. (Note that the assumption~\cite[(A2)]{modrap2} holds in the present setting by~\cite[Corollary~4.8]{modrap2}.) We will denote by $\cJ_!(\lambda)$, $\cJ_*(\lambda)$, $\IC^\mix_\lambda$ and $\cT(\lambda)$ the corresponding standard, costandard, simple, and tilting objects respectively. (In the conventions of~\cite{modrap2}, the objects $\cJ_!(\lambda)$, $\cJ_*(\lambda)$, $\cT(\lambda)$ would rather be denoted $\Delta^\mix_\lambda$, $\nabla^\mix_\lambda$, $\mathcal{T}^\mix_\lambda$.)

\begin{rmk}
\label{rmk:order-Pervmix}
It follows from the proof of Lemma~\ref{lem:properties-order} and Remark~\ref{rmk:comparison-ar-mr}\eqref{it:comp-ar-mr} that the order on $\bX$ induced by inclusions of closures of orbits $\Gr_\lambda$ is precisely the order $\leq$ introduced in~\S\ref{ss:orders-X}.
\end{rmk}

Now that this notation is introduced, we can finally give the proof of Lemma~\ref{lem:properties-order}.

\begin{proof}[Proof of Lemma~{\rm \ref{lem:properties-order}}]
Let $\Iw_- \subset \dot G^\vee(\scO)$ be the Iwahori subgroup associated with the Borel subgroup $\dot B^\vee$, and consider the ``opposite'' affine Grassmannian
\[
\Gr' := \dot G^\vee(\scO) \backslash \dot G^\vee(\scK).
\]
This ind-variety is endowed with natural actions of $\Iw_-$ and $\dot G^\vee(\scO)$ induced by right multiplication on $\dot G^\vee(\scK)$. For any $\lambda \in \bX$ we set $\Gr'_\lambda := \dot G^\vee(\scO) \backslash \dot G^\vee(\scO) \cdot z^\lambda \cdot \Iw_-$. Then the length function and Bruhat order on $\Waff$ describe dimensions of $\Iw_-$-orbits and inclusions between the closures of these orbits, respectively, in $\Iw_- \backslash \dot G^\vee(\scK)$. We deduce that we have
\begin{equation}
\label{eqn:order-Gr'}
\ell(w_\lambda) = \dim(\Gr'_\lambda) \quad \text{and} \quad \bigl( \lambda \leq \mu \ \Leftrightarrow \ \overline{\Gr'_\lambda} \subset \overline{\Gr'_\mu} \bigr).
\end{equation}
When $\lambda \in \bX^+$ we also set $(\Gr')^\lambda = \dot G^\vee(\scO) \backslash \dot G^\vee(\scO) z^\lambda \dot G^\vee(\scO)$. Then it is well known that
\begin{equation}
\label{eqn:order-Gr'-dominant}
\overline{(\Gr')^\lambda} \subset \overline{(\Gr')^\mu} \ \Leftrightarrow \ \lambda \preceq \mu.
\end{equation}
Moreover, $\Gr'_\lambda$ is dense in $(\Gr')^\lambda$.

Now we can prove part~\eqref{it:orders-dominant}. Let $\lambda, \mu \in \bX^+$. Then as explained above $\Gr'_\lambda$ is dense in $(\Gr')^\lambda$ and $\Gr'_\mu$ is dense in $(\Gr')^\mu$. We deduce that $\overline{\Gr'_\lambda} \subset \overline{\Gr'_\mu}$ if and only if $\overline{(\Gr')^\lambda} \subset \overline{(\Gr')^\mu}$. Comparing with~\eqref{eqn:order-Gr'} and~\eqref{eqn:order-Gr'-dominant}, we deduce that $\lambda \leq \mu$ if and only if $\lambda \preceq \mu$.

Then we prove part~\eqref{it:orders-orbit}. Let $\dot P_\lambda^\vee$ be the stabilizer in $\dot G^\vee$ of the point $\dot G^\vee(\scO) \cdot z^{w_S(\lambda)} \in \Gr'$. Then $\dot P_\lambda^\vee$ is the parabolic subgroup containing $\dot B^\vee$ associated with the subset $K=\{w_S s w_S, \, s \in I\}$ of $S$. Moreover there exists a natural morphism $(\Gr')^\lambda \to \dot P_\lambda^\vee \backslash \dot G^\vee$ which is an affine fibration and sends the point $\dot G^\vee(\scO) \cdot z^{w_S(\lambda)}$ to the base point $\dot P_\lambda^\vee \cdot 1$. For any $w \in W$, this fibration restricts to a fibration $\Gr'_{w(\lambda)} \to \dot P_\lambda^\vee \backslash \dot P_\lambda^\vee (w_S w^{-1}) \dot B^\vee$ with the same fiber. Hence the inclusions between closures of orbits in $(\Gr')^\lambda$ are governed by the inclusions between closures of $\dot B^\vee$-orbits in $\dot P_\lambda^\vee \backslash \dot G^\vee$, which is itself governed by the Bruhat order on $W_K \backslash W$. More precisely, let $v,w \in W$ be such that $v=\min(vW_I)$ and $w=\min(wW_I)$. Then using~\eqref{eqn:order-Gr'} we have
\begin{multline*}
v(\lambda) \leq w(\lambda) \ \Leftrightarrow \ \overline{\Gr'_{v(\lambda)}} \subset \overline{\Gr'_{w(\lambda)}} \ \Leftrightarrow \ \overline{\Gr'_{v(\lambda)}} \cap (\Gr')^\lambda \subset \overline{\Gr'_{w(\lambda)}} \cap (\Gr')^\lambda \\
\Leftrightarrow \ \overline{\dot P_\lambda^\vee \backslash \dot P_\lambda^\vee (w_S v^{-1}) \dot B^\vee} \subset \overline{\dot P_\lambda^\vee \backslash \dot P_\lambda^\vee (w_S w^{-1}) \dot B^\vee}.
\end{multline*}
Now we have $w_S v^{-1} = \max(W_K w_S v^{-1})$ and $w_S w^{-1} = \max(W_K w_S w^{-1})$. Hence this last condition is equivalent to $w_S v^{-1} \leq w_S w^{-1}$, and finally to $w \leq v$, which finishes the proof.

Finally we prove part~\eqref{it:order-conv}. If $\lambda \leq \mu$ then $\overline{\Gr'_\lambda} \subset \overline{\Gr'_\mu}$ (see~\eqref{eqn:order-Gr'}), hence $\overline{\Gr'_\lambda} \subset \overline{(\Gr')^{\dom(\mu)}}$, which implies that $\overline{(\Gr')^{\dom(\lambda)}} \subset \overline{(\Gr')^{\dom(\mu)}}$, and finally that $\dom(\lambda) \preceq \dom(\mu)$ (see~\eqref{eqn:order-Gr'-dominant}). By~\eqref{eqn:conv-order}, this implies that $\lambda \in \conv(\mu)$.
\end{proof}

\subsection{Geometric Satake equivalence}
\label{ss:geometric-Satake}

Let $\Perv_\sph(\Gr,\bk)$ be the abelian category of (ordinary, i.e.~non-mixed) $\dot G^\vee(\scO)$-equivariant perverse sheaves on $\Gr$. (The $\dot G^\vee(\scO)$-orbits on $\Gr$ are sometimes called the \emph{spherical orbits}, and the objects of $\Perv_\sph(\Gr,\bk)$ are then called \emph{spherical perverse sheaves}.) This category is equipped with a symmetric monoidal structure given by the convolution product $\star$; moreover there exists an equivalence of abelian tensor categories
\begin{equation}\label{eqn:satake}
\Sat: (\Perv_\sph(\Gr,\bk), \star) \simto (\Repf(\dot G), \otimes),
\end{equation}
which sends the intersection cohomology sheaf associated with an orbit $\dot G^\vee(\scO) \cdot L_\lambda$ with $\lambda \in \bX^+$ to the simple $\dot G$-module with highest weight $\lambda$.
This equivalence is known as the \emph{geometric Satake equivalence}; in this generality, it is due to Mirkovi{\'c}--Vilonen~\cite{mv:gld}.

Following~\cite[\S 2.4]{ar:agsr}, one can define a right action of $\Perv_\sph(\Gr,\bk)$ on $\Dmix_{(\Iw)}(\Gr, \bk)$ as follows. Let $\PerPar_\sph(\Gr,\bk)$ be the subcategory of $\Perv_\sph(\Gr,\bk)$ consisting of objects which are parity. In view of the geometric Satake equivalence~\eqref{eqn:satake}, the category $\Perv_\sph(\Gr,\bk)$ admits a natural structure of highest weight category, and the objects $\PerPar_\sph(\Gr,\bk)$ are exactly the tilting objects for this structure. (In most cases, this follows from the main result of~\cite{jmw2}. The general case is discussed in detail in~\cite[\S 1.5]{mr:etsps}.) In particular, the natural functor
\[
\Kb(\PerPar_{\sph}(\Gr,\bk)) \to \Db \Perv_{\sph}(\Gr,\bk)
\]
is an equivalence of categories, so that we can consider $\Perv_{\sph}(\Gr,\bk)$ as a full subcategory in $\Kb(\PerPar_{\sph}(\Gr,\bk))$. The convolution product induces a symmetric monoidal structure on $\PerPar_{\sph}(\Gr,\bk)$, and hence on $\Kb(\PerPar_{\sph}(\Gr,\bk))$, so that the monoidal structure can also be recovered from this equivalence (see~\cite{jmw2}). In conclusion, to construct an action of $\Perv_\sph(\Gr,\bk)$ on $\Dmix_{(\Iw)}(\Gr, \bk)$ it suffices to construct an action of $\Kb(\PerPar_{\sph}(\Gr,\bk))$ on $\Dmix_{(\Iw)}(\Gr, \bk)$. Now the convolution product also restricts to a bifunctor
\[
\Par_{(\Iw)}(\Gr, \bk) \times \PerPar_{\sph}(\Gr,\bk) \to \Par_{(\Iw)}(\Gr, \bk);
\]
see~\cite[Theorem~4.8]{jmw}.
Passing to bounded homotopy categories we deduce the desired action of the monoidal category $\Kb(\PerPar_{\sph}(\Gr,\bk))$ on $\Dmix_{(\Iw)}(\Gr, \bk)$.

\subsection{Relation with coherent sheaves on the Springer resolution}
\label{ss:ar-mr-thms}

The following theorem is the main result of~\cite{ar:agsr}; see also~\cite{mr:etsps} for a different construction of such an equivalence. (See Remark~\ref{rmk:comparison-ar-mr}\eqref{it:comp-ar-mr} below for a comparison of the two constructions.)

\begin{thm}
\label{thm:ar-equivalence}
There exists an equivalence of triangulated categories
\[
P : \Dmix_{(\Iw)}(\Gr, \bk) \simto \Db\Coh^{\dot G \times \Gm}(\tcN)
\]
with the following properties:
\begin{enumerate}
\item
\label{it:ar-shift}
there exists an isomorphism of functors $P \circ \langle 1 \rangle \cong \langle 1 \rangle [1] \circ P$;
\item
\label{it:ar-standard-costandard}
for any $\lambda \in \bX$,
there exist isomorphisms
\[
P(\cJ_!(\lambda)) \cong \Delta(\lambda), \quad P(\cJ_*(\lambda)) \cong \nabla(\lambda);
\]
\item
\label{it:ar-Satake}
for any $\cF$ in $\Dmix_{(\Iw)}(\Gr, \bk)$ and any $\cG \in \Perv_\sph(\Gr,\bk)$, there exists a bifunctorial isomorphism $P(\cF \star \cG) \cong P(\cF) \otimes \Sat(\cG)$.
\end{enumerate}
\end{thm}

\begin{rmk}
\label{rmk:comparison-ar-mr}
\begin{enumerate}
\item
The difference of sign between property~\eqref{it:ar-shift} in Theorem~\ref{thm:ar-equivalence} and the statement of~\cite[Theorem~1.1]{ar:agsr} is due to the difference of conventions in the definition of the functor $\langle 1 \rangle$ for coherent sheaves in~\cite{ar:agsr} and in the present paper. Property~\eqref{it:ar-standard-costandard} is not stated explicitly in~\cite[Theorem~1.1]{ar:agsr}, but it appears in the proof of~\cite[Theorem~8.3]{ar:agsr}.
\item
\label{it:comp-ar-mr}
In~\cite{mr:etsps}, a different construction of an equivalence between $\Dmix_{(\Iw)}(\Gr, \bk)$ and $\Db\Coh^{\dot G \times \Gm}(\tcN)$ is given. The main difference between the two constructions is that the compatibility with the geometric Satake equivalence (Property~\eqref{it:ar-Satake}) is not clear from the proof in~\cite{mr:etsps}. Another difference appears in the labeling of objects: the equivalence of~\cite{mr:etsps} exchanges the (co)standard mixed perverse sheaf labeled by $\lambda$ and the (co)standard exotic sheaf labeled by $-\lambda$. To resolve this apparent contradiction, one should recall that the Iwahori subgroup used in~\cite{mr:etsps} is the \emph{negative} one, denoted $\Iw_-$ in the proof of Lemma~\ref{lem:properties-order}. Hence, if $\varphi$ is an automorphism of $\dot G^\vee$ as in the proof of~\cite[Corollary~II.1.16]{jantzen}, then we have $\varphi(\dot B^\vee)=\dot B^\vee_+$ and $\varphi(t)=t^{-1}$ for $t \in \dot T^\vee$, so that the induced automorphism of $\Gr$ sends the orbit $\Iw_- \cdot L_\lambda$ to the orbit denoted $\Gr_{-\lambda}$ in the present paper; hence the induced equivalence $\Dmix_{(\Iw_-)}(\Gr, \bk) \simto \Dmix_{(\Iw)}(\Gr, \bk)$ will send the object denoted $\Delta^\mix_{\lambda}$ in~\cite{mr:etsps} to the object $\cJ_!(-\lambda)$ of the present paper, and similarly for costandard objects. Using the notation introduced in the proof of Lemma~\ref{lem:properties-order}, this comment also shows that the anti-automorphism $g \mapsto \varphi(g)^{-1}$ induces an isomorphism of varieties $\Gr \simto \Gr'$ which sends $\Gr_\lambda$ to $\Gr'_\lambda$.
\end{enumerate}
\end{rmk}

\subsection{The Finkelberg--Mirkovi\'c conjecture}
\label{ss:fmconjecture}

The category $\Repf(\dot G)$ embeds in the category $\Rep_\varnothing(G)$ via the functor $V \mapsto \For^{\dot G}_G(V)$ associated with the Frobenius morphism $G \to \dot G$. On the other hand, according to~\cite[Proposition~2.1]{mv:gld}, the category $\Perv_{\sph}(\Gr, \bk)$ is equivalent (via the natural forgetful functor) to the category of perverse sheaves on $\Gr$ constructible with respect to the $\dot G^\vee(\scO)$-orbits, so it embeds in the category $\Perv_{(\Iw)}(\Gr,\bk)$ of (ordinary) $\Iw$-constructible perverse sheaves.
In~\cite[\S1.5]{fm:sif1}, M.~Finkelberg and I.~Mirkovi{\'c} conjectured that~\eqref{eqn:satake} can be extended to an equivalence between these larger categories.
In the statement below, we denote by $\IC_\lambda$ the simple perverse sheaf associated with the $\Iw$-orbit $\Gr_\lambda$.  Recall also that the convolution action of the category $\Perv_\sph(\Gr,\bk)$ on the $\Iw$-constructible derived category $\Db_{(\Iw)}(\Gr,\bk)$ restricts to a right action of $\Perv_\sph(\Gr,\bk)$ on $\Perv_{(\Iw)}(\Gr,\bk)$. (This fact is proved for $\overline{\mathbb{Q}}_\ell$-coefficients in the \'etale setting in~\cite[Comments after Proposition~6]{gaitsgory}; the same proof applies also in our setting.)

\begin{conj}[Finkelberg--Mirkovi\'c~\cite{fm:sif1}]
\label{conj:fm}
There exists an equivalence of highest weight categories
\[
Q: \Perv_{(\Iw)}(\Gr,\bk) \simto \Rep_\varnothing(G)
\]
such that
\begin{enumerate}
\item 
for any $\lambda \in \bX$, we have $Q(\IC_\lambda) \cong \irr(w_\lambda \bullet 0)$;
\item 
for any $\cF \in \Perv_{(\Iw)}(\Gr,\bk)$ and any $\cG \in \Perv_\sph(\Gr,\bk)$, there exists a bifunctorial isomorphism $Q(\cF \star \cG) \cong Q(\cF) \otimes \For^{\dot G}_G(\Sat(\cG))$.
\end{enumerate}
\end{conj}

A characteristic-zero analogue of this conjecture (involving the principal block of a quantum group at a root of unity) was proved in~\cite{abg:qglg}.

\subsection{A graded version of the Finkelberg--Mirkovi\'c conjecture}
\label{ss:graded-fm-conj}

Conjecture~\ref{conj:fm} remains open at the moment. Our goal in this section is to establish a ``graded version'' of it, involving the following notion from~\cite{bgs}.

\begin{defn}
Let $\mathcal{A}$ be a $\bk$-linear abelian category in which every object has finite length.  A \emph{grading} on $\mathcal{A}$ is a triple $(\mathcal{M},v,\varepsilon)$ where $\mathcal{M}$ is a $\bk$-linear abelian category equipped with an autoequivalence $\la 1\ra: \mathcal{M} \to \mathcal{M}$, $v: \mathcal{M} \to \mathcal{A}$ is an exact functor whose essential image includes all simple objects in $\mathcal{A}$, and $\varepsilon: v \simto v \circ \la 1\ra$ is an isomorphism of functors such that the induced map
\[
\bigoplus_{n \in \Z} \Ext^k_{\mathcal{M}}(M,N\la n\ra) \to \Ext^k_{\mathcal{A}}(v(M),v(N))
\]
is an isomorphism for all $M,N \in \mathcal{M}$ and all $k \in \Z$.
\end{defn}

Our first result is that the convolution action of $\Perv_\sph(\Gr,\bk)$ on $\Dmix_{(\Iw)}(\Gr,\bk)$ introduced in~\S\ref{ss:geometric-Satake} is t-exact, in the following sense.

\begin{prop}
\label{prop:convolution-exact}
For any $\cF \in \Perv^\mix_{(\Iw)}(\Gr,\bk)$ and any $\cG \in \Perv_\sph(\Gr,\bk)$, we have $\cF \star \cG \in \Perv^\mix_{(\Iw)}(\Gr,\bk)$.
\end{prop}

This proposition will be proved simultaneously with the following theorem, which we view as a ``graded version'' of Conjecture~\ref{conj:fm}. In this statement, for $\mu \in \bX^+$, we denote by $\tilt(\mu)$ the tilting $G$-module with highest weight $\mu$.

\begin{thm}
\label{thm:graded-fm}
There exists an exact functor
\[
\bQ: \Perv^\mix_{(\Iw)}(\Gr,\bk) \to \Rep_\varnothing(G)
\]
together with an isomorphism $\varepsilon: \bQ \simto \bQ \circ \la 1\ra$ such that $(\Perv^\mix_{(\Iw)}(\Gr,\bk), \bQ, \varepsilon)$ is a grading on $\Rep_\varnothing(G)$.  In addition,
\begin{enumerate}
\item
\label{it:Q-images}
for any $\lambda \in \bX$, we have
\begin{gather*}
\bQ(\cJ_!(\lambda)) \cong \weyl(w_\lambda \bullet 0), \quad
\bQ(\cJ_*(\lambda)) \cong \coweyl(w_\lambda \bullet 0), \\
\bQ(\IC^\mix_\lambda) \cong \irr(w_\lambda \bullet 0), \quad \bQ(\cT(\lambda)) \cong \tilt(w_\lambda \bullet 0);
\end{gather*}
\item
\label{it:Q-Satake}
for any $\cF \in \Perv^\mix_{(\Iw)}(\Gr,\bk)$ and any $\cG \in \Perv_\sph(\Gr,\bk)$, there exists a bifunctorial isomorphism $\bQ(\cF \star \cG) \cong \bQ(\cF) \otimes \For^{\dot G}_G(\Sat(\cG))$.
\end{enumerate}
\end{thm}

\begin{rmk}
We expect that there also exists a functor
$v : \Perv^\mix_{(\Iw)}(\Gr, \bk) \to \Perv_{(\Iw)}(\Gr,\bk)$
and an isomorphism $\varepsilon : v \circ \langle 1 \rangle \simto v$ such that $(\Perv^\mix_{(\Iw)}(\Gr, \bk),v,\epsilon)$ is a grading on $\Perv_{(\Iw)}(\Gr,\bk)$. However, this fact is not known at present. (In~\cite{modrap2} we have constructed such a structure for finite-dimensional flag varieties of reductive groups and coefficients of good characteristic.)
\end{rmk}

\begin{proof}[Proof of Proposition~{\rm \ref{prop:convolution-exact}} and Theorem~{\rm \ref{thm:graded-fm}}]
Define
\[
\bQ: \Dmix_{(\Iw)}(\Gr,\bk) \to \Db\Rep_\varnothing(G)
\]
to be the composition
\[
\Dmix_{(\Iw)}(\Gr,\bk) \xrightarrow{P} \Db\Coh^{\dot G \times \Gm}(\tcN)
\xrightarrow{\varkappa_\varnothing}
\Dfg_{\dot B}(\bL_\varnothing)
\xrightarrow{\Omega_\varnothing}
\Db\Rep_\varnothing(G).
\]
In view of Property~\eqref{it:ar-shift} in Theorem~\ref{thm:ar-equivalence} and~\eqref{eqn:varkappa-shift}, we have functorial isomorphisms
\[
\bQ(\cF\la 1\ra) \cong \Omega_\varnothing(\varkappa_\varnothing(P(\cF)\la 1\ra [1])) \cong \Omega_\varnothing(\varkappa_\varnothing(P(\cF))) \cong \bQ(\cF)
\]
for any $\cF$ in $\Dmix_{(\Iw)}(\Gr,\bk)$.
In other words, there exists a natural isomorphism $\varepsilon: \bQ \simto \bQ \circ \la 1\ra$.

Let us next show that $\bQ$ is exact. 
In view of~\cite[Proposition~3.4]{modrap2}, it is enough to show that $\bQ(\cJ_!(\lambda))$ and $\bQ(\cJ_*(\lambda))$ lie in $\Rep_\varnothing(G)$. However, by Proposition~\ref{prop:exotic-weyl} and Property~\eqref{it:ar-standard-costandard} in Theorem~\ref{thm:ar-equivalence}, we have
\begin{equation}
\label{eqn:isom-Q-standard-costandard}
\bQ(\cJ_!(\lambda)) \cong \weyl(w_\lambda \bullet 0), \qquad \bQ(\cJ_*(\lambda)) \cong \coweyl(w_\lambda \bullet 0).
\end{equation}
This proves the desired exactness, and also the first two isomorphisms in~\eqref{it:Q-images}.

Proposition~\ref{prop:coh-degr} and Theorem~\ref{thm:ar-equivalence} imply that for any $\cF,\cG \in \Perv^\mix_{(\Iw)}(\Gr,\bk)$, $\bQ$ induces an isomorphism
\begin{equation}\label{eqn:bq-degr}
\bigoplus_{n \in \Z} \Hom_{\Dmix_{(\Iw)}(\Gr,\bk)}(\cF,\cG\la n\ra[k]) \simto \Ext^k_{\Rep_\varnothing(G)}(\bQ(\cF), \bQ(\cG)).
\end{equation}
On the other hand, we know from~\cite[Lemma~3.15]{modrap2} that the realization functor provides an equivalence $\Db\Perv^\mix_{(\Iw)}(\Gr,\bk) \cong \Dmix_{(\Iw)}(\Gr,\bk)$.  This means that on the left-hand side of~\eqref{eqn:bq-degr}, we can replace $\Hom(\cF,\cG\la n\ra[k])$ by $\Ext^k_{\Perv^\mix_{(\Iw)}(\Gr,\bk)}(\cF,\cG\la n\ra)$. 

Next, the simple object $\IC^\mix_\lambda$ is the image of any nonzero morphism $\cJ_!(\lambda) \to \cJ_*(\lambda)$, while the simple object $\irr(w_\lambda \bullet 0)$ is the image of any nonzero morphism $\weyl(w_\lambda \bullet 0) \to \coweyl(w_\lambda \bullet 0)$.  In view of~\eqref{eqn:isom-Q-standard-costandard}, and since $\bQ$ is exact and faithful (as follows from~\eqref{eqn:bq-degr}),
we find that
\[
\bQ(\IC^\mix_\lambda) \cong \irr(w_\lambda \bullet 0).
\]
We have thus shown that $(\Perv^\mix_{(\Iw)}(\Gr,\bk),\bQ,\varepsilon)$ is a grading on $\Rep_\varnothing(G)$.

We now turn to the fourth isomorphism in~\eqref{it:Q-images}. The exactness of $\bQ$ and~\eqref{eqn:isom-Q-standard-costandard} (together with Lemma~\ref{lem:orders-dominant-weights} and Remark~\ref{rmk:order-Pervmix}) imply that $\bQ(\cT(\lambda))$ is a tilting $G$-module which admits $\tilt(w_\lambda \bullet 0)$ as a direct summand. Using the isomorphism~\eqref{eqn:bq-degr} for $\cF=\cG=\cT(\lambda)$ and $k=0$, together with~\cite[Theorem~3.1]{gg}, we see that the ring $\End(\bQ(\cT(\lambda)))$ is local, and hence that $\bQ(\cT(\lambda))$ is indecomposable, which proves that $\bQ(\cT(\lambda)) \cong \tilt(w_\lambda \bullet 0)$.

Finally, using Property~\eqref{it:ar-Satake} in Theorem~\ref{thm:ar-equivalence},~\eqref{eqn:isom-psi-tensoring}, and the tensor identity,
one can check that for any $\cF \in \Dmix_{(\Iw)}(\Gr,\bk)$ and $\cG \in \Perv_\sph(\Gr,\bk)$, there exists a bifunctorial isomorphism
\[
\bQ(\cF \star \cG) \cong \bQ(\cF) \otimes \For^{\dot G}_G(\Sat(\cG))
\]
in $\Db\Rep_\varnothing(G)$.  In particular, if $\cF \in \Perv^\mix_{(\Iw)}(\Gr,\bk)$, then $\bQ(\cF \star \cG)$ lies in $\Rep_\varnothing(G)$.  Now, $\bQ$ is t-exact, and~\eqref{eqn:bq-degr} implies that $\bQ$ kills no nonzero object.  Since $\bQ(\cF \star \cG)$ has cohomology only in degree~$0$, $\cF \star \cG$ must have perverse cohomology only in degree~$0$.  In other words, $\cF \star \cG$ is perverse. This proves Proposition~\ref{prop:convolution-exact}, and also Property~\eqref{it:Q-Satake} of the theorem.
\end{proof}

\subsection{Application to characters of simple $G$-modules}

It is well known that the classes of the modules $\coweyl(w_\lambda \bullet 0)$ form a $\Z$-basis of the Grothendieck group $[\Rep_\varnothing(G)]$ of the abelian category $\Rep_\varnothing(G)$.
As a direct application of Theorem~\ref{thm:graded-fm}, we obtain the following result. 

\begin{prop}
\label{prop:ch-simples}
For any $\lambda, \mu \in \bX$, the coefficients of $[\coweyl(w_\lambda \bullet 0)]$ in the expansion of $[\irr(w_\mu \bullet 0)]$ on the basis $\bigl( [\coweyl(w_\nu \bullet 0)] : \nu \in \bX \bigr)$ of $[\Rep_\varnothing(G)]$ is
\[
\sum_{\substack{i,j \in \Z}} (-1)^{i} \cdot \dim_\bk \bigl( \Hom_{\Dmix_{(\Iw)}(\Gr,\bk)}(\cJ_!(\lambda), \IC^\mix_\mu \langle j \rangle [i]) \bigr).
\]
\end{prop}

\begin{proof}
It is clear that the classes $[\cJ_*(\nu) \langle j \rangle]$ for $\nu \in \bX$ and $j \in \Z$ form a basis of the Grothendieck group $[\Perv^\mix_{(\Iw)}(\Gr,\bk)]$, and that the coefficient of $[\cJ_*(\nu) \langle j \rangle]$ in the expansion of the class of an object $\cF$ in this basis is equal to
\[
\sum_{i \in \Z} (-1)^{i} \cdot \dim_\bk \bigl( \Hom_{\Dmix_{(\Iw)}(\Gr,\bk)}(\cJ_!(\lambda), \cF \langle -j \rangle [i] \bigr).
\]
Applying $\bQ$ to the expansion of $[\IC^\mix_\mu]$, we obtain the desired equality.
\end{proof}

\begin{rmk}
\begin{enumerate}
\item
Since the characters of the induced modules $\coweyl(w_\lambda \bullet 0)$ are given by Weyl's character formula, see~\cite[Proposition~II.5.10]{jantzen}, determining the character of a module is equivalent to expressing the class of this module in terms of the classes $[\coweyl(w_\lambda \bullet 0)]$. In particular, this proposition gives a geometric character formula for all simple $G$-modules in $\Rep_\varnothing(G)$ (which, admittedly, is not computable in practice).
\item
Using adjunction, the sum in Proposition~\ref{prop:ch-simples} can be interpreted (up to sign) as the Euler characteristic of the costalk at $L_\lambda$ of $\IC^\mix_\mu$, in the sense of mixed derived categories.
\end{enumerate}
\end{rmk}

Let $Y_1 \subset \Gr$ be the union of the $\Iw$-orbits $\Gr_\lambda$ such that $w_\lambda \bullet 0$ is restricted, i.e.~satisfies $0 \leq \langle w_\lambda \bullet 0, \alpha^\vee \rangle <\ell$ for any simple root $\alpha$. (This subvariety is independent of $\ell$ under our assumptions, but is not closed in general.) On the other hand, let $Y_2 \subset \Gr$ be the union of the $\Iw$-orbits $\Gr_\lambda$ such that $\langle w_\lambda \bullet 0 +\rho, \alpha^\vee \rangle \leq \ell(\ell-h+2)$ for any positive root $\alpha$. (This subvariety is closed, but depends on $\ell$.) We will assume that $\ell \geq 2h-3$, so that $Y_1 \subset Y_2$ (see~\cite[\S 1.13]{williamson}). For any $\lambda \in \bX$, we will denote by $\mathcal{E}_\lambda \in \Par_{(\Iw)}(\Gr, \bk)$ the indecomposable object supported on $\overline{\Gr_\lambda}$ and whose restriction to $\Gr_\lambda$ is $\underline{\bk}_{\Gr_\lambda} [ \dim (\Gr_\lambda) ]$, and by $\IC_\lambda$ the (ordinary) intersection cohomology complex associated with the constant rank-$1$ local system on $\Gr_\lambda$.

Let $(h_{y,x} : x,y \in \Waff)$ be the affine Kazhdan--Lusztig polynomials for $\Waff$ normalized as in~\cite{soergel}. (To be really precise, the setting we consider does not fit exactly with the setting of~\cite{soergel} since $\Waff$ is not a Coxeter group. However a direct generalization applies, see e.g.~\cite[\S 1.8]{williamson-takagi}.)
Lusztig's conjecture~\cite{lusztig} predicts that
\[
[\irr(w \bullet 0)] = \sum_{\substack{y \in \Waffmin \\ y \leq w}} (-1)^{\ell(w)+\ell(y)} h_{w_S y, w_S w}(1) \cdot [\coweyl(y \bullet 0)]
\]
for any $w \in \Waffmin$ such that $\langle w \bullet 0 + \rho, \alpha^\vee \rangle \leq \ell(\ell-h+2)$ for any positive root $\alpha$. It was proved by Kato that it is equivalent to require this formula for any $w \in \Waffmin$ such that $w \bullet 0$ is restricted, see~\cite[\S 5.4]{kato}

\begin{thm}
\label{thm:criterion-lusztig}
Assume that $\ell \geq 2h-3$ and $\ell>h$.
\begin{enumerate}
\item
\label{it:criterion-1}
If $\cE_\lambda \cong \IC_\lambda$ for any $\lambda \in \bX$ such that $\Gr_\lambda \subset Y_1$, then
Lusztig's conjecture holds.
\item
\label{it:criterion-2}
If Lusztig's conjecture holds, then
for any $\lambda$ such that $\Gr_\lambda \subset Y_2$, we have $\mathcal{E}_\lambda \cong \IC_\lambda$.
\end{enumerate}
\end{thm}

\begin{rmk}
\begin{enumerate}
\item
It is well known that the condition in~\eqref{it:criterion-1} holds if $\ell \gg 0$, see~\cite{williamson-IC}. Hence Theorem~\ref{thm:criterion-lusztig} provides a new proof of Lusztig's conjecture in large characteristic.
\item
The criterion~\eqref{it:criterion-1} is similar to a criterion obtained by Fiebig, see~\cite[\S 7.5]{fiebig}. In his setting, no analogue of~\eqref{it:criterion-2} is obtained, however. (Note that Fiebig's criterion is in terms of the affine flag variety of $\dot G^\vee$, while ours is in terms of the affine Grassmannian.)
\item
Theorem~\ref{thm:criterion-lusztig} can also be compared with~\cite[Corollary~1.0.3]{soergel-relation}, which gives a similar result relating Lusztig's conjecture ``around the Steinberg weight" and parity complexes on the \emph{finite} flag variety.
\item
Combining~\eqref{it:criterion-1} and~\eqref{it:criterion-2}, we see that the absence of $\ell$-torsion in stalks and costalks of intersection cohomology complexes associated with orbits in $Y_1$ (which is equivalent to the condition $\mathcal{E}_\lambda \cong \IC_\lambda$, see~\cite{williamson-IC}) implies the same condition on $Y_2$, a portion which might be much larger (in particular if $\ell \gg 0$). The fact that this property follows from the ``ordinary" Finkelberg--Mirkovi\'c conjecture was noted in~\cite[Remarks after Theorem~2.14]{williamson}.
This property has no known geometric explanation.
\end{enumerate}
\end{rmk}

Before proving this statement we need a preliminary result. Recall that for any complex algebraic variety $X$ endowed with a finite algebraic stratification
\[
X = \bigsqcup_{s \in \mathscr{S}} X_s
\]
where each $X_s$ is isomorphic to an affine space, and for any field $\mathbb{F}$, we define the mixed derived category $\Dmix_{\mathscr{S}}(X,\mathbb{F})$ as the bounded homotopy category of the additive category $\Par_{\mathscr{S}}(X,\mathbb{F})$ of parity complexes on $X$.
This category has an autoequivalence $\{1\}$ induced by the cohomological shift in parity complexes, another autoequivalence $[1]$ given by the cohomological shift of complexes, and the Tate twist $\langle 1 \rangle = \{-1\}[1]$. 

In particular, if $s \in \mathscr{S}$, the category $\Dmix_{\{s\}}(X_s,\mathbb{F})$ is equivalent to the bounded derived category of the category of graded $\bk$-vector spaces, with $\bk$ (concentrated in internal degree $0$ and considered as a complex in degree $0$) corresponding to $\underline{\mathbb{F}}_{X_s}\{\dim(X_s)\}$, regarded as a complex concentrated in degree $0$, see~\cite[Lemma~3.1]{modrap2}. Here the Tate twist $\langle 1 \rangle$ corresponds to the shift-of-grading functor normalized as in~\cite[\S 3.1]{modrap2}, which we will also denote $\langle 1 \rangle$.

If we denote by
$i_s : X_s \to X$ the embedding, then we have
a standard object $\Delta^\mix_s = (i_s)_! \underline{\mathbb{F}}_{X_s} \{\dim(X_s)\}$ and a costandard object $\nabla^\mix_s =(i_s)_* \underline{\mathbb{F}}_{X_s} \{\dim(X_s)\}$ in the category $\Dmix_{\mathscr{S}}(X,\mathbb{F})$.
(See~\cite[\S\S2.4--2.5]{modrap2} for the definition of the functors $(i_s)_*$ and $(i_s)_!$; these functors are part of a ``recollement" formalism.)

\begin{lem}
\label{lem:characterization-parity-Dmix}
Let $\cF \in \Dmix_{\mathscr{S}}(X,\mathbb{F})$ be an object which satisfies
\[
\Hom_{\Dmix_{\mathscr{S}}(X,\mathbb{F})}(\Delta^\mix_s, \cF \{i\}[j]) = \Hom_{\Dmix_{\mathscr{S}}(X,\mathbb{F})}(\cF, \nabla^\mix_s \{i\}[j])=0
\]
unless $j=0$.
Then $\cF$ is isomorphic to an object of $\Par_{\mathscr{S}}(X,\mathbb{F})$, considered as a complex concentrated in degree $0$.
\end{lem}

\begin{proof}
We prove the claim by induction on $\#\mathscr{S}$.

We can assume that $\cF$ is indecomposable. We
choose $s \in \mathscr{S}$ such that $X_s$ is closed in $X$, and set $U:=\bigsqcup_{t \in \mathscr{S} \smallsetminus \{s\}} X_t$. We denote by $j : U \hookrightarrow X$ the embedding. If $j^* \cF=0$, then using the canonical triangle
\[
j_! j^* \cF \to \cF \to (i_s)_* i_s^* \cF \xrightarrow{[1]}
\]
in $\Dmix_{\mathscr{S}}(X,\mathbb{F})$
we see that $\cF \cong (i_s)_* i_s^* \cF$. The assumption implies that $i_s^* \cF$ is isomorphic to a complex concentrated in degree $0$, and the claim follows.

From now on we assume that $j^* \cF \neq 0$.
By induction, there exists a parity complex $\cE_U$ on $U$ such that $j^* \cF \cong \cE_U$. By the classification of parity complexes on $X$ (see~\cite{jmw}), there exists a parity complex $\cE$ on $X$ such that $\cE_U \cong j^* \cE$. Then we have canonical distinguished triangles
\[
(i_s)_* i_s^! \cF \to \cF \to j_* j^* \cF \xrightarrow{[1]} \quad \text{and} \quad (i_s)_* i_s^! \cE \to \cE \to j_* j^* \cE \xrightarrow{[1]}
\]
in $\Dmix_{\mathscr{S}}(X,\mathbb{F})$.

We claim that the functor $j^*$ induces a surjection
\begin{equation}
\label{eqn:j-surjection}
\Hom_{\Dmix_{\mathscr{S}}(X,\mathbb{F})}(\mathcal{H}, \mathcal{H}') \twoheadrightarrow \Hom_{\Dmix_{\mathscr{S}}(U,\mathbb{F})}(j^* \mathcal{H}, j^* \mathcal{H}')
\end{equation}
when $\mathcal{H}$ and $\mathcal{H}'$ are either $\cF$ or $\cE$. Indeed, using the distinguished triangle above for $\mathcal{H}'$ we obtain an exact sequence
\begin{multline*}
\Hom_{\Dmix_{\mathscr{S}}(X,\mathbb{F})}(\mathcal{H}, \mathcal{H}') \to \Hom_{\Dmix_{\mathscr{S}}(X,\mathbb{F})}(\mathcal{H}, j_* j^* \mathcal{H}') \\
\to \Hom_{\Dmix_{\mathscr{S}}(X,\mathbb{F})}(\mathcal{H}, (i_s)_* i_s^! \mathcal{H}' [1]),
\end{multline*}
where the first map is the morphism appearing in~\eqref{eqn:j-surjection}. Hence to conclude it suffices to prove that the third term in this exact sequence vanishes. However, by adjunction we have
\[
\Hom_{\Dmix_{\mathscr{S}}(X,\mathbb{F})}(\mathcal{H}, (i_s)_* i_s^! \mathcal{H}'[1]) \cong \Hom_{\Dmix_{\{s\}}(X_s,\mathbb{F})}(i_s^*\mathcal{H}, i_s^! \mathcal{H}' [1]).
\]
And either by our assumption (for the case of $\cF$) or by~\cite[Remark~2.7]{modrap2} (for the case of $\cE$), the 
objects $i_s^*\mathcal{H}$ and $i_s^! \mathcal{H}'$, considered as complexes of graded $\bk$-vector spaces (see above), are direct sums of objects of the form $\bk \langle n \rangle [-n]$ with $n \in \Z$. Hence the desired vanishing indeed holds.

From this surjectivity we deduce that $j^* \cF \cong \cE_U$ is indecomposable, hence that $\cE$ can be chosen to be indecomposable also. Next we fix isomorphisms $j^* \cF \simto \cE_U$ and $\cE_U \simto j^* \cF$. By surjectivity again, these isomorphisms can be lifted to morphisms $f : \cF \to \cE$ and $g : \cE \to \cF$. Since $g \circ f$ does not belong to the maximal ideal in $\End(\cE)$, it must be invertible. Similarly $f \circ g$ is invertible, and finally $\cF \cong \cE$.
\end{proof}

\begin{proof}[Proof of Theorem~{\rm \ref{thm:criterion-lusztig}}]
\eqref{it:criterion-1} Assume that $\cE_\lambda \cong \IC_\lambda$ for any $\lambda \in \bX$ such that $\Gr_\lambda \subset Y_1$.
Then, by~\cite[Lemma~3.7]{modrap2}, in $\Dmix_{(\Iw)}(\Gr,\bk)$, for any $\lambda$ such that $\Gr_\lambda \subset Y_1$, the simple mixed perverse sheaf $\IC^\mix_\lambda$ is simply $\mathcal{E}_\lambda$, considered as a complex concentrated in degree $0$. Hence, using~\cite[Remark~2.7]{modrap2} and Proposition~\ref{prop:ch-simples}, we see that
\[
[\irr(w_\lambda \bullet 0)] = \sum_{\substack{\mu \in \bX \\ w_\mu \leq w_\lambda}} (-1)^{\ell(w_\lambda)+\ell(w_\mu)} \left( \sum_{i \in \Z} \dim \mathsf{H}^i(\Gr_\mu, i_\mu^!(\mathcal{E}_\lambda)) \right) \cdot [\coweyl(w_\mu \bullet 0)],
\]
where $i_\mu : \Gr_\mu \to \Gr$ is the embedding. However, it is easy to see (using the defining property of Kazhdan--Lusztig elements) that if $\mathcal{E}_\lambda \cong \IC_\lambda$ then we have
\[
\sum_{i \in \Z} \dim \mathsf{H}^i(\Gr_\mu, i_\mu^!(\mathcal{E}_\lambda)) = h_{w_S w_\mu, w_S w_\lambda}(1),
\]
see~\cite[Implication $(3) \Rightarrow (4)$ in Proposition~3.11]{williamson-IC}. Hence Lusztig's formula holds for any 
$\lambda \in \bX$ such that
$w_\lambda \bullet 0$ is restricted. As noted above, this is known to imply Lusztig's conjecture via results of Kato~\cite{kato}.

\eqref{it:criterion-2} We assume now that Lusztig's conjecture holds. 
The theory of mixed derived categories developed in~\cite{modrap2} also applies to coefficients $\Z_\ell$ or $\Q_\ell$.
In particular we will consider the mixed derived categories $\Dmix_{(\Iw)}(Y_2,\bE)$, for any $\bE \in \{\Q_\ell, \Z_\ell, \bk\}$. (Here the subscript $(\Iw)$ means constructible with respect to the stratification by orbits of $\Iw$.)  In each case we have standard and costandard perverse sheaves $\cJ_!(\lambda, \bE)$ and $\cJ_*(\lambda, \bE)$, and an intermediate extension $\IC^\mix_\lambda(\bE)$. We also have ``extension of scalars" functor
\[
\bk : \Dmix_{(\Iw)}(Y_2, \Z_\ell) \to \Dmix_{(\Iw)}(Y_2, \bk), \qquad \Q_\ell : \Dmix_{(\Iw)}(Y_2, \Z_\ell) \to \Dmix_{(\Iw)}(Y_2, \Q_\ell),
\]
which satisfy
\begin{gather}
\label{eqn:kK-J!}
\bk(\cJ_!(\lambda, \Z_\ell)) \cong \cJ_!(\lambda, \bk), \quad \Q_\ell(\cJ_!(\lambda, \Z_\ell)) \cong \cJ_!(\lambda, \Q_\ell), \\
\label{eqn:kK-J*}
\bk(\cJ_*(\lambda, \Z_\ell)) \cong \cJ_*(\lambda, \bk), \quad \Q_\ell(\cJ_*(\lambda, \Z_\ell)) \cong \cJ_*(\lambda, \Q_\ell), \\
\label{eqn:Qp-IC}
\Q_\ell(\IC_\lambda^\mix(\Z_\ell)) \cong \IC^\mix_\lambda(\Q_\ell).
\end{gather}
(To be precise, in~\cite{modrap2} we only consider triples $(\bK, \bO, \mathbb{F})$ such that 
$\mathbb{F}$ is the residue field of $\bO$. But the same constructions apply for our present triple $(\Q_\ell, \Z_\ell, \bk)$.)

We will also consider the indecomposable parity complex $\cE_\lambda(\Q_\ell)$ and the ``ordinary" intersection cohomology complex $\IC_\lambda(\Q_\ell)$ with coefficients in $\Q_\ell$.
Note that, as in the proof of~\eqref{it:criterion-1}, $\IC^\mix_\lambda(\Q_\ell)$ is isomorphic to $\cE_\lambda(\Q_\ell) \cong \IC_\lambda(\Q_\ell)$, considered as a complex concentrated in degree $0$. In particular, the coefficients of $\IC^\mix_\lambda(\Q_\ell)$ in the basis of costandard perverse sheaves are given by $(-1)^{\ell(w_\lambda)+\ell(w_\mu)} h_{w_S w_\mu, w_S w_\lambda}(1)$.

First, we claim that
\begin{equation}
\label{eqn:k-IC}
\bk ( \IC^\mix_\lambda(\Z_\ell)) \cong \IC^\mix_\lambda(\bk)
\end{equation}
for any $\lambda \in \bX$ such that $\Gr_\lambda \subset Y_2$.
In fact, as in the proof of Proposition~\ref{prop:ch-simples}, it is not difficult to see that the classes of the objects $\cJ_*(\lambda, \bE) \langle i \rangle$ form a $\Z$-basis of the Grothendieck group $[\Dmix_{(\Iw)}(Y_2,\bE)]$, for any $\bE \in \{\Q_\ell, \Z_\ell, \bk\}$. In view of~\eqref{eqn:kK-J*}, this implies that the functors $\bk$ and $\Q_\ell$ induce canonical isomorphisms
\[
\xymatrix{
[\Dmix_{(\Iw)}(Y_2,\Q_\ell)] & [\Dmix_{(\Iw)}(Y_2,\Z_\ell)] \ar[l]_-{r_{\Q_\ell}}^-{\sim} \ar[r]^-{r_\bk}_-{\sim} & [\Dmix_{(\Iw)}(Y_2,\bk)],
}
\]
such that $r_\bk( r_{\Q_\ell}^{-1}([\IC_\lambda^\mix(\Q_\ell)])) = [\bk ( \IC^\mix_\lambda(\Z_\ell))]$. Now, if we consider these Gro\-thendieck groups as $\Z[v,v^{-1}]$-modules where $v$ acts via $\langle 1 \rangle$, Lusztig's conjecture and the existence of $\bQ$ imply that
\[
r_\bk( r_{\Q_\ell}^{-1}([\IC_\lambda^\mix(\Q_\ell)]))_{|v=1} = [\IC^\mix_\lambda(\bk)]_{|v=1},
\]
where $(-)_{|v=1}$ is the map to the quotient by the submodule generated by $v-1$. Hence the mixed perverse sheaf $\bk ( \IC^\mix_\lambda(\Z_\ell))$ has only one composition factor, which is a Tate twist of $\IC^\mix_\lambda(\bk)$. Considering the restrictions to $\Gr_\lambda$, we deduce~\eqref{eqn:k-IC}.

Next, we claim that
\begin{multline}
\label{eqn:stalks-costalhs-IC-free}
\Hom_{\Dmix_{(\Iw)}(Y_2,\Z_\ell)}(\cJ_!(\lambda,\Z_\ell), \IC^\mix_\mu(\Z_\ell) \langle j \rangle [i]) \text{ and} \\
\Hom_{\Dmix_{(\Iw)}(Y_2,\Z_\ell)}(\IC^\mix_\mu(\Z_\ell), \cJ_*(\lambda,\Z_\ell) \langle j \rangle [i])\text{ are $\Z_\ell$-free}
\end{multline}
for any $i,j \in \Z$ and any $\lambda, \mu \in \bX$ such that $\Gr_\lambda, \Gr_\mu \subset Y_2$. We treat the second case; the first one is similar. First we remark that we can replace $\Dmix_{(\Iw)}(Y_2,\Z_\ell)$ by $\Dmix_{(\Iw)}(Y_2,\Z_\ell)$ in these $\Hom$-spaces. Now,
assume that the finitely generated $\Z_\ell$-module $\Hom_{\Dmix_{(\Iw)}(\Gr,\Z_\ell)}(\IC^\mix_\mu(\Z_\ell), \cJ_*(\lambda,\Z_\ell) \langle j \rangle [i])$ has torsion. Using~\eqref{eqn:k-IC} and~\cite[Lemma~2.10]{modrap2}, we deduce that 
$\Hom_{\Dmix_{(\Iw)}(\Gr,\bk)}(\IC^\mix_\mu(\bk), \cJ_*(\lambda,\bk) \langle j \rangle [k])$ is nonzero for $k \in \{i,i+1\}$. Then applying $\bQ$ we obtain that the $\bk$-vector space $\Ext^k_{\Rep_\varnothing(G)}(\irr(w_\lambda \bullet 0), \coweyl(w_\lambda \bullet 0))$ is nonzero for $k \in \{i,i+1\}$. This contradicts a parity-vanishing result of Andersen~\cite{andersen} (obtained as a consequence of Lusztig's conjecture), see~\cite[Proposition~C.2(b)]{jantzen}.

From~\eqref{eqn:kK-J!}, \eqref{eqn:k-IC}, \eqref{eqn:stalks-costalhs-IC-free} and~\cite[Lemma~2.10]{modrap2}, we deduce that
\begin{multline*}
\dim_\bk \bigl( \Hom_{\Dmix_{(\Iw)}(Y_2,\bk)}(\cJ_!(\lambda,\bk), \IC^\mix_\mu(\bk) \langle j \rangle [i]) \bigr) \\
= \dim_{\Q_\ell} \bigl( \Hom_{\Dmix_{(\Iw)}(Y_2,\Q_\ell)}(\cJ_!(\lambda,\Q_\ell), \IC^\mix_\mu(\Q_\ell) \langle j \rangle [i]) \bigr)
\end{multline*}
and that
\begin{multline*}
\dim_\bk \bigl( \Hom_{\Dmix_{(\Iw)}(Y_2,\bk)}(\IC^\mix_\mu(\bk), \cJ_*(\lambda,\bk) \langle j \rangle [i]) \bigr) \\
= \dim_{\Q_\ell} \bigl( \Hom_{\Dmix_{(\Iw)}(Y_2,\Q_\ell)}(\IC^\mix_\mu(\Q_\ell), \cJ_*(\lambda,\Q_\ell) \langle j \rangle [i]) \bigr)
\end{multline*}
for any $i, j \in \Z$ and $\lambda, \mu \in \bX$ such that $\Gr_\lambda, \Gr_\mu \subset Y_2$. Since, in the case of $\Q_\ell$, we know that these spaces vanish unless $i+j=0$ (see above), we deduce the same property over $\bk$. Then, Lemma~\ref{lem:characterization-parity-Dmix} implies that $\IC^\mix_\lambda(\bk)$ is isomorphic to a parity complex considered as a complex concentrated in degree $0$, for any $\lambda \in \bX$ such that $\Gr_\lambda \subset Y_2$. By indecomposibility and considering supports, we even have $\IC^\mix_\lambda(\bk) \cong \cE_\lambda(\bk)$ for any such $\lambda$.

By the well-known characterization of simple objects in the recollement formalism, see~\cite[Corollaire~1.4.24]{bbd}, the fact that $\IC^\mix_\lambda(\bk) \cong \cE_\lambda(\bk)$ implies that $i_\mu^*(\cE_\lambda(\bk))$ is in perverse degrees $\leq -1$ and that $i_\mu^!(\cE_\lambda(\bk))$ is in perverse degrees $\geq 1$, for any $\mu \in \bX$ such that $\Gr_\mu \subset \overline{\Gr_\lambda} \smallsetminus \Gr_\lambda \subset Y_2$. On the other hand, these complexes are just the ordinary restriction and corestriction of $\cE_\lambda(\bk)$ to $\Gr_\mu$, considered as complexes concentrated in degree $0$; see~\cite[Remark~2.7]{modrap2}. Hence these conditions mean that $i_\mu^*(\cE_\lambda(\bk))$ belongs to $D^{\leq -\dim(\Gr_\mu)-1}(\Gr_\mu,\bk)$, and that $i_\mu^!(\cE_\lambda(\bk))$ belongs to $D^{\geq -\dim(\Gr_\mu)+1}(\Gr_\mu,\bk)$, where $i_\mu^*$ and $i_\mu^!$ now mean the \emph{ordinary} restriction and corestriction functors. Using the characterization of \emph{ordinary} intersection cohomology complexes given by~\cite[Corollaire~1.4.24]{bbd}, it follows that $\IC_\lambda(\bk) \cong \cE_\lambda(\bk)$.
\end{proof}

\begin{rmk}
Theorem~\ref{thm:criterion-lusztig}\eqref{it:criterion-2}, \cite[Lemma~3.7]{modrap2} and~\cite[Corollary~3.17]{modrap3} imply that if Lusztig's conjecture holds, then the category $\Perv^\mix_{(\Iw)}(Y_2,\bk)$ (which is part of a grading on the Serre subcategory of $\Rep_\varnothing(G)$ generated by the simple objects $\irr(w \bullet 0)$ with $w \in \Waffmin$ such that $\langle w \bullet 0 + \rho, \alpha^\vee \rangle \leq \ell(\ell-h+2)$ for any positive root $\alpha$) is a Koszul category.
\end{rmk}


\addtocontents{toc}{\protect\addvspace{1.5em}}

\newpage

\section*{Index of notation}

\noindent
\begin{center}
\begin{tabular}{|@{}p{0.32\textwidth\relax}|@{}p{0.32\textwidth\relax}|@{}p{0.32\textwidth\relax}@{}|}
\hline
\begin{tabular}{@{}p{2.6cm}@{}p{\dimexpr0.32\textwidth-2.6cm\relax}@{}}
\hspace{0.05cm} $\preceq$ & \S\ref{ss:Hom-calculations} \\
\hspace{0.05cm} $\le$, $\le'$ & \S\ref{ss:orders-X} \\
\hspace{0.05cm} $\sA \aq \sa$ & \S\ref{ss:normal-subalg} \\
\hspace{0.05cm} $\sA \rtimes \sD$ & \S\ref{ss:semi-direct-products} \\
\hspace{0.05cm} $\sA\ldgmod$ & \S\ref{ss:notation} \\
\hspace{0.05cm} $\sA\ldgmod_H$, \\
\quad $\sA\ldgmod^+_H$ & \S\ref{ss:equiv-dgmod} \\
\hspace{0.05cm} $\sA\lmod$ & \S\ref{ss:notation} \\
\hspace{0.05cm} $\sA\lmod_H$ & \S\ref{ss:equiv-dgmod} \\
\hspace{0.05cm} $B$, $\fb$, $B_1$, $\dot B$  & \S\ref{ss:notation-alg-groups} \\
\hspace{0.05cm} $B^+$, $\fb^+$  & \S\ref{ss:notation-alg-groups} \\
\hspace{0.05cm} $\conv(\lambda)$, \\
\quad $\convo(\lambda)$ & \S\ref{ss:orders-X} \\
\hspace{0.05cm} $D(\sA)$, $\Dfg(\sA)$ & \S\ref{ss:notation} \\
\hspace{0.05cm} $D_H(\sA)$, $D^+_H(\sA)$, \\
\quad $\Dfg_H(\sA)$   & \S\ref{ss:equiv-dgmod} \\
\hspace{0.05cm} $\Db_{\Stein}(P_{I,1})$ & \S\ref{ss:statement-equiv-formality} \\
\hspace{0.05cm} $\Db_{\Stein}(P_JM_{I,1})$ & \S\ref{sec:formality-non-equiv} \\
\hspace{0.05cm} $\Dmix_{(\Iw)}(\Gr, \bk)$ & \S\ref{ss:mixed-der} \\
\hspace{0.05cm} $d_I$ & \S\ref{ss:Springer-res} \\
\hspace{0.05cm} $\Dist(H)$ & \S\ref{ss:induction} \\
\hspace{0.05cm} $\dom(\lambda)$, \\
\quad $\dom_I(\lambda)$ & \S\ref{ss:orders-X} \\
\hspace{0.05cm} $\Delta(\lambda)$, $\Delta_I(\lambda)$ & \S\ref{ss:standard-costandard-exotic}, \S\ref{ss:reminder-exotic-empty} \\
\hspace{0.05cm} $\se_{J,I}$ & \S\ref{ss:Ind-Res-Springer} \\
\hspace{0.05cm} $f^*$, $f_*$ & \S\ref{ss:notation}, \S\ref{ss:equiv-dgmod} \\
\hspace{0.05cm} $\For^K_H$ & \S\ref{ss:induction}, \S\ref{ss:equiv-dgmod} \\
\hspace{0.05cm} $\Fr$ & \S\ref{ss:notation-alg-groups} \\
\hspace{0.05cm} $G$, $\fg$, $\dot G$ & \S\ref{ss:notation-alg-groups} \\
\hspace{0.05cm} $\Gr$, $\Gr_\lambda$ & \S\ref{ss:mixed-der} \\
\hspace{0.05cm} $\mathbb{I}^H$ & \S\ref{ss:induction}, \S\ref{ss:spectral-sequence-quotient} \\
\hspace{0.05cm} $\IC^\mix_\lambda$ & \S\ref{ss:mixed-der} \\
\end{tabular}
&
\begin{tabular}{@{}p{2.6cm}@{}p{\dimexpr0.32\textwidth-2.6cm\relax}@{}}
\hspace{0.05cm} $\incl_I$ & \S\ref{ss:setting} \\
\hspace{0.05cm} $\inc$ & \S\ref{ss:setting} \\
\hspace{0.05cm} $\Ind^K_H$ & \S\ref{ss:induction}, \S\ref{ss:equiv-dgmod} \\
\hspace{0.05cm} $\cJ_!(\lambda)$, $\cJ_*(\lambda)$ & \S\ref{ss:mixed-der} \\
\hspace{0.05cm} $\bk_H(\lambda)$ & \S\ref{ss:induction} \\
\hspace{0.05cm} $\irr_I(\lambda)$ & \S\ref{ss:ssrk1}, \S\ref{ss:setting} \\
\hspace{0.05cm} $\cL_{\dot G/\dot P_I}(V)$, \\
\quad $\cL_{\tcN_I}(V)$ & \S\ref{ss:Springer-res} \\
\hspace{0.05cm} $\bL_I$ & \S\ref{ss:RnI} \\
\hspace{0.05cm} $M_I$, $\fm_I$, \\
\quad $M_{I,1}$, $\smm_I$ & \S\ref{ss:notation-alg-groups} \\
\hspace{0.05cm} $\weyl_I(\lambda)$ & \S\ref{ss:ssrk1}, \S\ref{ss:setting} \\
\hspace{0.05cm} $\mu_{J,I}$ & \S\ref{ss:Ind-Res-Springer} \\
\hspace{0.05cm} $N$, $\fn$, $N_I$, $\fn_I$, \\
\quad $N_{I,1}$, $\snn_I$ & \S\ref{ss:notation-alg-groups}  \\
\hspace{0.05cm} $\dot N$, $\fnt$, $\dot N_I$, $\fnt_I$ & \S\ref{ss:notation-alg-groups} \\
\hspace{0.05cm} $\fN_I$ & \S\ref{ss:notation-alg-groups} \\
\hspace{0.05cm} $\coweyl_I(\lambda)$ & \S\ref{ss:ssrk1}, \S\ref{ss:setting} \\
\hspace{0.05cm} $\tcN$, $\tcN_I$ & \S\ref{ss:Springer-res} \\
\hspace{0.05cm} $\tcN_{J,I}$ & \S\ref{ss:Ind-Res-Springer} \\
\hspace{0.05cm} $n_I$ & \S\ref{ss:Springer-res} \\
\hspace{0.05cm} $\nabla(\lambda)$, $\nabla_I(\lambda)$ & \S\ref{ss:standard-costandard-exotic}, \S\ref{ss:reminder-exotic-empty} \\
\hspace{0.05cm} $\Omega_I$ & \S\ref{ss:setting} \\
\hspace{0.05cm} $P_I$, $\fp_I$, \\
\quad $P_{I,1}$, $\spp_I$ & \S\ref{ss:notation-alg-groups} \\
\hspace{0.05cm} $\Perv_\sph(\Gr,\bk)$ & \S\ref{ss:geometric-Satake} \\
\hspace{0.05cm} $\Phi$, $\Phi^+$, $\Phi_I^+$ & \S\ref{ss:notation-alg-groups} \\
\hspace{0.05cm} $\phi$ & \S\ref{ss:main-translation} \\
\hspace{0.05cm} $\varphi_I$ & \S\ref{ss:formality-PI1} \\
\hspace{0.05cm} $\Pi_I$, $\Pi_{J,I}$ & \S\ref{ss:Ind-Res-Springer} \\
\hspace{0.05cm} $\Pi^I$, $\Pi^{J,I}$ & \S\ref{ss:Ind-Res-Springer} \\
\hspace{0.05cm} $\pi_I$ & \S\ref{ss:RnI} \\
\end{tabular}
&
\begin{tabular}{@{}p{2.3cm}@{}p{\dimexpr0.32\textwidth-2.3cm\relax}@{}}
\hspace{0.05cm} $\pr_I$ & \S\ref{ss:setting} \\
\hspace{0.05cm} $\psi_I$, $\psi_{J,I}$ & \S\ref{ss:Theta}, \S\ref{ss:statement-equiv-formality} \\
\hspace{0.05cm} $\bR_I$ & \S\ref{ss:RnI} \\
\hspace{0.05cm} $\Rep(H)$, \\
\quad $\Repf(H)$ & \S\ref{ss:induction} \\
\hspace{0.05cm} $\Rep_I(G)$ & \S\ref{ss:setting} \\
\hspace{0.05cm} $\rho$, $\rho_I$ & \S\ref{ss:notation-alg-groups} \\
\hspace{0.05cm} $\Rn_I$ & \S\ref{ss:RnI} \\
\hspace{0.05cm} $S$ & \S\ref{ss:notation-alg-groups} \\
\hspace{0.05cm} $\Sat$ & \S\ref{ss:geometric-Satake} \\
\hspace{0.05cm} $\Sigma$ & \S\ref{ss:notation-alg-groups} \\
\hspace{0.05cm} $\sigma_I$ & \S\ref{ss:fromLtoRn} \\
\hspace{0.05cm} $\varsigma$, $\varsigma_I$ & \S\ref{ss:notation-alg-groups}, \S\ref{ss:setting} \\
\hspace{0.05cm} $\St$, $\St_I$ & \S\ref{ss:Steinberg} \\
\hspace{0.05cm} $T$, $\ft$, $\dot T$  & \S\ref{ss:notation-alg-groups} \\
\hspace{0.05cm} $T_I^J$, $T_J^I$ & \S\ref{ss:setting} \\
\hspace{0.05cm} $\tilt(\lambda)$, $\tilt_s(\lambda)$ & \S\ref{ss:ssrk1}, \S\ref{ss:graded-fm-conj} \\
\hspace{0.05cm} $\cT(\lambda)$ & \S\ref{ss:mixed-der} \\
\hspace{0.05cm} $t_\lambda$ & \S\ref{ss:setting} \\
\hspace{0.05cm} $\Theta_{J,I}$, $\Theta^{J,I}$ & \S\ref{ss:Theta} \\
\hspace{0.05cm} $\theta$ & \S\ref{ss:main-translation} \\
\hspace{0.05cm} $W$, $W_I$ & \S\ref{ss:notation-alg-groups} \\
\hspace{0.05cm} $W_\aff$ & \S\ref{ss:setting} \\
\hspace{0.05cm} $\WaffCox$, $\Waff^\circ$ & \S\ref{ss:setting} \\
\hspace{0.05cm} $\Waffmin$, $\WaffminI$ & \S\ref{ss:combinatorics-weights} \\
\hspace{0.05cm} $w_I$ & \S\ref{ss:notation-alg-groups} \\
\hspace{0.05cm} $w_\lambda$ & \S\ref{ss:orders-X} \\
\hspace{0.05cm} $\bX$, $\bX^+$ & \S\ref{ss:notation-alg-groups} \\
\hspace{0.05cm} $\bX_I^+$ & \S\ref{ss:setting} \\
\hspace{0.05cm} $\bXpp_I$ & \S\ref{ss:standard-costandard-exotic} \\
\hspace{0.05cm} $\fZ_I$ & \S\ref{ss:notation-alg-groups} \\
\hspace{0.05cm} $\rmZ_I$ & \S\ref{ss:Steinberg} \\
\ \\
\end{tabular} \\
\hline
\end{tabular}
\end{center}

\newpage


\end{document}